\documentclass[12pt]{amsart}
\usepackage{hyperref}
\hypersetup{
    colorlinks=true,
    linkcolor=red,
    filecolor=magenta,      
    urlcolor=cyan,
}
\usepackage{amssymb}
\usepackage{graphicx}
\usepackage{mathrsfs}
\usepackage[all]{xy}
\makeindex
\newcommand{\ad}{\mathrm{ad}}
\newcommand{\Ou}{\scriptsize\bullet}
\newcommand{\Bu}{\scriptsize\bullet,\scriptsize\bullet}
\newcommand{\coh}{\mathrm{coh}}
\newcommand{\cBC}{\ch_{\mathrm{BC}}}
\newcommand{\Db}{\mathrm{D^{b}_{\mathrm{coh}}}}

\newcommand{\LW}{\Lambda\left(W_{\C}^{*}\right)}
\newcommand{\As}{A^{E\prime \prime}}
\newcommand{\sXBC}{\square^{X}_{\mathrm{BC}}}

\newcommand{\cO}{\mathcal{O}}

\newcommand{\Hom}{\mathrm{Hom}}

\newcommand{\TsX}{T^{*}X}

\newcommand{\Aut}{\mathrm{Aut}}

\newcommand{\we}{\wedge}
\newcommand{\pa}{\partial}

\newcommand{\N}{\mathbf{N}}
\newcommand{ \Z}{\mathbf{Z}}
\newcommand{\R}{\mathbf{R}}
\newcommand{\C}{\mathbf{C}}

\newcommand{\n}{\nabla}

\newcommand{\fa}{f^{\alpha}}
\newcommand{\fb}{f^{\beta}}
\newcommand{\ho}{ \widehat { \otimes }}

\def \Hom{\mathrm {Hom}}
\def \Trs{\mathrm {Tr_{s}}}
\def \Tr{\mathrm{Tr}}
\def \Td{\mathrm {Td}}
\def \ch{\mathrm {ch}}
\def \End{\mathrm {End}}

\def \even{\mathrm{even}}

\def\deg{\mathrm{deg}}

\theoremstyle{plain}
\newtheorem{theorem}{Theorem}

\newtheorem{proposition}[theorem]{Proposition}
\newtheorem{cor}[theorem]{Corollary}
\theoremstyle{definition}
\newtheorem{definition}[theorem]{Definition}
\newtheorem{example}[theorem]{Example}

\theoremstyle{remark}
\newtheorem{remark}[theorem]{Remark}

\numberwithin{equation}{section}
\numberwithin{theorem}{section}
\counterwithin*{figure}{section}
\begin{document} 
 
 \bibliographystyle{alpha}
 
\title[Coherent sheaves and RRG]{Coherent 
sheaves, superconnections, and RRG}


\author{Jean-Michel \textsc{Bismut}}
\address{Institut de Math\'ematique d'Orsay \\ Universit\'e 
Paris-Saclay
\\ B\^atiment 307 \\ 91405 Orsay \\ France} 
\curraddr{}
\email{Jean-Michel.Bismut@universite-paris-saclay.fr}
\author{Shu \textsc{Shen}}
\address{Institut de Mathématiques de Jussieu-Paris Rive-Gauche\\
Sorbonne Université\\ Case Courrier 247\\4 place Jussieu\\75252 Paris 
Cedex 05\\ France}
\email{shu.shen@imj-prg.fr}
\author{Zhaoting \textsc{Wei}}
\address{Department of Mathematics\\
Texas A\&M University-Commerce\\Commerce\\ TX 75429
\\ USA}
\email{zhaoting.wei@tamuc.edu}
\thanks{}
\date{}


\dedicatory{}

\begin{abstract}
	Given a compact complex manifold, the purpose of this paper is to construct the Chern character for 
	coherent sheaves with values in  Bott-Chern cohomology, and to 
	prove a corresponding  
	 Riemann-Roch-Grothendieck formula. Our paper is based 
	on a fundamental construction of Block.
\end{abstract}
\maketitle
\tableofcontents
\section{Introduction}%
\label{sec:intro}
Let $X$ be a compact complex manifold. Let 
\index{KX@$K^{\cdot}\left(X\right)$}%
\index{KX@$K_{\cdot}\left(X\right)$}%
$K^{\cdot}\left(X\right),K_{\cdot}\left(X\right)$ be the 
Grothendieck groups of the locally free $\mathcal O_{X}$-sheaves, 
and of the $\mathcal O_{X}$-coherent sheaves on $X$,  
and let $H_{\mathrm{BC}}^{\scriptsize\bullet,\scriptsize\bullet}\left(X,\R\right)$  be the 
Bott-Chern cohomology of $X$. Set
\begin{equation}\label{eq:intro1}
H^{(=)}_{\mathrm{BC}}\left(X,\R\right)=\bigoplus_{p=0}^{\dim 
X}H_{\mathrm{BC}}^{p,p}\left(X,\R\right).
\end{equation}

There are   characteristic classes 
$K^{\cdot}\left(X\right)\to H^{(=)}_{\mathrm{BC}}\left(X,\R\right)$, like the Chern character 
$\ch_{\mathrm{BC}}$, or the Todd class 
$\Td_{\mathrm{BC}}$, that refine on the classical Chern classes with 
values in  the even de Rham cohomology $H_{\mathrm{dR}}^{\even}\left(X,\R\right)$.

In the present paper, we extend the definition of the   Chern 
character  $\ch_{\mathrm{BC}}:K^{\cdot}\left(X\right)\to 
H_{\mathrm{BC}}^{(=)}\left(X,\R\right)$ to a map
$K_{\cdot}\left(X\right)\to 
H_{\mathrm{BC}}^{(=)}\left(X,\R\right)$. In degree $(1,1)$, this map 
is just the first Chern class of the corresponding determinant line
bundle of Knudsen-Mumford \cite{KnudsenMumford}.  Also we prove a 
Riemann-Roch-Grothendieck  
formula. Namely, if $Y$ is another compact complex manifold, let $f:X\to Y$ 
be a holomorphic map.  Let $f_{!}:K_{\cdot}\left(X\right)\to 
K_{\cdot}\left(Y\right)$ be the direct image, 
and let $f_{*}:H_{\mathrm{BC}}^{(=)}\left(X,\R\right)\to 
H^{(=)}_{\mathrm{BC}}\left(Y,\R\right)$ denote the   
push-forward map. 
\begin{theorem}\label{thm:rrg}
	If $\mathscr F\in K_{\cdot}\left(X\right)$, then 
	\begin{equation}\label{eq:intro1a1}
\Td_{\mathrm{BC}}\left(TY\right)\ch_{\mathrm{BC}}\left(f_{!} \mathscr 
F\right)=f_{*}\left[\Td_{\mathrm{BC}}\left(TX\right)\ch_{\mathrm{BC}}\left(\mathscr F\right)\right]
\ \mathrm{in}\  H_{\mathrm{BC}}^{(=)}\left(Y,\R\right).
\end{equation}
\end{theorem}

When $X$ is projective, then 
$K_{\cdot}\left(X\right)=K^{\cdot}\left(X\right)$, and our definition 
of $\ch_{\mathrm{BC}}$ is the standard one. Similarly, when $X,Y$ are 
projective, Theorem \ref{thm:rrg} is a consequence of 
Riemann-Roch-Grothendieck.

We will describe the ideas and techniques that are used to 
obtain Theorem \ref{thm:rrg}.
\subsection{The construction of $\cBC$}%
\label{subsec:cocBC}
If $E$ is a holomorphic vector bundle on $X$, and if $g^{E}$ is a Hermitian metric on 
$E$, there is an associated  unitary Chern connection, whose curvature is of 
type $(1,1)$. Using Chern-Weil theory,  one can construct 
a corresponding Chern character form,
$\ch\left(E,g^{E}\right)$, which is closed and lies in 
$\Omega^{(=)}\left(X,\R\right)$. Results of Bott and Chern 
\cite{BottChern65} show that the class of $\ch\left(E,g^{E}\right)$ 
in $H^{(=)}_{\mathrm{BC}}\left(X,\R\right)$ does not depend on 
$g^{E}$, which gives us a Chern character map
$\cBC:K^{\cdot}\left(X\right)\to H^{(=)}_{\mathrm{BC}}\left(X,\R\right)$.

If $X$ is projective,  if $E\in K_{\cdot}\left(X\right)$, it has a finite 
locally free projective 
resolution  $R$. It can be shown 
that $\cBC\left(R\right)$ does not depend on $R$, so that we get a 
Chern character map $\cBC:K_{\cdot}\left(X\right)\to 
H^{(=)}_{\mathrm{BC}}\left(X,\R\right)$. As shown by Voisin 
\cite{Voisin02}, if $X$ is not projective, 
such  projective resolutions may well not exist\footnote{In particular, 
for a generic torus of dimension $\ge 3$, the ideal sheaf of a point 
does not admit a finite locally free resolution.}.

To sidestep this problem,   we use a fundamental construction of Block 
\cite{Block10}. To explain this construction, we  replace 
$K_{\cdot}\left(X\right)$ by $\Db\left(X\right)$, the derived category of 
bounded $\mathcal{O}_{X}$-complexes with coherent cohomology. In 
\cite{Block10}, Block established an equivalence of categories 
between $\Db\left(X\right)$ and another category 
$\underline{\mathrm{B}}\left(X\right)$ of objects which extend holomorphic 
vector bundles. More precisely, an  object in 
$\underline{\mathrm{B}}\left(X\right)$ is represented by $\mathscr E:\left(E, 
A^{E \prime \prime }\right)$, where $E$ a $\Z$-graded vector bundle which 
is a free $\Lambda\left(\overline{\TsX}\right)$-module, and $A^{E 
\prime \prime }$ is an antiholomorphic superconnection on $E$ in the sense 
of Quillen \cite{Quillen85a} such that $A^{E \prime \prime,2}=0$.

Following earlier   work by Qiang \cite{Qiang16,Qiang17}, as in the case of holomorphic vector bundles, one can define 
generalized metrics on $E$, and use these to  
define a Chern character $\cBC:\underline{\mathrm{B}}\left(X\right)\to 
H^{(=)}_{\mathrm{BC}}\left(X,\R\right)$, from which we  get a map 
$\cBC:\Db\left(X\right)\to H^{(=)}_{\mathrm{BC}}\left(X,\R\right)$, and finally 
an associated map $\cBC:K_{\cdot}\left(X\right)\to 
H^{(=)}_{\mathrm{BC}}\left(X,\R\right)$. Also we show that in degree 
$(1,1)$, our Chern character coincides with the first Chern class of 
the Knudsen-Mumford determinant  bundle \cite{KnudsenMumford}. These results are established by showing 
that the analytic constructions are compatible with the  underlying 
functorial properties of $\Db\left(X\right)$.  
\subsection{The Riemann-Roch theorem}%
\label{subsec:rrint}
Let us now describe the main steps in our proof of Theorem 
\ref{thm:rrg}. As in Borel-Serre \cite[Section 7]{BorelSerre58}, the proof is reduced to the case where $f$ 
is an embedding or a projection. Indeed, let $i:H \subset X\times Y$ be the graph of $f$, and 
let $p$ be the projection $X\times Y\to Y$. Then $f=pi$. Also $X\sim 
H$, so that $i$ is an embedding $X\to X\times Y$.  By functoriality, 
we will reduce the proof to the case where $f=i $ or $f=p$.
\subsubsection{The case of embeddings}%
\label{subsubsec:emb}
In the case of embeddings, the proof is obtained by a suitable 
adaptation of the deformation to the normal cone. The proof relies 
mostly on the functorial properties of the above objects $\mathscr E$.
\subsubsection{The case of projections}%
\label{subsubsec:pro}
In the case of a projection $p:M=S\times X\to S$,  we reduce the proof to the 
case where $\mathscr F$ is represented by an   object  
$\mathscr  E$ in $\underline{\mathrm{B}}\left(M\right)$. By a theorem of 
Grauert \cite[Theorem 10.4.6]{GrauertRemmert84}, $p_{*}\mathscr E$ is known to define an object 
in $\Db\left(S\right)$, but it is not an object in 
$\underline{\mathrm{B}}\left(S\right)$, because it is infinite-dimensional. 

Then, we adapt the methods of  \cite{Bismut10b}, 
where the case of proper holomorphic submersions was considered, when  $\mathscr E$ is a 
 holomorphic vector bundle, and the direct image 
$Rp_{*}\mathscr E$ is   locally free, in which case $\cBC\left(Rp_{*}\mathscr E\right)$ 
can be defined without using the theory of Block \cite{Block10}. The proof was obtained using the 
infinite-dimensional elliptic superconnections on the infinite-dimensional 
$p_{*} \mathscr E$ of 
\cite{Bismut86d,BismutGilletSoule88b,BismutGilletSoule88c},  
their hypoelliptic deformations\footnote{This refers to  properties 
 of the curvatures of the superconnections, that are fiberwise
elliptic differential operators, or fiberwise hypoelliptic operators.}, as well 
as local index theory\footnote{Local index theory is a method 
which computes the local asymptotics of the supertraces of certain 
heat kernels as $t\to 0$ in terms of characteristic forms.}.

 Given metric data 
which involve in particular  a Hermitian metric on $TX$ with Kähler form 
$\omega^{X}$, we define infinite-dimensional Chern character 
forms,   and we prove that these forms represent 
$\cBC\left(Rp_{*}\mathscr E\right)$. For $t>0$, when replacing 
$\omega^{X}$ by $\omega^{X}/t$, if 
$\overline{\pa}^{X}\pa^{X}\omega^{X}=0$, 
we take the limit of our elliptic superconnection forms, and we prove 
Theorem \ref{thm:rrg}. When there is no such a metric, we inspire 
ourselves from the methods of \cite{Bismut10b}, where 
superconnections with hypoelliptic curvature are defined,  which 
ultimately allow us to prove Theorem \ref{thm:rrg}. In the 
constructions  and in the proofs, the fact we deal with 
antiholomorphic superconnections instead of holomorphic vector 
bundles introduces extra 
difficulties with respect to \cite{Bismut10b}. 
\subsubsection{The spectral truncations}
In Subsections \ref{subsec:trunc} and \ref{subsec:sptrin}, we develop 
a technique of spectral truncations, which is used when 
dealing with the infinite-dimensional $p_{*} \mathscr E$. This method 
extends classical spectral truncations to families, while preserving critical 
cohomological information. 
\subsection{Earlier work}
Let us describe  earlier related work.  
In \cite{OBrianToledoTong81a}, O'Brian, Toledo and Tong defined a 
Chern character on $K_{\cdot}\left(X\right)$ with values in the 
Hodge cohomology $H^{\Ou}\left(X,\Omega^{\Ou}_{X}\right)$. In 
\cite{OBrianToledoTong85}, they proved a corresponding 
Riemann-Roch-Grothendieck formula. Since Bott-Chern cohomology maps 
into Hodge cohomology, the above results follow from Theorem 
\ref{thm:rrg}. 

In \cite{RLevy87}, Levy 
proved  a topological version of our result, which is  valid for 
 general complex spaces. Let 
$K^{0}_{\mathrm{top}}\left(X\right)$ be the Grothendieck group generated 
by smooth vector bundles on $X$. Levy 
constructed a morphism  $K_{\cdot}\left(X\right)\to 
K^{0}_{\mathrm{top}}\left(X\right)$ that extends the obvious morphism 
$K^{\cdot}\left(X\right)\to K^{0}_{\mathrm{top}}\left(X\right)$, and he 
proved the compatibility of this morphism to direct images. Since 
Bott-Chern cohomology maps into de Rham cohomology, Levy's results 
are now consequences of Theorem \ref{thm:rrg}.

In \cite{Grivaux10}, Grivaux defined a Chern character on 
$K_{\cdot}\left(X\right)$ with values in the rational Deligne 
cohomology ring $A\left(X\right)$,  and he proved
Riemann-Roch-Grothendieck for projective morphisms.  He also proved a 
uniqueness result for characteristic classes on 
$K_{\cdot}\left(X\right)$ with values in a cohomology  theory that 
verifies a 
simple set of axioms. Recently, 
following earlier work by Green \cite{Green80},  Toledo-Tong \cite{ToledoTong86}, Hosgood 
\cite{Hosgood20a,Hosgood20b} constructed resolutions of coherent sheaves by simplicially locally free sheaves, 
and obtained Chern classes  of coherent sheaves in de Rham 
cohomology (and in twisted de Rham cohomology), which he showed 
to be compatible with the construction of Grivaux 
\cite{Grivaux10}. 

In \cite{Wu20}, Wu has proved a Riemann-Roch-Grothendieck 
theorem for coherent sheaves and projective morphisms. His Chern 
character takes values in rational Bott-Chern cohomology.  He proves 
a Riemann-Roch-Grothendieck for projective morphisms, while our 
results are unconditional.

As explained before, the constructions and results of Block \cite{Block10} play a 
fundamental role in our paper, as an inspiration, and also for 
fundamental ideas. Also  in \cite{Qiang16,Qiang17}, Qiang had 
constructed a Chern character $\cBC:K_{\cdot}\left(X\right)\to 
H^{(=)}_{\mathrm{BC}}\left(X,\R\right)$  by 
essentially the same methods as ours.  More 
detailed comparison with his results will be given in Section  
\ref{sec:chfoge}.

Moreover,  we show that our Chern character 
$\cBC:K_{\cdot}\left(X\right)\to H_{\mathrm{BC}}^{(=)}\left(X,\R\right)$ verifies 
the uniqueness conditions of Grivaux \cite{Grivaux10}, and so, it is 
compatible with the previous constructions of the Chern 
character. 
\subsection{The organization of the paper}%
\label{subsec:org}
This paper is organized as follows. In Section  \ref{sec:boch}, we 
recall various properties of Bott-Chern cohomology, and we recall the 
construction of characteristic classes with values in Bott-Chern 
cohomology.

In Section \ref{sec:derblo}, we state various properties of the 
derived category $\Db\left(X\right)$, and of the associated functors.

In Section \ref{sec:lia}, we recall elementary properties of  
 free modules over an exterior algebra.

In Section \ref{sec:ahscn}, we introduce antiholomorphic 
superconnections $\mathscr E=\left(E,A^{E \prime \prime }\right)$, we 
show that the corresponding cohomology sheaf $\mathscr H \mathscr E$ 
is coherent, and that the associated spectral sequence degenerates at 
$\mathscr E_{2}$. In particular, we prove the result of Block 
\cite{Block10} asserting that locally,  $\mathscr E$ is just a  complex of holomorphic 
vector bundles.

In Section \ref{sec:blo}, we establish the result of Block 
\cite{Block10} on the equivalence of categories 
 $\underline{\mathrm{B}}\left(X\right) \simeq \Db\left(X\right)$.
 
 In Section \ref{sec:ascngm}, given a splitting $E \simeq E_{0}$, 
 with 
 \index{ho@$\ho$}%
 $E_{0}=\Lambda\left(\overline{\TsX}\right)\ho D$\footnote{In 
 the paper, $D$ is called the diagonal bundle associated with $E$. 
 If $\mathcal{A},\mathcal{A'}$ are $\Z_{2}$-graded algebras, 
 $\mathcal{A}\ho\mathcal{A'}$ is the $\Z_{2}$-graded algebra which is 
 the tensor product of $\mathcal{A},\mathcal{A'}$.}, we define the generalized metrics on $E_{0}$, and the adjoint 
 $A^{E_{0} \prime }$ of the antiholomorphic superconnection $A^{E_{0} 
 \prime \prime }$. We also introduce the curvature $\left[A^{E_{0} 
 \prime \prime }, A^{E_{0} \prime }\right]$. This chapter is the 
 extension of what is done for holomorphic Hermitian vector bundles, 
 when constructing the corresponding Chern connection.
 
 In Section \ref{sec:chfoge}, given a generalized metric $h$, we 
 construct the Chern character form $\ch\left(A^{E_{0}\prime \prime} ,h\right)\in \Omega^{(=)}\left(X,\R\right)$, whose Bott-Chern 
 cohomology class does not depend on $h$, or on  the splitting and 
 will be denoted $\cBC\left(A^{E \prime \prime }\right)$. We show 
 that $\cBC$ induces a map $\Db\left(X\right)\to 
 H^{(=)}_{\mathrm{BC}}\left(X,\R\right)$, and ultimately a map 
 $K_{\cdot}\left(X\right)\to H^{(=)}_{\mathrm{BC}}\left(X,\R\right)$. 
 This Section gives another approach to earlier results by Qiang 
 \cite{Qiang16,Qiang17}.
 
 In Section \ref{sec:imm}, we establish Theorem \ref{thm:rrg} when 
 $f$ is an embedding $i:X\to Y$. Using the deformation to the normal 
 cone, we reduce the proof to the embedding $X\to 
 \mathbf{P}\left(N_{X/Y}\oplus \C\right)$. Also we prove that our 
 Chern character $\cBC:K_{\cdot}\left(X\right)\to 
 H^{(=)}_{\mathrm{BC}}\left(X,\R\right)$ verifies the uniqueness conditions of 
 Grivaux \cite{Grivaux10}.
 
 In Section \ref{sec:suel}, when $f$ is the projection $p:M=S\times 
 X\to S$, and $\mathscr E$ is an antiholomorphic superconnection on 
 $M$, we construct the infinite-dimensional antiholomorphic 
 superconnection $\left( p_{*} \mathscr E, A^{p_{*} \mathscr E \prime 
 \prime } \right) $ over $S$. Given a splitting $E \simeq 
 E_{0}$, a metric $g^{TX}$ on $TX$ 
 with Kähler form $\omega^{X}$,  and a Hermitian metric $g^{D}$ on 
 $D$,  we construct the $L_{2}$-adjoint $A^{p_{*} \mathscr 
 E_{0}\prime}$ of $A^{p_{*} \mathscr E_{0} \prime \prime }$, 
and we 
 give a Lichnerowicz formula for the curvature $\left[A^{p_{*} 
 \mathscr E_{0} \prime \prime },A^{p_{*} \mathscr E_{0}\prime 
 }\right]$.  Also we construct the elliptic Chern character forms 
 $\ch\left(A^{p_{*} \mathscr E_{0} \prime \prime }, \omega^{X},g^{D}\right)$,  and 
 their common Bott-Chern class $\cBC\left(A^{p_{*} \mathscr E \prime \prime 
 }\right)\in H^{(=)}_{\mathrm{BC}}\left(S,\R\right)$.
 
 In Section \ref{sec:lift}, we prove that $\cBC\left(A^{p_{*} \mathscr E \prime \prime 
 }\right)=\cBC\left(Rp_{*} \mathscr E\right)$. The proof is 
 based on the construction of a quasi-isomorphism from a classical  
 antiholomorphic superconnection $\underline{\mathscr E}$ to $p_{*} 
 \mathscr E$. Spectral truncations play a critical role in the proofs.
 
 In Section \ref{sec:prka}, we prove Theorem \ref{thm:rrg} for the projection 
 $p$ when $\overline{\pa}^{X}\pa^{X}\omega^{X}=0$. The proof is 
 obtained by taking the limit as $t\to 0$ of the forms 
 $\ch\left(A^{p_{*} \mathscr E_{0} \prime \prime}, \omega^{X}/t,g^{D}\right)$ by 
 methods of local index theory. 
 
 In Section \ref{sec:hypo}, we construct superconnections with 
 hypoelliptic curvature. More precisely, if $\widehat{TX}$ is another 
 copy of $TX$, we introduce the total space $\mathcal{M}$ of 
 $\widehat{T_{\R}X}$, and also the tautological Koszul complex 
 associated with the embedding $M\to \mathcal{M}$.  Given a splitting 
 $E \simeq E_{0}$, and metrics 
 $g^{TX},g^{\widehat{TX}},g^{D}$ on $TX,\widehat{TX},D$, we construct 
 a nonpositive Hermitian form $\epsilon_{X}$. The adjoint of our new 
 antiholomorphic superconnection $\mathcal{A}^{\prime \prime }_{Y}$ 
 over $S$
 is constructed using $\epsilon_{X}$. The corresponding 
 curvature is a fiberwise hypoelliptic operator. Our constructions 
 combine the methods of \cite{Bismut10b} with the 
 results of the previous Sections.
 
 In Section \ref{sec:hypofo}, the hypoelliptic superconnection  forms 
 $\ch\left(\mathcal{A}''_{Y},\omega^{X},g^{D},g^{\widehat{TX}}\right)$ are constructed, and  their Bott-Chern class is shown to be the same as the class of the elliptic superconnection forms.
 
 In Section \ref{sec:hypodb}, when 
 $\overline{\pa}^{X}\pa^{X}\omega^{X}=0$, we give another proof of 
Theorem \ref{thm:rrg} for $p$. The results of this Section duplicate 
the results of Section \ref{sec:prka}, with techniques that prepare 
for a proof of the Theorem \ref{thm:rrg} in full generality.

Finally, in Section \ref{sec:RBC}, we introduce a new deformation of 
our hypoelliptic superconnection, in which the 
Kähler form is  deformed to an object that is still a 
$(1,1)$-form depending explicitly on the extra coordinate 
$Y$, and we prove Theorem \ref{thm:rrg} for $p$, which completes the 
proof of this Theorem in full generality.
 
In the whole paper, if $A=A_{+} \oplus A_{-}$ is a $\Z_{2}$-graded 
algebra, if $a,b\in A$, we denote by 
\index{ab@$\left[a,b\right]$}%
$\left[a,b\right]$ the 
supercommutator of $a,b$, which is bilinear in $a,b$. More 
precisely, if $a\in A$, set 
\begin{align}\label{eq:dega1}
\deg\, a=&0\,\mathrm{if}\,a\in A_{+}, \\
&1\,\mathrm{if}\, a\in A_{-}. \nonumber 
\end{align}
If $a,b\in A_{\pm}$, then
\begin{equation}\label{eq:dega2}
\left[a,b\right]=ab-\left(-1\right)^{\deg a\deg b}ba.
\end{equation}
If $A$ is trivially $\Z_{2}$-graded, then $\left[a,b\right]$ is the 
standard commutator.
\section{Bott-Chern cohomology  and characteristic classes}%
\label{sec:boch}
The purpose of this Section is to recall  basic facts on 
Bott-Chern cohomology of a compact complex manifold $X$.

This Section is organized as follows. In Subsection \ref{subsec:boc}, 
we recall the definition of the Bott-Chern 
cohomology $H^{\Bu}_{\mathrm{BC}}\left(X,\C\right)$.

In Subsection \ref{subsec:kospl}, we give various properties of the 
Bott-Chern
Laplacian of Kodaira-Spencer, which imply that 
$H_{\mathrm{BC}}^{\Bu}\left(X,\C\right)$ is 
finite-dimensional, that it can be represented by smooth forms or by 
currents, and also that it is a bigraded algebra.

In Subsection \ref{subsec:unpro}, we describe various functors
that act on Bott-Chern cohomology.

Finally, in Subsection \ref{subsec:bovb}, we explain the construction 
of characteristic classes of holomorphic vector bundles with values 
in Bott-Chern cohomology.
\subsection{The Bott-Chern cohomology}%
\label{subsec:boc}
Let $X$ be a compact complex manifold of  dimension $n$.  Let $TX$ be 
the holomorphic tangent bundle, and let $T_{\R}X$ be the 
corresponding real tangent bundle. If $T_{\C}X=T_{\R}X \otimes 
_{\R}\C$, then $T_{\C}X=TX \oplus \overline{TX}$.

Let 
$\left(\Omega\left(X,\R\right),d\right)$ be the de Rham 
complex of smooth real differential forms on $X$. The  de Rham cohomology groups 
\index{HXR@$H^{\Ou}\left(X,\R\right)$}%
$H\left(X,\R\right)$ are defined by
\begin{equation}\label{eq:flug2}
H^{\Ou}\left(X,\R\right)=\ker 
d\cap\Omega^{\Ou}\left(X,\R\right)/d\Omega^{\scriptsize\bullet-1}\left(X,\R\right).
\end{equation}
If $\alpha\in \Omega\left(X,\R\right)$ is closed, let
\index{a@$\left[\alpha\right]$}%
$\left[\alpha\right]\in H\left(X,\R\right)$ be its cohomology 
class. We may as well replace $\R$ by $\C$, and obtain 
the complexification of the above objects.

We follow   \cite[Section 6.8]{Demaillylivre} and 
\cite[Section 1]{Schweitzer07}. For $0\le p,q\le n$, 
let $\Omega^{p,q}\left(X,\C\right)$ be the vector space of smooth 
sections of $\Lambda^{p,q}\left(T^{*}_{\C}X 
\right)=\Lambda^{p}\left(T^{*}X\right)\ho\Lambda^{q}\left(\overline{T^{*}X}\right)$. Note that $d$ splits as $d=\overline{\pa}+\pa$.
\begin{definition}\label{BottChe}
The Bott-Chern cohomology groups
\index{HBC@$H^{p,q}_{\mathrm{BC}}\left(X,\C\right)$}%
$H^{p,q}_{\mathrm{BC}}\left(X,\C\right)$ are given by
\begin{equation}\label{eq:BoCh1}
H^{p,q}_{\mathrm{BC}}\left(X,\C\right)= \left( 
\Omega^{p,q}\left(X,\C\right)\cap \ker d\right) 
/\overline{\pa}\pa\Omega^{p-1,q-1}\left(X,\C\right).
\end{equation}
Then $H^{\Bu}_{\mathrm{BC}}\left(X,\C\right)$ is a bigraded 
vector space. 
If $\alpha\in 
\Omega^{p,q}\left(X,\C\right)$ is closed, let
\index{a@$\left\{\alpha\right\}$}%
$\left\{\alpha\right\}$ be the class of $\alpha$ in 
$H^{p,q}_{\mathrm{BC}}\left(X,\C\right)$.
There is a  canonical map 
$H^{p,q}_{\mathrm{BC}}\left(X,\C\right)\to 
H^{p+q}\left(X,\C\right)$ that maps 
$\left\{\alpha\right\}$ to $\left[\alpha\right]$.
\end{definition}
Then $H^{\Bu}_{\mathrm{BC}}\left(X,\C\right)$ inherits from 
$\Omega^{\Bu}\left(X,\C\right)$ the 
structure of a bigraded algebra. 

Put
\index{OXC@$\Omega^{(=)}\left(X,\C\right)$}%
\index{HXC@$H^{(=)}_{\mathrm{BC}}\left(X,\C\right)$}%
\begin{align}\label{eq:BoCh0}
    &\Omega^{(=)}\left(X,\C\right)=\bigoplus_{0\le p\le 
n}\Omega^{p,p}\left(X,\C\right),
&H^{(=)}_{\mathrm{BC}}\left(X,\C\right)=\bigoplus_{0\le p\le 
n}H^{p,p}_{\mathrm{BC}}\left(X,\C\right).
\end{align}
The vector spaces  
$\Omega^{(=)}\left(X,\C\right),H_{\mathrm{BC}}^{(=)}\left(X,\C\right)$ are preserved by 
conjugation. We obtain this way the  real vector space
\index{OXR@$\Omega^{(=)}\left(X,\R\right)$}%
$\Omega^{(=)}\left(X,\R\right)$ and the corresponding real Bott-Chern 
cohomology  
\index{HXR@$H^{(=)}_{\mathrm{BC}}\left(X,\R\right)$}%
$H^{(=)}_{\mathrm{BC}}\left(X,\R\right)$. Also 
$H^{(=)}_{\mathrm{BC}}\left(X,\R\right)$ is  an algebra.
\subsection{The  Bott-Chern Laplacian of Kodaira-Spencer}%
\label{subsec:kospl}
Let $g^{TX}$ be a Hermitian  metric on $TX$. We equip 
$\Omega\left(X,\C\right)$ with the associated $L_{2}$ Hermitian 
product, so that the $\Omega^{\Bu}\left(X,\C\right)$ are mutually 
orthogonal. If $A$ is a differential operator acting on 
$\Omega\left(X,\C\right)$, we denote by $A^{*}$ its formal 
adjoint. If $A,B$ are two operators, let 
\index{AB@$\left[A,B\right]_{+}$}%
$\left[A,B\right]_{+}=AB+BA$ be their anticommutator.

We introduce the Kodaira-Spencer Laplacian \cite[Proposition 
5]{KodairaSpencer60}, \cite[Section 2b]{Schweitzer07},
\begin{equation}\label{eq:or1}
\square^{X}_{\mathrm{BC}}=\left[\overline{\pa}\pa,\left(\overline{\pa}\pa\right)^{*}\right]_{+}+
\left[\overline{\pa}^{*}\pa,\left(\overline{\pa}^{*}\pa\right)^{*}\right]_{+}+\overline{\pa}^{*}\overline{\pa}+\pa^{*}\pa.
\end{equation}
Then $\square^{X}_{\mathrm{BC}}$ is a formally self-adjoint 
nonnegative real 
operator   acting on $\Omega\left(X,\R\right)$ and 
preserving the bidegree. This operator is elliptic of order $4$, with 
scalar principal symbol $\left\vert  \xi\right\vert^{4}_{T_{\R}^{*}X}/4$.

Let
\index{DXR@$\mathscr D\left(X,\R\right)$}%
$\mathscr D\left(X,\R\right)$ be the vector space of real currents 
on $X$, and let 
\index{DXC@$\mathscr D\left(X,\C\right)$}%
$\mathscr D\left(X,\C\right)$ be its complexification.  Then $\Omega\left(X,\R\right) \subset \mathscr 
D\left(X,\R\right)$.  
If $u\in \mathscr D\left(X,\C\right)$, if 
$\square^{X}_{\mathrm{BC}}u=0$, since $\square^{X}_{\mathrm{BC}}$ is 
elliptic, then $u\in\Omega\left(X,\C\right)$.

Set\begin{equation}\label{eq:or2}
\mathcal{H}\left(X\right)=\ker 
\square^{X}_{\mathrm{BC}}.\end{equation}
Then $\mathcal{H}\left(X\right)$ is also bigraded, and moreover,
\begin{equation}\label{eq:or3}
\mathcal{H}\left(X\right)=\ker 
\overline{\pa}\cap\ker \pa\cap\ker\overline{\pa}^{*}\pa^{*}.
\end{equation}
By \cite[Theorem 2.2]{Schweitzer07}, we have the orthogonal splitting
\begin{equation}\label{eq:or3a1}
\Omega\left(X,\C\right)=\mathcal{H}\left(X\right) \oplus \mathrm{im} 
\overline{\pa}\pa \oplus \left( 
\mathrm{im}\pa^{*}+\mathrm{im}\overline{\pa}^{*} \right),
\end{equation}
and we have the canonical isomorphism of bigraded vector spaces,
\begin{equation}\label{eq:or4}
H_{\mathrm{BC}}\left(X,\C\right) \simeq 
\mathcal{H}\left(X\right).
\end{equation}
In particular, the Bott-Chern groups 
$H^{\Bu}_{\mathrm{BC}}\left(X,\C\right)$ are finite 
dimensional.

By the above, if we define Bott-Chern cohomology 
using instead $\mathscr D\left(X,\C\right)$, it is 
the same as before, and it is still represented canonically by the 
smooth forms in 
$\mathcal{H}\left(X\right)$. While, in general,  
the multiplication of currents is  not well-defined, the 
multiplication of the Bott-Chern classes of currents is still 
well-defined via the multiplication of corresponding smooth 
representatives. 

If $s\in \ker \overline{\pa}\cap\ker \pa$, then
\begin{equation}\label{eq:or5}
\square^{X}_{\mathrm{BC}}s=\overline{\pa}\pa 
\left(\overline{\pa}\pa\right)^{*}s,
\end{equation}
so   that if $s\in \Omega\left(X,\C\right)$, then
\begin{equation}\label{eq:or6}
\square^{X}_{\mathrm{BC}}\overline{\pa}\pa 
s=\overline{\pa}\pa\left(\overline{\pa}\pa\right)^{*}\overline{\pa}\pa s.
\end{equation}
In particular $\square^{X}_{\mathrm{BC}}$ preserves $\ker \overline{\pa}\cap\ker \pa$ 
and $\mathrm{im}\overline{\pa}\pa$, while mapping the first vector 
space into the second one.  By (\ref{eq:or3}), we have the orthogonal 
splitting,
\begin{equation}\label{eq:or6a1}
\ker \overline{\pa}\cap\ker\pa=\mathcal{H}\left(X\right) \oplus 
\mathrm{im}\overline{\pa}\pa,
\end{equation}
$\square^{X}_{\mathrm{BC}}$ preserves the splitting in 
(\ref{eq:or6a1}), and $\mathcal{H}\left(X\right)$ is its kernel.

For $t>0$, let $P_{t}\left(x,x'\right)$ denote the smooth kernel 
associated with the operator 
$\exp\left(-t\square^{X}_{\mathrm{BC}}\right)$ with respect to the 
Riemannian volume $dx'$. The operator 
$\exp\left(-t\square^{X}_{\mathrm{BC}}\right)$ preserves $\ker 
\overline{\pa}\cap\ker \pa$. 

For $t>0$, if $s\in \mathscr D\left(X,\C\right)$, then
\begin{equation}\label{eq:or7}
\exp\left(-t\square^{X}_{\mathrm{BC}}\right)s-s=-\int_{0}^{t}\square^{X}_{\mathrm{BC}}\exp\left(-u\sXBC\right)sdu.
\end{equation}
For $t>0$, if  $s\in \mathscr D\left(X,\C\right)$, then 
$\exp\left(-t\sXBC\right)s\in 
\Omega\left(X,\C\right)$.
If $s\in \ker\overline{\pa}\cap\ker\pa$, by (\ref{eq:or5}), 
(\ref{eq:or7}), we deduce that
\begin{equation}\label{eq:or8}
\exp\left(-t\square^{X}_{\mathrm{BC}}\right)s-s=-\overline{\pa}\pa\int_{0}^{t}\left(\overline{\pa}\pa\right)^{*}\exp\left(-u\sXBC\right)sdu.
\end{equation}
Therefore if $s\in \mathscr D\left(X,\C\right)$ lies in 
$\ker\overline{\pa}\cap\ker\pa$, and if $\left\{s\right\}\in 
H_{\mathrm{BC}}\left(X,\C\right)$ is the 
corresponding Bott-Chern class, the form 
$\exp\left(-t\sXBC\right)s\in 
\Omega\left(X,\C\right)$ is a smooth representative 
of $\left\{s\right\}$. 
Also, as  $t\to 0$, $\exp\left(-t\sXBC\right)s$ converges to $s$ in 
$\mathscr D\left(X,\C\right)$.
\begin{proposition}\label{prop:plim}
	If for $n\in \N$, $s_{n}\in \mathscr  D\left(X,\C\right)$ lies in $\ker 
	\overline{\pa}\cap \ker \pa$,  if $s\in \mathscr D\left(X,\C\right)$ 
	and $s_{n}\to s$ in $\mathscr D\left(X,\C\right)$, then $s\in \ker 
	\overline{\pa}\cap \ker \pa$, and 
	$\left\{s_{n}\right\}\to\left\{s\right\}$ in 
	$H_{\mathrm{BC}}\left(X,\C\right)$.
\end{proposition}
\begin{proof}
It is obvious that $s\in\ker 
	\overline{\pa}\cap \ker \pa$. For $t>0$, 
$\exp\left(-t\square^{X}_{\mathrm{BC}}\right)s_{n}$  converges 
to $\exp\left(-t\square^{X}_{\mathrm{BC}}\right)s$ in 
$\Omega\left(X,\C\right)$. Also the operator 
$\exp\left(-t\square^{X}_{\mathrm{BC}}\right)$ does not change the 
Bott-Chern class. This completes the proof of our proposition.
\end{proof}
\subsection{Functorial properties of Bott-Chern cohomology}%
\label{subsec:unpro}
Let $Y$ be a compact complex manifold of dimension $n'$, and let $f:X\to Y$ be a 
holomorphic map. Then $f^{*}$ maps $\Omega\left(Y,\C\right)$ into 
$\Omega\left(X,\C\right)$ as a morphism of bigraded algebras. Therefore 
$f^{*}$ induces a morphism of bigraded algebras $H_{\mathrm{BC}}\left(Y,\C\right)\to 
H_{\mathrm{BC}}\left(X,\C\right)$. Also $f^{*}$ preserves the 
corresponding real vector spaces.

By duality, $f_{*}$ maps $\mathscr D^{\Bu}\left(X,\C\right)$ into 
$\mathscr 
D^{n'-n+\scriptsize\bullet,n'-n+\scriptsize\bullet}\left(Y,\C\right)$, and it preserves the associated real vector spaces. Since the Bott-Chern cohomology can 
be defined  using currents, we  get a morphism of bigraded vector 
spaces  $f_{*}:H_{\mathrm{BC}}^{\Bu}\left(X,\C\right)\to 
H_{\mathrm{BC}}^{n'-n+\scriptsize\bullet,
n'-n+\scriptsize\bullet}\left(Y,\C\right)$. This makes 
sense, in spite of the fact that, in general, 
$f^{*}$ does not act on $\mathscr D\left(Y,\C\right)$, and $f_{*}$ 
does not map $\Omega\left(X,\C\right)$ into 
$\Omega\left(Y,\C\right)$\footnote{Still, if $f$ is a submersion, 
both assertions are true.}.  

Let $\mathcal{Y}$ be 
the total space of $T^{*}_{\R}Y$, so that $Y$ is viewed as the zero 
section of $\mathcal{Y}$.  Let $\mathscr 
D_{f}\left(Y,\C\right)$ be the set of currents $s$ on $Y$ such that  its 
wave-front set $\mathrm{WF}\left(s\right)\in \mathcal{Y}\setminus Y$ 
does not intersect $N_{f}=f_{*}T_{\R}X^{\perp}$\footnote{This is the set of 
the $\left(f\left(x\right),\eta\right)\vert_{x\in X, \eta\in 
T^{*}_{\R,f\left(x\right)}Y}$ such that $f^{*}\left(x\right)\eta=0$.}. By Hörmander 
\cite[Theorem 8.2.4]{Hormander83a},  if $s\in 
\mathscr D_{f}\left(Y,\C\right)$ is a current on $Y$, the current $f^{*}s$ is 
well-defined. If $s_{n}\vert_{n\in\N}\in \Omega\left(Y,\C\right)$ is a microlocal 
approximation of $s$ in $\mathscr 
D_{f}\left(Y,\C\right)$\footnote{This means that there is there is a 
closed conic subset $\Gamma \subset \mathcal{Y}$ such that  
$\Gamma\cap N_{f}=\emptyset$,
$\mathrm{WF}\left(s\right)\subset 
\Gamma,\mathrm{WF}\left(s_{n}\right) \subset \Gamma$, and 
$s_{n}\vert_{n\in\N}$ 
converges to $s$ in $\mathscr D_{\Gamma}\left(Y\right)$ in the sense 
of \cite[Definition 8.2.2]{Hormander83a}. By \cite[Theorem 
8.2.3]{Hormander83a}, given $s$ as before, such a sequence 
$s_{n}\vert_{n\in\N}$ 
always exists.} , then 
$f^{*}s_{n}$ converges to $f^{*}s$ in $\mathscr D\left(X,\C\right)$. If 
$s\in \mathscr D_{f}^{\Bu}\left(Y,\C\right)$ is closed, the above shows that $f^{*}s\in \mathscr 
D^{\Bu}\left(X,\C\right)$ is closed, and that 
$\left\{f^{*}s\right\}=f^{*}\left\{s\right\}\,\mathrm{in}\, 
H_{\mathrm{BC}}^{\Bu}(X,\C)$. 

In \cite[Theorem 8.2.10]{Hormander83a}, the above argument is used to 
show that if $\mathcal{X}$ is the total space of $T^{*}_{\R}X$,  if 
$s,s'$ are two currents in $\mathscr D\left(X,\C\right)$ such that there 
are no $\left(x,\xi\right)\in \mathcal{X}$ such that 
$\left(x,\xi\right)\in \mathrm{WF}\left(s\right), \left(x,-\xi\right)\in 
\mathrm{WF}\left(s'\right)$, the product of currents $ss'$ is 
well-defined and verifies suitable continuity properties. In 
particular, if $s,s'\in \mathscr D^{\Bu}\left(X,\C\right)$ are 
closed, $ss'$ is still closed and moreover,
\begin{equation}\label{eq:pro1}
\left\{ss'\right\}=\left\{s\right\}\left\{s'\right\}\,\mathrm{in}\, 
H_{\mathrm{BC}}^{\Bu}\left(X,\C\right).
\end{equation}
\subsection{Bott-Chern characteristic classes of holomorphic vector 
bundles}%
\label{subsec:bovb}
Here we follow \cite[Section (e)]{BismutGilletSoule88a}. Let $\left(F,\n^{F \prime 
\prime }\right)$ be a holomorphic 
vector bundle of rank $m$ on $X$. Let $g^{F}$ be a Hermitian 
metric on $F$,  let $\n^{F}=\n^{F \prime \prime}+\n^{F \prime}$ denote the corresponding Chern 
connection on $F$. Then $\n^{F \prime  \prime },\n^{F \prime }$ act 
as odd operators on $\Omega\left(X,F\right)$. Let $R^{F}$ be its curvature. Then 
$R^{F}$ is a $(1,1)$-form on $X$ with values in 
skew-adjoint sections of $\End\left(F\right)$. Then\footnote{Recall that $\left[\,\right]$ is our notation for the supercommutator.}
\begin{align}\label{eq:su1}
&\n^{F \prime \prime,2}=0,&\n^{F \prime,2}=0,\qquad
&R^{F}=\left[\n^{F \prime \prime },\n^{F \prime }\right].
\end{align}
By (\ref{eq:su1}), we obtain the Bianchi identities,
\begin{align}\label{eq:su2}
&\left[\n^{F \prime \prime },R^{F}\right]=0,
&\left[\n^{F \prime },R^{F}\right]=0.
\end{align}

Let $P\left(B\right)$ be an invariant polynomial on 
$\mathfrak{gl}\left(m,\C\right)$ under the adjoint action of 
$\mathrm{GL}\left(m,\C\right)$. 
The polynomial $P$ is said to be real if 
\begin{equation}\label{eq:cw0}
\overline{P\left(B\right)}=P\left(\overline{B}\right).
\end{equation}

Put
\index{PF@$P\left(F,g^{F}\right)$}%
\begin{equation}\label{eq:cw1}
P\left(F,g^{F}\right)=P\left(-R^{F}/2i\pi\right).
\end{equation}
Then the form $P\left(F,g^{F}\right)$ lies in 
$\Omega^{(=)}\left(X,\C\right)$.
By Chern-Weil theory, the form $P\left(F,g^{F}\right)$ 
is closed, and  its cohomology class does not depend on $g^{F}$. 
This characteristic class will be denoted
\index{PF@$P\left(F\right)$}%
$P\left(F\right)\in 
H^{\even}\left(X,\C\right)$. If $P$ is a real polynomial, then 
$P\left(F,g^{F}\right)\in \Omega^{(=)}\left(X,\R\right)$, and 
$P\left(F\right)\in H^{\even}\left(X,\R\right)$.

 Let $\mathscr M^{F}$ be the space of 
metrics on $F$, and let $d^{\mathscr M^{F}}$ be the de Rham operator 
on $\mathscr M^{F}$. Then $\theta=\left(g^{F}\right)^{-1}d^{\mathscr M^{F}}g^{F}$ 
is a $1$-form on $\mathscr M^{F}$ with values in self-adjoint 
sections of $\End\left(F\right)$ with respect to the given $g^{F}$. 

Note that $d^{\mathscr M^{F}}P\left(F,g^{F}\right)$ is a $1$-form on 
$\mathscr M^{F}$ with values in $\Omega^{(=)}\left(X,\C\right)$. Let 
$P'$ be the derivative of $P$. Put
\begin{equation}\label{eq:cw4}
\gamma_{P}\left(F,g^{F}\right)=\left\langle  P'\left(-\frac{R^{F}}{2i\pi}\right),\theta\right\rangle.
\end{equation}
Then $\gamma_{P}\left(F,g^{F}\right)$ a $1$-form on $\mathscr M^{F}$ with values in 
$\Omega^{(=)}\left(X,\C\right)$. If $P$ is real, this form is also 
real.

Let us recall a result of 
Bott-Chern \cite[Proposition 3.28]{BottChern65}, \cite[Theorem 
1.27]{BismutGilletSoule88a}. 
\begin{proposition}\label{prop:pbc}
	The following identity holds:
	\begin{equation}\label{eq:cw2}
d^{\mathscr 
M^{F}}P\left(F,g^{F}\right)=-\frac{\overline{\pa}\pa}{2i\pi}\gamma_{P}\left(F,g^{F}\right).
\end{equation}
In particular the Bott-Chern class 
$\left\{P\left(F,g^{F}\right)\right\}\in 
H_{\mathrm{BC}}^{(=)}\left(X,\C\right)$ does not depend on $g^{F}$. 
If $P$ is real, this class lies in 
$H^{(=)}_{\mathrm{BC}}\left(X,\R\right)$.
\end{proposition}
\begin{proof}
Note that $\n^{F \prime }$ depends on $g^{F}$, and we have the 
	identity
	\begin{equation}\label{eq:cw7}
d^{\mathscr M^{F}}\n^{F \prime}=-\left[\n^{F \prime },\theta\right].
\end{equation}
Since $R^{F}=\left[\n^{F \prime \prime }, \n^{F \prime }\right]$, we 
get
\begin{equation}\label{eq:cw7a1}
d^{\mathscr M^{F}}R^{F}=-\left[\n^{F \prime \prime },d^{\mathscr 
M^{F}}\n^{F \prime }\right].
\end{equation}
Also
\begin{equation}\label{eq:cw7a2}
d^{\mathscr M^{F}}P\left(-R^{F}\right)=-\left\langle  
P'\left(-R^{F}\right),d^{\mathscr M^{F}}R^{F}\right\rangle.
\end{equation}
By (\ref{eq:cw7})--(\ref{eq:cw7a2}), we obtain
\begin{equation}\label{eq:cw7a3}
d^{\mathscr M^{F}}P\left(-R^{F}\right)=-\left\langle  
P'\left(-R^{F}\right),\left[\n^{F \prime \prime },\left[\n^{F \prime 
},\theta\right]\right]\right\rangle.
\end{equation}
Using the Bianchi identities and (\ref{eq:cw7a3}), we get 
(\ref{eq:cw2}). If $P$ is real, it is well known the form 
$P\left(F,g^{F}\right)$ is real.

Let us now give another argument taken from \cite[Theorem 
1.29]{BismutGilletSoule88a} that shows that 
$\left\{P\left(F,g^{F}\right)\right\}$ does not depend on $g^{F}$. On 
$\mathbf{P}^{1}$, we consider the canonical meromorphic coordinate 
$z$ and the Poincaré-Lelong equation
\begin{equation}\label{eq:bc5}
\frac{\overline{\pa}^{\mathbf{P}^{1}}\pa^{\mathbf{P}^{1}}}{2i\pi}\log\left(\left\vert  z\right\vert^{2}\right)=\delta_{0}-\delta_{\infty }.
\end{equation}
Let $p_{1},p_{2}$ be the projections from $X\times \mathbf{P}^{1}$ on 
$X,\mathbf{P}^{1}$. Let $g^{F'}$ be another Hermitian metric. Let 
$g^{p_{1}^{*}F}$ be a Hermitian metric  on 
$p_{1}^{*}F$ that restricts to $g^{F}, g^{F \prime }$ on $X\times 
\left\{0\right\},X\times \left\{ \infty\right\} $. By (\ref{eq:bc5}), we get
\begin{equation}\label{eq:bc5a1r0}
\frac{\overline{\pa}^{X}\pa^{X}}{2i\pi}p_{1*}\left[p_{2}^{*}\log\left(\left\vert  z\right\vert^{2}\right)P\left(p_{1}^{*}F,g^{p_{1}^{*}F}\right)\right]=P\left(F,g^{F}\right)-P\left(F,g^{F \prime }\right),
\end{equation}
which shows that $\left\{P\left(F,g^{F}\right)\right\}$ does not 
depend on $g^{F}$. The proof of our proposition is completed. 
\end{proof}

 \begin{definition}\label{def:Dboch}
 	Let
\index{PBC@$P_{\mathrm{BC}}\left(F\right)$}%
$P_{\mathrm{BC}}\left(F\right)\in H^{(=)}_{\mathrm{BC}}\left(X,\C\right)$ 
be the Bott-Chern class of $P\left(F,g^{F}\right)$.
 \end{definition}
 
If $P$ is real, then $P_{\mathrm{BC}}\left(F\right)\in 
H^{(=)}_{\mathrm{BC}}\left(X,\R\right)$.

Recall that if $B\in \mathfrak{gl}\left(m,\C\right)$, then
\index{Td@$\Td\left(B\right)$}%
\index{A@$\widehat{A}\left(B\right)$}%
\index{ch@$\ch\left(B\right)$}%
\index{c@$c_{1}\left(B\right)$}%
\begin{align}\label{eq:cw6}
&\Td\left(B\right)=\det\left[\frac{B}{1-e^{-B}}\right], 
&\widehat{A}\left(B\right)=\det\left[\frac{B/2}{\sinh\left(B/2\right)}\right],\\
&\ch\left(B\right)=\Tr\left[\exp\left(B\right)\right], &c_{1}\left(B\right)=\Tr\left[B\right]. \nonumber 
\end{align}
The associated  characteristic classes are called the Todd class, the 
$\widehat{A}$ class,   
the Chern character, and the first Chern class. Also
\begin{equation}\label{eq:crisp1}
\Td\left(B\right)=\widehat{A}\left(B\right)e^{c_{1}\left(B\right)/2}.
\end{equation}

In Section \ref{sec:prka}, we will also use the real form of the 
class $\widehat{A}$. More precisely, given $k\in \N$, if $\mathfrak 
{so}(k)$ is the Lie algebra of $\mathrm{SO}\left(k\right)$, if 
$B\in \mathfrak{so}(k)$, set
\begin{equation}\label{eq:crisp2}
\widehat{A}\left(B\right)=\det\left[\frac{B/2}{\sinh\left(B/2\right)}\right]^{1/2}.
\end{equation}
To $\widehat{A}\left(B\right)$, one can associate corresponding 
Pontryagin forms for real Euclidean vector bundles equipped with a 
metric connection.

In the present paper, analogues of equation (\ref{eq:cw2}) will reappear in other contexts, 
that include vector bundles of infinite dimension.\footnote{The 
theory of \textit{Bott-Chern classes} developed in 
\cite{BismutGilletSoule88a} is a secondary version of Bott-Chern 
cohomology. It is an analogue of Chern-Simons theory for holomorphic 
vector bundles. Here, this secondary theory will not play any role.}
\section{The derived category $\Db\left(X\right)$}%
\label{sec:derblo}
The purpose of this Section is to recall basic properties of the 
derived category $\Db\left(X\right)$  of 
bounded complexes of $\mathcal{O}_{X}$-modules with coherent cohomology.

This Section is organized as follows. In Subsection 
\ref{subsec:catdb}, we give the  definition of 
$\Db\left(X\right)$ and of the associated $K$-group $K\left(X\right)$.

In Subsection  \ref{subsec:pubaa}, if $Y$ is a compact complex 
manifold, and if $f:X\to Y$ is a holomorphic map, we define the  
derived functor $Lf^{*}:\Db\left(Y\right)\to\Db\left(X\right)$.

In Subsection \ref{subsec:ten}, we consider the  derived 
tensor products in $\Db\left(X\right)$.

Finally, in Subsection \ref{subsec:dira}, we recall the properties of the 
 derived function $Rf_{*}:\Db\left(X\right)\to 
\Db\left(Y\right)$.

\subsection{Definition of the derived category $\Db\left(X\right)$}%
\label{subsec:catdb}
Let $X$ be a compact complex manifold. Let 
\index{coh@$\coh\left(X\right)$}%
$\coh\left(X\right)$ be 
the abelian category of $\mathcal{O}_{X}$-coherent sheaves on $X$, 
and let $K\left(\mathrm{coh}(X)\right)$ denote the corresponding 
Grothendieck group. In the sequel, we use the notation\footnote{In the introduction, we 
used the notation $K_{\cdot}\left(X\right)$.}
\index{KX@$K\left(X\right)$}%
\begin{equation}\label{eq:deri2}
K\left(X\right)=K\left(\coh\left(X\right)\right).
\end{equation}

Let 
\index{CbX@$\mathrm{C^{b}}\left(X\right)$}%
$\mathrm{C^{b}}\left(X\right)$ be the 
category of bounded complexes of  $\mathcal O_{X}$-modules.  Let 
\index{DbX@$\mathrm{D^{b}\left(X\right)}$}%
$\mathrm{D^{b}\left(X\right)}$ be the  corresponding derived category. 

Let 
\index{Cbc@$\mathrm{C}_{\coh}^{\mathrm{b}}\left(X\right)$}%
$\mathrm{C}_{\coh}^{\mathrm{b}}\left(X\right)$ be the full 
subcategory of ${\rm C^{b}}(X)$  whose objects have coherent cohomology.  
Let 
\index{Dbc@$\Db\left(X\right)$}%
$\Db\left(X\right)$ be the corresponding derived category. Then ${\rm 
D^{b}_{coh}}(X)$ is the full subcategory of  ${\rm 
D^{b}}(X)$ whose objects have coherent cohomology.

The same arguments as in  	\cite[\href{https://stacks.math.columbia.edu/tag/0FCP}{Tag
0FCP}]{stacks-project} show that the map 
	\begin{equation}\label{eq:hipi1}
		\mathcal{E}\in \Db\left(X\right)\to 
		\sum_{i}(-1)^{i}\mathcal{H}^{i}\mathcal{E}\in K\left(X\right)
	\end{equation} 
induces an isomorphism of $K$-groups,
	\begin{equation}\label{eq:hipi2}
K({\rm 
D_{\mathrm{coh}}^{b}}(X)) \simeq K\left(X\right).
\end{equation} 
	\subsection{Pull-backs}
\label{subsec:pubaa}%
Let $Y$ be another compact complex manifold, and let $f:X\to Y$ be a 
holomorphic map.

By \cite[\href{https://stacks.math.columbia.edu/tag/0095}{Tag 
0095}]{stacks-project}, if $\mathcal{E}\in \mathrm{C}^{\mathrm{b}}_{\mathrm{coh}}\left(Y\right)$, one can define 
its pull-back
 $f^{*}\mathcal{E}$, which is a complex of $\mathcal{O}_{X}$-modules, by the formula
\begin{equation}\label{eq:pu1x1}
		f^{*}  \mathcal{E}= f^{-1}\mathcal E\otimes_{f^{-1}\mathcal 
		O_{Y}} \mathcal 
		O_{X}. 
	\end{equation} 
	By Grauert-Remmert \cite[Section 1.2.6]{GrauertRemmert84}, 
	$f^{*}\mathcal{E}\in 
	\mathrm{C}^{\mathrm{b}}_{\mathrm{coh}}\left(X\right)$. We can define the 
left-derived functor 
\index{Lf@$Lf^{*}$}%
$Lf^{*}$,  which to $\mathcal{E}\in 
\Db\left(Y\right)$ associates $Lf^{*}\mathcal{E}\in 
\Db\left(X\right)$.

If $\mathcal{E}$ is a bounded complex of flat $\cO_{Y}$-modules, by 
\cite[\href{https://stacks.math.columbia.edu/tag/06YJ}{Tag 06YJ}]{stacks-project}, we have
\begin{equation}\label{eqLf}
		Lf^{*}\mathcal{E}=f^{*}\mathcal{E}. 
	\end{equation} 
	Also  $Lf^{*}$ induces a morphism on Grothendieck groups, 
	\index{f@$f^{!}$}%
	\begin{equation}
		f^{!}: K(Y)\to K(X).
	\end{equation} 

If $Z$ is another compact complex manifold, if $g:Y\to Z$ is a holomorphic map, using 
\cite[Proposition I.9.15]{Borel87}, there is a canonical isomorphism 
between the functors $Lg^{*}Lf^{*}$ and $L\left(fg\right)^{*}$. In 
particular, we  get an identity of morphisms 
of Grothendieck groups,  
\begin{equation}
		f^{!}g^{!}=(gf)^{!}: K(Z)\to K(X).
	\end{equation}
	\subsection{Tensor products}%
\label{subsec:ten}
Let $\mathcal{E},\mathcal{F}$ be objects in $\Db\left(X\right)$. In  
\cite[\href{https://stacks.math.columbia.edu/tag/064M}{Tag 
064M}]{stacks-project}, a  
derived tensor product $\mathcal{E} \ho 
^{\mathrm{L}}_{\mathcal{O}_{X}}\mathcal{F}$\footnote{We have replaced 
the usual notation $\otimes ^{\mathrm{L}}$ by $\ho^{\mathrm{L}}$, to 
underline the fact that the objects in $\Db\left(X\right)$ are 
$\Z$-graded.}, also an object in 
$\Db\left(X\right)$ is defined. By 
\cite[\href{https://stacks.math.columbia.edu/tag/079U}{Tag 
079U}]{stacks-project}, if $Y$ is a compact complex manifold, and if 
$f:X\to Y$ is holomorphic, if $\mathcal{E},\mathcal{F}$ are objects 
in $\Db\left(Y\right)$, then
\begin{equation}\label{eq:tenx1}
Lf^{*}\left(\mathcal{E} 
\ho_{\mathcal{O}_{Y}}^{\mathrm{L}}\mathcal{F}\right)=Lf^{*}\mathcal{E} \ho _{\mathcal{O}_{X}}^{\mathrm{L}}Lf^{*}\mathcal{F}.
\end{equation}

Let $i:X\to X\times X$ be the diagonal embedding, and let 
$p_{1},p_{2}:X\times X\to X$ be the two projections. Since $p_{1}i$, 
$p_{2}i$ are the identity in $X$, using the results of Subsection 
\ref{subsec:pubaa} and equation (\ref{eq:tenx1}), we find that there is a canonical isomorphism,
\begin{equation}\label{eq:tenx1z1}
\mathcal{E} \ho 
^{\mathrm{L}}_{\mathcal{O}_{X}}\mathcal{F} \simeq Li^{*}\left(Lp_{1}^{*} \mathcal{E}\ho _{\mathcal{O}_{X\times 
X}}^{\mathrm{L}}Lp_{2}^{*}\mathcal{F}\right).
\end{equation}

By \cite[\href{https://stacks.math.columbia.edu/tag/06XY}{Tag 
06XY}]{stacks-project}, if $\mathcal{E}, \mathcal{F}$ are objects in $\Db\left(X\right)$, if 
one of them consists of flat modules over $\mathcal{O}_{X}$, then
\begin{equation}\label{eq:tenx2}
\mathcal{E}\ho^{\mathrm{L}}_{\mathcal{O}_{X}} 
 \mathcal{F}=\mathcal{E}\ho_{\mathcal{O}_{X}}\mathcal{F}.
\end{equation}
\subsection{Direct images}%
\label{subsec:dira}	
	Let 
	\index{RF@$Rf_{*}$}%
	$Rf_{*}$ be the right derived functor of the direct image 
	$f_{*}$. By  a theorem of
Grauert \cite[Theorem 10.4.6]{GrauertRemmert84}, if $\mathcal{E}$ is an object in ${\rm D_{\mathrm{coh}}^{b}}(X)$, 
$Rf_{*}\mathcal{E}$ is  an object in  ${\rm 
D_{\mathrm{coh}}^{b}}(Y)$. 

By \cite[Theorem 
9.5]{Borel87} and by \cite[Proposition IV.4.14]{Demaillylivre}, if $\mathcal{E}$ is a  bounded complex 
of soft $\cO_X$-modules, then 
	\begin{equation}\label{eq:fon1}
		Rf_{*}\mathcal{E}=f_{*}\mathcal{E}. 
	\end{equation} 
Also $Rf_{*}$ induces a morphism of Grothendieck groups,
\index{f@$f_{!}$}%
	\begin{equation}\label{eq:fon2}
		f_{!}: K(X)\to K(Y). 
	\end{equation}

If  $Z$ is a   compact  complex	manifold, and if $g: Y\to 
Z$ is a holomorphic  map, by 
\cite[Proposition I.9.15]{Borel87}, there is a canonical 
isomorphism between  the
functors $Rg_{*}Rf_{*}$ and $R(gf)_{*}$.
In particular, we have an 
identity of morphisms of Grothendieck groups, 
	\begin{equation}
		g_{!}	f_{!}=(gf)_{!}: K(X)\to K(Z). 
	\end{equation}
\section{Preliminaries on linear algebra and differential geometry}%
\label{sec:lia}
The purpose of this Section is to explain elementary facts of linear 
algebra and  differential geometry that will be used in the rest of the paper.

This Section is organized as follows. In Subsection 
\ref{subsec:filal}, if $W$ is a complex vector space, we introduce 
the category $\mathscr C_{W}$ of free 
$\Lambda\left(\overline{W}^{*}\right)$-modules.

In Subsection \ref{subsec:scmtr}, we define the 
supertraces, and give its main properties.

In Subsection \ref{subsec:moco}, we recall the construction of cones 
in homological algebra.

In Subsection \ref{subsec:dua}, we construct generalized Hermitian 
metrics on objects in $\mathscr C_{W}$.

In Subsection \ref{subsec:cli}, we recall elementary properties of the Clifford 
algebras of complex Hermitian vector spaces.

Finally, in Subsection \ref{subsec:remdi}, we explain various 
relations  of the de Rham operator to connections with nonzero torsion on the 
tangent bundle of a real manifold.
\subsection{Filtered vector space and exterior algebras}%
\label{subsec:filal}
Let $W$ be a finite-dimensional complex vector space of dimension 
$m$, and let $W_{\R}$ be the corresponding real vector space. If
 $W_{\C}=W_{\R} \otimes _{\R}\C$ is its complexification, 
 then $W_{\C}=W \oplus \overline{W}$. 
 
If $r,r'\in\Z,r\le r'$, let
$E=\bigoplus_{i=r}^{r'}E^{i}$ be a $\Z$-graded vector space  on which 
$\Lambda^{p}\left(\overline{W}^{*}\right)$ acts on the left as linear morphisms 
of degree $p$ (i.e., they increase the degree by $p$). We can also 
define a right action of $\Lambda\left(\overline{W}^{*}\right)$ on 
$E$,  so that if $e\in 
E,\alpha\in\Lambda\left(\overline{W}^{*}\right)$, then
\begin{equation}\label{eq:til2}
e\alpha=\left(-1\right)^{\deg e\deg \alpha}\alpha e.
\end{equation}

For $0\le p\le m$, put
\begin{equation}\label{eq:bc6r1}
F^{p}E=\Lambda^{p}\left(\overline{W}^{*}\right)E.
\end{equation}
The $F^{p}E$ form a decreasing filtration
\begin{equation}\label{eq:bc7r1}
F^{0}E=E\supset F^{1}E\supset\ldots\supset F^{m+1}E=0.
\end{equation}
Put
\begin{equation}\label{eq:bc8r1}
E_{0}^{p,q}=F^{p}E^{p+q}/F^{p+1}E^{p+q}.
\end{equation}
For $0\le p'\le m$,  $\Lambda^{p'}\left(\overline{W}^{*}\right)$ maps 
$E_{0}^{p,q}$ into 
$E_{0}^{p+p',q}$ surjectively.

For $r\le q\le r'$, set
\begin{equation}\label{eq:bc13a}
D^{q}=E_{0}^{0,q}=E^{q}/F^{1}E^{q}.
\end{equation}
Then $\Lambda^{p}(\overline{W}^{*})$ maps $D^{q}$ into  
$E_{0}^{p,q}$ surjectively.

In the sequel, we assume that we have the identification of left
$\Lambda\left(\overline{W}^{*}\right)$-modules
\begin{equation}\label{eq:bc12}
E_{0}^{\Bu}=\Lambda^{\Ou}\left(\overline{W}^{*}\right) 
\ho D^{\Ou}.
\end{equation}
We will call 
 \index{D@$D$}%
$D$ the diagonal vector space 
associated with $E$. There is a canonical degree preserving morphism 
$E\to D$, and we have the exact sequence
\begin{equation}\label{eq:exab1}
0\to F^{1}E\to E\to D\to 0.
\end{equation}
Note that a basis of $D$ lifts in $E$ to a basis of $E$ as a  free
$\Lambda\left(\overline{W}^{*}\right)$-module.

For $r\le i\le r'$, put
\index{E0@$E_{0}$}%
\begin{align}\label{eq:bc12a0}
&E_{0}^{i}=\bigoplus_{p+q=i}E_{0}^{p,q},
&E_{0}=\bigoplus_{i=r}^{r'}E^{i}_{0}. 
\end{align}
Then $E_{0}$ is a $\Z$-graded 
filtered vector 
space, which has the same properties as $E$, with the same $D$.

Let $\underline{E}$ be another $\Z$-graded vector space like $E$.  We underline the 
objects associated with $\underline{E}$. If $k\in \Z$, if $\phi$ is a linear morphism 
from $E$ into $\underline{E}$ that maps $E^{\Ou}$ into 
$\underline{E}^{\Ou+k}$, we will say that $\phi$ is of degree  $k$.  Let 
$\mathrm{Hom}\left(E,\underline{E}\right)$ be the set of 
 linear maps from $E$ into 
$\underline{E}$  such that if
$\alpha\in\Lambda\left(\overline{W}^{*}\right)$, then
\begin{equation}\label{eq:com1}
\phi\alpha-\left(-1\right)^{\deg\alpha\deg\phi}\alpha\phi=0,
\end{equation}
If $\alpha\in \Lambda\left(\overline{W}^{*}\right),e\in E$, equation 
(\ref{eq:com1}) is equivalent to the fact that
\begin{equation}\label{eq:com1z1}
\phi\left(e\alpha\right)=(\phi e)\alpha.
\end{equation}
Then $\Hom\left(E,\underline{E}\right)$ is a $\Z$-graded algebra. 
There is a left action of $\Lambda^{p}\left(\overline{W}^{*}\right)$ on 
$\Hom\left(E,\underline{E}\right)$ by morphisms of degree $p$.

Given $W$, let
\index{CW@$\mathscr C_{W}$}%
$\mathscr C_{W}$ be the category consisting of the above objects $E$ that verify 
(\ref{eq:bc12}),  and of the 
morphisms in $\Hom\left(E,\underline{E}\right)$.
If $E,\underline{E}$ are objects in $\mathscr C_{W}$, then 
$\Hom\left(E,\underline{E}\right)$ is still an object in $\mathscr C_{W}$, and  
the 
corresponding diagonal vector space  is given by 
$\Hom\left(D,\underline{D}\right)$. 

If 
$\underline{E}=\Lambda\left(\overline{W}^{*}\right)$, then 
$\underline{D}=\C$, and 
$\Hom\left(\Lambda\left(\overline{W}^{*}\right),E\right)=E$. 
Put
\index{Ed@$E^{\dag}$}%
\begin{equation}\label{eq:dua1}
E^{\dag}=\Hom\left(E,\Lambda\left(\overline{W}^{*}\right)\right).
\end{equation}
Then $D^{\dag}=D^{*}$. If $e\in E$, the map $\phi\in 
\Hom\left(E,\Lambda\left(\overline{W}^{*}\right)\right)\to 
\left(-1\right)^{\deg \phi\deg e}\phi 
e\in\Lambda\left(\overline{W}^{*}\right)$ gives the identification 
\begin{equation}\label{eq:dua2}
\left(E^{\dag}\right)^{\dag}=E.
\end{equation}
Also
\begin{equation}\label{eq:crac1}
E^{\dag}_{0}=\Lambda\left(\overline{W}^{*}\right)\ho D^{*}.
\end{equation}

If $W'$ has the same properties as $W$,  by \cite[Section 
2.3]{Keller06}, there is   a tensor product that is a functor
from $\mathscr C_{W} \ho\mathscr C_{W'}$ to $\mathscr 
C_{W \oplus W'}$. 
\begin{proposition}\label{prop:pis}
	If $E,\underline{E}$ are objects in $ \mathscr C_{W}$, the morphism $\phi\in \Hom\left(E,\underline{E}\right)$ is an isomorphism 
	if and only if $\phi$ induces an isomorphism from  $D$ into 
	$\underline{D}$.
\end{proposition}
\begin{proof}
	We only need to show that if $\phi$ maps isomorphically  
	$D$ into $\underline{D}$, then it acts isomorphically from $E$ into $\underline{E}$. 
	By (\ref{eq:bc8r1}), (\ref{eq:bc12}), $\phi$ also acts as an 
	isomorphism from $F^{p}E/F^{p+1}E$ into $F^{p}\underline{E}/F^{p+1}\underline{E}$.   A recursion argument shows that $\phi$ is an 
	isomorphism from $F^{p}E/F^{p+i}E$ into 
	$F^{p}\underline{E}/F^{p+i}\underline{E}$.  For $p=0, i=m+1$, we 
	obtain our proposition.
\end{proof}

Observe that $\Aut\left(E\right)$ is a Lie group, and 
 its Lie algebra  is just $\End\left(E\right)$. There is a  surjective group homomorphism  $\rho:
\Aut\left(E\right)\to\Aut\left(D\right)$, that induces a corresponding morphism of Lie algebras 
 $\End\left(E\right)\to \End\left(D\right)$. Put
 \index{Aut@$\Aut_{0}\left(E\right)$}%
 \index{End@$\End_{0}\left(E\right)$}%
 \begin{align}\label{eq:defa1}
&\Aut_{0}\left(E\right)=\left\{g\in\Aut\left(E\right),\rho 
g=1\right\},
&\End_{0}\left(E\right)=\left\{f\in \End\left(E\right),\rho 
f=0\right\}. 
\end{align}
Then $\Aut_{0}\left(E\right)$ is a Lie subgroup of $\Aut\left(E\right)$ with Lie algebra 
$\End_{0}\left(E\right)$.

Let 
\index{Aut@$\Aut^{0}\left(E\right)$}%
$\Aut^{0}\left(E\right)$ be the Lie subgroup of the elements of 
$\Aut\left(E\right)$ of degree $0$. Its Lie algebra $\End^{0}\left(E\right)$  consists of 
the elements of $\End\left(E\right)$ of degree $0$. We use the 
notation
\begin{align}\label{eq:defa2}
&\Aut_{0}^{0}\left(E\right)=\Aut_{0}\left(E\right)\cap \Aut^{0}\left(E\right),
&\End_{0}^{0}\left(E\right)=\End_{0}\left(E\right)\cap\End^{0}\left(E\right).
\end{align}

 If $g\in \Aut_{0}\left(E\right)$, then $\left(g-1\right)^{m+1}=1$. Similarly if 
 $f\in \End_{0}\left(E\right)$, then $f^{m+1}=0$.
\begin{proposition}\label{prop:fact}
	If $g\in \Aut\left(E_{0}\right)$, there exists a unique pair 
	$h\in \Aut\left(D\right), f\in \End_{0}\left(E_{0}\right)$ such that
\begin{equation}\label{eq:fact1}
g=\exp\left(f\right)h,
\end{equation}
and $h=\rho g$.
\end{proposition}
\begin{proof}
	Replacing $g$ by $g\left(\rho 
	g\right)^{-1}$, we may as well assume that $g\in 
	\Aut_{0}\left(E_{0}\right)$. Then $g-1$ 
	is nilpotent, so that if $f=\log 
	g=\log\left(1+\left(g-1\right)\right)$,
	then $f\in \End_{0}\left(E_{0}\right)$. The proof of our proposition is completed. 
\end{proof}

Observe that
\begin{equation}\label{eq:com1z1x}
\End\left(E_{0}\right)=\Lambda\left(\overline{W}^{*}\right)\ho\End\left(D\right).
\end{equation}
Any $f\in \End\left(E_{0}\right)$ can be written in the form
\begin{align}\label{eq:com1z1y}
&f=\sum_{0}^{m}f_{i}, & f_{i}\in 
\Lambda^{i}\left(\overline{W}^{*}\right)\ho\End\left(D\right).
\end{align}
Then $f\in \Aut\left(E_{0}\right)$ if and only if $f_{0}\in 
\Aut\left(D\right)$.

By (\ref{eq:exab1}), there are  non-canonical splittings,
\begin{equation}\label{eq:bc12a1}
E^{q}=D^{q} \oplus F^{1}E^{q}.
\end{equation}
By (\ref{eq:bc12}), (\ref{eq:bc12a1}), we get the non-canonical 
identification of  $\Z$-graded $\Lambda\left(\overline{W}^{*}\right)$-modules, 
\begin{equation}\label{eq:bc13}
E \simeq E_{0}.
\end{equation}

Let $\Hom^{0}\left(D,F^{1}E_{0}\right)$ be the homomorphisms of degree 
$0$ of $D$ in $F^{1}E_{0}$.
The  other 
splittings are obtained via  the action of $A\in 
\Hom^{0}\left(D,F^{1}E_{0}\right)$.

If we consider another splitting as in (\ref{eq:bc12a1}) associated with $A\in \Hom^{0}\left(D,F^{1}E_{0}\right)$, the corresponding $E_{0}$ is unchanged, but the 
identification to $E$ is different. This way, we obtain an 
automorphism of $E_{0}$, 
  $g=1+A\in \Aut^{0}_{0}\left(E_{0}\right)$.
  \subsection{Supercommutators and supertraces}%
  \label{subsec:scmtr}
  Let $\mathcal{A}=\mathcal{A}_{+} \oplus \mathcal{A}_{-}$ be a 
  $\Z_{2}$-graded algebra. If $a,b\in \mathcal{A}$, recall that 
  \index{ab@$\left[a,b\right]$}%
  $\left[a,b\right]$ is the supercommutator of $a,b$.
  
  If $E$ is an object in $\mathscr C_{W}$, then 
$\End\left(E\right)=\Hom\left(E,E\right)$ is a $\Z$-graded algebra, 
and so it inherits a corresponding $\Z_{2}$-grading. If we fix 
splittings as in (\ref{eq:bc13}), then 
$\End\left(E\right)=\End\left(E_{0}\right)$, and 
$\End\left(E_{0}\right)$ is given by (\ref{eq:com1z1x}).

Let $\tau=\pm 1$ be the involution of $D$ that defines its 
$\Z_{2}$-grading. If $A\in \End\left(D\right)$, its supertrace
\index{Trs@$\Trs$}%
$\Trs\left[A\right]\in \C$ is defined by
\begin{equation}\label{eq:scm2}
\Trs\left[A\right]=\Tr\left[\tau A\right].
\end{equation}
We extend $\Trs$ to a map from 
$\Lambda\left(W^{*}_{\C}\right)\ho\End\left(D\right)$ into 
$\Lambda\left(W^{*}_{\C}\right)$, with the convention that if 
$\alpha\in \Lambda\left(W^{*}_{\C}\right), A\in \End\left(D\right)$, 
then
\begin{equation}\label{eq:scm3}
\Trs\left[\alpha A\right]=\alpha\Trs\left[A\right].
\end{equation}
By \cite{Quillen85a},   $\Trs$ vanishes on supercommutators in  
$\Lambda\left(W^{*}_{\C}\right)\ho\End\left(D\right)$.
\begin{proposition}\label{prop:psm}
	The linear map $\Trs$ induces a morphism 
	$\Lambda\left(W^{*}\right)\ho\End\left(E\right)\to 
	\Lambda\left(W^{*}_{\C}\right)$ that does not 
	depend on the splitting (\ref{eq:bc12a1}), and vanishes on  
	supercommutators.
\end{proposition}
\begin{proof}
	As we saw at the end of Subsection \ref{subsec:filal}, the choice 
	of another splitting is equivalent to the choice of $g\in 
	\Aut_{0}^{0}\left(E_{0}\right)$. Such a $g$ is necessarily even, 
	so that if $a\in 
	\Lambda\left(W^{*}_{\C}\right)\ho\End\left(D\right)$, then
	\begin{equation}\label{eq:scm4}
\Trs\left[gag^{-1}\right]=\Trs\left[a\right],
\end{equation}
from which our proposition follows.
\end{proof}
\subsection{Morphisms and cones}%
\label{subsec:moco}
If    $\left(Q,d^{Q}\right)$ is a bounded complex of complex vector 
spaces with  differential $d^{Q}$ of degree $1$, we denote by 
\index{HQ@$HQ$}%
$HQ$ its cohomology.

Let $\left(Q,d^{Q}\right), 
\left(\underline{Q},d^{\underline{Q}}\right)$  be two such bounded 
complexes of
	 complex vector spaces.  Let $\phi:Q\to\underline{Q}$ be a 
	 morphism of complexes, so that $\phi$ preserves the degree, so 
	 that
	 \begin{equation}\label{eq:compl1}
\phi d^{Q}=d^{\underline{Q}}\phi.
\end{equation}
Then $\phi$ induces a morphism $\phi:HQ\to 
H\underline{Q}$. Also $\phi$ is said to be a quasi-isomorphism 
if $\phi:HQ\to H\underline{Q}$ is an 
isomorphism.

	 Put
	 \begin{equation}\label{eq:res4a0b}
C=\mathrm{cone}\left(Q,\underline{Q}\right).
\end{equation}
Then $\left(C,d^{C}\right)$ is also  a bounded complex such that
\begin{equation}\label{eq:res4a0z1b}
C^{\Ou}=Q^{\Ou+1}\oplus \underline{Q}^{\Ou}.
\end{equation}
The morphism $\phi$ can be identified with an endomorphism of $C$, that 
maps $C^{\Ou}$ into $C^{\Ou +1}$, i.e., $\phi$ can be viewed as a 
morphism of degree $1$ in $\End\left(C\right)$, 
and the differential $d^{C}_{\phi}$ is given by 
\begin{equation}\label{eq:res4a0z2b}
d^{C}_{\phi}=
\begin{bmatrix}
	d^{Q} & 0 \\
	\phi\left(-1\right)^{\deg} & d^{\underline{Q}}
\end{bmatrix}.
\end{equation}

We have the exact sequence of complexes,
\begin{equation}\label{eq:res4a-1b}
\xymatrix{0\ar[r] &\underline{Q}^{\Ou}\ar[r]& C^{\Ou}\ar[r]& 
Q^{\Ou+1}\ar[r]& 0.}
\end{equation}
There is a corresponding long exact sequence in cohomology,
\begin{equation}\label{eq:loex0}
\xymatrix{H^{\Ou}\underline{Q}\ar[r] &H^{\Ou}C
\ar[r] &H^{\Ou+1}Q\ar[r]^{\phi\left(-1\right)^{\Ou+1}}
&H^{\Ou+1}\underline{Q}}.
\end{equation}
By (\ref{eq:loex0}), $\phi$ is a quasi-isomorphism if and only if 
$HC=0$.

For $t\in \C$, let $M_{t}\in \End\left(C\right)$ be given in matrix 
form by
\begin{equation}\label{eq:for1}
M_{t}=
\begin{bmatrix}
	1 & 0 \\
	0 & t
\end{bmatrix}.
\end{equation}
For $t\neq 0$, $M_{t}$ is invertible, and
\begin{equation}\label{eq:for2}
d^{C}_{t\phi}=M_{t}d^{C}_{\phi}M_{t}^{-1},
\end{equation}
so that $M_{t}$ is an isomorphism from $\left(C,d^{C}_{\phi}\right)$ 
into $\left(C,d^{C}_{t\phi}\right)$.
\subsection{Generalized Hermitian metrics}%
\label{subsec:dua}
We make the same assumptions as in Subsection \ref{subsec:filal}.
 If $e\in W^{*}_{\C} $, set
\begin{equation}\label{eq:bc13a1}
\widetilde e =-e.
\end{equation}
We still denote by\, 
\index{ti@$\widetilde{}$}%
$ \widetilde {}$\, the corresponding anti-automorphism of algebras 
of 
$\Lambda\left(W^{*}_{\C}\right)$, so that if $\alpha\in 
\Lambda^{p}\left(W^{*}_{\C}\right)$, then
\begin{equation}\label{eq:bc13ax1}
\widetilde \alpha=\left(-1\right)^{p\left(p+1\right)/2}\alpha.
\end{equation}
Put
\index{a@$\alpha^{*}$}%
\begin{equation}\label{eq:bc1A3ax2}
\alpha^{*}=\overline{\widetilde \alpha}.
\end{equation}

If $A\in \End\left(D\right)$, let $\widetilde A\in 
\End\left(D^{*}\right)$ be its transpose. 
Then \ $\widetilde {}$ \  extends to an anti-automorphism from  
$\End\left(E_{0}\right)=\Lambda\left(\overline{W}^{*}\right)\ho\End\left(D\right)$ to $\End\left(E^{\dag}_{0}\right)=\Lambda\left(\overline{W}^{*}\right)\ho\End\left(D^{*}\right)$.

If $f\in \End\left(D\right)$, we define $f^{*}\in 
\End\left(\overline{D}^{*} \right) $ by the formula
\begin{equation}\label{eq:crac3}
f^{*} =\overline{\widetilde f}.
\end{equation}
Our conventions here coincide with the ones in \cite[Section 1 
(c)]{BismutLott95} and \cite[Section 3.5]{Bismut10b}.

Recall that
\begin{equation}\label{eq:supp1}
\LW=\Lambda\left(W^{*}\right)\ho\Lambda\left(\overline{W^{*}}\right),
\end{equation}
and that the natural grading $\LW$ is given by the sum of the degrees 
in $\Lambda\left(W^{*}\right)$ and 
$\Lambda\left(\overline{W^{*}}\right)$. If $\alpha\in 
\Lambda^{i}\left(W^{*}_{\C}\right)$, we write
\index{deg@$\deg$}%
\begin{equation}\label{eq:supp2}
\deg \,\alpha=i.
\end{equation}

We will often count the degree in 
$\Lambda^{p}\left(\overline{W}^{*}\right),\Lambda^{q}\left(W^{*}\right)$ to be $p,-q$, and we introduce the corresponding degree 
\index{deg@$\deg_{-}$}%
$\deg_{-}$ on $\Lambda\left(W^{*}_{\C}\right) $ so that if $\alpha\in \Lambda^{p}\left(\overline{W}^{*}\right), \beta\in \Lambda^{q}\left(W^{*}\right)$, 
\begin{equation}\label{eq:supp3}
\deg_{-}\,\alpha\we\beta=p-q.
\end{equation}
Most of the time, we will use this second grading on 
$\LW$. We will still use the notation  $\Lambda^{i}\left(W^{*}_{\C}\right)$ 
for the forms of total degree  equal to $i$.

In the sequel, elements in $D^{i},\overline{D}^{i*}$ will be said to 
be of degree $i$. If $A\in \mathrm{Hom}\left(D,\overline{D}^{*}\right)=D^{*} \otimes 
\overline{D}^{*}$, we count its 
degree as the difference of the degrees in $\overline{D}^{*}$ and 
$D^{*}$.   Let $A^{*}\in 
\Hom\left(D,\overline{D}^{*}\right)$ denote the conjugate of the 
transpose of $A$. 
We equip  $\LW\ho\Hom\left(D,\overline{D}^{*}\right)$ with 
the obvious antilinear involution $*$ and with the degree induced by $\deg_{-}$ on 
$\LW$ and by the above degree on 
$\Hom\left(D,\overline{D}^{*}\right)$. Under $*$, this degree is 
changed into its negative. An element of 
$\LW\ho
\Hom\left(D,\overline{D}^{*}\right)$ is said to be self-adjoint if it is invariant under $*$.

If $h\in \LW\ho
\Hom\left(D,\overline{D}^{*}\right)$, we can write $h$ in the 
form
\begin{align}\label{eq:bc13a2}
	&h=\sum_{i=0}^{2m}h_{i}, & h_{i}\in 
\Lambda^{i}\left(W^{*}_{\C}\right)\ho\Hom\left(D,\overline{D}^{*}\right).
\end{align}
\begin{definition}\label{def:bc13a2}
	An element $h\in \LW\ho\Hom\left(D,\overline{D}^{*}\right)$ is 
	said to be a generalized metric on $D$ if it is  of degree $0$, 
	self-adjoint, and such that $h_{0}$ defines a $\Z$-graded Hermitian metric on 
	the $\Z$-graded vector space $D$\footnote{This just means that 
	the $D^{i}$ are mutually orthogonal.}.   Let $\mathscr M^{D}$ be the 
	set of generalized metrics on $D$. 
	
	A metric $h\in \mathscr M^{D}$ is said to be pure if $h=h_{0}$.
\end{definition}

If $h\in \mathscr M^{D}$, then $h^{-1}\in \mathscr M^{\overline{D}^{*}},\overline{h}^{-1}\in 
\mathscr M^{D^{*}}$.  If 
$f\in \Lambda\left(W^{*}_{\C}\right)\ho\End\left(D\right)$, we define $f^{*}\in 
\Lambda\left(W^{*}_{\C}\right)\ho\End\left(\overline{D}^{*}\right)$ by the same procedure as 
before. 

Observe that $\Aut^{0}\left(E_{0}\right)$ acts on 
$\mathscr M^{D}$ so that if $g\in\Aut^{0}\left(E_{0}\right)$,  the action of $g^{-1}$ on $\mathscr M^{D}$ 
is given by 
\index{gh@$g^{\dag}h$}%
\begin{equation}\label{eq:bc13a3}
g^{\dag} h=g^{*}h g.
\end{equation}
\subsection{Clifford algebras}%
\label{subsec:cli}
Let $V$ be a finite-dimensional complex vector space of dimension $n$, and let 
$V_{\R}$ be its real form. Let $V_{\C}=V_{\R} \otimes _{\R}\C$ be the 
complexification of $V_{\R}$, so that $V_{\C}=V \oplus \overline{V}$.

Let $g^{V^{*}}$ be a Hermitian metric on $V^{*}$, and let 
$g^{V^{*}_{\R}}$ be the corresponding scalar product on  $V^{*}_{\R}$. 
Then $g^{V^{*}}$ identifies $V^{*}$ and 
$\overline{V}$. Let $c\left(V^{*}_{\R}\right)$ be the Clifford algebra of 
$\left(V^{*}_{\R}, g^{V^{*}_{\R}}\right)$\footnote{For convenience, 
we prefer to introduce the Clifford algebra of $V^{*}_{\R}$ instead 
of the usual $c\left(V_{\R}\right)$.}, which is  generated by $1, f\in V^{
*}_{\R}$, 
and  the commutation relations  for $f,f'\in V^{*}_{\R}$, 
\footnote{The  commutation relations
are usually written with an extra factor $2$ in the right-hand side. 
In complex geometry, it is more natural to drop the factor $2$.} 
\begin{equation}\label{eq:xiao2}
ff'+f'f=-\left\langle  f,f'\right\rangle_{g^{V^{*}_{\R}}}.
\end{equation} 

As a $\Z_{2}$-graded vector space, 
$c\left(V^{*}_{\R}\right)$ is canonically isomorphic to the exterior 
algebra $\Lambda\left(V^{*}_{\R}\right)$. Let 
$f^{1},\ldots,f^{2n}$ be an orthonormal basis of $V^{*}_{\R}$.  If 
$0\le p\le n$, and if
$1\le i_{1}<\ldots<i_{p}\le n$, the 
identification maps $f^{i_{1}}\ldots f^{i_{p}}\in 
c\left(V^{*}_{\R}\right)$ to $f^{i_{1}}\we\ldots\we 
f^{i_{p}}\in\Lambda\left(V^{*}_{\R}\right)$. If $B\in 
\Lambda\left(V^{*}_{\R}\right)$, we denote by 
\index{cB@${}^{c}B$}%
${}^{c}B$ the 
corresponding element in $c\left(V^{*}_{\R}\right)$.

If $f\in V^{*}$, if $f_{*}=g^{V^{*}}f\in \overline{V}$, put
\begin{align}\label{eq:xiao2a1}
&c\left(\overline{f}\right)=\overline{f}\we,
&c\left(f\right)=-i_{f_{*}}.
\end{align}
By (\ref{eq:xiao2a1}), $\Lambda\left(\overline{V}^{*}\right)$ is a 
$c\left(V^{*}_{\R}\right)$-Clifford module. We have the 
identification of $\Z_{2}$-graded algebras,
\begin{equation}\label{eq:bing1}
c\left(V^{*}_{\R}\right) \otimes 
_{\R}\C=\End\left(\Lambda\left(\overline{V}^{*}\right)\right).
\end{equation}
In the sequel, we will not distinguish between the two sides of 
(\ref{eq:bing1}). In particular, if $B\in 
\Lambda\left(V^{*}_{\C}\right)$, ${}^{c}B$ is viewed as an element of 
$\End\left(\Lambda\left(\overline{V}^{*}\right)\right)$. In 
particular, if $B\in\Lambda\left(\overline{V}^{*}\right)$, ${}^{c}B$ 
is just the exterior product by $B$.

Let $U$ be another finite-dimensional complex vector 
space.  We will view $\Lambda\left(U^{*}_{\R}\right)$ as an algebra 
acting upon itself by left multiplication. 

Then
\begin{equation}\label{eq:jix2a2}
	\Lambda \left(U^{*}_{\R}\oplus 
V_{\R}^{*}\right)=\Lambda\left(U^{*}_{\R}\right)\ho\Lambda \left(V^{*}_{\R}\right).
\end{equation}
Also  $\Lambda \left(U^{*}_{\R} \oplus 
V^{*}_{\R}\right)$  can be identified with 
$\Lambda \left(U^{*}_{\R}\right)\ho c\left(V^{*}_{\R}\right)$. If 
$B\in \Lambda \left(U^{*}_{\C} \oplus 
V^{*}_{\C}\right)$, let ${}^{\mathrm{c}}B$ the 
associated element in 
$\Lambda \left(U^{*}_{\C}\right)\ho\End\left(\Lambda \left(\overline{V}^{*}\right)\right)$. If $B\in \Lambda \left(\overline{U}^{*} \oplus \overline{V}^{*}\right)$, ${}^{\mathrm{c}}B$ is still multiplication by $B$.  
\subsection{A simple remark of differential geometry}%
\label{subsec:remdi}
Let $Z$ be a real manifold, with real tangent bundle $TZ$, and let $\n^{TZ}$ be a connection on 
$TZ$, with torsion $T$. Let $\n^{\Lambda\left(T^{*}Z\right)}$ denote 
the induced connection on $\Lambda\left(T^{*}Z\right)$.
Let 
\index{nLTZ@$\n^{\Lambda\left(T^{*}Z\right)}_{\mathrm a}$}%
$\n^{\Lambda\left(T^{*}Z\right)}_{\mathrm a}$ denote the 
antisymmetrization of  $\n^{\Lambda\left(T^{*}Z\right)}$, that acts on 
$\Omega\left(Z,\R\right)$ and increases the degree by $1$. 

Let $d^{Z}$ be the de Rham operator on $Z$.  Then 
\index{iT@$i_{T}$}%
$i_{T}$ is a $2$-form on $Z$ with values in contraction operators\footnote{We will 
always use the normal ordering of such operators, so that contraction 
operators act before multiplication by forms.}. Let $\left(f_{\alpha}\right)$ be a basis of $TZ$, and let 
$\left(\fa\right)$ be the corresponding dual basis of $T^{*}Z$. Then 
$i_{T}$ is given by
\begin{equation}\label{eq:it1}
i_{T}=\frac{1}{2}\fa\fb i_{T\left(f_{\alpha},f_{\beta}\right)}.
\end{equation}
We have the classical identity
\begin{equation}\label{eq:cla1}
d^{Z}=\n^{\Lambda\left(T^{*}Z\right)}_{\mathrm a}+i_{T}.
\end{equation}

For $a\in \R$, let
\index{nTZ@$\n^{TZ,a}$}%
$\n^{TZ,a}$ be the connection on $TZ$ such that if 
$U,V\in TZ$, 
\begin{equation}\label{eq:cla2}
\n^{TZ,a}_{U}V=\n^{TZ}_{U}V+aT\left(U,V\right).
\end{equation}
Then the torsion of $\n^{TZ,a}$ is given by $\left(1+2a\right)T$. 
Let $\n^{\Lambda\left(T^{*}Z\right),a}$ be the connection on 
$\Lambda\left(T^{*}Z\right)$ that is induced by $\n^{TZ,a}$. 

In the sequel, we define the operator $i_{T\left(U,\cdot\right)}$ 
using normal ordering, i.e.,
\begin{equation}\label{eq:clap4}
i_{T}\left(U,\cdot\right)=\fa i_{T\left(U,f_{\alpha}\right)}.
\end{equation}
By (\ref{eq:cla2}), if $U\in TZ$,
\begin{equation}\label{eq:cla2a1}
\n^{\Lambda\left(T^{*}Z\right),a}_{U}=\n_{U}^{\Lambda\left(T^{*}Z\right)}-a i_{T\left(U,\cdot\right)}.
\end{equation}

 We 
have the identity of connections on $\Lambda\left(T^{*}Z\right)$, 
\begin{equation}\label{eq:cla3}
	\n^{\Lambda\left(T^{*}Z\right),a}=\n^{\Lambda\left(T^{*}Z\right)}-a
\left\langle  
T\left(\cdot,f_{\beta}\right),f^{\gamma}\right\rangle\fb 
i_{f_{\gamma}}.
\end{equation}
By (\ref{eq:it1}), (\ref{eq:cla3}), we deduce that
\begin{equation}\label{eq:cla4}
	\n^{\Lambda\left(T^{*}Z\right),a}_{\mathrm{a}}=\n^{\Lambda\left(T^{*}Z\right)}_{\mathrm a}-2ai_{T}.
\end{equation}

By (\ref{eq:cla1}), (\ref{eq:cla4}), we get
\begin{equation}\label{eq:cla4b1}
d^{Z}=\n^{\Lambda\left(T^{*}Z\right),-1/2}_{\mathrm{a}}.
\end{equation}

If $U$ is a smooth section of $TZ$, then
\begin{equation}\label{eq:clap1}
\left[U,V\right]=\n^{TZ}_{U}V-T\left(U,V\right)-\n^{TZ}_{V}U.
\end{equation}
Equation (\ref{eq:clap1}) can be rewritten in the form
\begin{equation}\label{eq:clap2}
\left[U,V\right]=\n^{TZ,-1}_{U}V -\n^{TZ}_{V}U.
\end{equation}
We define the operator $i_{\n^{TZ}U}$ using normal ordering, i.e., 
\begin{equation}\label{eq:clap2a1}
i_{\n^{TZ}U}=\fa i_{\n^{TZ}_{f_{\alpha}}U}.
\end{equation}

By (\ref{eq:clap2}), when acting on $\Omega\left(Z,\R\right)$, the Lie derivative operator $L_{U}$ is 
given by
\begin{equation}\label{eq:clap3}
L_{U}=i_{\n^{TZ}U}+\n^{\Lambda\left(T^{*}Z\right),-1}_{U}.
\end{equation}

Combining (\ref{eq:cla2a1}) and (\ref{eq:clap3}), we obtain
\begin{equation}\label{eq:clap4b}
L_{U}=i_{\n^{TZ}U}+\n^{\Lambda\left(T^{*}Z\right)}_{U}+i_{T\left(U,\cdot\right)}.
\end{equation}
Equation (\ref{eq:clap4b}) is compatible with equation (\ref{eq:cla1}) 
and with Cartan formula $L_{U}=\left[d^{Z},i_{U}\right]$, which can 
also be written in the form
\begin{equation}\label{eq:clap4ax1}
L_{U}=\left[\n^{\Lambda\left(T^{*}Z\right)}_{\mathrm{a}}+i_{T},i_{U}\right]=
\left[\n^{\Lambda\left(T^{*}Z\right),-1/2}_{\mathrm{a}},i_{U}\right].
\end{equation}

Let $Z$ be a complex manifold,  let $TZ$ be its holomorphic 
tangent bundle, and let  $T_{\R}Z$ be its real tangent bundle. Let $g^{TZ}$ be a Hermitian metric on $TZ$, let 
$\n^{TZ}$ be the associated holomorphic Hermitian connection on $TZ$, 
and let $\n^{T_{\R}Z}$ be the induced connection on $T_{\R}Z$, with 
torsion $T$. Then 
$T$ splits as $\left(2,0\right)$-form with values in $TZ$ and a 
$\left(0,2\right)$-form with values in $\overline{TZ}$. For $a\in 
\R$, $\n^{T_{\R}Z,a}$ preserves $TZ$ and $\overline{TZ}$. More 
precisely, if $u,v\in TZ$ are smooth sections of $TZ$,
\begin{align}\label{eq:clap5}
&\n^{TZ, a \prime \prime }=\n^{TZ \prime \prime },&\n^{TZ,a \prime 
}_{u}v=\n_{u}^{TZ}v+aT\left(u,v\right).
\end{align}
Also  equation (\ref{eq:cla4b1}) splits into separate formulas 
for $\overline{\pa}^{Z}, \pa^{Z}$.
\section{The antiholomorphic superconnections of Block}%
\label{sec:ahscn}
The purpose of this Section is to describe the antiholomorphic superconnections introduced by Block \cite{Block10}.

This Section is organized as follows. In Subsection \ref{subsec:def}, 
we define the antiholomorphic superconnections $\left(E,A^{E\prime 
\prime}\right)$ on a compact complex manifold $X$.  

In Subsection \ref{subsec:co}, we establish the local conjugation  
result of Block, which asserts that locally, $A^{E \prime \prime }$ 
can be reduced to a canonical form. This result is an extension of 
the Newlander-Nirenberg theorem for holomorphic vector bundles.

In Subsection \ref{subsec:applco}, when viewing $\mathscr 
E=\left(E,A^{E \prime \prime }\right)$ as a sheaf of 
$\mathcal{O}_{X}$-complexes, we show that its cohomology  is a 
$\Z$-graded coherent sheaf, and that the natural spectral sequence degenerates at $ E_{2}$.

In Subsection \ref{subsec:detli}, we show that the complex line $\det 
D$ has a natural holomorphic structure, and that $\det D$ coincides 
with the 
 determinant line  of Knudsen-Mumford 
 \cite{KnudsenMumford}.
 
 In Subsection \ref{subsec:acyc}, we show that $\mathscr H \mathscr 
 E=0$ if and only if for any $x\in X$, the complex 
 $\left(D,v_{0}\right)_{x}$ is exact.
 
 In Subsection \ref{subsec:scmoco}, given a morphism of 
 antiholomorphic superconnections, we construct the corresponding cone.
 
In Subsection \ref{subsec:pullten}, we construct the 
 pull-backs of antiholomorphic superconnections.
 
 Finally, in Subsection \ref{subsec:tenspro}, we obtain a tensor 
 product of antiholomorphic superconnections.
\subsection{The antiholomorphic superconnections}%
\label{subsec:def}
Let $X$ be a compact complex manifold of dimension $n$. If $r,r'\in 
\Z, r\le r'$, let 
$E=\bigoplus_{i=r}^{r'}E^{i}$ be  a  
finite-dimensional  complex  
$\Z$-graded vector 
bundle on $X$, which is also a left
$\Lambda \left(\overline{\TsX}\right)$-module, so that elements 
of $\Lambda^{p}\left(\overline{\TsX}\right)$ act on $E$ as morphisms of 
degree $p$. 

We  use the notation of Section \ref{sec:lia}, with $W$ 
replaced by $TX$. We assume that the $D^{q}$ as defined 
(\ref{eq:bc13a}) have constant rank, so that the corresponding 
diagonal $D$ is a complex 
$\Z$-graded vector bundle on $X$. Also we assume that the analogue of 
(\ref{eq:bc12}) holds. 

Then $C^{\infty }\left(X,E\right)$ is a $\Z$-graded vector space. The 
degree of an operator acting on $C^{\infty }\left(X,E\right)$ is 
counted as in Section \ref{sec:lia}.

	Let
	\index{P@$P$}%
	$P$ denote the projection $E\to D$. Then $P$ 
	preserves the $\Z$-grading. Also $P$ vanishes on $F^{1}E$.

As in (\ref{eq:bc12a0}),  set
\index{E0@$E_{0}$}%
\begin{align}\label{eq:bot1}
&E_{0}^{i}=\bigoplus_{p+q=i}\Lambda^{p}\left(\overline{\TsX}\right)
\otimes D^{q},
&E_{0}=\bigoplus_{i=r}^{r'} E_{0}^{i}.
\end{align}
Then $E_{0}$ is a $\Z$-graded filtered vector bundle with the 
same $D$ as $E$.

We can write the analogue of (\ref{eq:bc12a1}) as a non-canonical 
splitting of smooth vector bundles,
\begin{equation}\label{eq:iv4a-1}
E^{q}=D^{q} \oplus F^{1}E^{q}.
\end{equation}
As in (\ref{eq:bc13}), we get the smooth non-canonical identification of 
$\Z$-graded $\Lambda\left(\overline{\TsX}\right)$-modules,
\begin{equation}\label{eq:iv4}
E \simeq E_{0}.
\end{equation}
Any other splitting is obtained from smooth sections of 
$\Hom^{0}\left(D,F^{1}E_{0}\right)$.

Now we follow Quillen \cite[Section 2]{Quillen85a} and Block \cite[Definition 
2.4]{Block10}.
\begin{definition}\label{def:scn}
	A differential operator 
\index{AE@$A^{E \prime \prime }$}%
	$A^{E \prime \prime}$ acting on 
	$C^{\infty }\left(X,E\right)$ is said to be an antiholomorphic 
	flat
	superconnection\footnote{In \cite[Sections 2 and 4]{Block10}, Block writes instead 
	that $A^{E \prime \prime }$ is a $\Z$-connection.} if it acts as a differential operator of degree 
	$1$ on $C^{\infty}\left(X,E\right)$  such that
	\begin{itemize}
		\item  If $\alpha\in \Omega^{0,\scriptsize\bullet}\left(X,\C\right), s\in 
		C^{\infty}\left(X,E\right)$, then
		\begin{equation}\label{eq:iv5}
A^{E \prime \prime }\left(\alpha s\right)=\left( 
\overline{\pa}^{X}\alpha \right)  
s+\left(-1\right)^{\mathrm{deg}\alpha}\alpha A^{E \prime \prime }s.
\end{equation}
\item The following identity holds:
\begin{equation}\label{eq:iv6}
A^{E \prime \prime,2}=0.
\end{equation}
	\end{itemize}
\end{definition}
In the sequel, we will omit `flat', and  say instead that $A^{E \prime \prime }$ is an 
antiholomorphic superconnection.

The operator $\As$ preserves the filtration $F$. In particular $A^{E 
\prime \prime }$ acts on $D=E/F^{1}E$ like a smooth section 
\index{v@$v_{0}$}%
$v_{0}\in \End\left(D\right)$ 
which is of degree $1$, and such that $v_{0}^{2}=0$. Then $P$ is a morphism of  complexes 
$\left(C^{\infty }\left(X,E \right),A^{E \prime \prime }\right)\to 
\left( C^{\infty }\left(X,D\right),v_{0}\right)$. 

If $g$ is a smooth section of $\Aut^{0}\left(E\right)$, then 
$\underline{A}^{E \prime \prime }=g A^{E \prime \prime  }g^{-1}$ is another antiholomorphic 
superconnection on $E$. In the sequel, we will write that $\underline{A}^{E 
\prime \prime } \simeq A^{E \prime \prime }$.

Assume for the moment that we fix a non-canonical identification as in 
(\ref{eq:iv4a-1}), (\ref{eq:iv4}). Let $A^{E_{0}\prime \prime }$ be the 
corresponding antiholomorphic superconnection on $E_{0}$. 	We can write $A^{E_{0} \prime \prime }$ in the form
	\begin{equation}\label{eq:iv36a6}
A^{E_{0} \prime \prime }=v_{0}+\n^{D \prime \prime }+\sum_{i\ge 
2}^{}v_{i},
\end{equation}
where   $\n^{D \prime \prime }$ a degree preserving antiholomorphic 
connection on $D$\footnote{In general, its square does not 
vanish, i.e., it does not define a holomorphic structure on $D$. The 
distinction between antiholomorphic connection and holomorphic 
structure will reappear in the sequel.},  and 
for $i=0\ \mathrm{or}\ i\ge 2$, $v_{i}\in 
\Lambda^{i}\left(\overline{\TsX}\right)\ho 
\End^{1-i}\left(D\right)$. Since $A^{E \prime \prime ,2}=0$, from 
(\ref{eq:iv36a6}), we get
\begin{align}\label{eq:carr1}
&v_{0}^{2}=0,&\left[\n^{D \prime \prime },v_{0}\right]=0, \qquad
\n^{D \prime \prime ,2}+\left[v_{0},v_{2}\right]=0.
\end{align}

Put
\index{B@$B$}%
\begin{equation}\label{eq:baba1}
B=v_{0}+\sum_{i\ge 2}^{}v_{i}.
\end{equation}
Then $B$ is a section of degree $1$ of
$\Lambda\left(\overline{\TsX}\right)\ho\End\left(D\right)$, and (\ref{eq:iv36a6}) can be written in the form,
\begin{equation}\label{eq:baba2}
A^{E_{0} \prime \prime }=\n^{D \prime \prime}+B.
\end{equation}

Let $g$ be a smooth section of $\Aut^{0}_{0}\left(E_{0}\right)$ 
associated with a different splitting in (\ref{eq:iv4a-1}), 
(\ref{eq:iv4}). Let $\underline{A}^{E_{0} \prime \prime }$ be the superconnection 
on $E_{0}$ that corresponds to $A^{E \prime  \prime }$ as before. Then
\begin{equation}\label{eq:carr1a1}
\underline{A}^{E_{0} \prime \prime }=gA^{E_{0} \prime \prime }g^{-1}.
\end{equation}
Then $\underline{A}^{ E\prime \prime }$ has an expansion similar to 
(\ref{eq:iv36a6}), in which the corresponding terms will also be 
underlined. Note that $\underline{v}_{0}=v_{0}$, and also that there 
is a smooth section $\alpha$  of 
$\overline{\TsX}\ho\End^{-1}\left(D\right)$ such that
\begin{equation}\label{eq:roug}
\underline{\n}^{D \prime \prime}=\n^{D \prime \prime}+\left[\alpha,v_{0}\right].
\end{equation}
\begin{example}\label{exa:triva}
Here are  two trivial cases of antiholomorphic superconnections: 
	\index{CX@$\mathscr C^{X}$}%
	\index{CX@$\underline{\mathscr C}^{X}$}%
	\begin{align}\label{eq:triva1}
&\mathscr 
C^{X}=\left(\Lambda\left(\overline{\TsX}\right),\overline{\pa}^{X}\right),
&\underline{\mathscr C}^{X}=\left(\Lambda\left(T^{*}_{\C}X \right) 
,\overline{\pa}^{X}\right).
\end{align}
In the first case, $D=\C$, in the second case, 
$D=\Lambda\left(\TsX\right)$. 
In both cases, we have canonical identifications $E=E_{0}$.
\end{example}
\subsection{A conjugation result}%
\label{subsec:co}
The following extension of the Newlander-Nirenberg theorem has been 
established in  \cite[p.17]{Block10}.
\begin{theorem}\label{thm:Conj}
	If $\left(E,A^{E \prime \prime }\right)$ is an antiholomorphic superconnection, 
	if $x\in X$, there is an open neighborhood $U$ of $x$,  a 
	holomorphic structure $\underline{\n}^{D\vert_{U} \prime \prime }$ on $D\vert 
	_{U}$ such that $\underline{\n}^{D\vert_{U} \prime \prime }v_{0}=0$, and moreover,
	\begin{equation}\label{eq:raa1}
	(E,A^{E\prime\prime })|_{U} \simeq 
	\left(\Lambda(\overline{T^{*}U})\otimes 
	D\vert_{U},v_{0}+\underline{\nabla}^{D\vert_{U}\prime\prime 
	}\right).
	\end{equation} 
\end{theorem}
\begin{proof}
	We will proceed as in proof of the  Newlander-Nirenberg 
	theorem \cite[Subsection 2.2.2]{Donaldson88}. Recall that $\dim X=n$. Let
	$(z^{1},\ldots,z^{n})$ be a coordinate system near $x$ such that $x$ is represented by	
	$z=0$. For $r>0$ small enough, let  $U_{r}=\{z\in \C^{n}: |z^{i}|< 
	r \}$ be the corresponding  polydisc. For $r$ small enough, we 
	can view $U_{2r}$ as an open neighborhood  of $x$ in $X$, such 
	that 
	$D\vert_{U_{2r}}$ is a trivial smooth vector bundle. In particular, $D$ is equipped with the 
	holomorphic structure $\overline{\pa}$. By (\ref{eq:iv36a6}), on 
	$U_{2r}$, there is  a smooth section $B_{0}$ of 
	$\Lambda\left(\overline{\C}^{n}\right)\ho\End\left(D\right)$ of 
	degree  $1$  such that
	\begin{equation}\label{eq:adi1}
	A^{E_{0}\prime\prime }|_{U_{2r}}=\overline{\partial}+B_{0}.
	\end{equation}
	If $B_{0}^{(\ge 1)}$ is the sum of the components of $B_{0}$ of degree $\ge 1$ in 
	$\Lambda\left(\overline{\C}^{n}\right)$, then
	\begin{equation}\label{eq:dec1}
B_{0}=v_{0}+B_{0}^{(\ge 1)}.
\end{equation}

	For $t\in \C,\left\vert  t\right\vert\le 1$, consider  the 
	dilation $z\to \delta_{t}(z)= tz$. Replacing $r$ by $tr$, and  
	$A^{E_{0} \prime \prime }\vert_{U_{2r}}$ 
	by its conjugate by $\delta_{t}^{*}$, we may as well assume that 
	$\left\Vert  B_{0}^{(\ge 1)}\right\Vert_{C^{b}\left(U_{2r}\right)}$ is arbitrary small.

We will show by 
induction  that for $1\le i\le n$, there are  smooth sections  
$J_{i}$ of $\Aut^{0}(\Lambda\left(\overline{\C}^{n}\right)\ho D)$,  and  antiholomorphic superconnections 
	$\overline{\partial}+B_{i}$ on $U_{r}$ such 	that  $(J_{i},B_{i})$ do not contain the Grassmannian variable 	
	$d\overline{z}^{j},j\le i$, and that 
	\begin{equation}\label{eq:lab0}
	J_{i}^{-1}\left(\overline{\partial}+B_{i-1}\right)J_{i}=	\overline{\partial}+B_{i}. 
	\end{equation} 
If $J=J_{1}\cdots J_{n}$ , we will obtain part of our theorem by conjugating 
$A^{E_{0} \prime \prime }$ by $J^{-1}$.
	
For $i=1$,  we need to find 
$J_{1}$  such that $J_{1}$ and 
	$\overline{\partial} J_{1}+B_{0}J_{1}
$	do not contain  $d\overline{z}^{1}$, in which case
	\begin{align}
		B_{1}=J_{1}^{-1}\left( \overline{\partial} J_{1}\right) +J_{1}^{-1}B_{0}J_{1}. 
	\end{align} 
	In the sequel we view $\Lambda\left(\overline{\C}^{n-1}\right)$ 
	as being generated by 
	$d\overline{z}^{2},\ldots,d\overline{z}^{n}$.
If $z=(z^{1},z')\in U_{r}$, with $z'\in\C^{n-1}$ fixed for the moment, we  have to  find an invertible section $J_{1}\in 
C^{\infty}(U_{r},\Lambda (\overline{\C}^{n-1})\ho
\End(D))$ of  degree $0$ such that 
	\begin{equation}\label{ew:lab0}
		\frac{\partial J_{1}}{\partial\overline{ z}^{1}} 
		+\left(i_{\frac{\partial}{\partial 
		\overline{z}^{1}}}B_{0}\right)J_{1}=0.
	\end{equation} 

Let $\phi\in C^{\infty,c}\left(\C^{n}, [0,1]\right)$ be  such 
that $\phi=1$ on $U_{r}$, the support of which is included  in $U_{2r}$. It is enough 
to show that for  $|z'|$ small 
enough, there is an invertible section $j_{z'}\in 
C^{\infty}(\C,\Lambda (\C^{n-1})\ho 
\End(D))$ of total degree $0$, which depends smoothly on $z'$,  such that 
	\begin{equation}\label{eq:ors1}
		\frac{\partial j_{z'}}{\partial \overline{z}^{1}} 
		+\phi \left (i_{\frac{\partial}{\partial 
		\overline{z}^{1}}}B_{0}\right)j_{z'}=0. 
	\end{equation} 
If $j_{z'}=1+g_{z'}$, equation (\ref{eq:ors1}) is equivalent to 
	\begin{equation}\label{eqgz}
		\left(\frac{\pa}{\pa\overline{z}^{1}}+\phi i_{\frac{\partial}{\partial 
		\overline{z}^{1}}}B_{0} \right)g_{z'} +\phi i_{\frac{\partial}{\partial 
		\overline{z}^{1}}}B_{0}=0.
	\end{equation} 

	The form $\frac{dz^{1}}{2i\pi z^{1}}$ on $\C$ is locally integrable, and 
	 by Poincaré-Lelong (\ref{eq:bc5}), we have the equation of currents,
\begin{align}\label{eqfunda}
\overline{\pa}\frac{dz^{1}}{ {2i\pi}	 z^{1}}=\delta_{0}.
	\end{align}
 Let $\pi_{1},\pi_{2}$ be the two canonical projections $\C^{2}\to \C$. Let $\pi:\C^{2}\to \C$ be given by $\left(z^{1},z^{1 \prime 
	}\right)=z^{1}+z^{1 \prime }$.
Let $C^{c}\left(\C,\C\right)$ be the vector space of continuous 
complex functions on $\C$ with compact support. If $f\in C^{c}\left(\C,\C\right)$, put
\begin{equation}\label{eq:ra1}
Rf=\pi_{*}\left[\pi_{1}^{*}fd\overline{z}^{1}\pi_{2}^{*}\frac{dz^{1}}{2i\pi z^{1}}\right]. 
\end{equation}
Equation (\ref{eq:ra1}) gives the convolution of two 
distributions.  Let $C^{b}\left(\C,\C\right)$ be the vector space of 
bounded continuous functions from $\C$ into itself. Then $Rf\in C^{b}\left(\C,\C\right)$. 
By (\ref{eqfunda}), we get
\begin{equation}\label{eq:ra2}
\frac{\pa}{\pa\overline{z}^{1}}Rf=f.
\end{equation}

Let 
$C^{0,b}\left(\C,\Lambda\left(\C^{n-1}\right)\ho\End\left(D\right)\right)$ be the space of bounded continuous sections of degree $0$ of $\Lambda\left(\C^{n-1}\right)\ho\End\left(D\right)$.  Let $F_{z'}$ be the bounded operator  acting on  
$C^{0,b}\left(\C,\Lambda(\C^{n-1})\ho \End(D)\right)$,
\begin{equation}\label{eqFdef}
		F_{z'}=R\phi i_{\frac{\pa}{\pa \overline{z}^{1}}}B_{0}.
	\end{equation}
Since $B_{0}$ is of degree $1$,  $F_{z'}$ is of degree $0$. 

	By \eqref{eq:ra2}, we get  
	\begin{align}\label{eqFonda2}
		\frac{\partial}{\partial \overline{ 
		z}_{1}}F_{z'}=\phi
		i_{\frac{\pa}{\pa\overline{z}^{1}}}B_{0}. 
	\end{align} 
By  (\ref{eq:ra1}), \eqref{eqFdef}, there is $C>0$ such that 
	\begin{equation}\label{eq:ineq1}
		\|F_{z'}\| \le C \left\|B_{0}^{(\ge 1)}\right\|_{C^{b}(U_{2r})}. 
	\end{equation} 

Using (\ref{eqFonda2}), equation (\ref{eqgz}) can be 
written in the form
\begin{equation}\label{eq:ra3}
\frac{\pa}{\pa\overline{z}^{1}}\left[\left(1+F_{z'}\right)g_{z'}+F_{z'}1\right]=0.
\end{equation}

We may and we will assume that $\left\Vert  B_{0}^{(\ge 
1)}\right\Vert_{C^{b}\left(U_{2r}\right)}$ is 
small enough, so that, when using (\ref{eq:ineq1}),   $\left\Vert  
F_{z'}\right\Vert<1$. Then
		 $1+F_{z'}$ is 
		invertible with uniformly bounded inverse.
		
		Put 
	\begin{align}\label{eqgz1}
		g_{z'}=-(1+F_{z'})^{-1}F_{z'}1. 
	\end{align} 
Then $g_{z'}$ is a solution of degree $0$ of (\ref{eq:ra3}), and so it solves 
(\ref{eqgz}).   Also it is obviously smooth in all variables. 
Moreover, 
	\begin{equation}\label{eq:lab3}
		\|g_{z'}\|\le \frac{\|F_{z'}\|}{1-\|F_{z'}\|}. 
	\end{equation} 
If $\left\Vert  B_{0}^{(\ge 1)}\right\Vert_{C^{b}\left(U_{2r}\right)}$ is small enough so that  
$\|F_{z'}\|\le1/3$, then	 $\|g_{z'}\| \le 1/2 $. In particular, $j_{z'}=1+g_{z'}$ is invertible. 	

Assume that $i\ge1$ and that we have constructed $(J_{k}, B_{k})$ for 
 $k\le i-1$.  By (\ref{eq:lab0}), we have to find $J_{i}$ such 
that $J_{i}$ and $\overline{\partial} J_{i}+B_{i-1}J_{i}$
do not contain $d\overline{z}^{1},\ldots,d\overline{z}^{i}$.  This  
is equivalent to the fact that 
$J_{i}$ is holomorphic in $z^{1},\ldots,z^{i-1}$, and 
	\begin{align}\label{eqJi1}
		\frac{\partial J_{i}}{\partial \overline{z}^{i}} +\left(i_{\frac{\partial}{\partial 
		\overline{z}^{i}}}B_{i-1}\right)J_{i}=0. 
	\end{align} 
By the argument given in (\ref{eq:lab0})--(\ref{eq:lab3}), using 
again a dilation $\delta_{t},\left\vert  t\right\vert<1$,we may 
assume that 
$\left\Vert  B_{i-1}^{(\ge 1)}\right\Vert_{C^{b}\left(U_{2r}\right)}$ is small enough, so that
\eqref{eqJi1} can be solved. Since 
$\left(\overline{\partial}+B_{i-1}\right)^{2}=0$ and since $B_{i-1}$ does not contain
$d\overline{z}^{1},\ldots,d\overline{z}^{i-1}$, $B_{i-1}$ is 
a holomorphic function  of 
$z^{1},\ldots,z^{i-1}$. The arguments that were used before show that 
the solution $J_{i}$ is also holomorphic in $z^{1},\ldots,z^{i-1}$.

As explained after (\ref{eq:lab0}), for $r>0$ small enough, on 
$U=U_{r}$,  
 there exists a smooth section $\underline{v}_{0}$ of 
$\End^{0}\left(D\right)$ such that
\begin{equation}\label{eq:comp1}
J^{-1}A^{E_{0},\prime \prime }\vert_{U}J=\overline{\pa}+\underline{v}_{0}.
\end{equation}
Let $J^{(0)}$ be the component of $J$ of degree $0$ in 
$\Lambda\left(\overline{\TsX}\right)$. By (\ref{eq:iv36a6}), (\ref{eq:comp1}), we get
\begin{equation}\label{eq:comp2}
\left[J^{(0)}\right]^{-1}v_{0}J^{(0)}=\underline{v}_{0}.
\end{equation}
Put
\begin{align}\label{eq:comp3}
&K=J\left[J^{(0)}\right]^{-1},&\underline{\n}^{D\vert_{U} \prime \prime 
}=J^{(0)}\overline{\pa}\left[J^{(0)}\right]^{-1}.
\end{align}
Then $\underline{\n}^{D\vert_{U} \prime \prime }$ is a holomorphic structure 
on $D\vert_{U}$. By (\ref{eq:comp1})--(\ref{eq:comp3}), we get
\begin{equation}\label{eq:comp5}
K^{-1}A^{E_{0} \prime \prime }K=\underline{\n}^{D\vert_{U} \prime \prime 
}+v_{0},
\end{equation}
which completes the proof of our theorem.
\end{proof} 
\subsection{Antiholomorphic superconnections and coherent sheafs}%
\label{subsec:applco}
Let $F$ be a smooth vector bundle on $X$,  and let
\index{OXF@$\mathcal{O}_{X}^{\infty }\left(F\right)$}%
$\mathcal{O}_{X}^{\infty }\left(F\right)$ be  
the sheaf of smooth sections of $F$. 
\begin{definition}\label{def:Ds}
	Let 
	\index{E@$\mathscr E$}%
	$\mathscr E$ be the  sheaf of $\mathcal{O}_{X} $-complexes  
	$\left(\mathcal{O}_{X}^{\infty }\left(E\right),\As\right)$,	 and let 
	\index{HE@$\mathscr H \mathscr E$}%
	$\mathscr H \mathscr E$ denote 
	its cohomology.
	For $r\ge 0$,  let $\left( \mathscr E_{r}, d_{r}\right) $ denote the  spectral 
	sequence of $\mathcal{O}_{X}$-sheaves associated with the 
	filtration by 
	$\mathcal{O}_{X}^{\infty}\left(\Lambda\left(\overline{\TsX}\right)\right)$.
	
	Let 
	\index{HD@$HD$}%
	$HD$ be the cohomology of  $\left( \mathcal{O}_{X}^{\infty 
}\left(D\right),v_{0}\right)$.
\end{definition}

The sheaf $\mathscr H\mathscr E$ is a $\Z$-graded filtered sheaf of 
$\mathcal{O}_{X}$-modules.
 Also 
\begin{align}\label{eq:do1}
&\left(\mathscr E_{0},d_{0}\right)=\left( \mathcal{O}_{X}^{\infty 
}\left(\Lambda\left(\overline{\TsX}\right)\ho D\right),v_{0} \right), &\mathscr E_{1}=\mathcal{O}_{X}^{\infty }\left( 
\Lambda \left(\overline{\TsX}\right) \right) \ho_{\mathcal{O}_{X}^{\infty }} HD.
\end{align}
The differential $d_{1}$ on $\mathscr E_{1}$ is such that 
$d_{1}^{2}=0$.  Also it verifies Leibniz rule with respect to 
multiplication by $\Omega^{0,\scriptsize\bullet}\left(X,\C\right)$. 
Therefore $d_{1}$ induces a holomorphic structure on the 
$\mathcal{O}_{X}^{\infty }$-module $HD$, which will 
also be denoted
\index{nHD@$\n^{HD \prime \prime}$}%
$\n^{HD \prime \prime}$.
\begin{remark}\label{rem:locf}
	Assume that $HD$ has locally constant rank. Then $HD$ is a smooth 
	vector bundle on $X$.  Then $\n^{HD \prime \prime }$ defines a canonical 
 holomorphic structure on $HD$. 
\end{remark}

When fixing a splitting in (\ref{eq:iv4a-1}), using (\ref{eq:iv4}),  and writing 
$A^{E_{0} \prime 
\prime }$ as in (\ref{eq:iv36a6}), from the three equations in 
(\ref{eq:carr1}), we can also derive a construction of $\n^{HD \prime 
\prime }$. Indeed, by (\ref{eq:carr1}),  the connection $\n^{D \prime \prime }$ induces a holomorphic 
structure on the $\mathcal{O}_{X}^{\infty }$-module $HD$, which is just $d_{1}=\n^{HD \prime 
\prime }$.  This holomorphic structure is canonical, and does not 
depend on the splitting in (\ref{eq:iv4a-1}), as is also clear from 
(\ref{eq:roug}).
\begin{definition}\label{def:cohdH}
	 Let 
	\index{HHD@$\mathscr H HD$}%
	 $\mathscr H HD$ be the  cohomology 
 of  the complex of $\mathcal{O}_{X}$-modules $ \left( \mathcal{O}_{X}^{\infty 
 }\left(\Lambda\left(\overline{\TsX}\right)\right)\ho_{\mathcal{O}_{X}^{\infty}} HD,\n^{HD 
 \prime \prime }\right) $.
\end{definition}
Then $\mathscr H HD$ is a $\Z$-graded $\mathcal{O}_{X}$-module. By 
(\ref{eq:do1}), 
\begin{equation}\label{eq:cons1}
\mathscr E_{2}=\mathscr H  HD.
\end{equation}

We recall a fundamental result of Block \cite[Lemma 4.5]{Block10}.
\begin{theorem}\label{thm:cohco}
	The complex $\mathscr E$ has $\Z$-graded coherent cohomology $\mathscr H 
	\mathscr E$. The spectral sequence $ \mathscr E_{r}$ degenerates at 
	$\mathscr E_{2}= \mathscr H \mathscr E$, and $F^{1} \mathscr H 
	\mathscr E=0$. 
	
	For $i\ge 1$, 
	\begin{equation}\label{eq:cons2-a}
\mathscr H^{i}HD=0.
\end{equation}
We have the identity of $\Z$-graded coherent sheaves,
	\begin{equation}\label{eq:cons2}
\mathscr H \mathscr E= \mathscr H HD,
\end{equation} 
and also the identity of $\mathcal{O}_{X}^{\infty }$-modules,
\begin{equation}\label{eq:cons2a1}
\mathscr H \mathscr E \otimes 
_{\mathcal{O}_{X}}\mathcal{O}_{X}^{\infty }=HD.
\end{equation}
	\end{theorem}
\begin{proof}
	By Theorem \ref{thm:Conj}, if $U$ is sufficiently 
small, after conjugation by a smooth  section of 
$\Aut^{0}_{0}\left(E_{0}\right)$ on $U$, we may and we 
will assume that on $U$, $A^{E_{0} \prime \prime }$ has the form
\begin{equation}\label{eq:iv8}
A^{E_{0} \prime \prime }=v_{0}+\n^{D \prime \prime}.
\end{equation}
In (\ref{eq:iv8}), $\n^{D \prime \prime }$ is a holomorphic 
structure on $D\vert_{U}$ that preserves the $\Z$-grading and is 
such that
\begin{equation}\label{eq:iv8a1}
\n^{D \prime \prime }v_{0}=0.
\end{equation}
In (\ref{eq:iv8}), $A^{E_{0}\prime \prime }$ now preserves the 
decreasing filtration associated with the degree in $D$, which was 
not the case in the original form of $A^{E_{0} \prime \prime }$. We 
will compute $\mathscr H \mathscr E$ using the associated spectral 
sequence.

Let $\mathscr HD$ be the cohomology of 
$\left(\mathcal{O}_{U}\left(D\right),v_{0}\right)$.
 By \cite[Theorems II.3.13 and II.3.14]{Demaillylivre}, $\mathscr HD$ is a coherent sheaf.
Using the Poincaré lemma for Dolbeault cohomology \cite[Lemma 
I.3.29]{Demaillylivre}, if  $V$ is a small polydisk, 
 then
 \begin{equation}\label{eq:iv9}
\mathscr H\mathscr 
E\left(V\right)=\mathscr HD\left(V\right),  
\end{equation}
so that
\begin{equation}\label{eq:iv9a1}
\mathscr H \mathscr E= \mathscr H D.
\end{equation}
Therefore $\mathscr H\mathscr E$ is a $\Z$-graded coherent sheaf. Also, 
the natural
filtration $F$ is trivial on $\mathscr H \mathscr E$, so that $F^{1}\mathscr H 
\mathscr E=0$, a property which is intrinsic to $\mathscr H \mathscr 
E$, and does not depend on the choice of a conjugation.

By \cite[Corollary VI.1.12]{Malgrange67}, $\mathcal{O}_{X}^{\infty }$ 
is flat on $\mathcal{O}_{X}$, so that the map
\begin{equation}\label{eq:iv11}
\rho: \mathscr H D\otimes _{\mathcal{O}_{U}}\mathcal
O_{U}^{\infty }\to HD
\end{equation}
is an isomorphism of $\mathcal{O}_{U }^{\infty }$-modules.  Moreover, 
the action of $d_{1}$ on $HD$ is just $\n^{D 
\prime \prime }$. Using the Poincaré  lemma, we deduce that for $i\ge 
1$,  $\mathscr 
H^{i}HD=0$, and also that
\begin{equation}\label{eq:iv11a1}
\mathscr H D= \mathscr H H D.
\end{equation}

By the above, the 
cohomology of $\left(\mathscr E_{1},d_{1}\right)$ coincides with $\mathscr H D= \mathscr H \mathscr E$. Therefore, the original 
spectral sequence degenerates at $\mathscr E_{2}$, and  $\mathscr E_{2}=\mathscr H 
\mathscr E$. Using (\ref{eq:iv9a1}), (\ref{eq:iv11}), we get 
(\ref{eq:cons2a1}). The proof of our theorem is completed. 
\end{proof}
\begin{remark}\label{rem:loccst}
	Assume that $HD$ has locally constant rank. As we saw in Remark 
	\ref{rem:locf}, $HD$ is equipped with a canonical holomorphic 
	structure $\n^{HD \prime \prime }$. By Theorem \ref{thm:cohco}, 
	we get
	\begin{equation}\label{eq:locrk}
\mathscr H\mathscr E=\mathcal{O}_{X}\left(HD\right).
\end{equation}

In general,  $HD$ depends only on $D,v_{0}$, and  $\mathscr H HD$ 
depends only on $D,v_{0}, \n^{D \prime \prime}$. By (\ref{eq:cons2}), 
 $\mathscr H \mathscr E$ depends only on $D,v_{0},\n^{D \prime \prime}$, and 
not on the full $A^{E_{0} \prime \prime }$.
\end{remark}
\subsection{The determinant line bundle}%
\label{subsec:detli}
Since $ \mathscr H \mathscr E=\mathscr H HD$ is a $\Z$-graded coherent 
	sheaf, by Knudsen-Mumford \cite{KnudsenMumford}, its determinant 
	\index{dHE@$\det \mathscr H \mathscr E$}%
	$\det \mathscr H \mathscr E$ is a canonically defined holomorphic 
	line bundle on $X$.
	
Let 
\index{dD@$\det D$}%
$\det D$ denote the smooth determinant line bundle, 
\begin{equation}\label{eq:det0}
\det D=\bigotimes_{r}^{r'} \left( \det D^{i}\right) ^{\left(-1\right)^{i}}.
\end{equation}

Now we fix identifications  in (\ref{eq:iv4a-1}), (\ref{eq:iv4}), 
and we write $A^{E_{0} \prime \prime }$ as in (\ref{eq:iv36a6}). 
 The 
antiholomorphic connection $\n^{D \prime \prime }$ on $D$ induces a 
corresponding antiholomorphic connection 
\index{ndet@$\n^{\det D \prime \prime}$}%
$\n^{\det D \prime \prime}$ on $\det D$. 
\begin{theorem}\label{thm:pdet}
	The antiholomorphic connection $\n^{\det D \prime \prime }$ on 
	$\det D$ does not depend on the splitting in (\ref{eq:iv4a-1}), 
	(\ref{eq:iv4}). It defines a holomorphic structure on $\det D$. 
Finally, we have the canonical isomorphism of holomorphic line 
bundles on $X$, 
\begin{equation}\label{eq:sq1a1}
\det D \simeq \det \mathscr H \mathscr E.
\end{equation}
\end{theorem}
\begin{proof}
We use the notation in  (\ref{eq:carr1a1}), (\ref{eq:roug}). Let 
$\underline{\n}^{\det D \prime \prime }$ denote the antiholomorphic 
connection on $\det D$ associated with $\underline{\n}^{D}$. By 
(\ref{eq:roug}), we obtain
\begin{equation}\label{eq:sq2}
\underline{\n}^{\det D \prime \prime }=\n^{\det D \prime \prime 
}+\Trs\left[\left[\alpha,v_{0}\right]\right].
\end{equation}
Since supertraces vanish on supercommutators, by (\ref{eq:sq2}), we 
get
\begin{equation}\label{eq:sq3}
\underline{\n}^{\det D\prime \prime }=\n^{\det D \prime \prime },
\end{equation}
which proves the first part of our proposition. Also
\begin{equation}\label{eq:sq4}
\n^{\det D \prime \prime ,2}=\Trs\left[\n^{D \prime \prime ,2}\right].
\end{equation}
Using the third identity in (\ref{eq:carr1}), (\ref{eq:sq4}), and the 
fact that supertraces vanish on supercommutators, we get  
	\begin{equation}\label{eq:sq1}
\n^{\det D \prime \prime, 2}=0.
\end{equation}
so that $\n^{\det D \prime \prime }$ is a holomorphic structure on $\det D$. 

Now we use the notation in the proof of Theorem \ref{thm:cohco}. On a small open set $U$ in $X$, after a local conjugation by a smooth 
section of $\Aut^{0}_{0}\left(E_{0}\right)$, we can write $A^{E_{0} 
\prime \prime}$ in the form (\ref{eq:iv8}),  $\n^{D \prime \prime }$ 
being now a 
holomorphic structure on $D\vert_{U}$. By Knudsen-Mumford \cite{KnudsenMumford}, on $U$, 
the holomorphic line bundle $\det \mathscr H \mathscr E$ is just the 
holomorphic line bundle $\det D$ equipped with the holomorphic 
structure induced by this specific $\n^{D \prime  \prime}$. As we saw before, this 
is just our original $\n^{\det D \prime \prime }$ whose construction 
does not depend on the conjugation. The proof of our theorem is completed. 
\end{proof}
\subsection{The case where $\mathscr H \mathscr E=0$}%
\label{subsec:acyc}
We will now give conditions under which $\mathscr H \mathscr E=0$.
\begin{theorem}\label{thm:simp}
	The $\mathcal{O}_{X}$-module $\mathscr H \mathscr E$ vanishes if 
	and only if there exists a smooth section $k_{0}$ of degree $-1$ 
	of $\End\left(D\right)$ such that 
	\begin{equation}\label{eq:locrk1}
\left[v_{0},k_{0}\right]=1.
\end{equation}
An 
	equivalent condition is that 
 for any $x\in X$, the complex 
	$\left(D,v_{0}\right)_{x}$ is exact. 
	
	Another equivalent condition is that there exists a smooth section $k$ of 
	$\End\left(E\right)$ of degree $-1$ such that
	\begin{equation}\label{eq:exa1}
\left[A^{E \prime \prime },k\right]=1.
\end{equation}
\end{theorem}
\begin{proof}
By Theorem \ref{thm:cohco}, $\mathscr H \mathscr E=0$ if and only if  
we have the identity of $\mathcal{O}_{X}^{\infty }$-modules 
$HD=0$. Since $D$ is soft, by \cite[Proposition 
IV.4.14]{Demaillylivre}, this is equivalent to the exactness of the 
complex $\left(C^{\infty }\left(X,D\right),v_{0}\right)$. Since 
$C^{\infty }\left(X,D\right)$ is projective over $C^{\infty 
}\left(X,\C\right)$,  this is 
equivalent to the existence of a smooth section $k_{0}$  of 
$\End\left(D\right)$ of degree $-1$ such that 
$\left[v_{0},k_{0}\right]=1$, which implies that for any $x\in X$, 
the complex $\left(D,v_{0}\right)_{x}$ is exact. Conversely, if this last 
assumption is verified, we find easily there is a global smooth 
$k_{0}$ such that $\left[v_{0},k_{0}\right]=1$, from which we deduce 
that $HD=0$.

To complete the proof of our theorem, we  need to show that if $\mathscr H \mathscr E=0$, then $k$ 
exists. We use a splitting as in (\ref{eq:iv4a-1}), (\ref{eq:iv4}), 
we replace $E$ by $E_{0}$, and we write $A^{E_{0} \prime \prime }$ as 
in (\ref{eq:iv36a6}).   We write $k$ in the form
\begin{equation}\label{eq:exa2}
k=\sum_{i\ge 0}^{}k_{i},
\end{equation}
with $k_{i}$ a smooth section of 
$\Lambda^{i}\left(\overline{\TsX}\right)\ho\End\left(D\right)$ of degree $-1$.  We choose 
$k_{0}$ as before. 

If $\alpha$ is a smooth section of 
$\Lambda\left(\overline{\TsX}\right)\ho\End\left(D\right)$, we denote by $\alpha^{(\le i)}$ 
the sum of the components of $\alpha$ whose degree in 
$\Lambda\left(\overline{\TsX}\right)$ is $\le i$, and by $\alpha^{(i)}$ 
the component that lies in 
$\Lambda^{i}\left(\overline{\TsX}\right)\ho \End\left(D\right)$. 

Assume that we 
found $k_{j}, j\le i$ such that
\begin{equation}\label{eq:exa3}
\left[A^{E_{0} \prime \prime },\sum_{j\le i}^{}k_{j}\right]^{(\le i)}=1.
\end{equation}
We will show how to construct $k_{i+1}$ so that (\ref{eq:exa3}) is 
solved with $i$ replaced by $i+1$.  Assuming that 
(\ref{eq:exa3}) has been solved, the corresponding equation with $i$ 
replaced by $i+1$ can be written in the form, 
\begin{equation}\label{eq:exa4}
	\left[v_{0},k_{i+1}\right]+\left[A^{E_{0} \prime \prime },\sum_{j\le 
i}^{}k_{j}\right]^{(i+1)}=0.
\end{equation}
Because of (\ref{eq:exa3}), we get
\begin{equation}\label{eq:exa4a1}
\left[v_{0},\left[A^{E_{0} \prime \prime },\sum_{j\le 
i}^{}k_{j}\right]\right]^{(i+1)}=\left[A^{E_{0}\prime \prime },\left[A^{E_{0} \prime \prime },\sum_{j\le 
i}^{}k_{j}\right]\right]^{(i+1)}=0.
\end{equation}
By (\ref{eq:exa4a1}), to solve (\ref{eq:exa4}), we can take 
\begin{equation}\label{eq:exa4a2}
k_{i+1}=-\left[k_{0},\left[A^{E_{0} \prime \prime },\sum_{j\le 
i}^{}k_{j}\right]^{(i+1)}\right].
\end{equation}
This completes the proof of the existence of $k$. 

Another more direct proof is as follows. The vector bundle 
$\End\left(E\right)$ can be equipped with the antiholomorphic 
superconnection $\ad\left(A^{E \prime \prime}\right)$. Moreover, 
$1$ is a section of $\End\left(E\right)$ such that
\begin{equation}\label{eq:exa4a3}
\left[A^{E \prime \prime },1\right]=0.
\end{equation}
We claim that the complex $\left(C^{\infty 
}\left(X,\End\left(E\right)\right),\ad\left(A^{E \prime \prime 
}\right)\right) $ is exact, 
which implies the existence of $k$.
  Indeed let $L\left(k_{0}\right)$ denote 
left multiplication by $k_{0}$ acting on $\End\left(D\right)$. Then 
\begin{equation}\label{eq:exa4a3-x}
\left[\ad\left(v_{0}\right),L\left(k_{0}\right)\right]=1, 
\end{equation}
so that the  complex 
$\left(\End\left(D\right),\ad\left(v_{0}\right)\right)$ is exact. Consider the global spectral 
sequence associated with the filtration $F$. By (\ref{eq:exa4a3-x}),  the first term of the 
spectral sequence vanishes identically, so that the complex 
$\left(C^{\infty 
}\left(X,\End\left(E\right)\right),\ad\left(A^{E \prime \prime 
}\right)\right) $  is exact.
The proof of our theorem is completed. 
\end{proof}
\subsection{Superconnections, morphisms, and cones}%
\label{subsec:scmoco}
Let $\left(\underline{E},A^{\underline{E} \prime \prime }\right)$ be 
another couple  similar 
to $\left(E,A^{E \prime \prime }\right)$. The objects associated with 
$\underline{E}$ will be underlined. A morphism $\phi: E\to 
\underline{E}$ is a smooth morphism of degree $0$ of 
$\Z$-graded $\Lambda\left(\overline{\TsX}\right)$-vector bundles, 
which is such that $A^{\underline{E} \prime \prime}\phi=\phi A^{E \prime 
\prime }$. Then $\phi$  induces a morphism of 
$\mathcal{O}_{X}$-complexes 
$\mathscr E\to \underline{\mathscr E}$ that commutes with 
multiplication by $\mathcal{O}^{\infty 
}_{X}\left(\Lambda\left(\overline{\TsX}\right)\right)$. Then $\phi$ induces a 
morphism  $E_{0}\to 
\underline{E}_{0}$,  a morphism of $\mathcal{O}_{X}^{\infty }$-complexes 
$\left(D,v_{0}\right)\to \left( \underline{D}, \underline{v}_{0} 
\right) $,  a morphism of $\mathcal{O}_{X}$-complexes
  $\left(HD,\n^{HD \prime \prime }\right)\to 
\left(H\underline{D},\n^{H\underline{D}\prime \prime }\right)$, and 
morphisms
of coherent sheaves 
$\mathscr H\mathscr E\to \mathscr H\underline{\mathscr E}$. 

Here, we follow the conventions of Subsection \ref{subsec:moco}. We form the cone
\begin{equation}\label{eq:hue1}
\mathscr C=\mathrm{cone}\left(\mathscr E, \underline{\mathscr 
E}\right).
\end{equation}
By (\ref{eq:res4a0z1b}), the underlying vector bundle $C$ is such that
\begin{equation}\label{eq:hue2}
C^{\Ou}=E^{\Ou+1} \oplus \underline{E}^{\Ou}.
\end{equation}
Then $C$ is an object similar to $E$. The corresponding 
antiholomorphic superconnection is denoted $A^{C \prime \prime }_{\phi}$, 
and as in (\ref{eq:res4a0z2b}), we have the identity,
\begin{equation}\label{eq:res4a0z2bx}
A^{C \prime \prime }_{\phi}=
\begin{bmatrix}
	A^{E \prime \prime } & 0 \\
	\phi\left(-1\right)^{\deg} & A^{\underline{E}\prime \prime}.
\end{bmatrix}
\end{equation}

As in (\ref{eq:res4a-1b}), we have the exact sequence of complexes,
\begin{equation}\label{eq:loex1}
\xymatrix{&0\ar[r] &\underline{\mathscr E}^{\Ou}\ar[r] &\mathscr 
C^{\Ou}\ar[r] &\mathscr E^{\Ou+1}\ar[r] &0,} 
\end{equation}
from which we get the long exact sequence of 
$\mathcal{O}_{X}$-modules,
\begin{equation}\label{eq:loex}
\xymatrix{
\ldots\ar[r] &\mathscr H^{\Ou}\underline{\mathscr E}\ar[r] &\mathscr H^{\Ou} 
\mathscr C
\ar[r] & \mathscr H^{\Ou+1}\mathscr E\ar[r]^{\phi\left(-1\right)^{\cdot+1}}
&\mathscr H^{\Ou+1}\underline{\mathscr E}\ldots}.
\end{equation}
We have corresponding exact sequences involving the $\mathcal{O}_{X 
}^{\infty }$-modules $HD$ and their pointwise version.

If $\phi:\mathscr H \mathscr E\to \mathscr H\underline{\mathscr 
E}$ is an isomorphism, then $\phi$ is said to be a 
quasi-isomorphism. As we saw in Subsection \ref{subsec:moco},  this 
is equivalent to
\begin{equation}\label{eq:hue3}
\mathscr H \mathscr C=0.
\end{equation}
Using Theorem \ref{thm:simp}, we find  that $\phi$ is 
a quasi-isomorphism if and only if for any $x\in X$, 
$\phi:D_{x}\to \underline{D}_{x}$ is a quasi-isomorphism.

Given $t\in \C$, we can still define $M_{t}\in \End\left(C\right)$ as 
in (\ref{eq:for1}). Namely, $M_{t}$ acts by multiplication by $1$ on 
$E$, and by multiplication by $t$ on $\underline{E}$. As in 
(\ref{eq:for2}),  for $t\neq 0$, we have the identity
\begin{equation}\label{eq:co1}
A^{C \prime \prime }_{t\phi}=M_{t}A^{C \prime \prime 
}_{\phi}M_{t}^{-1}.
\end{equation}
\subsection{Pull-backs}%
\label{subsec:pullten}
Let $Y$ be a compact complex manifold, and let $f:X\to Y$ be a
holomorphic map.  If $H$ is a vector bundle on $Y$, $f^{*}H$ denote the 
pull-back of $H$ to $X$. Then $df^{*}$ maps $f^{*}T^{*}Y$ into $T^{*}X$. 
 Therefore 
$\Lambda \left(\overline{T^{*}X}\right)$  is a 
$\Lambda \left(\overline{f^{*}T^{*}Y}\right)$-module.

We use the notation of Subsection \ref{subsec:def}, except that now 
$\left(F,A^{F \prime \prime}\right)$ is an antiholomorphic 
superconnection on $Y$ with diagonal bundle $D_{F}$. Put
\begin{equation}\label{eq:iv12a-1}
	E=\Lambda \left(\overline{T^{*}X}\right)\ho
	_{\Lambda \left(\overline{f^{*}T^{*}Y}\right)}f^{*}F .
\end{equation}
Then $E$ is a $\Z$-graded vector bundle on $X$, the degree in $E$ 
being the sum of the degrees in the two factors of the right 
hand-side. Also $E$ is   a 
$\Lambda \left(\overline{T^{*}X}\right)$-module, so that it is 
equipped with the corresponding filtration. 

By (\ref{eq:bc12}), 
we get
\begin{equation}\label{eq:iv12a-2}
E_{0}^{p,q}=\Lambda^{p}\left(\overline{T^{*}X}\right) \ho 
f^{*}D_{F}^{q}.
\end{equation}
In particular,  $E$ on $X$ has exactly the same properties as $F$ on 
$Y$, and the associated diagonal bundle is $D_{E}=f^{*}D_{F}$.

Let $\mu$ be the canonical morphism  $f^{*}F\to E$. Then $\mu$ induces a 
corresponding map $C^{\infty }\left(Y,F\right)\to C^{\infty 
}\left(X,E\right)$ which is such that if $\alpha\in 
\Omega^{0,\Ou}\left(Y,\C\right), s\in C^{\infty }\left(Y,F\right)$, then
\begin{equation}\label{eq:iv12a-3}
\mu\left(\alpha s\right)=df^{*}\left(\alpha\right)\mu s.
\end{equation}
\begin{proposition}\label{prop:rest}
	There is a unique antiholomorphic superconnection $A^{E \prime \prime }$  
	on $E$  such that if $s\in 
	C^{\infty }\left(Y,F\right)$, then 
	\begin{equation}\label{eq:iv12a-4}
A^{E \prime \prime }\left( \mu s \right) =\mu\left(A^{F \prime \prime }s\right).
\end{equation}
\end{proposition}
\begin{proof}
	Let $\alpha\in \Omega^{0,\Ou}\left(Y,\C\right)$. By 
	(\ref{eq:iv5}), (\ref{eq:iv12a-3}), we get
	\begin{equation}\label{eq:iv12a-5}
\mu\left(A^{F \prime \prime }\alpha 
s\right)=\overline{\pa}^{X}\left( df^{*}\alpha \right)  \mu \left(s\right)+
\left(-1\right)^{\deg\,\alpha} \left( df^{*}\alpha\right)\mu\left(A^{F \prime \prime 
}s\right).
\end{equation}
In particular, if $df^{*}\alpha=0$, (\ref{eq:iv12a-5}) vanishes. This 
completes the proof of our proposition.
\end{proof}

We will use the notation
\index{fb@$f_{b}^{*}$}%
\begin{equation}\label{eq:notz1}
E=f^{*}_{b}F.
\end{equation}
The notation emphasizes the fact that $f^{*}_{b}F$ is not the 
classical pull-back $f^{*}F$.

Similarly, if  $\mathscr E=\left(E,A^{E \prime \prime }\right), 
\mathscr F=\left(F,A^{F \prime \prime}\right)$, we use the notation
\begin{equation}\label{eq:iv12a-7}
\mathscr E=f^{*}_{b} \mathscr F.
\end{equation}
There is an associated morphism of $\Z$-graded 
$\mathcal{O}_{X}$-modules $\mathcal{O}_{X} \otimes 
_{f^{-1}\mathcal{O}_{Y}}f^{-1}\mathscr H \mathscr F 
\to \mathscr H \mathscr E $.

If $\phi:\mathscr F\to \mathscr F'$ is a morphism, it  
induces a morphism $f^{*}_{b}\phi: f^{*}_{b}\mathscr F\to f^{*}_{b} 
\mathscr F'$.

We  fix a smooth splitting as in (\ref{eq:iv4a-1}), that induces  the 
non-canonical identification $F \simeq F_{0}$ in (\ref{eq:iv4}). This choice induces a corresponding 
  identification $E \simeq E_{0}$. Also, $\mu$ induces a canonical map $f^{*}F_{0}\to E_{0}$.
We write 
$A^{F_{0} \prime \prime }$
 as in (\ref{eq:baba2}). 

 The connection 
$\n^{D_{F} \prime \prime }$ induces a connection $\n^{f^{*}D_{F} \prime 
\prime }$ on $f^{*}D_{F}$. Let $\mu B$ be the section of 
 $\Lambda\left(\overline{T^{*}X}\right)\ho\End\left(f^{*}D_{F}\right)$ deduced 
 from $B$ by the above pull-back map.  Then
\begin{equation}\label{eq:iv12a-6}
A^{E_{0} \prime \prime }=\n^{f^{*}D_{F} \prime \prime 
}+\mu B.
\end{equation}

If $\alpha\in\Omega^{0,\scriptsize\bullet}\left(Y,\C\right)$, from 
now on, we will also use the  
notation $f^{*}\alpha$ instead of $df^{*}\alpha$.
\subsection{Tensor products}%
\label{subsec:tenspro}
Let $\mathscr E=\left(E,A^{E \prime \prime }\right), \mathscr F=\left(F,A^{F \prime \prime 
}\right)$ be two antiholomorphic superconnections  on $X$, with associated diagonal vector bundles 
$D_{E},D_{F}$. Put
\index{EbF@$E\ho_{b}F$}%
\begin{equation}\label{eq:redten1}
E\ho_{b}F=E\ho_{\Lambda\left(\overline{\TsX}\right)}F.
\end{equation}
This defines an  antiholomorphic superconnection 
\index{EbF@$\mathscr E\ho_{b} \mathscr F$}%
$\mathscr E\ho_{b} \mathscr F=\left(E\ho_{b} 
F,A^{E\ho_{b} F \prime \prime }\right)$.  The corresponding 
diagonal vector bundle is  $D_{E}\ho D_{F}$.
\section{An equivalence of categories}%
\label{sec:blo}%
The purpose of this Section is to establish a result of Block 
\cite{Block10}, who showed there is an equivalence of categories 
between the homotopy
category of antiholomorphic superconnections on $X$ and 
$\Db\left(X\right)$.

This Section is organized as follows. In Subsection 
\ref{subsec:prodb}, if $\mathscr F$ is an object in 
$\Db\left(X\right)$, we construct a corresponding 
$\mathcal{O}_{X}$-complex of $\mathcal{O}_{X}^{\infty }$-modules  
$\overline{\mathscr F}^{\infty }$, and we show it is 
quasi-isomorphic to $\mathscr F$.

In Subsection \ref{subsec:thrca}, we introduce the $\mathrm{dg}$-category 
$\mathrm{B}_{\mathrm{dg}}\left(X\right)$ of the antiholomorphic 
superconnections $\mathscr E$, the $0$-cycle category 
$\mathrm{B}\left(X\right)$, and the homotopy category 
$\underline{\mathrm{B}}\left(X\right)$. In the next Subsections, we construct 
an equivalence of categories between $\underline{\mathrm{B}}\left(X\right)$ 
and $\Db\left(X\right)$.

In Subsection  \ref{subsec:essur}, we show that the natural functor 
$F_{X}:\mathrm{B}\left(X\right)\to \Db\left(X\right)$ is essentially 
surjective.

In Subsection \ref{subsec:hoca}, we describe in more detail the 
homotopy category $\underline{\mathrm{B}}\left(X\right)$.

In Subsection \ref{subsec:eqcat}, we show that the induced functor 
$\underline{F}_{X}=\underline{\mathrm{B}}\left(X\right)\to\Db\left(X\right)$ is an equivalence of categories.

In Subsection \ref{subsec:puba}, we show the above equivalence of 
categories is compatible with pull-backs.

In Subsection \ref{subsec:tepro}, we prove the compatibility of this 
equivalence with tensor products.

Finally, in Subsection \ref{subsec:dirim}, we consider the case of 
direct images.
\subsection{Some properties of $\Db\left(X\right)$}%
\label{subsec:prodb}
Let $\mathscr  F\in {\rm D_{\mathrm{coh}}^{b}}(X)$, and let $d^{\mathscr F}$ 
be its differential. 
Set 
\index{F@$\mathscr   F^{\infty}$}%
\begin{equation}\label{eq:vl0}
	\mathscr   F^{\infty}=\mathcal O_{X}^{\infty}\otimes_{\mathcal 
	O_{X}}\mathscr  F.
\end{equation} 
Then  $\mathscr F^{\infty}$ is a perfect $\mathcal  
O_{X}^{\infty}$-complex. Let $d^{\mathscr F^{\infty}}$ be the 
differential on $\mathscr F^{\infty }$, and let $H \mathscr F^{\infty 
}$ be the cohomology of $\mathscr F^{\infty }$.
Since $\mathcal 
O_{X}^{\infty}$ is flat over $\mathcal O_{X}$ \cite[Corollary 
VI.1.12]{Malgrange67}, we have 
\begin{equation}\label{eq:cov1}
	H\mathscr{F}^{\infty}=\mathcal O_{X}^{\infty}\otimes_{\mathcal 
	O_{X}}\mathscr{HF}. 
\end{equation} 

 Put
 \index{F@$\overline{\mathscr F}^{\infty}$}%
 \begin{equation}\label{eq:coi0}
\overline{\mathscr F}^{\infty}=\mathcal{O}_{X}^{\infty 
}\left(\Lambda\left(\overline{\TsX}\right)\right)\ho_{\mathcal{O}_{X}} \mathscr 
F.
\end{equation}
Since $\mathscr F$ is a complex of $\mathcal{O}_{X}$-modules, we can 
equip $\overline{\mathscr F}^{\infty }$ with the differential 
\index{dF@$d^{\overline{\mathscr F}^{\infty}}$}%
$d^{\overline{\mathscr F}^{\infty}}$ given by
\begin{equation}\label{eq:exp1}
d^{\overline{\mathscr F}^{\infty}}=\overline{\pa}^{X}+d^{\mathscr 
F}.
\end{equation}

 By Poincaré Lemma and by a theorem of 
Malgrange \cite[Corollary 
VI.1.12]{Malgrange67}, we have a quasi-isomorphism of $\mathcal O_{X}$-complexes,
	\begin{equation}\label{eq:coq1}
	\mathscr F\to 	\overline{\mathscr F}^{\infty }.
	\end{equation}
Equivalently $\mathscr H \mathscr F, \mathscr H \overline{\mathscr 
	F}^{\infty }$ denote the corresponding cohomology sheaves, we 
	have the canonical  isomorphism of $\mathcal{O}_{X}$-modules,
	\begin{equation}\label{eq:coq2}
\mathscr H \mathscr F=\mathscr H \overline{\mathscr F}^{\infty }.
\end{equation}

	Also $ \overline{\mathscr F}^{\infty }$ is equipped with the filtration 
	induced by  $\Lambda\left(\overline{\TsX}\right)$. Let 
	$\overline{\mathscr F}^{\infty }_{r},r\ge 0$ denote the corresponding 
	spectral sequence of $\mathcal{O}_{X}$-modules.
	
	Then
	\begin{equation}\label{eq:do1bis}
\overline{\mathscr F}^{\infty }_{1}=\mathcal{O}_{X}^{\infty 
}\left(\Lambda\left(\overline{\TsX}\right)\right)\ho_{\mathcal{O}^{\infty }_{X}}  H \mathscr F^{\infty}.
\end{equation}
Also $\overline{\mathscr F}^{\infty 
}_{1}$ inherits a differential $d_{1}$, which can be viewed as a holomorphic 
structure $\n^{ H\mathscr F^{\infty 
} \prime \prime }$ on $ H\mathscr F^{\infty 
}$. Let $\mathscr H \overline{\mathscr F}^{\infty 
}_{1}$ denote the cohomology of $\left(\overline{\mathscr F}^{\infty 
}_{1},d_{1}\right)$. Then 
\begin{equation}\label{eq:do1bisa}
\overline{\mathscr F}^{\infty 
}_{2}=\mathscr H \overline{\mathscr F}^{\infty 
}_{1}.
\end{equation}
\begin{proposition}\label{prop:coco}
		The spectral sequence $\overline{\mathscr F}^{\infty }_{r}$ 
		degenerates at $\overline{\mathscr F}^{\infty }_{2}$, and 
		$F^{1} \mathscr H\overline{\mathscr F}^{\infty}=0$.
		
		For $i\ge 1$, 
	\begin{equation}\label{eq:do1bisb}
\mathscr H^{i}\overline{\mathscr F}^{\infty 
}_{1}=0.
\end{equation}
We have the identity of $\Z$-graded coherent sheaves,
\begin{equation}\label{eq:dos1bisc}
\mathscr H\mathscr F = \mathscr H\overline{\mathscr F}^{\infty 
}=\mathscr H^{0} \overline{\mathscr F}^{\infty }_{1}.
\end{equation}
\end{proposition}
\begin{proof}
	Using the special form of $d^{\overline{\mathscr F}^{\infty }}$ in 
	(\ref{eq:exp1}), the proof is essentially the same as the proof of 
Theorem \ref{thm:cohco}. It is left to the reader.
\end{proof}
\subsection{Three categories}%
\label{subsec:thrca}
Let $X$ be a compact complex manifold. If $\mathscr E=\left(E,A^{E 
\prime \prime }\right)$ is an antiholomorphic superconnection on $X$, 
put
\index{EX@$\mathscr E_{X}$}%
\begin{equation}\label{eq:kel1}
\mathscr E_{X}=C^{ \infty }\left(X,E\right).
\end{equation}

 Let $\mathscr E=\left(E,A^{E 
\prime \prime}\right), 
\underline{\mathscr E}=\left(\underline{E},A^{\underline{E} \prime \prime }\right)$ be  
antiholomorphic superconnections on $X$.  Let 
$\Hom\left(E,\underline{E}\right)$ be the vector bundle of 
$\Lambda\left(\overline{\TsX}\right)$-morphisms from $E$ into 
$\underline{E}$.  Then $\Hom\left(E,\underline{E}\right)$ is a 
$\Z$-graded vector bundle. If $k\in \Z$, 
$\Hom^{k}\left(E,\underline{E}\right)$ consists of the morphisms 
that increase the degree by $k$. Then $\Hom\left(E,\underline{E}\right)$ is also 
equipped with the induced antiholomorphic superconnection 
$A^{\Hom\left(E,\underline{E}  \right)\prime \prime}$. We use the 
notation $\Hom\left(\mathscr E,\underline{\mathscr E}\right)=
\left(\Hom\left(E,\underline{E}\right),A^{\Hom\left(E,\underline{E}\right)
\prime \prime}\right)$.

Following  \cite[Subsection 2.2]{Keller06}, we define   $\mathrm 
B_{\rm dg}(X)$ to be the $\mathrm{dg}$-category whose objects are the 
$\mathscr E_{X}$, and the morphisms  $\Hom\left(\mathscr 
E,\underline{\mathscr E}\right)_{X}$.

Let 
\index{BX@$\mathrm B(X)$}%
$\mathrm B(X)=Z^{0}\left(\mathrm{B}\left(X\right)\right)$ be the $0$-cycle 
category associated with $\mathrm 
B_{\mathrm{dg}}(X)$. Its objects coincide with 
the objects of $\mathrm B_{\mathrm{dg}}(X)$, and  if 
$\mathscr E, \underline{\mathscr E}$ are taken as before, its 
morphisms are given by
\index{ZH@$Z^{0}\Hom\left(\mathscr E, \underline{\mathscr E}\right)_{X}$}%
\begin{equation}\label{eq:kel2}
Z^{0}\Hom\left(\mathscr E, \underline{\mathscr E}\right)_{X}=\left\{\phi\in \Hom^{0}\left(\mathscr E, 
\underline{\mathscr E}\right)_{X}, A^{\Hom\left(E,\underline{E} 
\right)\prime \prime }\phi=0\right\}.
\end{equation}
Also $\mathscr E, \underline{\mathscr E}$ are said to be isomorphic 
if there exists $\phi\in Z^{0}\Hom\left(\mathscr 
E_{X},\underline{\mathscr E}_{X}\right),\psi\in 
Z^{0}\Hom\left(\underline{\mathscr E}_{X}, \mathscr E_{X}\right)$ such that
\begin{align}\label{eq:lep1}
&\psi\phi=1_{\mathscr E_{X}},&\phi\psi=1_{\underline{\mathscr E}_{X}}.
\end{align}

Let 
\index{BX@$\underline{\mathrm B}(X)$}%
$\underline{\mathrm B}(X)$ be the homotopy category associated 
with 
$\mathrm  
B(X)$. Its objects coincide with 
the objects of $\mathrm B(X)$, and the morphisms are given by 
the morphisms of $H^{0}\Hom\left(\mathscr E, 
\underline{\mathscr E}\right)_{X}$, the cohomology group of 
degree $0$ associated with the complex $\left(\Hom\left(\mathscr 
E, \underline{\mathscr 
E}\right)_{X},A^{\Hom\left(E,\underline{E}\right) \prime \prime}\right)$.

The aim of this section is to establish an equivalence of categories  between $\underline{\mathrm B}(X)$ and ${\rm D_{\mathrm{coh}}^{b}}(X)$. The proof 
consists of  the following three steps:
\begin{enumerate}
	\item  Construct an essentially surjective functor  $F_{X}:\mathrm 
	B(X)\to{\rm D_{\mathrm{coh}}^{b}}(X) $. 

	\item  Prove that the above functor factors though a functor 
	$\underline{F}_{X}:\underline{\mathrm 
	B}(X)\to{\rm D_{\mathrm{coh}}^{b}}(X) $.

	\item Show that $\underline{F}_{X}$ is  fully faithful. 
\end{enumerate} 

\subsection{Essential surjectivity}%
\label{subsec:essur}
Let  $\mathscr E=\left(E,A^{E\prime\prime }\right)$ be an antiholomorphic 
superconnection. By Theorem 
\ref{thm:cohco}, 
$\mathscr E\in  \Db\left(X\right)$.  We obtain this way a 
functor 
\index{FX@$F_{X}$}%
$F_{X}:\mathrm{B}\left(X\right)\to \Db\left(X\right)$.
\begin{theorem}\label{thm:thmesssur}
	The functor $F_{X}$ is essentially 
	surjective, i.e., if $\mathscr F\in {\rm D^{b}_{\rm 
	coh}}(X)$, there is $\mathscr E\in \mathrm B(X)$ 
	and 	an isomorphism  $F_{X}\left( \mathscr E \right) \simeq  \mathscr F  $ in ${\rm D^{b}_{\rm coh}}(X)$. 
\end{theorem} 
\begin{proof}
	The proof of our theorem is divided into the next four steps.
\end{proof}

The following result is established by Illusie in \cite[Proposition 
II.2.3.2]{SGA6}. We give a proof for completeness.
\begin{proposition}\label{prop:pit-bis}
	There exists a bounded   complex of  finite 
	dimensional complex smooth vector bundles  $\left( Q,d^{Q}\right) $ 
		and a quasi-isomorphism of $\mathcal 
O_{X}^{\infty}$-complexes
$\phi_{0}:\mathcal{O}^{\infty }_{X}Q\to \mathscr F^{\infty}$. 
\end{proposition} 
\begin{proof}We may and we will assume that for $i<0$, $\mathscr 	
	F^{i}=0$.  Put
	 \begin{equation}\label{eq:roro1}
k=\sup\left\{i\ge 0,\mathscr H^{i}\mathscr F\neq 0\right\}.
\end{equation}
We will show our result by  induction on $k$, while also proving that 
we may take $ Q^{i}=0$ for $i>k$.

\noindent $\bullet$ \underline{The case $k=0$}\newline		
Assume first that $k=0$. Since $\mathscr H^{0} \mathscr F$ is a 
coherent sheaf, if $x\in X$, there exist  an open neighborhood  $U$ 
	of $x$,   a bounded holomorphic complex 
	  $ R_{U}$ of  trivial holomorphic vector bundles in nonpositive 
	  degree on $U$, and a holomorphic morphism 
	  $r_{U}:\mathcal{O}_{U}R_{U}^{0}\to \mathscr 
	  H^{0} \mathscr F\vert_{U}$  such that we have the exact sequence of 
$\mathcal{O}_{U}$-modules,
\begin{equation}\label{eq:blai1}
\xymatrix{0\ar[r] &\mathcal{O}_{U}
	R_{U}^{-\ell_{U}}\ar[r]&\ldots\ar[r]&\mathcal{O}_{U}
	R_{U}^{0}\ar[r]^-{r_{U}}&\mathscr H^{0} 
	\mathscr F |_{U}\ar[r]&0.}
\end{equation}
Since $\mathcal{O}_{U}^{\infty }$ is flat over $\mathcal{O}_{U}$, we get the  exact sequence of 
$\mathcal{O}_{U}^{\infty }$-modules
 \begin{equation}\label{eq:blai2}
\xymatrix{0\ar[r] &\mathcal{O}_{U}^{\infty }
	R_{U}^{-\ell_{U}}\ar[r]&\ldots\ar[r]&\mathcal{O}_{U}^{\infty }	R_{U}^{0}\ar[r]^-{r_{U}}&\mathscr H^{0} 
	\mathscr F^{\infty } |_{U}\ar[r]&0.}
\end{equation} 
	 Since $X$ is compact, there is a finite cover $\mathcal{U}$ of $X$ by such 
	 $U$.   	Set 
		\begin{align}\label{eqr02}
		\ell=\sup_{U\in \mathcal U}\ell_{U}. 
	\end{align} 

We will establish our result by induction on $\ell$. If $\ell=0$, then $ \mathscr H^{0}\mathscr F$ is locally 
free, i.e., there is a smooth holomorphic  bundle $F$ on  $X$ such that 
$ \mathscr H^{0} \mathscr F=\mathcal{O}_{X}F $. We   consider $Q=F$ as a trivial
complex concentrated in degree $0$ and $r_{0}:\mathcal{O}^{\infty }_{X} Q\to 
\mathscr F^{\infty}$ to be the canonical injection. 

Assume now that  $\ell\ge 1$, and  that our proposition holds  if 
the  corresponding length defined in \eqref{eqr02} is  $\le \ell-1$.  
From the exact sequence (\ref{eq:blai2}), we get a morphism of 
complexes $\mathcal{O}_{U}^{\infty }R_{U}\to \mathscr F^{\infty 
}\vert_{U}$.
Put
\begin{equation}\label{eq:diag1a1}
\mathscr C_{U}=\mathrm{cone}\left(\mathcal{O}_{U}^{\infty }
R_{U}, \mathscr F^{\infty }\vert_{U}\right).
\end{equation}
Since  $k=0$, 
$\mathscr C_{U}$ is exact. Observe that $\mathscr C_{U}$ is just 
the complex
\begin{equation}\label{eq:compl1a1}
\xymatrix{0\ar[r] &\mathcal{O}_{U}^{\infty }
	R_{U}^{-\ell_{U}}\ar[r]&\ldots\ar[r]&\mathcal{O}_{U}^{\infty }	R_{U}^{0}\ar[r]^-{r_{U}}& 
	\mathscr F^{0, \infty } \vert_{U}\ar[r]&\mathscr F^{1, \infty 
	}\ar[r]&.}
\end{equation}

Since  $R^{0}_{U}$ is trivial on $U$, we can extend $R^{0}_{U}$ to a trivial  vector bundle on $X$.  Set 
	\begin{align}\label{eqsiza1}
	R=\bigoplus_{U\in \mathcal U} R^{0}_{U}. 
	\end{align} 
Then, $ R$ is a trivial     vector bundle on $X$, which we 
identify with  a trivial complex in 
degree $0$. Also we have an obvious morphism of complexes 
$R\vert_{U}\to R_{U}$.

Let $\left(\varphi_{U}\vert_{U\in \mathcal U}\right)$ be a smooth partition of unity 
subordinated to  $\mathcal U$. Put
\begin{equation}\label{eq:res4}
r=\sum_{U\in \mathcal U}^{}\varphi_{U}r_{U}.
\end{equation}
Then  $r$ defines  a 
morphism of $\mathcal O_{X}^{\infty}$-complexes 
$\mathcal{O}_{X}^{\infty } R\to \mathscr F^{\infty }$. Moreover, for 
any $U$, $r$ induces  surjections $r\vert_{U}:\mathcal{O}_{U}^{\infty }R\to 
\mathscr H^{0} \mathscr F^{\infty }\vert_{U}$.

We claim that  $r\vert_{U}$  lifts to a smooth 
morphism of complexes $ R\vert_{U}\to R_{U}$
such that the following diagram commutes:
\begin{align}\label{diag1}
		\begin{aligned}
		\xymatrix{
  &  \mathcal{O}_{U}^{\infty }  R_{U}\ar[d]^-{r_{U}}\\  
  \mathcal{O}^{\infty }_{U}R
  \ar[r]^-{r\vert_{U}} 
  \ar@{.>}[ru]
  & \mathscr F^{\infty}|_{U}.}
  \end{aligned}
	\end{align} 
Let us prove this claim. Since $\mathscr F^{\infty }\vert_{U}$ is a subcomplex of $\mathscr 
C_{U}$, from the morphism $r\vert_{U}:\mathcal{O}_{U}^{\infty }
R\to \mathscr F^{\infty }\vert_{U}$, we get a corresponding 
morphism $s_{U}:\mathcal{O}^{\infty}_{U}R\to \mathscr C_{U}$, which is 
also clear by (\ref{eq:compl1a1}).

Since $\mathscr C_{U}$ is exact and 
		since $\mathcal{O}_{X}^{\infty }R$ is  free and concentrated 
		in degree $0$,  by \cite[Lemma 
		3.3]{BismutGilletSoule88c}, there is a null-homotopy 
		$h_{U}:\mathcal{O}_{U}^{\infty }R^{\Ou}\to \mathscr 
		C_{U}^{\Ou-1}$ such that
\begin{equation}\label{eq:compl2}
s_{U}=d^{\mathscr C_{U}}h_{U}.
\end{equation}
Inspection of (\ref{eq:compl1a1}) shows that $h_{U}$ provides the 
desired lift in (\ref{diag1}).
 
Since $r_{U}$ is a quasi-isomorphism, 
by \eqref{diag1} and \cite[Proposition 1.4.4]{KashShap90}, we 
get a quasi-isomorphism of $\mathcal O_{U}^{\infty}$-complexes, 
		\begin{equation}\label{eq:deca1}
		\mathcal{O}^{\infty}_{U} {		\rm 
cone}( R|_{U}, R_{U})\to {		\rm cone}(\mathcal{O}^{\infty }_{U}
R,\mathscr  F^{\infty}|_{U}). 
	\end{equation} 
	
	We use the notation $\mathscr R_{U}, \mathscr 
	R\vert_{U}$ to designate the sheaves $\mathcal{O}^{\infty 
	}_{U}R_{U}, \mathcal{O}^{\infty }_{U} R$.
By proceeding as in  (\ref{eq:compl1a1}),  we can rewrite 
(\ref{eq:deca1}) as  a 
quasi-isomorphism of complexes of $\mathcal{O}^{\infty }_{U}$-modules,
\begin{align}\label{diag3}
		\xymatrix{0\ar[r] & \mathscr 
	R_{U}^{-\ell_{U}}\ar[r]\ar[d]&\cdots\ar[r]\ar[d]&\mathscr R^{-1}_{U}\oplus 
	\mathscr R|_{U}\ar[r]^{\gamma}\ar[d]&\mathscr 
	R_{U}^{0}\ar[r]\ar[d]&0 \ar[d]
\ar[r]&\cdots\\ 0\ar[r]
& 0\ar[r]&0\ar[r]&\mathscr R|_{U}\ar[r]&\mathscr 
	F^{0,\infty}|_{U}\ar[r]&\mathscr F^{1,\infty}|_{U} 
\ar[r]&\cdots,}
\end{align}
so that the corresponding bicomplex is exact.
By  construction, the cohomology   of the second row in
\eqref{diag3} is concentrated in degree $-1$, and will be denoted 
$H^{-1}$.  This forces $\gamma$ 
to be a surjective morphism of $\mathcal{O}_{U}^{\infty}$-modules. 
Since $\mathscr R_{U}$ and $\mathscr R$ are free, for any $x\in U$, 
$\gamma_{x}$ is surjective, and $\ker \gamma$ is a 
sheaf defined by a smooth vector bundle $K_{U}$ on 
$U$. Since $U$ is contractible, $K_{U}$ is a trivial vector bundle,   so that  the complex 
\begin{equation}
	\xymatrix{0\ar[r] &
	\mathscr R_{U}^{-\ell_{U}}\ar[r]&\cdots\ar[r]& \mathscr R_{U}^{-2} \ar[r]&\ker \gamma\to 0}
\end{equation} 
is  a resolution of $H^{-1}$. Since  here $\ell$ is replaced here by 
$\ell-1$, by our induction argument, there exists a bounded   complex of  finite 
	dimensional complex vector bundles  $\left( Q',d^{Q'} \right) $ 
 with $
Q^{\prime i}=0$ if $i\ge 0$,  such that if $\mathscr 
Q'=\mathcal{O}^{\infty }_{X}Q'$,  we have a  quasi-isomorphism of 
$\mathcal{O}_{X}^{\infty }$-complexes,
\begin{equation}\label{eqfcoco11}
	\mathscr Q'\to {		\rm cone}(\mathscr 
R,\mathscr  F^{\infty}). 
\end{equation} 
By (\ref{eq:res4a-1b}), we have a canonical morphism $\rm{ 
	cone}^{\scriptscriptstyle\bullet } (\mathscr 
	R,\mathscr  F^{\infty})\to \mathscr R^{\Ou+1}$.
We obtain a tautological diagram 
\begin{equation}\label{diag11}
		\begin{aligned}
		\xymatrix{ \mathscr Q^{\prime \Ou}\ar[r]\ar[dr]
  &  {		\rm cone}^{\Ou}(\mathscr 
R,\mathscr  F^{\infty})\ar[d]\\ 
  & \mathscr R^{\scriptscriptstyle\bullet +1}.}
  \end{aligned}
	\end{equation} 
Since the horizontal 
morphism in \eqref{diag11} is a quasi-isomorphism, by \cite[Proposition 
1.4.4 (TR5)]{KashShap90}, we get a 
quasi-isomorphism of $\mathcal O_{X}^{\infty}$-complexes
\begin{equation}\label{eqfcoco2}
  {		\rm cone}(\mathscr 
Q^{\prime \Ou},\mathscr  R^{\scriptscriptstyle\bullet +1})\to {\rm 
cone}({	\rm cone}^{\Ou}(\mathscr 
R,\mathscr  F^{\infty}),\mathscr R^{\scriptscriptstyle\bullet+1 }). 
\end{equation}

Put
\begin{align}\label{eq:cocoa1}
	&Q^{\Ou}=\mathrm{cone}\left(Q^{\prime \Ou-1},R^{\Ou}\right),
&\mathscr 
Q^{\Ou}={		\rm cone}(\mathscr 
Q'^{\scriptscriptstyle\bullet -1},\mathscr  R^{\Ou}),
\end{align}
so that $\mathscr Q=\mathcal{O}^{\infty }_{X}Q$. Observe 
that for $i>0$, $Q^{i}=0$.

 By  \cite[Proposition 1.4.4 
(TR3)]{KashShap90}, we have a homotopy equivalence of $\mathcal O_{X}^{\infty}$-complexes
\begin{equation}\label{eqfcoco1}
	{\rm cone} ( {		\rm 
	cone}^{\scriptscriptstyle\bullet -1} (\mathscr 
R,\mathscr  F^{\infty}), \mathscr R^{\Ou})\to\mathscr F^{\infty}.
\end{equation} 
By \eqref{eqfcoco2}--\eqref{eqfcoco1}, we get a quasi-isomorphism, 
\begin{equation}\label{eqfcoco3}
	 \mathscr Q\to \mathscr 
F^{\infty}, 
\end{equation} 
which completes the proof of our theorem when $k=0$.

\noindent $\bullet$ \underline{The case $k\ge 1$}\newline
Assume  that $k\ge 1$ and  that our theorem holds for $k'\le 
k-1$.   Consider the truncated complex 
 	\begin{equation}
	\tau_{\le k-1}\mathscr F^{\infty}: \xymatrix{0\ar[r] & \mathscr 
	F^{0,\infty}\ar[r]&\cdots\ar[r]&\mathscr 
	F^{k-2,\infty}\ar[r]&\mathscr 
	\ker\, d^{\mathscr F^{\infty }}\vert_{\mathscr F^{k-1,\infty}}\ar[r]&0}.
\end{equation}
Then $\tau_{\le k-1}\mathscr F^{\infty}$ is a subcomplex of $\mathscr 
F^{\infty }$. 
The cohomology of $\tau_{\le k-1}\mathscr F^{\infty}$ is given by  
 $\mathscr H^{i} \mathscr F^{\infty },0\le i\le k-1$. By 
our induction argument, there exists a complex $Q_{1}$ of complex 
vector bundles on $X$ such that  $Q_{1}^{i}=0$ for $i\ge k$, and that 
if $\mathscr Q_{1}=\mathcal{O}^{\infty}_{X}Q_{1}$,  there exists a
quasi-isomorphism
\begin{equation}
	\mathscr Q_{1}\to \tau_{\le k-1}\mathscr F^{\infty}. 
\end{equation} 
By composition with the inclusion $\tau_{\le k-1}\mathscr F^{\infty}\to \mathscr F^{\infty}$, we get a morphism of complex 
\begin{equation}
		\mathscr Q_{1}\to \mathscr F^{\infty}. 
\end{equation}
The cohomology of ${\rm cone}(\mathscr Q_{1},\mathscr F^{\infty})$ is 
concentrated in degree $k$.

Let $d$ be the differential of $\mathrm{cone(\mathscr Q_{1}}, 
\mathscr F^{\infty })$. Consider the truncated complex 
\begin{equation}
	\tau_{\ge k}{\rm cone}(\mathscr Q_{1},\mathscr F^{\infty}):0\to  {\rm 
	cone}^{k}(\mathscr Q_{1},\mathscr F^{\infty})/d\mathrm{cone}^{k-1}(\mathscr Q_{1},\mathscr F^{\infty})\to {\rm 
	cone}^{k+1}\to \ldots
\end{equation} 
By our  induction  hypothesis, there    exists a smooth complex of vector bundles $Q_{2}$ 
 such that $Q_{2}^{i}=0$ for $i>k$, and that if $\mathscr Q_{2}=\mathcal{O}^{\infty 
}_{X}Q_{2}$, we have a quasi-isomorphism
\begin{align}\label{eqQ2tauk}
	\mathscr Q_{2}\to \tau_{\ge k}{\rm cone}(\mathscr Q_{1},\mathscr F^{\infty}). 
\end{align} 
By proceeding as in  \eqref{diag1}, \eqref{eq:compl2} the morphism 
\eqref{eqQ2tauk} lifts to a quasi-isomorphism 
\begin{align}\label{eqFcc2}
		\mathscr Q_{2}\to {\rm cone}(\mathscr Q_{1},\mathscr 
		F^{\infty}).
\end{align} 

Put
\begin{align}\label{eq:clop1}
&Q=\mathrm{cone}\left(Q_{2}^{\Ou-1},Q_{1}^{\Ou}\right),
&\mathscr Q=\mathrm{cone}\left(\mathscr Q_{2}^{\Ou-1}, \mathscr 
Q_{1}^{\Ou}\right),
\end{align}
so that $\mathscr Q=\mathcal{O}^{\infty }_{X}Q$. Then $\mathscr Q^{i}=0$ for $i>k$. 
By proceeding  as in \eqref{eqfcoco11}-\eqref{eqfcoco3}, from \eqref{eqFcc2}, we get 
a quasi-isomorphism 
\begin{align}
	\mathscr Q\to \mathscr F^{\infty}, 
\end{align} 
which completes the proof of our proposition.
\end{proof}

	We take $ Q,\phi_{0}$  as in Proposition 
\ref{prop:pit-bis}. 
 Put
 \begin{equation}\label{eq:zer1}
\mathscr C=\mathrm{cone}\left(\mathscr Q,\mathscr F^{\infty}\right).
\end{equation}
  As in 
 (\ref{eq:res4a0z1b})--(\ref{eq:res4a-1b}), we have the identity 
 \begin{equation}\label{eq:gina1}
 \mathscr C^{\Ou}=\mathscr Q^{\Ou+1} \oplus \mathscr F^{\infty\Ou},
 \end{equation}
 the differential $d^{\mathscr C}$ is given by
  \begin{equation}\label{eq:gina2}
 d^{\mathscr C}=
 \begin{bmatrix}
 d^{\mathscr Q} & 0 \\
 	\phi_{0}(-1)^{\rm deg} & d^{\mathscr F^{\infty}}
 \end{bmatrix},
 \end{equation}
and we have  the exact sequence of complexes,
 \begin{equation}\label{eq:gina3}
 0\to \mathscr F^{\infty \Ou}\to \mathscr C^{\Ou}\to \mathscr Q^{\Ou+1}\to 0.
 \end{equation}
 Since $\phi_{0}$ is a quasi-isomorphism, $\mathscr C$ is exact.

As in (\ref{eq:coi0}), if $\mathscr A$ is a  $\Z$-graded 
$\mathcal{O}_{X}^{\infty}$-module, we use the notation
\begin{equation}\label{eq:noti1}
\overline{\mathscr A}=\mathcal{O}_{X}^{\infty 
}\left(\Lambda\left(\overline{\TsX}\right)\right)\ho \mathscr A.
\end{equation}

Put
\begin{equation}\label{eq:coi-1}
\overline{Q}=\Lambda\left(\overline{\TsX}\right)\ho Q.
\end{equation}
Then $\overline{\mathscr Q}$ is the $\mathcal{O}_{X}^{\infty 
}$-module of smooth sections of the vector bundle $\overline{Q}$.

Since $\mathcal{O}^{\infty 
}_{X}\left(\Lambda\left(\overline{\TsX}\right)\right)$ is flat over 
$\mathcal{O}_{X}^{\infty }$, 
\begin{align}
\phi_{0}:	\overline{\mathscr Q}	\to \overline{ \mathscr F}^{\infty }  
	\end{align}
is still a  quasi-isomorphism, and  we have the exact sequence
 \begin{equation}\label{eq:gina3bi}
 0\to \overline{\mathscr F}^{\infty \Ou}\to \overline{\mathscr C}^{\Ou}\to \overline{\mathscr 
 Q}^{\Ou+1}\to 0.
 \end{equation}
Also $\overline{\mathscr  C}$ is exact.

Put
\begin{equation}\label{eq:zeph1}
\mathscr R=\Hom\left( \overline{\mathscr Q}^{\Ou+1},\overline{\mathscr 
C}^{\Ou}\right).
\end{equation}
Let $d^{\mathscr R}$ be the differential on $\mathscr R$  which comes 
from $d^{Q}, d^{C}$.  Also  the 
supercommutator with $d^{\mathscr C}$ is a differential on $\End\left(\overline{\mathscr 
C}\right)$.

Let  $A\in \End\left(\overline{\mathscr C}\right)$ be a lower 
triangular.  We 
write $A$ as a matrix with respect to the splitting (\ref{eq:gina1}),
\begin{equation}\label{eq:16a2x1}
A=
\begin{bmatrix}
	\alpha & 0 \\
	\beta & \gamma
\end{bmatrix}.
\end{equation}
Put
\begin{equation}\label{eq:162a2x1z1}
\underline{A}=
\begin{bmatrix}
	\alpha \\
	\beta
\end{bmatrix}.
\end{equation}
Then $\underline{A}\in \mathscr R$.  
\begin{proposition}\label{prop:pdifb}
	The following identity holds:
	\begin{equation}\label{eq:zeph0b}
\underline{\left[d^{\overline{\mathscr C}},A\right]}=d^{\mathscr 
R}\underline{A}-\left(-1\right)^{\mathrm{deg}\left(A\right)}\begin{bmatrix}
	0 \\
	\gamma\phi_{0}\left(-1\right)^{\deg}
\end{bmatrix}.
\end{equation}
\end{proposition}
\begin{proof}
	We only need to establish (\ref{eq:zeph0b}) when $\alpha=0,\beta=0$, which 
	is easy.
\end{proof}

As before,  we denote by
\index{RX@$\mathscr R_{X}$}%
$\mathscr R_{X}$ the global 
sections of $\mathscr R$ on $X$. 
\begin{proposition}\label{prop:pexab}
	The complexes $\mathscr R$ and $\mathscr R_{X}$ are 
	exact.
\end{proposition}
\begin{proof}
	We know that $\overline{\mathscr C}$ is exact. Also 
	$\overline{\mathscr Q}$ is locally free.  It is now easy to 
	see that $\mathscr R$ is exact. Since $\mathscr R$ is soft,  the functor $\mathscr R \to 
	\mathscr R_{X}$ is exact, and we obtain the corresponding 
	result for $\mathscr R_{X}$. 
\end{proof}
\begin{remark}\label{rem:exi}
	For a related argument, we refer to 
\cite[\href{https://stacks.math.columbia.edu/tag/0647}{Tag
0647}]{stacks-project}.
\end{remark}

We will establish a fundamental result of Block \cite[Lemma 
4.6]{Block10}.
\begin{theorem}\label{thm:exib}
	Given $\mathscr F\in \Db\left(X\right)$, there exists 	$ \mathscr E=\left(E,A^{E
	\prime \prime }\right)\in {\rm B}(X)$ and a morphism of 
	$\mathcal{O}^{\infty 
	}_{X}\left( \Lambda\left(\overline{\TsX}\right)\right) $-modules 
	$\phi:\mathscr E\to \overline{\mathscr F}^{\infty }$,  which is a quasi-isomorphism of 
	$\mathcal O_{X}$-complexes, and induces a quasi-isomorphism of 
	$\mathcal{O}_{X}^{\infty }$-complexes 
	$\left(D,\overline{v}_{0}\right)\to \mathscr F^{\infty 
	}$.
In particular, $\mathscr E$ and $\mathscr F$ are isomorphic in ${\rm 
D^{b}_{\mathrm{coh}}}(X)$. 	
\end{theorem}
\begin{proof}
We use the notation introduced before. 
We will 
reformulate our problem as  being the construction  of a generalized 
antiholomorphic superconnection 
$d^{\overline{\mathscr C}}$ on $\overline{ \mathscr  C}$ which will have the structure
\begin{equation}\label{eq:gina6}
d^{\overline{\mathscr C}}=
\begin{bmatrix}
  A^{\overline{E} \prime \prime } & 0 \\
	\phi (-1)^{\rm deg} & d^{\overline{\mathscr F}^{\infty}}
\end{bmatrix}.
\end{equation}
such that if
\begin{equation}\label{eq:gina7}
d^{\overline{\mathscr C}}=\sum_{k\in\N}^{}d^{\overline{\mathscr C}}_{k}
\end{equation}
is the canonical expansion of $d^{\overline{\mathscr C}}$ with 
respect to $\Lambda\left(\overline{\TsX}\right)$,  then
\begin{equation}\label{eq:gina6a1}
d^{\overline{\mathscr C}}_{0}=d^{\mathscr C}.
\end{equation}

By (\ref{eq:exp1}),  for $k\ge 2$, 
\begin{equation}\label{eq:ann1}
d^{\overline{\mathscr F}^{\infty }}_{k}=0.
\end{equation}
By (\ref{eq:gina6}), we get
\begin{equation}\label{eq:gina9a1}
d^{\overline{\mathscr C}}_{k}=
\begin{bmatrix}
	v_{k} & 0 \\
	\phi_{k}(-1)^{\rm deg} & d^{\overline{\mathscr F}^{\infty }}_{k}
\end{bmatrix}.
\end{equation}
In particular, for $k=0$, we get
\begin{align}\label{eq:cru1-z}
&v_{0}=d^{Q},&d_{0}^{\overline{\mathscr F}^{\infty }}=d^{\mathscr 
F^{\infty }}.
\end{align}

For $i\ge 0$, put
\begin{equation}\label{eq:gina8}
d^{\overline{\mathscr C}}_{\le i}=\sum_{k=0}^{i}d^{\overline{\mathscr C}}_{k}.
\end{equation}
We will use a similar notation for other related expressions.
We will construct  $d^{\overline{\mathscr C}}_{i}$ by induction so 
that for  $i\ge 1$,
\begin{equation}\label{eq:gina9}
\left[d^{\overline{\mathscr C}}_{\le i}\right]^{2}_{\le i}=0.
\end{equation}

For $i=0$, we take $d^{\overline{\mathscr C}}_{0}=d^{\mathscr C}$. For $i\ge 0$, we  will  show how to 
construct $d^{\overline{\mathscr C}}_{\le i+1}$ from $d^{\overline{\mathscr C}}_{\le i}$. We have the 
identity
\begin{equation}\label{eq:gina10}
d^{\overline{\mathscr C}}_{\le i+1}=d^{\overline{\mathscr C}}_{\le i}+d^{\overline{\mathscr C}}_{i+1}.
\end{equation}
Then
\begin{equation}\label{eq:gina11}
\left[d^{\overline{\mathscr C}}_{\le i+1}\right]^{2}_{\le i+1}=\left[d^{\overline{\mathscr C}}_{\le 
i},d^{\overline{\mathscr C}}_{i+1}\right]_{\le i+1}+\left[d^{\overline{\mathscr C}}_{\le 
i}\right]^{2}_{\le i+1}.
\end{equation}
Clearly,
\begin{equation}\label{eq:gina12}
\left[d^{\overline{\mathscr C}}_{\le 
i},d^{\overline{\mathscr C}}_{i+1}\right]_{\le 
i+1}=\left[d^{\mathscr C},d^{\overline{\mathscr C}}_{i+1}\right].
\end{equation}
For (\ref{eq:gina11}) to vanish, using (\ref{eq:gina12}), we should have
\begin{equation}\label{eq:gina16}
\left[d^{\mathscr C},d^{\overline{\mathscr C}}_{i+1}\right]+\left[d^{\overline{\mathscr C}}_{\le 
i}\right]^{2}_{\le i+1}=0.
\end{equation}
By (\ref{eq:gina9}), (\ref{eq:gina16}) can also be written in the form
\begin{equation}\label{eq:gina16a1}
\left[d^{\mathscr C},d^{\overline{\mathscr C}}_{i+1}\right]+\left[d^{\overline{\mathscr C}}_{\le 
i}\right]^{2}_{i+1}=0.
\end{equation}

Since $d^{\overline{\mathscr F}^{\infty},
2}=0$\footnote{By (\ref{eq:ann1}), the expansion of 
$d^{\overline{\mathscr F}^{\infty}}$ terminates at 
$k=1$. Here, this fact will be ignored.}, we have the identity
\begin{equation}\label{eq:gina16a11}
\left[d^{\overline{\mathscr F}^{\infty}}_{\le i+1}\right]^{2}_{\le i+1}=0.
\end{equation}
By (\ref{eq:gina16a1}), (\ref{eq:gina16a11}),  we deduce that 
\begin{equation}\label{eq:gina16a2}
\left[d^{\mathscr F^{\infty}},d^{\overline{\mathscr 
F}^{\infty}}_{i+1}\right]+\left[d^{\overline{\mathscr F}^{\infty} }_{\le i}\right]^{2}_{i+1}=0.
\end{equation}

First, we solve (\ref{eq:gina16a1}) for $i=0$. Namely, we need to find 
$d_{1}^{\overline{\mathscr C}}$ such that
\begin{equation}\label{eq:gina17a1}
\left[d^{\mathscr C},d_{1}^{\overline{\mathscr C}}\right]=0.
\end{equation}
By Proposition \ref{prop:pdifb}, equation (\ref{eq:gina17a1}) can be written in the 
form
\begin{equation}\label{eq:gina18a1}
d^{\mathscr R}\underline{d_{1}^{\overline{\mathscr C}}}+
\begin{bmatrix}
	0 \\
	d_{1}^{\overline{\mathscr F}^{\infty}
	}\phi_{0}\left(-1\right)^{\deg}
\end{bmatrix}=0.
\end{equation}
Recall that by (\ref{eq:cru1-z}), $v_{0}=d^{Q}$. Then
\begin{equation}\label{eq:gina18a2}
d^{\mathscr R}\begin{bmatrix}
	0 \\
	d_{1}^{\overline{\mathscr F}^{\infty}	}\phi_{0}\left(-1\right)^{\deg}
\end{bmatrix}=
\begin{bmatrix}
	0 \\
	d_{0}^{\overline{\mathscr 
	F}^{\infty}}d^{\overline{\mathscr 
	F}^{\infty}
	}_{1}\phi_{0}\left(-1\right)^{\deg}-d_{1}^{\overline{\mathscr 
	F}^{\infty}}\phi_{0}\left(-1\right)^{\deg}v_{0}
\end{bmatrix}.
\end{equation}
Since $\left[d_{0}^{\overline{\mathscr F}^{\infty}}, d_{1}^{ 
\overline{\mathscr F}^{\infty}}\right]=0$, and $d_{0}^{\overline{\mathscr F}^{\infty}}\phi_{0}=\phi_{0}v_{0}$, and since $v_{0}$ is odd, by (\ref{eq:gina18a2}), we get
\begin{equation}\label{eq:gina18a3}
d^{\mathscr R}\begin{bmatrix}
	0 \\
	d_{1}^{\overline{\mathscr F}^{\infty}
	}\phi_{0}\left(-1\right)^{\deg}
\end{bmatrix}=0.
\end{equation}

 Let 
$\n^{Q\prime \prime }$  be an antiholomorphic 
connection on  $Q$.
Put
\begin{align}\label{eq:senia1}
&e_{1}^{\mathscr R}=\underline{d_{1}^{\overline{\mathscr C}}}-
\begin{bmatrix}
	\n^{Q \prime \prime} \\
	0
\end{bmatrix},
&f_{1}^{\mathscr R}=d^{\mathscr R}\begin{bmatrix}
	\n^{Q \prime \prime} \\
	0
\end{bmatrix}+\begin{bmatrix}
	0 \\
	d_{1}^{\overline{\mathscr F}^{\infty}
	}\phi_{0}\left(-1\right)^{\deg}
\end{bmatrix}.
\end{align}
An elementary computation shows that
\begin{equation}\label{eq:senia2}
f_{1}^{\mathscr R}=
\begin{bmatrix}
	 \n^{Q \prime \prime }v_{0} \\
\phi_{0}\left(-1\right)^{\deg}\n^{Q \prime \prime }+d_{1}^{\overline{\mathscr F}^{\infty }}\phi_{0}\left(-1\right)^{\deg}	
\end{bmatrix}.
\end{equation}
Differentiation on $X$ having disappeared, $e_{1}^{\mathscr R}, f_{1}^{\mathscr R}$ lie 
in  $\mathscr R_{X}$.

By (\ref{eq:gina18a3}), (\ref{eq:senia1}), solving equation 
(\ref{eq:gina18a1}) is equivalent to finding $e_{1}^{\mathscr R}\in \mathscr R_{X}$  such that
\begin{equation}\label{eq:senia2a2}
d^{\mathscr R}e_{1}^{\mathscr R}+f_{1}^{\mathscr R}=0.
\end{equation}
By (\ref{eq:gina18a3}), (\ref{eq:senia1}), we get
\begin{equation}\label{eq:senia3}
d^{\mathscr R}f_{1}^{\mathscr R}=0.
\end{equation}
By Proposition \ref{prop:pexab}, the complex $\mathscr 
R_{X}$ is exact. Using (\ref{eq:senia3}), there is a solution 
$e_{1}^{\mathscr R}$ to (\ref{eq:senia2a2}), which gives  the existence of a solution 
for (\ref{eq:gina17a1}).

Now we  assume that $i\ge 1$. We 
will still find a solution for (\ref{eq:gina16a1}).
By Proposition \ref{prop:pdifb}, we get
\begin{equation}\label{eq:gina14}
\underline{\left[d^{\mathscr C},d^{\overline{\mathscr 
C}}_{i+1}\right]}=d^{ \mathscr R}\underline{d^{\overline{\mathscr C}}_{i+1}}+
\begin{bmatrix}
	0 \\
	d_{i+1}^{\overline{\mathscr F}^{\infty}}\phi_{0}
	\left(-1\right)^{\deg}
\end{bmatrix}.
\end{equation}
By (\ref{eq:gina16a1}),   (\ref{eq:gina14}), we obtain,
\begin{equation}\label{eq:gina14a1}
d^{\mathscr R}\underline{d^{\overline{\mathscr C}}_{i+1}}+\underline{\left[d^{\overline{\mathscr C}}_{\le 
i}\right]^{2}_{i+1}}+
\begin{bmatrix}
	0 \\
	d_{i+1}^{\overline{\mathscr 
	F}^{\infty}}\phi_{0}\left(-1\right)^{\deg}
\end{bmatrix}=0.
\end{equation}

Using (\ref{eq:gina9}), we get
\begin{equation}\label{eq:gina15b}
\left[d^{\mathscr C},\left[d^{\overline{\mathscr C}}_{\le 
i}\right]^{2}_{ i+1}\right]=\left[d^{\mathscr C},\left[d^{\overline{\mathscr C}}_{\le 
i}\right]^{2}\right]_{ i+1}=\left[d^{\mathscr C}_{\le i},d^{\overline{\mathscr C},2}_{\le 
i}\right]_{ i+1}=0.
\end{equation}
By Proposition \ref{prop:pdifb} and by (\ref{eq:gina15b}), we get
\begin{equation}\label{eq:gina17b}
d^{\mathscr R}\underline{\left[d^{\overline{\mathscr C}}_{\le 
i}\right]^{2}_{ i+1}}-
\begin{bmatrix}
	0 \\
	\left[d^{\overline{\mathscr F}^{\infty}}_{\le 
	i}\right]^{2}_{i+1}\phi_{0}\left(-1\right)^{\deg}
\end{bmatrix}=0.
\end{equation}
An easy computation shows that
\begin{equation}\label{eq:easy1a}
d^{\mathscr R}\begin{bmatrix}
	0 \\
	d^{\overline{\mathscr F}^{\infty }
	}_{i+1}\phi_{0}\left(-1\right)^{\deg}
\end{bmatrix}=
\begin{bmatrix}
	0 \\
	d^{\mathscr F^{\infty }}d^{\overline{\mathscr F}^{\infty }
	}_{i+1}\phi_{0}\left(-1 \right) ^{\deg}-d^{\overline{\mathscr F}^{\infty }}_{i+1}\phi_{0}\left(-1\right)^{\deg}v_{0}
\end{bmatrix}.
\end{equation}
Since $\phi_{0}$ is a morphism of complexes, and since $v_{0}$ is odd,  using 
(\ref{eq:gina16a2}), we can rewrite (\ref{eq:easy1a}) in the form
\begin{equation}\label{eq:easy2a}
d^{\mathscr R}\begin{bmatrix}
	0 \\
	d^{\overline{\mathscr F}^{\infty } 
	}_{i+1}\phi_{0}\left(-1\right)^{\deg}
\end{bmatrix}=-
\begin{bmatrix}
	0 \\
	\left[d_{\le i}^{\overline{\mathscr F}^{\infty },2}\right]_{i+1}\phi_{0}\left(-1\right)^{\deg}
\end{bmatrix}
\end{equation}

By combining (\ref{eq:gina17b}) and (\ref{eq:easy2a}), we get
\begin{equation}\label{eq:gina18}
d^{\mathscr R}\left[\underline{\left[d^{\overline{\mathscr C}}_{\le 
i}\right]^{2}_{i+1}}+
\begin{bmatrix}
	0 \\
	d_{i+1}^{\overline{\mathscr 
	F}^{\infty}}\phi_{0}\left(-1\right)^{\deg}
\end{bmatrix}\right]=0.
\end{equation}
Since the complex $\mathscr  R_{X}$ is exact, using 
(\ref{eq:gina18}),    equation 
(\ref{eq:gina14a1}) for $\underline{d^{\overline{\mathscr C}}_{i+1}}$ can be 
solved.

In summary, we have constructed   $\mathscr E=(E,A^{E\prime\prime 
})\in {\rm B}(X)$ and a morphism of $\mathcal{O}_{X}$-complexes 
$\phi:\mathscr E\to \overline{\mathscr F}^{\infty }$. Therefore 
$\phi$ induces a morphism
\begin{equation}\label{eq:iso0}
\phi:\mathscr H \mathscr E\to \mathscr H 
\overline{\mathscr F}^{\infty }.
\end{equation}

By construction, $\phi$ induces an isomorphism of 
$\mathcal{O}_{X}^{\infty }$-modules,
\begin{equation}\label{eq:isoa1}
HD \simeq H \mathscr F^{\infty }.
\end{equation}
Both sheaves in (\ref{eq:isoa1}) are equipped with antiholomorphic connections, which 
necessarily correspond by $\phi$.  By combining equation 
(\ref{eq:cons2}) in  Theorem 
\ref{thm:cohco} and equation (\ref{eq:dos1bisc}) in Proposition 
\ref{prop:coco}, we deduce that in (\ref{eq:iso0}), $\phi$ is an 
isomorphism. The proof of our theorem is completed. 
	\end{proof}
\subsection{The homotopy category}%
\label{subsec:hoca}
Let $\mathscr E,\underline{\mathscr E}$ 
be two objects in $\mathrm B(X)$. Note that $\Hom\left(\mathscr 
E,\underline{\mathscr E}\right)$ is an object in $\mathrm{B}\left(X\right)$.  
Recall that 
\index{ZH@$Z^{0}\Hom\left(\mathscr E, \underline{\mathscr E}\right)_{X}$}%
$Z^{0}\Hom\left(\mathscr E,\underline{\mathscr E} 
\right)_{X}$ was defined in (\ref{eq:kel2}). A morphism $\phi\in  Z^{0}\Hom\left(\mathscr E,\underline{\mathscr E} 
\right)_{X}$ is said to be null-homotopic if there exists $\psi\in 
\Hom^{-1}\left(\mathscr E,\underline{\mathscr E}\right)_{X}$ such that
\begin{equation}\label{eq:hoc1}
\phi=A^{\Hom\left(E,\underline{E}\right)\prime \prime }\psi.
\end{equation}
Also $\phi\in Z^{0}\Hom\left(\mathscr E,\underline{\mathscr E} 
\right)_{X}$ is said to be a homotopy equivalence if there is $\psi\in
Z^{0}\Hom\left(\underline{\mathscr E}, \mathscr E
\right)_{X}$ such that $\psi\phi-1\vert_{\mathscr E}$ is null-homotopic 
in $Z^{0}\Hom\left(\mathscr E, \mathscr E\right)_{X}$,  and 
$\phi\psi-1\vert_{\underline{\mathscr E}}$ is null-homotopic in 
$Z^{0}\Hom\left(\underline{\mathscr E},\underline{\mathscr 
E}\right)_{X}$.

If   $\phi\in Z^{0}\Hom\left(\mathscr E,\underline{\mathscr 
E}\right)_{X}$ is  null-homotopic, the image of this morphism in 
$\mathrm{C}^{\mathrm{b}}_{\coh}\left(X\right)$ is also
  null-homotopic. In particular,  the  functor $F_{X}:{\rm B}(X)\to {\rm 
 D_{\mathrm{coh}}^{b}}(X)$ factorizes 
through  a functor
\index{FX@$\underline{F}_{X}$}%
	$\underline{F}_{X}:	\underline{{\rm B}}(X)\to {\rm D_{\mathrm{coh}}^{b}}(X)$. 
	By Theorem	\ref{thm:thmesssur}, $\underline{F}_{X}$ is   essentially surjective. 
\begin{proposition}\label{prop:prophtpeq}
	If $\phi\in Z^{0}\Hom\left(\mathscr E,\underline{\mathscr 
	E}\right)_{X}$,  the 
	following conditions are equivalent: 
	\begin{enumerate}
		\item $\phi$ is a homotopy equivalence,
		\item $\phi$ induces a quasi-isomorphism $  \mathscr E\to 	  \underline{\mathscr E}$,
		
	\item  For any $x\in X$, $\phi$ induces a quasi-isomorphism 	
	$D_{x}\to \underline{D}_{x}$.
	\item  $\phi$ induces a quasi-isomorphism 	$\mathcal{O}^{\infty 
	}_{X}\left(D\right)\to\mathcal{O}^{\infty 
	}_{X}\left(\underline{D}\right)$.	
	\item  $\phi$ induces a homotopy equivalence $C^{\infty}(X,D)\to 	
	C^{\infty}(X,\underline{D})$. 
\end{enumerate} 
\end{proposition}
\begin{proof}
	This is an easy consequence of Theorem \ref{thm:simp} and of the 
	construction of $\mathrm{cone}\left(\mathscr 
	E,\underline{\mathscr E}\right)$.		
\end{proof} 
\subsection{An equivalence of categories}%
\label{subsec:eqcat}
We have a fundamental theorem of Block \cite[Theorem 4.3]{Block10}. 
 \begin{theorem}\label{thm:eqcat}
The functor 
	$\underline{F}_{X}:  \underline{\mathrm B}(X)\to {\rm D}_{\rm coh}^{\rm 
b}(X)$ is	an equivalence of triangulated categories. 
 \end{theorem} 
\begin{proof}
	Since  $\underline{F}_{X}$ is essentially surjective, we only need 
to show that $\underline{F}_{X}$ is fully faithful. Equivalently,  if $\mathscr E,\underline{\mathscr E}\in 
\underline{\mathrm{B}}\left(X\right)$, 
 we have to show that the corresponding map
	$\Hom_{\underline{\rm B}(X)}( \mathscr E,\underline{ \mathscr E})\to \Hom_{\rm 
		D_{\mathrm{coh}}^{b}(X)}(\mathscr E,\underline{\mathscr E})
$	 
is bijective. 

Let us prove surjectivity. A morphism in $\Hom_{\rm 
		D_{\mathrm{coh}}^{b}(X)}(\mathscr E,\underline{\mathscr E})$ is represented  
		by the diagram of morphisms of $\mathcal{O}_{X}$-complexes,
	\begin{align}\label{diag12}
		\begin{aligned}
		\xymatrix{\mathscr E& \mathscr F\ar[r]\ar[l]_-{\mathrm{q.i.}}
  &  \underline{\mathscr E}},
  \end{aligned}
	\end{align} 
where $\mathscr F\in{\rm 
		D_{\mathrm{coh}}^{b}}(X)$, and `\textrm{q.i.}'  stands for 
		`quasi-isomorphism'.  Also there are obvious maps 
		$\mathcal{O}^{\infty 
		}_{X}\Lambda\left(\overline{\TsX}\right)\ho _{\mathcal{O}_{X}}
		\mathscr E\to \mathscr E,\mathcal{O}^{\infty 
		}_{X}\Lambda\left(\overline{\TsX}\right)\ho_{\mathcal{O}_{X}} 
		\underline{\mathscr E}\to \underline{\mathscr E}$. By 
		(\ref{diag12}), we obtain the commutative diagram,
		\begin{align}\label{diag12a}
	\begin{aligned}
	\xymatrix{
\mathscr E   &\mathscr F \ar[l]_-{\rm 
q.i.}\ar[d]\ar[r]&\underline{\mathscr E} \\
         & \overline{\mathscr F}^{\infty}\ar@{.>}[lu]\ar@{.>}[ur]&}
	\end{aligned}
\end{align}  
By  (\ref{diag12a}),  we 
		get the diagram of $\mathcal{O}_{X}$-complexes,
		\begin{align}\label{diag13}
		\xymatrix{\mathscr E& 
		\overline{\mathscr F}^{\infty}\ar[r]\ar[l]_-{ \mathrm{q.i.}}
  & \underline{\mathscr E}},
	\end{align}
	with $\mathcal{O}_{X}^{\infty 
		}\Lambda\left(\overline{\TsX}\right)$-morphisms.
By Theorem \ref{thm:exib}, there exists 
$\underline{\underline{\mathscr E}}\in {\rm B}(X)$ and a quasi-isomorphism
$\underline{\underline{\mathscr E}}\to 
\overline{\mathscr F}^{\infty}$.   We can replace (\ref{diag13}) by
\begin{align}\label{diag14}
		\xymatrix{\mathscr E & \underline{\underline{\mathscr E}}\ar[r]\ar[l]_-{ \mathrm{q.i.}}
  &  \underline{\mathscr E}}, 
	\end{align} 
the morphisms in (\ref{diag14}) being morphisms of ${\rm B}(X)$. By 
Proposition  \ref{prop:prophtpeq}, the 
first morphism in \eqref{diag14} is a
homotopy equivalence. Using its homotopic inverse, we get  the 
desired  morphism from $\mathscr E\to \underline{\mathscr E}$.

We now prove injectivity. Let $\phi_{1},\phi_{2}:\mathscr E\to\underline{\mathscr E}$ be  two morphisms in 
${\rm B}(X)$, whose images  coincide   in ${\rm D_{\mathrm{coh}}^{b}}(X)$. By definition, 
there is $\mathscr F\in {\rm D_{\mathrm{coh}}^{b}}(X)$ and a quasi-isomorphism 
	$\phi:		\mathscr F\to \mathscr E$,
such that $\phi_{1}\phi=\phi_{2}\phi$ up to   homotopy  in 
$\Hom_{\mathrm{C^{b}_{coh}\left(X\right)}}\left(\mathscr F,\underline{\mathscr E}\right)$.

Using again Theorem \ref{thm:exib}, there exists 
$\underline{\underline{\mathscr E}}\in \mathrm{B}\left(X\right)$ and a 
quasi-isomorphism $\underline{\underline{ \mathscr  E}}\to\overline{ 
\mathscr F}^{\infty }$. Therefore,  in the above 
construction, we may as well replace $\mathscr F$ by 
$\underline{\underline{\mathscr E}}$, so that $\phi$ is a 
quasi-isomorphism $\underline{\underline{ \mathscr E}}\to \mathscr 
E$, and $\phi_{1}\phi=\phi_{2}\phi$ up to homotopy in 
$Z^{0}\Hom\left(\mathscr E,\underline{\underline{\mathscr 
E}}\right)_{X}$. Since $\phi$ is a quasi-isomorphism, this is a 
homotopy equivalence, so  that
$\phi_{1}=\phi_{2}\,\mathrm{in}\,\underline{\mathrm{B}}\left(X\right)$, which completes the proof of our theorem.
\end{proof} 
\subsection{Compatibility with pull-backs}%
\label{subsec:puba}
We make the same assumptions as in Subsection \ref{subsec:pullten} 
and we use the corresponding notation. 
By the results of that Subsection, we have defined a 
pull-back functor $f^{*}_{b}:\mathrm{B}\left(Y\right)\to \mathrm{B}\left(X\right)$. Recall 
that the definition of the left-derived functor 
$Lf^{*}:\Db\left(Y\right)\to\Db\left(X\right)$ was given in 
Subsection \ref{subsec:pubaa}.
\begin{proposition}\label{prop:puba}
		The following diagram commutes up to isomorphism:
		\begin{equation}\label{eq:commu1}
\xymatrix{&\underline{\mathrm{B}}\left(Y\right)\ar[r]^{\underline{F}_{Y}}\ar[d]^{f_{b}^{*}}&\Db\left(Y\right)\ar[d]^{Lf^{*}}\\
&\underline{\mathrm{B}}\left(X\right)\ar[r]^{\underline{F}_{X}}&\Db\left(X\right)}
\end{equation}
	\end{proposition}
\begin{proof}
	If $\mathscr E$ is an object in 
	$\underline{\mathrm{B}}\left(Y\right)$, if 
	$\mathcal{E}=\underline{F}_{Y}\left(\mathscr E\right)$, then $\mathcal{E}$ is 
	an object in 
	$\Db\left(Y\right)$. Also $f^{*}\mathcal{E}$ is given by 
	(\ref{eq:pu1x1}).  Since $\mathcal E$ is a bounded 
	complex of flat $\mathcal{O}_{Y}$-modules, by (\ref{eqLf}), we have
\begin{align}\label{eqLfbis}
		Lf^{*}\mathcal{E}=f^{*}\mathcal{E}. 
	\end{align}

By Proposition \ref{prop:rest}, there is a morphism of 
	$\mathcal{O}_{X}$-complexes  $\mu:f^{*}\mathcal{E}\to 
	f^{*}_{b}\mathscr E$. Using (\ref{eqLf}), 
	to establish our proposition,  we only need  to prove  that if  $x\in	X$, the morphism $\mu$ induces a 
	quasi-isomorphism of the corresponding stalks at $x$.
	
Take $x\in X, y=f\left(x\right)\in Y$. By Theorem \ref{thm:Conj},  to establish our theorem, we may and we 
will assume that on a small open neighborhood $V$ of $y$,  $D$ is a 
trivial holomorphic vector bundle, 
$E=\Lambda\left(\overline{T^{*}V}\right)\ho D$, and 
$A^{E \prime \prime }\vert_{V}=v+\n^{D \prime \prime }$. Put 
$U=f^{-1}V$.
By proceeding as in the proof of Theorem \ref{thm:cohco}, we get
\begin{align}\label{eq:pou1}
&\mathscr H  f^{*}\mathcal{E}=\mathscr 
H\left(f^{*}\mathcal{O}_{Y}(D),f^{*}v\right),
&\mathscr H f^{*}_{b}\mathscr 
E\vert_{U}=\mathscr H 
\left(\mathcal{O}_{U}f^{*}D,f^{*}v\vert_{U}\right).
\end{align}

Also note that
\begin{equation}\label{eq:pou2}
f^{-1}\mathcal{O}_{Y} \otimes 
_{f^{-1}O_{Y}}\mathcal{O}_{X}=\mathcal{O}_{X}.
\end{equation}
Since $D$ is trivial on $V$, by (\ref{eq:pou2}), we get
\begin{equation}\label{eq:pou3}
f^{-1}\mathcal{O}_{Y}\left(D\right) \otimes 
_{f^{-1}\mathcal{O}_{Y}}\mathcal{O}_{X}\vert_{U}=\mathcal{O}_{X}\left(f^{*}D\right)\vert_{U}.
\end{equation}

By (\ref{eq:pou1}), (\ref{eq:pou3}), we obtain
\begin{equation}\label{eq:pou4}
 \mathscr H f^{*}\mathcal{E}= \mathscr Hf_{b}^{*} \mathscr E.
\end{equation}
This completes the proof that $\mu$ is a quasi-isomorphism.
\end{proof}
\subsection{Compatibility with tensor products}%
\label{subsec:tepro}
Let $\mathscr E, \mathscr F$ be objects in 
$\underline{\mathrm{B}}\left(X\right)$. We use the notation of Subsections 
\ref{subsec:ten} and \ref{subsec:pullten}.
\begin{proposition}\label{prop:pte}
We have a canonical isomorphism,
\begin{equation}\label{eq:isar1}
\underline{F}_{X}\left(\mathscr E \ho _{b} \mathscr 
F\right) \simeq  \underline{F}_{X} \left( \mathscr E \right)  \ho
_{\mathcal{O}_{X}}^{\mathrm{L}} \underline{F}_{X}\left( \mathscr F \right)
\,\mathrm{in}\,\Db\left(X\right).
\end{equation}
\end{proposition}
\begin{proof}
	We will identify $\mathscr E, \mathscr F$ with their images by 
	$\underline{F}_{X}$ in $\Db\left(X\right)$.
	Since $\mathscr E, \mathscr F$ are flat over 
$\mathcal{O}_{X}$, by (\ref{eq:tenx2}), we get
\begin{equation}\label{eq:rev1}
 \mathscr E\ho^{\mathrm{L}}_{\mathcal{O}_{X}} \mathscr F= \mathscr 
 E\ho_{\mathcal{O}_{X}} \mathscr F.
\end{equation}
There is a natural morphism 
$\phi:\mathscr E\ho_{\mathcal{O}_{X}} \mathscr F\to \mathscr E\ho_{b} 
\mathscr F$ of objects in $\Db\left(X\right)$. We will show that it 
is a quasi-isomorphism.

We write $\mathscr E, \mathscr F$ in the form,
\begin{align}\label{eq:rev2}
&\mathscr E=\left(E,A^{E \prime \prime }\right),&\mathscr 
F=\left(F,A^{F \prime \prime }\right).
\end{align}
Let $D_{E},D_{F}$ be the diagonal vector bundles associated with 
$\mathscr E, \mathscr F$, and let $v_{0,E},v_{0,F}$ be the corresponding 
endomorphisms of $D_{E},D_{F}$.  The objects associated with 
$E\ho_{b}F$ are denoted in the same way. In particular,
\begin{align}\label{eq:rev2a1}
&D_{E\ho_{b}F}=D_{E}\ho D_{F},&v_{0,E\ho_{b}F}=v_{0,E}\ho 1+1\ho 
v_{0,F}.
\end{align}

By Theorem \ref{thm:Conj}, if $U \subset X$ is a small open set, after 
conjugation, we may assume that
\begin{align}\label{eq:rev3}
&A^{E \prime \prime }\vert_{U}=v_{0,E}+\n^{D_{E}\vert_{U} 
\prime \prime},
&A^{F \prime \prime }\vert_{U}=v_{0,F}+\n^{D_{F}\vert_{U} 
\prime \prime}.
\end{align}
In particular $D_{E}\vert_{U}, D_{F}\vert_{U}$ are holomorphic vector 
bundles, so that $D_{E\ho_{b}F} \vert_{U}$ is also a 
holomorphic vector bundle. Let $\n^{D_{E\ho_{b}F}\vert_{U} \prime \prime }$ be the corresponding 
holomorphic structure.  By (\ref{eq:rev3}), we get
\begin{equation}\label{eq:rev4}
A^{E\ho_{b} F \prime \prime 
}\vert_{U}=v_{0,E\ho_{b}F}+\n^{D_{E\ho_{b}F} \vert_{U} \prime 
\prime}.
\end{equation}

By Theorem \ref{thm:cohco},   the morphisms $\left(\mathcal{O}_{U}\left(D_{E} \right) 
,v_{0,E}\right)\to \mathscr E\vert_{U}$,
$\left( \mathcal{O}_{U}\left(D_{F}\right),v_{0,F} 
\right)\to \mathscr F\vert_{U} $ are 
quasi-isomorphisms. Since $\mathscr E\vert_{U}$ and $\mathcal{O}_{U}\left(D_{E}\right)$ are flat 
over $\mathcal{O}_{U}$, and the same is true for the objects 
associated with $F$, the obvious morphism
\begin{equation}\label{eq:rev4a1}
\left( \mathcal{O}_{U}\left(D_{E}\right),v_{0,E} \right)\ho 
_{\mathcal{O}_{U}}
\left(\mathcal{O}_{U}\left(D_{F} \right) ,v_{0,F}\right) \to \mathscr E\ho_{\mathcal{O}_{U}} \mathscr F
\end{equation}
is a quasi-isomorphism.

Observe that
\begin{equation}\label{eq:rev5}
\left(\mathcal{O}_{U}\left(D_{E} \right) ,v_{0,E}\right)\ho_{\mathcal{O}_{U}}
\left(\mathcal{O}_{U}\left(D_{F}\right),v_{0,F}\right)=
\left(\mathcal{O}_{U}\left(D_{E\ho_{b}F}\right),v_{0,E\ho_{b}F}\right).
\end{equation}
  It is now clear that $\phi$ is a quasi-isomorphism of objects in 
$\Db\left(X\right)$. The proof of our proposition is completed. 
\end{proof}
\subsection{Direct images}%
\label{subsec:dirim}
Let $Y$ be a compact complex  manifold, and let $f:X\to Y$ be a 
holomorphic map.

Let  $\mathscr E$ be an object in 
$\underline{\mathrm{B}}\left(X\right)$, so that $\underline{F}_{X}\left(\mathscr 
E\right)$ is an object in $\Db\left(X\right)$. Since $\mathscr E$ is 
a bounded complex of soft $\mathcal{O}_{X}^{\infty}$-modules, by 
(\ref{eq:fon1}), we get
\begin{equation}\label{eq:fon1bis}
Rf_{*} \mathscr E=f_{*} \mathscr E.
\end{equation}
Also $f^{*}$ maps $\mathcal{O}^{\infty }_{Y}\left(\Lambda\left(\overline{T^{*}Y}\right)\right)$ to 
$f_{*}\mathcal{O}^{\infty 
}_{X}\left(\Lambda\left(\overline{\TsX}\right)\right)$. The
 differential of $f_{*} \mathscr E$ verifies Leibniz's rule with 
respect to multiplication by $\mathcal{O}_{Y}^{\infty 
}\left(\Lambda\left(\overline{T^{*}Y}\right)\right)$, so that
$f_{*} \mathscr E$ is a complex of $\mathcal{O}^{\infty 
}_{Y}\Lambda\left(\overline{T^{*}Y}\right)$-modules.
\begin{proposition}\label{prop:exibb}
	There exists $\underline{\mathscr 
	E}=\left(\underline{E},A^{\underline{E} \prime \prime }\right)\in 
	\mathrm{B}\left(Y\right)$ and a morphism of $\mathcal{O}_{Y}^{\infty 
	}\left(\Lambda\left(\overline{T^{*}Y}\right)\right)$-modules
	$\phi:\underline{\mathscr E}\to f_{*} \mathscr E$, which is also 
	a quasi-isomorphism of $\mathcal{O}_{Y}$-complexes.
\end{proposition}
\begin{proof}
We use the notation in Subsection \ref{subsec:prodb}. In 
particular, as in (\ref{eq:coi0}), we get
\begin{equation}\label{eq:vl1}
\overline{f_{*}\mathscr E}^{\infty }=\mathcal{O}_{Y}^{\infty 
}\left(\Lambda\left(\overline{T^{*}Y}\right)\right)\ho_{\mathcal{O}_{Y}}f_{*} \mathscr E.
\end{equation}
Also 
$\mathcal{O}_{Y}^{\infty}\left(\Lambda\left(\overline{T^{*}Y}\right)\right)$ acts on $f_{*} \mathscr E$. We get  this way a morphism of complexes $\psi:\overline{f_{*}\mathscr E}^{\infty }\to f_{*} \mathscr E$, which is also a morphism of $\mathcal{O}_{Y}^{\infty }\left(\Lambda\left(\overline{T^{*}Y}\right)\right)$-modules. By (\ref{eq:coq1}), the map $f_{*}\mathscr E\to \overline{f_{*} \mathscr E}^{\infty }$ is a quasi-isomorphism of $\mathcal{O}_{Y}$-complexes. These maps fit into the triangle
\begin{equation}\label{eq:vl2}
\xymatrix{f_{*} \mathscr E\ar[r] \ar[rd]^{\mathrm{id}}&\overline{f_{*} \mathscr E}^{\infty 
}\ar[d]^{\psi}\\
& f_{*} \mathscr E.}
\end{equation}
From (\ref{eq:vl2}), we deduce that  $\psi$ is a quasi-isomorphism.

By a theorem of Grauert \cite[Theorem 10.4.6]{GrauertRemmert84}, 
$f_{*} \mathscr E$ defines an object in $\Db\left(Y\right)$. By 
Theorem \ref{thm:exib}, there exists an object $\underline{\mathscr 
E}=\left(\underline{E},A^{\underline{E} \prime \prime }\right)$ in 
$\mathrm{B}\left(Y\right)$ and a morphism of $\mathcal{O}_{Y}^{\infty 
}\left(\Lambda \left( \overline{T^{*}Y}\right)\right) $-modules $\rho:\underline{\mathscr E}\to 
\overline{f_{*} \mathscr E}^{\infty }$, which is a quasi-isomorphism of 
$\mathcal{O}_{Y}$-complexes. Then
$\phi=\psi\rho:\underline{\mathscr E}\to f_{*} \mathscr E$ is the 
quasi-isomorphism we are looking for, which completes our proof.	
\end{proof}

\section{Antiholomorphic superconnections and generalized metrics}%
\label{sec:ascngm}
In this Section, if $\left(E,A^{E \prime \prime }\right)$ is an 
antiholomorphic superconnection, if $E$ is equipped with a splitting 
so that $E \simeq E_{0}$, if $h$ is a generalized metric on $D$, we 
define the adjoint superconnection $A^{E_{0} \prime }$ of 
$A^{E_{0}\prime \prime }$ with respect 
to $h$. This is just the obvious analogue of the construction of the 
holomorphic part of the Chern connection on a holomorphic Hermitian 
vector bundle, which we briefly reviewed in 
Subsection \ref{subsec:bovb}.

This Section is organized as follows. In Subsection 
\ref{subsec:adjscn}, we construct $A^{E_{0} \prime }$ as the formal 
adjoint of $A^{E_{0} \prime \prime }$ with respect to a 
non-degenerate Hermitian form $\theta_{h}$.

In Subsection \ref{subsec:curva}, we conclude the Section by 
introducing   the curvature of the 
superconnection $A^{E_{0}}=A^{E_{0} \prime \prime }+A^{E_{0} \prime 
}$.
\subsection{The adjoint of an antiholomorphic superconnection}%
\label{subsec:adjscn}
We make the same assumptions as in Subsection \ref{subsec:def} and we 
use the corresponding notation.

We define the antiautomorphism\, 
 \index{ti@$\widetilde{}$}%
 $\widetilde{}$\, of 
 $\Lambda\left(T^{*}_{\C}X\right)$ as in Subsection \ref{subsec:dua}, 
 so that equation (\ref{eq:bc13ax1}) holds. Also we define 
 \index{deg@$\deg_{-}$}%
 $\deg_{-}$ 
 on $\Lambda\left(T^{*}_{\C}X\right)$ as in (\ref{eq:supp3}). If 
 $\Omega\left(X,\C\right)$ denote the vector space of smooth forms on 
 $X$, we denote by 
 \index{OX@$\Omega^{-,i}\left(X,\C\right)$}%
 $\Omega^{-,i}\left(X,\C\right)$ the vector space 
 of the $\alpha\in \Omega\left(X,\C\right)$ such that 
 $\deg_{-}\,\alpha=i$.
 
Here, we follow \cite[Section 6.6]{Bismut10b}. If $\alpha,\alpha'\in 
\Omega\left(X,\C\right)$, 
put
\begin{equation}\label{eq:iv13}
\theta\left(\alpha,\alpha'\right)=\frac{i^{n}}{\left(2\pi\right)^{n}}\int_{X}^{} \widetilde \alpha\we\overline{\alpha'}.
\end{equation}
One verifies easily that $\theta$ defines a non-degenerate Hermitian form on 
$\Omega\left(X,\C\right)$, and that the 
$\Omega^{-,i}\left(X,\C\right)$ are mutually orthogonal.  If $B\in \Omega\left(X,\C\right)$ acts on 
$\Omega\left(X,\C\right)$ by exterior 
multiplication, we denote by $B^{\dag}$ the adjoint of 
$B$ with respect to $\theta$. An easy computation, which also 
follows from \cite[Proposition 6.5.1]{Bismut10b} shows that if $f\in 
T^{*}_{\C}X$, then
\begin{equation}\label{eq:iv14}
f^{\dag}\we=-\overline{f}\we
\end{equation}
By (\ref{eq:iv14}), we deduce that if $B\in 
\Omega\left(X,\C\right)$, $B^{\dag}$ is just the 
`adjoint' $B^{*}$ we 
defined in Subsection \ref{subsec:dua}. 

Since $\Lambda\left(\TsX\right)$ is a $\Z$-graded holomorphic vector bundle on $X$, using 
the results of Subsection \ref{subsec:tenspro},  
$\Lambda\left(\TsX\right)\ho E$ has  the same properties as $E$, i.e., it can be 
equipped with an antiholomorphic superconnection 
$A^{\Lambda\left(\TsX\right)\ho E \prime \prime}$, and it 
is also a $\Lambda\left(T^{*}_{\C}X\right)$-module. If $\alpha\in 
\Omega\left(X,\C\right),s\in C^{\infty 
}\left(X,\Lambda\left(\TsX\right)\ho E\right)$, the obvious extension of (\ref{eq:iv5}) 
holds, i.e., 
\begin{equation}\label{eq:zom1}
A^{\Lambda\left(\TsX\right)\ho E \prime \prime }\left(\alpha 
s\right)=\left(\overline{\pa}^{X}\alpha 
\right)s+\left(-1\right)^{\mathrm{deg}\alpha}\alpha A^{\Lambda\left(\TsX\right)\ho E 
\prime \prime }s.
\end{equation}
In the sequel, we will still use the notation $A^{E 
\prime \prime }$ instead of $A^{\Lambda\left(\TsX\right)\ho E \prime 
\prime}$.

We fix a splitting of $E$ as in (\ref{eq:iv4a-1}),  (\ref{eq:iv4}), 
so that
\begin{equation}\label{eq:iv15}
\Lambda \left(\TsX\right)\ho 
E \simeq \Lambda\left(T^{*}_{\C}X\right)\ho D.
\end{equation}
Put
\begin{equation}\label{eq:iv5a1}
\Omega\left(X,D\right)=C^{\infty 
}\left(X,\Lambda\left(T^{*}_{\C}X\right)\ho D\right).
\end{equation}
By (\ref{eq:iv15}), (\ref{eq:iv5a1}), we deduce that
\begin{equation}\label{eq:iv15a2}
\Omega\left(X,D\right)=C^{\infty }\left(X,\Lambda\left(T^{*}X\right)\ho 
E\right).
\end{equation}

We equip 
$\Omega\left(X,D\right)$ with the total degree 
associated with the degree $\deg_{-}$ on $\Omega\left(X,\C\right)$ and with the 
degree in $D$.  Let 
\index{OXD@$\Omega^{-,i}\left(X,D\right)$}%
$\Omega^{-,i}\left(X,D\right)$ be the vector 
space of the  $s\in \Omega\left(X,D\right)$ such that  $\deg\, s=i$.

In the sequel, we  write $A^{E_{0} \prime \prime }$ instead 
of $A^{\Lambda\left(\TsX\right)\ho E_{0} \prime \prime}$.
\begin{definition}\label{def:iv16}
	Put
	\begin{align}\label{eq:iv17}
&E^{*i}_{0}=\bigoplus_{p+q=i}\Lambda^{p}\left(\overline{\TsX}\right) \otimes 
	D^{q*},&E_{0}^{*}=\bigoplus_{i=r}^{r'}E_{0}^{i*}.
\end{align}
\end{definition}
Then
\begin{equation}\label{eq:iv18-r1}
C^{\infty }\left(X,\Lambda \left(\TsX\right)\ho 
E^{*}_{0}\right)=
\Omega\left(X,D^{*}\right).
\end{equation}

 If $s\in \Omega\left(X,D\right), s'\in 
\Omega\left(X,\overline{D}^{*}\right)$, we can write $s,s'$ in the form
\begin{align}\label{eq:iv14a1}
&s=\sum_{}^{}\alpha_{i}r_{i}, &s'=\sum_{}^{}\beta_{j}t_{j},
\end{align}
with $\alpha_{i},\beta_{j}\in \Omega\left(X,\C\right)$, and 
$r_{i}\in C^{\infty }\left( X,D \right) , t_{j}\in 
C^{\infty }\left(X,\overline{D}^{*}\right)$.
\begin{definition}\label{def:heg}
	Put
	\index{th@$\theta$}%
	\begin{equation}\label{eq:iv14a2}
\theta\left(s,s'\right)=\frac{i^{n}}{\left(2\pi\right)^{n}}\sum_{}^{}\int_{X}^{}\widetilde \alpha_{i}\we\overline{\beta}_{j}\left\langle  r_{i}, \overline{t}_{j}\right\rangle.
\end{equation}
\end{definition}
Note that given $i,j$, the expression in the right-hand side of 
(\ref{eq:iv14a2}) is nonzero if and only if 
$\deg_{-}\alpha_{i}=\deg_{-}\beta_{j}, \deg\, r_{i}=\deg\, t_{j}$, so that 
$\deg_{-}\alpha_{i}\deg\, r_{i}=\deg_{-}\,\beta_{j}\deg\, t_{j}$.

If $s,s'$ are given instead by
\begin{align}\label{eq:sign0}
&s=\sum_{}^{}r_{i}\alpha_{i},&s'=\sum_{}^{}t_{j}\beta_{j},
\end{align}
from (\ref{eq:iv14a2}), and from the previous considerations, we 
deduce that $\theta\left(s,s'\right)$ is still given by the right-hand side 
of (\ref{eq:iv14a2}).

If $k$ is a 
smooth section of $\End\left(D\right)$, let $\widetilde k$ be its 
transpose, which is a
 section of $\End\left(D^{*}\right)$. It follows from the previous 
 considerations that the adjoint of $k$ with respect to $\theta$ is 
 given by $\overline{\widetilde k}$, i.e., it is just the action of 
 the classical expected adjoint.

We define the  smooth bundle  $\mathscr M^{D}$ of generalized metrics 
on $D$ as in Subsection \ref{subsec:dua}.
As in Subsection \ref{subsec:dua}, we will say that a smooth section 
$h$ of $ \mathscr 
M^{D}$ is pure if $h=h_{0}$.  

Let $h$ be a smooth section of $\mathscr M^{D}$. Then $h$ defines an 
invertible 
morphism from $\Omega\left(X,D\right)$ into 
$\Omega\left(X,\overline{D}^{*}\right)$.
\begin{definition}\label{def:hege}
	If $s,s'\in \Omega\left(X,D\right)$, put
	\index{th@$\theta_{h}$}%
	\begin{equation}\label{eq:iv14a3}
\theta_{h}\left(s,s'\right)=\theta\left(s,hs'\right).
\end{equation}
\end{definition}

Then $\theta_{h}$ is a non-degenerate Hermitian form on 
$\Omega\left(X,D\right)$,   and the $\Omega^{-,i}\left(X,D\right)$ 
are mutually orthogonal with respect to $\theta_{h}$.
\begin{definition}\label{def:adjo}
	Let 
	\index{AE@$A^{E_{0} \prime}$}%
	$A^{E_{0} \prime}$ denote the formal adjoint of $A^{E_{0} \prime \prime }$ 
	with respect to  $\theta_{h}$.
\end{definition}

Then $A^{E _{0}\prime }$ verifies Leibniz rule when replacing 
$\overline{\pa}^{X}$ by $\pa^{X}$, and it is of degree $-1$. Also
\begin{equation}\label{eq:iv15y1}
A^{E_{0} \prime, 2}=0.
\end{equation}

Let $A^{\overline{E}^{*}_{0} \prime }$ denote the formal adjoint of 
$A^{E_{0} \prime \prime }$ with respect to $\theta$. 
Then
\begin{equation}\label{eq:iv16}
A^{E_{0}\prime }=h^{-1}A^{\overline{E}_{0}^{*} \prime }h.
\end{equation}

The above constructions are compatible with the conventions in 
\cite[Subsections 1 (c) and 1 (d)]{BismutLott95} and in \cite[Section 
3.5]{Bismut10b}. In \cite[Proposition 2.17]{Qiang16}, Qiang followed the 
same lines to define the adjoint $A^{E_{0}\prime }$ when $h$ is a 
pure metric.

Let 
\index{nD@$\n^{D \prime}$}%
$\n^{D \prime}$ be the adjoint of $\n^{D \prime \prime }$ with 
respect to $\theta_{h}$. If $i=0, \mathrm{or}\,i\ge 2$, let $v_{i}^{*}$ be the adjoint of $v_{i}$ 
with respect to $\theta_{h}$. By (\ref{eq:iv36a6}), we obtain
\begin{equation}\label{eq:coa1}
A^{E_{0} \prime }=v_{0}^{*}+\n^{D \prime}+\sum_{i\ge 
2}^{}v_{i}^{*}.
\end{equation}
Recall that 
\index{B@$B$}%
$B$ was defined in (\ref{eq:baba1}). Its adjoint $B^{*}$ is given by
\begin{equation}\label{eq:baba3}
B^{*}=v_{0}^{*}+\sum_{i\ge 2}^{}v_{i}^{*}.
\end{equation}
Then equation (\ref{eq:coa1}) can be written in the form
\begin{equation}\label{eq:baba4}
A^{E_{0} \prime}=\n^{D \prime }+B^{*}.
\end{equation}

Let 
\index{nD@$\n^{D \prime}_{0}$}%
$\n^{D \prime }_{0}$ be the adjoint of $\n^{D \prime \prime}$ 
when $h$ is replaced by the pure metric $h_{0}$. Then there is a 
smooth section $\gamma$ of total degree $-1$ in 
$\Lambda^{(\ge 2)}\left(T^{*}_{\C}X\right)\ho\End\left(D\right)$ such that
\begin{equation}\label{eq:coa1a1}
\n^{D \prime}=\n^{D \prime}_{0}+\gamma.
\end{equation}
Put
\begin{equation}\label{eq:eqco}
\n^{D}_{0}=\n^{D \prime \prime }+\n^{D \prime}_{0}.
\end{equation}
Then
 $\n^{D}_{0}$ is a classical unitary connection on  $D$ with respect 
 to $h_{0}$\footnote{In general, 
$\n^{D \prime \prime }$ does not define a holomorphic structure on 
$D$, so that $\n^{D}_{0}$ is not a Chern connection.}. 

The Hermitian metric $h_{0}$ on $D$ induces a classical Hermitian metric 
\index{hde@$h^{\det D}$}%
$h^{\det D}$ on $\det D$. Let $\n^{\det D}$ be the connection on 
$\det D$ that is induced by $\n^{D}_{0}$. Then $\n^{\det D}$ is 
a unitary connection with respect to $h^{\det D}$. By Theorem 
\ref{thm:pdet}, $\n^{\det D \prime \prime }$ is a holomorphic 
structure on $\det D$. 
\begin{proposition}\label{prop:pcond}
	The connection $\n^{\det D}$ coincides with the Chern 
	connection on $\left(\det D,h^{\det D}\right)$.
\end{proposition}
\begin{proof}
	Let $\n^{\det D \prime }$ be the holomorphic connection on $\det 
	D$ that is induced by $\n^{D \prime }_{0}$. By construction, 
	$\n^{\det D \prime }$ is the adjoint of $\n^{\det D \prime \prime 
	}$ with respect to $h^{\det D}$, which gives our proposition.
\end{proof}
\subsection{Curvature}%
\label{subsec:curva}
We make the same assumptions as in Subsection \ref{subsec:adjscn}, 
and  we use again the superconnection formalism of Quillen 
\cite{Quillen85a}.

Let 
\index{AE@$A^{E_{0}}$}%
$A^{E_{0}}$ be the superconnection on $E_{0}$, 
\begin{equation}\label{eq:iv17r1}
A^{E_{0}}=A^{E_{0} \prime\prime}+A^{E_{0} \prime }.
\end{equation}
Set
\index{C@$C$}%
\index{nD@$\n^{D}$}%
\begin{align}\label{eq:dep1}
&\n^{D}=\n^{D \prime \prime }+\n^{D \prime },&C=B+B^{*}.
\end{align}
By (\ref{eq:baba2}), (\ref{eq:baba4}), (\ref{eq:iv17r1}), and 
(\ref{eq:dep1}),  we get
\begin{equation}\label{eq:baba5}
A^{E_{0}}=\n^{D}+C.
\end{equation}

The curvature of $A^{E_{0}}$ is given by
\begin{equation}\label{eq:iv18}
A^{E_{0},2}=\left[A^{E_{0} \prime \prime },A^{E_{0} \prime }\right].
\end{equation}
Then $A^{E_{0},2}$ is a smooth section of 
$\Lambda \left(T^{*}_{\C}X\right)\ho\End\left(D\right)$ of  degree 
$0$. Also, it verifies Bianchi's identities,
\begin{align}\label{eq:iv19}
&\left[A^{E_{0} \prime \prime },A^{E_{0},2}\right]=0,&
\left[A^{E_{0} \prime },A^{E_{0},2}\right]=0.
\end{align}
\begin{proposition}\label{prop:curva}
	The following identity holds:
	\begin{equation}\label{eq:iv20}
\left[A^{E_{0},2}
\right]^{*}=A^{E_{0},2}.
\end{equation}
\end{proposition}
\begin{proof}
	This is the obvious consequence of the fact that $A^{E_{0},\prime 
	}=\left[A^{E_{0} \prime \prime }\right]^{*}$.
\end{proof}
\section{Generalized metrics and Chern character forms}%
\label{sec:chfoge}
In this Section, given an antiholomorphic superconnection and a 
generalized metric, we construct the corresponding Chern character 
forms, and we prove that their Bott-Chern class does not depend on 
the metric. Also, we show that the Chern character extends to 
$\Db\left(X\right)$.

As we explained in the introduction, some results which are proved in 
the present Section were already stated by Qiang in  
\cite{Qiang16,Qiang17}. We will refer to them in more detail in the 
text.

This Section is organized as follows. In Subsection 
\ref{subsec:chchf}, we construct the Chern character forms, and we 
establish their main properties.

In Subsection \ref{subsec:triex}, we give a trivial example 
involving Kähler forms, that will be of special importance in 
Sections \ref{sec:hypo}--\ref{sec:RBC}.

In Subsection \ref{subsec:pba}, we describe the behavior of the 
Chern character forms under pull-backs.

In Subsection \ref{subsec:chtenso}, we evaluate the Chern character 
of a tensor product.

In Subsection \ref{subsec:lofr}, we consider the case where $\mathscr 
H \mathscr E$ is locally free.

In Subsection \ref{subsec:scame}, we study the behaviour of 
the Chern character form under a suitable scaling of the metric. Also 
we show that in degree $(1,1)$, the Chern character is just the first 
Chern class of the Knudsen-Mumford determinant line bundle.

In Subsection \ref{subsec:chcone}, we evaluate the Chern character of 
a cone.

In Subsection \ref{subsec:und}, using the results of  Section \ref{sec:blo}, 
we show that the Chern character defines a map 
$\cBC:\Db\left(X\right)\to H^{(=)}_{\mathrm{BC}}\left(X,\R\right)$.

In Subsection \ref{subsec:chk}, we show that $\cBC$ factors through 
$K\left(X\right)$, i.e., it induces a map $K\left(X\right)\to 
H^{(=)}_{\mathrm{BC}}\left(X,\R\right)$.

Finally, in Subsection \ref{subsec:trunc}, we introduce spectral 
truncations for antiholomorphic superconnections. The results 
obtained there will be used in an infinite-dimensional context in 
Section \ref{sec:lift}.

We make the same assumptions as in Sections \ref{sec:ahscn} and 
\ref{sec:ascngm}, and we use the corresponding notation. 
\subsection{The Chern character forms}%
\label{subsec:chchf}
We fix a square root  of $i=\sqrt{-1}$. Our formulas will not depend on 
this choice. Let
\index{f@$\varphi$}%
$\varphi$ denote the morphism of 
$\Lambda \left(T^{*}_{\C}X\right)$ that maps $\alpha$ to 
$\left(2i\pi\right)^{-\deg \alpha/2}\alpha$.

Recall that by Proposition \ref{prop:psm}, there is a well defined 
supertrace $\Trs:\Lambda\left(T^{*}X\right)\ho\End\left(E\right)\to 
\Lambda\left(T^{*}_{\C}X\right)$, that vanishes on supercommutators. 
Here, we may as well view $\Trs$ as a map from 
$\Lambda\left(T^{*}_{\C}X\right)\ho\End\left(D\right)$ to 
$\Lambda\left(T^{*}_{\C}X\right)$.

Let $\left(E,A^{E \prime \prime }\right)$ be an antiholomorphic 
superconnection on $X$. We fix splittings as in (\ref{eq:iv4a-1}), 
(\ref{eq:iv4}). 
Let $h\in \mathscr M^{D}$ be a generalized metric, and let 
$A^{E_{0}}$ be the associated superconnection defined in 
(\ref{eq:iv17r1}).
\begin{definition}\label{def:chf}
	Set
	\index{chA@$\ch\left(A^{E_{0} \prime \prime},h\right)$}%
	\begin{equation}\label{eq:iv21}
\ch\left(A^{E_{0} \prime \prime 
},h\right)=\varphi\Trs\left[\exp\left(-A^{E_{0},2}\right)\right].
\end{equation}
\end{definition}
Then $\ch\left(A^{E_{0} \prime \prime },h\right)$ is a smooth form on 
$X$.

Let
\index{dMD@$d^{\mathscr M^{D}}$}%
$d^{\mathscr M^{D}}$ denote the de Rham operator on $\mathscr 
M^{D}$. Then $h^{-1}d^{\mathscr M^{D}}h$ is a $1$-form on $\mathscr 
M^{D}$ with values in  morphisms in 
$\left[\Lambda\left(\TsX\right)\ho \End\left(E_{0}\right)\right]^{0}$ that 
are self-adjoint with respect to $h$.

Recall that the classical metric $h_{0}$ on $D$ was defined in 
(\ref{eq:bc13a2}).

We will extend the result  of Bott-Chern \cite[Proposition 3.28]{BottChern65} 
described in Proposition \ref{prop:pbc}, and results of  Quillen \cite[Section 2]{Quillen85a}, and 
Bismut-Gillet-Soulé \cite[Theorems 1.15  and 
1.24]{BismutGilletSoule88a} on superconnection forms, and reestablish a result
proved earlier by Qiang \cite{Qiang16}. Some formal 
aspects of the proof are closely related to  the proof 
of Proposition \ref{prop:pbc}. 
\begin{theorem}\label{thm:carch}
	The form $\ch\left(A^{E_{0} \prime \prime},h\right)$ lies in 
	$\Omega^{(=)}\left(X,\R\right)$, it is closed, and its Bott-Chern 
	cohomology  class does not depend on $h$.
	More precisely, $\varphi\Trs\left[h^{-1}d^{\mathscr M^{D}}h
	\exp\left(-A^{E_{0},2}\right)\right]$ is a $1$-form on $\mathscr 
	M^{D}$ with values in $\Omega^{(=)}\left(X,\R\right)$, and 
	moreover,
	\begin{equation}\label{eq:iv21a1}
d^{\mathscr M^{D}}\ch\left(A^{E_{0} \prime \prime},h\right)=
-\frac{\overline{\pa}^{X}\pa^{X}}{2i\pi}\varphi\Trs\left[h^{-1}d^{\mathscr M^{D}}h\exp\left(-A^{E_{0},2}\right)\right].
\end{equation}

	If $g$ is a smooth 
	section of $\Aut^{0}\left(E_{0}\right)$, the Bott-Chern  cohomology 
	class  of $\ch\left(A^{E_{0} \prime \prime },h\right)$ is unchanged when 
	replacing $A^{E_{0} \prime \prime }$ by $g A^{E_{0} \prime \prime 
	}g^{-1}$. In particular, the Bott-Chern cohomology class of 
	$\ch\left(A^{E_{0} \prime \prime },h\right)$ does not depend on the 
	splitting in (\ref{eq:iv4a-1}), and depends only 
	on $A^{E \prime \prime }$.
	
	Finally, 
	\begin{equation}\label{eq:tot1}
\ch\left(A^{E_{0} \prime \prime },h\right)^{\le 2}=\ch\left(A^{E_{0} \prime 
\prime },h_{0}\right)^{\le 2}.
\end{equation}
\end{theorem}
\begin{proof}
	Let 
	\index{ND@$N^{D}$}%
	$N^{D}\in \End\left(D\right)$ count the degree in $D$, and 
	let 
	$N_{-}^{\Lambda\left(T^{*}_{\C}X\right)}$ count the degree in 
	$\Lambda\left(T_{\C}^{*}X\right)$ as the difference of the 
	antiholomorphic and the holomorphic degrees. Since $A^{E_{0},2}$ has total degree $0$, 
	for $t>0$, we get
	\begin{equation}\label{eq:tot1a1}
e^{tN_{-}^{\Lambda\left(T^{*}_{\C}X\right)}}\exp\left(-A^{E_{0},2}\right)e^{-tN_{-}^{\Lambda\left(T^{*}_{\C}X\right)}}=e^{-tN^{D}}\exp\left(-A^{E_{0},2}\right)e^{tN^{D}}.
\end{equation}
Since $e^{-tN^{D}}$ is an even operator, and since supertraces vanish 
on supercommutators, by (\ref{eq:iv21}), (\ref{eq:tot1a1}), we get
\begin{equation}\label{eq:xing1}
e^{tN_{-}^{\Lambda\left(T^{*}_{\C}X\right)}}\ch\left(A^{E_{0} \prime 
\prime},h\right)=\ch\left(A^{E_{0} \prime 
\prime},h\right),
\end{equation}
which just says that $\ch\left(A^{E_{0} \prime 
\prime},h\right)$ is a form in $\Omega^{(=)}\left(X,\C\right)$.

 Using the 
	Bianchi identities (\ref{eq:iv19}) and the fact that supertraces 
	vanish on supercommutators, we get
	\begin{align}\label{eq:oong1}
&\overline{\pa}^{X}\Trs\left[\exp\left(-A^{E_{0},2}\right)\right]=0,
&\pa^{X}\Trs\left[\exp\left(-A^{E_{0},2}\right)\right]=0,
\end{align}
i.e., the above form is closed.
	Since 
	$A^{E_{0},2}$  is self-adjoint, it is elementary to verify that 
	$\ch\left(A^{E_{0} \prime \prime },h\right)$ is a real form\footnote{In the 
	proof, we use explicitly the fact that $h$ is of degree $0$.}.

 We denote by $\left[\,\right]_{+}$ the anticommutator. Since 
 $\exp\left(-A^{E_{0},2}\right)$ is even, we get
\begin{equation}\label{eq:iv34a1}
\Trs\left[h^{-1}d^{\mathscr M^{D}}h\exp\left(-A^{E_{0},2}\right)\right]=
\frac{1}{2}\Trs\left[\left[h^{-1}d^{\mathscr M^{D}}h,\exp\left(-A^{E_{0},2}\right)\right]_{+}\right].
\end{equation}
The $1$-form $\left[h^{-1}d^{\mathscr 
M^{D}}h,\exp\left(-A^{E_{0},2}\right)\right]_{+}$ takes its values in  
$h$ self-adjoint sections of 
$\Lambda\left(T^{*}_{\C}X\right)\ho\End\left(D\right)$ and is 
 of degree $0$. By the same considerations as before, we deduce 
that the form obtained from  (\ref{eq:iv34a1}) by normalization by $\varphi$  lies 
in $\Omega^{(=)}\left(X,\R\right)$.

In the sequel, we  use (\ref{eq:iv18}) repeatedly. Also, we make  $1$-forms on 
$\mathscr M^{D}$  to anticommute with odd forms on $X$. Clearly, 
\begin{equation}\label{eq:iv34a2}
\left[d^{\mathscr M^{D}},A^{E_{0} \prime }\right]=-\left[A^{E_{0} \prime 
},h^{-1}d^{\mathscr 
M^{D}}h\right],
\end{equation}
so that
\begin{equation}\label{eq:iv34a2x1}
d^{\mathscr M^{D}}\left[A^{E_{0} \prime \prime },A^{E_{0} \prime }\right]=\left[A^{E_{0} \prime \prime},
\left[A^{E_{0} \prime },h^{-1}d^{\mathscr 
M^{D}}h\right]\right].
\end{equation}
By (\ref{eq:iv34a2x1}), we get
\begin{multline}\label{eq:iv34a3}
d^{\mathscr M^{D}}\Trs\left[\exp\left(-A^{E_{0},2}\right)\right]\\
=-\Trs\left[\left[A^{E_{0} \prime \prime },\left[A^{E_{0} \prime},h^{-1}d^{\mathscr 
M^{D}}h\right]\right]\exp\left(-A^{E_{0},2}\right)\right].
\end{multline}
Using the Bianchi identities  (\ref{eq:iv19}) and (\ref{eq:iv34a3}), 
we get (\ref{eq:iv21a1}), from which we deduce 
that the Bott-Chern cohomology class of $\ch\left(A^{E_{0} \prime \prime },h\right)$ 
does not depend on $h$.

Let $g$ be a smooth section of $\Aut^{0}\left(E_{0}\right)$, and let $h$ 
be a smooth section of $\mathscr M^{D}$. Let $g^{*}$ be the adjoint 
of $g$ with respect to $h$\footnote{Our definition of 
$g^{*}$ is different from the one in (\ref{eq:bc13a3}).}.  The adjoint of $gA^{E_{0}\prime 
\prime }g^{-1}$ with respect to $h$ is given by $g^{*-1}A^{E_{0} \prime 
}g^{*}$, so that the corresponding curvature is given by 
$\left[gA^{E_{0}\prime 
\prime }g^{-1},g^{*-1}A^{E_{0} \prime 
}g^{*}\right]$. Observe that
\begin{equation}\label{eq:iv35}
\left[gA^{E_{0}\prime 
\prime }g^{-1},g^{*-1}A^{E_{0} \prime 
}g^{*}\right]=g\left[A^{E_{0} \prime \prime }, 
\left(g^{*}g\right)^{-1}A^{E_{0} \prime }g^{*}g\right]g^{-1}.
\end{equation}
Also $\left(g^{*}g\right)^{-1}A^{E_{0} \prime }g^{*}g$ is just the 
adjoint of $A^{E_{0} \prime \prime }$ with respect to the metric 
\index{gh@$g^{\dag}h$}%
$g^{\dag}h$ as defined in  (\ref{eq:bc13a3}).  By (\ref{eq:iv35}), 
since $g$ is of degree $0$, we get
\begin{equation}\label{eq:iv36}
	\ch\left(gA^{E_{0} \prime \prime 
	}g^{-1},h\right)=\ch\left(A^{E_{0}\prime \prime 
	},g^{\dag}h\right).
\end{equation}
By (\ref{eq:iv36}), we deduce that the Bott-Chern cohomology class of 
$\ch\left(A^{E_{0} \prime \prime },h\right)$ is unchanged when replacing 
$A^{E_{0} \prime \prime }$ by $gA^{E_{0} \prime \prime }g^{-1}$. 

When changing the splitting in equation (\ref{eq:iv4a-1}), as explained 
at the end of Subsection \ref{subsec:filal},  we obtain a smooth 
section $g$ of $\Aut_{0}^{0}\left(E_{0}\right)$, and the new 
antiholomorphic superconnection on $E_{0}$ associated with the fixed 
$A^{E \prime \prime }$ is just $gA^{E_{0} \prime \prime }g^{-1}$. 
Using the above results, it is now clear that the Bott-Chern 
cohomology class of $\ch\left(A^{E_{0} 
\prime \prime },h\right)$ does not depend on the splitting, and so it 
only depends on $A^{E \prime \prime }$.

Let $v_{0,0}^{*}, A^{E_{0} \prime }_{0},A^{E_{0}}_{0}$ denote the 
analogues of $v^{*}_{0}, A^{E_{0} \prime 
\prime },A^{E_{0}}$ when replacing $h$ by the pure metric $h_{0}$.  
Proceeding as in the proof of (\ref{eq:coa1a1}),  there are
smooth sections $\delta,\epsilon$ of $T^{*}_{\C}X\ho 
\End\left(D\right),\Lambda^{(\ge 
2)}\left(T^{*}_{\C}X\right)\ho\End\left(D\right)$ of total degree $0,-1$ such 
that
\begin{equation}\label{eq:deg1}
A^{E_{0} \prime }=A^{E _{0}\prime 
}_{0}+\left[\delta,v_{0,0}^{*}\right]+\epsilon.
\end{equation}
By (\ref{eq:deg1}), we get
\begin{equation}\label{eq:deg2}
A^{E_{0}}=A^{E_{0}}_{0}+\left[\delta,v_{0,0}^{*}\right]+\epsilon.
\end{equation}

For $\ell\in\R$, let $A^{E_{0}}_{\ell}$ be obtained from $A^{E_{0}}$ 
by replacing $\epsilon$ by $\ell \epsilon$. A Chern-Simons 
transgression shows that
\begin{equation}\label{eq:deg3}
\frac{\pa}{\pa 
\ell}\Trs\left[\exp\left(-A^{E_{0},2}_{\ell}\right)\right]=-d^{X}\Trs\left[\epsilon\exp\left(-A^{E,2}_{\ell}\right)\right].
\end{equation}
Since $\epsilon$ is of degree $\ge 2$ in 
$\Lambda\left(T^{*}_{\C}X\right)$, the right-hand side of 
(\ref{eq:deg3}) is of degree $\ge 4$, so that (\ref{eq:deg3}) 
vanishes in degree $\le 2$. We see that in degree $\le 2$, 
\begin{equation}\label{eq:deg3a1}
\Trs\left[\exp\left(-A^{E_{0},2}\right)\right]=\Trs\left[\exp\left(-A^{E_{0},2}_{\ell}\vert_{\ell=0}\right)\right].
\end{equation}

For $m\in\R$, let $B^{E_{0}}_{m}$ be the analogue of $A^{E_{0}}$ in 
(\ref{eq:deg2}), in which $\delta,\epsilon$ are replaced by 
$m\delta,0$. The same argument as before shows that
\begin{equation}\label{eq:deg4}
\frac{\pa}{\pa 
m}\Trs\left[\exp\left(-B^{E_{0},2}_{m}\right)\right]=-d^{X}\Trs\left[\left[\delta,v_{0,0}^{*}\right]\exp\left(-B^{E,2}_{m}\right)\right].
\end{equation}
By (\ref{eq:deg4}), we deduce that
\begin{equation}\label{eq:deg5}
\frac{\pa}{\pa 
m}\Trs\left[\exp\left(-B^{E,2}_{m}\right)\right]^{(\le 2)}=-d^{X}\Trs
\left[\left[\delta,v_{0,0}^{*}\right]\exp\left(-\left[v_{0},v_{0,0}^{*}\right]\right)\right].
\end{equation}
Since $\left[v_{0,0}^{*},\left[v_{0},v_{0,0}^{*}\right]\right]=0$, we 
get
\begin{equation}\label{eq:deg6}
-\Trs\left[\left[\delta,v_{0,0}^{*}\right]\exp\left(-\left[v_{0},v_{0,0}^{*}\right]\right)\right]=\Trs\left[\left[v_{0,0}^{*},\delta\exp\left(-\left[v_{0},v_{0,0}^{*}\right]\right)\right]\right]=0.
\end{equation}
By (\ref{eq:deg6}), (\ref{eq:deg4}) vanishes. Combining 
(\ref{eq:deg3a1}) and the vanishing of (\ref{eq:deg4}), we get 
(\ref{eq:tot1}).
The proof of 
our theorem is completed. 
\end{proof}
\begin{remark}\label{rem:zeph}
The arguments of deformation over $\mathbf{P}^{1}$ given in 
\cite[Section 1 f)]{BismutGilletSoule88a} and in the proof of 
Proposition \ref{prop:pbc} can also be used to prove 
some of the results established in Theorem \ref{thm:carch}. In 
particular, to prove independence on the splitting of $E$, one can 
instead interpolate between two splittings of $E$ over 
$\mathbf{P}^{1}$ and proceed as in \cite{BismutGilletSoule88a}.

In \cite[Proposition 2.24 and Corollary 3.14]{Qiang16}, Qiang established that in the 
case where $h$ is a pure metric, the form $\ch\left(A^{E_{0} \prime 
\prime },h\right)$ is closed, and that its Bott-Chern class does not 
depend on $h$. In \cite[Subsection 4.18]{Qiang16}, Qiang shows that 
the $\cBC\left(A^{E_{0} \prime \prime }\right)$ is invariant  when 
conjugating $A^{E_{0}\prime \prime }$.
\end{remark}
\begin{definition}\label{def:BCcl}
Let
\index{cBC@$\cBC\left(A^{E^{\prime \prime }}\right)$}%
$\cBC\left(A^{E\prime \prime }\right)\in 
H^{(=)}_{\mathrm{BC}}\left(X,\R\right)$ be 
the common Bott-Chern cohomology class of the forms $\ch\left(A^{E_{0} \prime \prime 
},h\right)$.
\end{definition}
\subsection{A trivial example}%
\label{subsec:triex}
Here, we follow \cite[Remark 4.5.3]{Bismut10b}. We consider the 
trivial $\mathscr C^{X}$ in Example \ref{exa:triva},  with $D=\C$,  
and associated  antiholomorphic superconnection $\overline{\pa}^{X}$.  Let $\omega^{X}$ be a smooth real 
$(1,1)$-form. Multiplication by $i\omega^{X}$ 
is a self-adjoint operator with respect to the Hermitian form 
$\theta$ in (\ref{eq:iv13}). The exponential 
$e^{-i\omega^{X}}$ in $\Lambda\left(T^{*}_{\C}X\right)$ 
is also self-adjoint. In particular 
$h_{\omega^{X}}=\exp\left(-i\omega^{X}\right)$ defines a generalized metric on 
$\C$, i.e., it can be viewed as an element of $\mathscr M^{\C}$. 
Using the notation in (\ref{eq:iv14a3}), if $s,s'\in 
\Omega\left(X,\C\right)$, then
\begin{equation}\label{eq:tri1}
\theta_{h_{\omega^{X}}}\left(s,s'\right)=\theta\left(s,e^{-i\omega^{X}}s'\right).
\end{equation}
The adjoint $A^{\C \prime }$ of $A^{\C \prime \prime }$ is given by
\begin{equation}\label{eq:tri2}
A^{\C \prime }=\pa^{X}-i\pa^{X}\omega^{X}.
\end{equation}
If $A^{\C}=A^{\C \prime \prime }+A^{\C \prime }$, the curvature 
$A^{\C,2}$ is given by
\begin{equation}\label{eq:tri3}
A^{\C,2}=-\overline{\pa}^{X}\pa^{X}i\omega^{X}.
\end{equation}
By (\ref{eq:tri3}), we get
\begin{equation}\label{eq:tri4}
\ch\left(A^{\C \prime \prime 
},h_{\omega^{X}}\right)=\exp\left(-\frac{\overline{\pa}^{X}\pa^{X}i\omega^{X}}{4\pi^{2}}\right).
\end{equation}
If we take $\omega^{X}=0$, we get the trivial form $1$. 

More generally, if we make the same assumptions as in Subsection 
\ref{subsec:chchf}, if $h\in \mathscr M^{D}$, $he^{-i\omega^{X}}$ 
lies also in $\mathscr M^{D}$. Even when $h$ is pure, when 
$\omega^{X}$ is nonzero, $he^{-i\omega^{X}}$ is not pure.

The above example will play a fundamental role in Section 
\ref{sec:hypo}. Indeed 
the theory of the hypoelliptic Laplacian in \cite{Bismut10b} can be 
reinterpreted in terms of the above choice of a generalized metric on 
$\C $. in which it is crucial that $\omega^{X}$ is the Kähler form of 
a Hermitian metric on $X$.
\subsection{The Chern character of pull-backs}%
\label{subsec:pba}
Let $Y$ be a compact complex manifold, and let $f:X\to Y$ be a holomorphic map as in Subsection 
\ref{subsec:pullten}. Let $\left(F,A^{F \prime \prime }\right)$ be an 
antiholomorphic superconnection on $Y$, and let $\left(E,A^{E 
\prime \prime }\right)=
f^{*}_{b}\left( F,A^{F \prime \prime } \right) $ be the 
antiholomorphic superconnection on $X$ that was defined  
in Subsection \ref{subsec:pullten}.

We fix a smooth splitting of $F$ as in (\ref{eq:iv4a-1}), 
(\ref{eq:iv4}), that induces 
a corresponding smooth splitting of $E$ on $X$. Then we have identifications 
$F \simeq F_{0}, E \simeq E_{0}$.
Let $h$ be a smooth section of $\mathscr M^{D_{F}}$ on $Y$. Then 
$f^{*}h$ 
is a smooth section of $\mathscr 
M^{D_{E}}$ on $Y$.
\begin{proposition}\label{prop:idpu}
	The following identity of forms on $Y$ holds:
\begin{equation}\label{eq:pb1}
\ch\left(A^{f_{b}^{*}E_{0} \prime \prime 
},f^{*}h\right)=f^{*}\ch\left(A^{E_{0} \prime \prime },h\right).
\end{equation}
Also
\begin{equation}\label{eq:pb1a1}
\cBC\left(A^{f_{b}^{*}E \prime \prime}\right)=f^{*}\cBC\left(A^{E \prime 
\prime }\right)\,\mathrm{in}\,H^{(=)}_{\mathrm{BC}}\left(Y,\R\right).
\end{equation}
\end{proposition}
\begin{proof}
	The proof of equation (\ref{eq:pb1}) is easy and is left to the 
	reader. By taking the corresponding Bott-Chern cohomology class 
	in (\ref{eq:pb1}), we get (\ref{eq:pb1a1}). 
\end{proof}
\subsection{Chern character and tensor products}%
\label{subsec:chtenso}
Let  $\left(E,A^{E \prime \prime }\right),\left( \underline{E}, 
A^{\underline{E}\prime \prime }\right) $ be antiholomorphic 
superconnections on $X$.  Let $\left( E\ho_{b}\underline{E}, 
A^{E\ho_{b}\underline{E} \prime \prime }\right) $ denote the corresponding tensor 
product, which was defined in  Subsection \ref{subsec:tenspro}. Its diagonal bundle is given by $D\ho\underline{D}$. We fix  
splittings for $E,\underline{E}$ 
similar to (\ref{eq:iv4a-1}), (\ref{eq:iv4}), so that 
$E \simeq E_{0},\underline{E} \simeq \underline{E}_{0}$. These two splittings induce 
a corresponding splitting of $E\ho_{b}\underline{E}$, and 
$\left(E\ho_{b} 
\underline{E}\right)_{0}=E_{0}\ho_{b}\underline{E}_{0}$. If 
$h,\underline{h}$ are 
smooth sections of 
$\mathscr M^{D}, \mathscr M^{\underline{D}}$, then $h\ho \underline{h}$ is a 
smooth section of $\mathscr M^{D\ho\underline{D}}$.
\begin{proposition}\label{prop:ptens}
	The following identity of forms on $X$ holds:
	\begin{equation}\label{eq:ten1}
\ch\left(A^{\left( E\ho_{b}\underline{E}\right)_{0}\prime \prime },h\ho 
\underline{h}\right)=\ch\left(A^{E_{0} \prime \prime 
},h\right)\ch\left(A^{\underline{E}_{0} \prime \prime },\underline{h}\right).
\end{equation}
Also
\begin{equation}\label{eq:ten2}
\cBC\left(A^{E\ho_{b} \underline{E} \prime \prime }\right)=\cBC\left(A^{E 
\prime \prime }\right)\cBC\left(A^{\underline{E} \prime \prime }\right)\ \mathrm{in}\ H^{(=)}_{\mathrm{BC}}\left(X,\R\right).
\end{equation}
\end{proposition}
\begin{proof}
	The proof of (\ref{eq:ten1}) is easy and is left to the reader. 
	Equation (\ref{eq:ten2}) is just a consequence.
\end{proof}
\subsection{The case where $\mathscr H \mathscr E$ is locally free}%
\label{subsec:lofr}
We use the notation of Subsection \ref{subsec:applco}.
We fix a splitting of $E$ as in (\ref{eq:iv4a-1}), (\ref{eq:iv4}). We write 
$A^{E_{0} \prime \prime }$ as in (\ref{eq:iv36a6}). Let $h$ be a 
pure metric on $D$.

In this Subsection, we assume that the $\mathcal{O}_{X}$-module $ \mathscr 
H \mathscr E = \mathscr H HD$ is
locally free. By (\ref{eq:iv11a1}),  $HD$ is a smooth holomorphic vector bundle on $X$.

Put
\begin{equation}\label{eq:hod1}
\mathcal{H}D=\left\{f\in D, v_{0}f=0, v_{0}^{*}f=0\right\}.
\end{equation}
Then $\mathcal{H}$ is a smooth vector 
subbundle of $D$. Using finite-dimensional Hodge theory, we have a 
canonical isomorphism of smooth vector bundles,
\begin{equation}\label{eq:hod2}
\mathcal{H}D \simeq HD.
\end{equation}
As a subbundle of $D$, $\mathcal{H}D$ inherits a Hermitian metric. 
Let $h^{HD}$ denote the corresponding smooth Hermitian metric on $HD$. 
Let $\n^{HD}$ be the corresponding Chern connection on $HD$. 

Let $P:D\to \mathcal{H}D$ denote the orthogonal projection from $D$ 
on $\mathcal{H}D$. Then $P$ is a smooth section of 
$\End\left(D\right)$ that preserves the $\Z$-grading. 
Let $\n^{\mathcal{H}D}$ be the connection on $\mathcal{H}D$, 
\begin{equation}\label{eq:hod3}
\n^{\mathcal{H}D}=P\n^{D}.
\end{equation}
Then $\n^{\mathcal{H}D}$ is a unitary connection on $\mathcal{H}D$.
\begin{proposition}\label{prop:ide}
	Via the identification $\mathcal{H}D \simeq HD$, 
	\begin{equation}\label{eq:hod4}
\n^{\mathcal{HD}}=\n^{HD}.
\end{equation}
\end{proposition}
\begin{proof}
Via the identification $\mathcal{H}D \simeq HD$, we have the identity,
\begin{equation}\label{eq:ida0}
\n^{\mathcal{H}D \prime \prime }=\n^{HD \prime \prime }.
\end{equation}
By (\ref{eq:hod3}), we get
	\begin{equation}\label{eq:hod5}
\n^{\mathcal{H}D \prime \prime }=P\n^{D \prime \prime }.
\end{equation}
By (\ref{eq:hod5}), and taking adjoints, we obtain
\begin{equation}\label{eq:hod6}
\n^{\mathcal{H}D \prime }=P\n^{D \prime }.
\end{equation}
By (\ref{eq:ida0})--(\ref{eq:hod6}), we get (\ref{eq:hod4}). The proof of our proposition is completed. 
\end{proof}
\subsection{The Chern character form and the scaling of the metric}%
\label{subsec:scame}
In this Subsection, we assume that $h\in \mathscr M^{D}$ is a pure metric.
Recall that
\index{ND@$N^{D}$}%
$N^{D}$ is the number operator of $D$, i.e., $N^{D}$ acts by 
multiplication by $k$ on $D^{k}$.
\begin{definition}\label{def:newm}
	For $T>0$, let 
	\index{hT@$h_{T}$}%
	$h_{T}\in \mathscr M^{D}$ be the pure metric,
	\begin{equation}\label{eq:iv36a2}
h_{T}=hT^{N^{D}}.
\end{equation}
Let 
\index{AET@$A^{E_{0} \prime}_{T}$}%
$A^{E_{0} \prime}_{T}$ denote the adjoint of $A^{E_{0} \prime \prime }$ 
with respect to $h_{T}$. Put
\begin{equation}\label{eq:iv36a2x1r}
A^{E_{0}}_{T}=A^{E_{0} \prime \prime }+A^{E_{0} \prime }_{T}.
\end{equation}
\end{definition}

If $A^{E_{0} \prime }$ denote the adjoint of $A^{E_{0} \prime \prime }$ with 
respect to $h$, then
\begin{equation}\label{eq:iv36a3}
A^{E_{0} \prime }_{T}=T^{-N^{D}}A^{E_{0} \prime }T^{N^{D}}.
\end{equation}
Set
\begin{align}\label{eq:iv36a4}
&B^{E_{0} \prime \prime }_{T}=T^{N^{D}/2}A^{E_{0} \prime \prime 
}T^{-N^{D}/2}, \nonumber \\
&B^{E_{0} \prime }_{T}=T^{N^{D}/2}A^{E_{0} \prime }_{T}T^{-N^{D}/2},\\
& B^{E_{0}}_{T}=T^{N^{D}/2}A^{E_{0}}_{T}T^{-N^{D}/2}. \nonumber 
\end{align}
Then
\begin{equation}\label{eq:iv36a5}
B^{E_{0}}_{T}=B^{E_{0} \prime \prime }_{T}+B^{E_{0} \prime}_{T}.
\end{equation}
Moreover, $B^{E_{0} \prime }_{T}$ is the $h$-adjoint of $B^{E_{0} \prime 
\prime }_{T}$.

As we saw in Subsection \ref{subsec:lofr}, if $\mathscr H \mathscr E$ 
is locally free, then $HD$ is a holomorphic vector bundle, and the 
pure metric $h$ induces a Hermitian metric $h^{HD}$ on $HD$. Let 
\index{chHD@$\ch\left(HD,h^{HD}\right)$}%
$\ch\left(HD,h^{HD}\right)$ be the Chern form associated with the 
$\Z$-graded holomorphic Hermitian vector bundle $HD$.  This Chern 
form is defined using the methods of Subsection \ref{subsec:bovb}.

Recall that the metric 
\index{hde@$h^{\det D}$}%
$h^{\det D}$ on $\det D$ associated with $h$ was defined after equation 
(\ref{eq:eqco}).
Let $c_{1}\left(\det D,h^{\det D}\right)$ denote the first Chern form 
associated with the holomorphic Hermitian  line bundle $\left(\det 
D,h^{\det D}\right)$.

The Knudsen-Mumford determinant line bundle $\det 
\mathscr H \mathscr E$ was considered in Subsection 
\ref{subsec:detli}, and  we have the canonical isomorphism in 
(\ref{eq:sq1a1}). 

If $\alpha\in \Omega^{(=)}\left(X,\R\right)$, we denote by 
$\alpha^{(1,1)}$ its component in $\Omega^{(1,1)}\left(X,\R\right)$. 
If $\beta\in H^{(=)}_{\mathrm{BC}}\left(X,\R\right)$, let 
$\beta^{(1,1)}$ be the component of $\beta$ in 
$H^{(1,1)}_{\mathrm{B C}}\left(X,\R\right)$.

If $\alpha_{T}\vert_{T>0}$ is a family of smooth forms on $X$, we will 
say that as $T\to + \infty $, 
$\alpha_{T}=\mathcal{O}\left(1/\sqrt{T}\right)$ if for any $k\in \N$, 
there exists $C_{k}>0$ such that for $T\ge 1$,  the sup of the norm of $\alpha_{T}$ 
and its derivatives of order $\le k$ is dominated by $C_{k}/\sqrt{T}$.
\begin{theorem}\label{thm:conv}
	If $\mathscr H \mathscr E$ is locally free,  as $T\to + \infty $, 
	\begin{equation}\label{eq:iv36a2b}
\ch\left(A^{E_{0} \prime \prime },h_{T}\right)=\ch\left(HD, 
h^{HD}\right)+\mathcal{O}\left(1/\sqrt{T}\right).
\end{equation}
Under the same assumption, 
\begin{equation}\label{eq:iv36a3b}
\ch_{\mathrm{BC}}\left(A^{E \prime \prime 
}\right)=\ch_{\mathrm{BC}}\left(HD\right)\, \mathrm{in}\, 
H^{(=)}_{\mathrm{BC}}\left(X,\R\right).
\end{equation}
 In particular, if $HD=0$, then
 \begin{equation}\label{eq:iv36a4b}
\ch_{\mathrm{BC}}\left(A^{E \prime \prime }\right)=0 \, \mathrm{in}\, 
H^{(=)}_{\mathrm{BC}}\left(X,\R\right).
\end{equation}

As $T\to 0$, 
\begin{equation}\label{eq:iv36a4x1}
\ch^{(1,1)}\left(A^{E_{0} \prime \prime },h_{T}\right)\to c_{1}\left(\det 
D,h^{\det D}\right).
\end{equation}
We have the identity,
\begin{equation}\label{eq:terma1}
\cBC^{(1,1)}\left(A^{E \prime \prime 
}\right)=c_{1,\mathrm{BC}}\left(\det \mathscr H \mathscr 
E\right)\ \mathrm{in}\ H_{\mathrm{BC}}^{(1,1)}\left(X,\R\right).
\end{equation}
\end{theorem}
\begin{proof}
	By (\ref{eq:iv21}), (\ref{eq:iv36a4}), we have
	\begin{equation}\label{eq:iv36a3z1}
\ch\left(A^{E_{0} \prime \prime 
},h_{T}\right)=\varphi\Trs\left[\exp\left(-B^{E_{0},2}_{T}\right)\right].
\end{equation}

By (\ref{eq:iv36a6}), (\ref{eq:coa1}), and (\ref{eq:iv36a4}), we get
\begin{align}\label{eq:iv36a7}
&B^{E \prime \prime}_{T}=\sqrt{T}v_{0}+\n^{D \prime \prime }+\sum_{i\ge 
2}^{}T^{\left(1-i\right)/2}v_{i},\\
&B^{E \prime }_{T}=\sqrt{T}v_{0}^{*}+\n^{D \prime}+\sum_{i\ge 
2}^{}T^{\left(1-i\right)/2}v_{i}^{*}. \nonumber 
\end{align}
By (\ref{eq:iv36a7}), we get
\begin{equation}\label{eq:iv36a9}
B^{E}_{T}=\sqrt{T}\left(v_{0}+v_{0}^{*}\right)+\n^{D}+\sum_{i\ge 
2}^{}T^{\left(1-i\right)/2}\left(v_{i}+v_{i}^{*}\right).
\end{equation}

Using Proposition \ref{prop:ide}, (\ref{eq:iv36a9}), and proceeding 
as in  \cite[Corollary 9.6]{BerlineGetzlerVergne} in a finite-dimensional context, we get 
(\ref{eq:iv36a2b}). By (\ref{eq:iv36a2b}), we obtain (\ref{eq:iv36a3b}), 
and also (\ref{eq:iv36a4b}). 

Set
\begin{equation}\label{eq:tronc1}
B^{E_{0},\le 1}_{T}=\sqrt{T} \left( v_{0}+v_{0}^{*} \right) +\n^{D}.
\end{equation}
Then $B^{E_{0}}_{T}-B^{E_{0},\le 1}_{T}$ is a smooth section of 
$\Lambda\left(T^{*}_{\C}X\right)\ho\End\left(D\right)$ whose 
degree in $\Lambda\left(T^{*}_{\C}X\right)$ is $\ge 2$. By proceeding as in 
(\ref{eq:deg3}), we deduce that
\begin{equation}\label{eq:tronc2}
\Trs\left[\exp\left(-B^{E_{0},2}_{T}\right)\right]^{(2)}=\Trs\left[\exp\left(-B^{E_{0},\le 1,2}_{T}\right)\right]^{(2)}.
\end{equation}
As $T\to 0$, 
\begin{equation}\label{eq:tronc3}
\Trs\left[\exp\left(-B^{E_{0},\le 1,2}_{T}\right)\right]^{(2)}\to
\Trs\left[\exp\left(-\n^{D,2}\right)\right]^{(2)}=\Trs\left[-\n^{D,2}\right].
\end{equation}
Using Proposition \ref{prop:pcond}, we get
\begin{equation}\label{eq:tronc4}
\varphi\Trs\left[-\n^{D,2}\right]=c_{1}\left(\det D,h^{\det D}\right).
\end{equation}
By (\ref{eq:tronc2})--(\ref{eq:tronc4}), we obtain 
(\ref{eq:iv36a4x1}). Finally, by combining Theorem \ref{thm:pdet} and 
(\ref{eq:iv36a4x1}), we get (\ref{eq:terma1}). 
The proof of our theorem is completed. 
\end{proof}
\begin{remark}\label{rem:bfgs}
	The argument used in the proof of (\ref{eq:tronc2}) is exactly 
	the same as the argument used by Bismut-Freed in a smooth 
	infinite-dimensional context in their proof of 
	\cite[eq. (1.53) in Theorem 1.19]{BismutFreed86b}. This argument 
	reappears in a holomorphic infinite-dimensional context in 
	\cite[Theorems 1.13 and 1.15]{BismutGilletSoule88c}. However, in 
	these references,  the 
	addition of `bad' diverging terms in the considered 
	superconnections plays a key role in the convergence of the 
	superconnection forms as $T\to 0$, while here, the situation  is the 
	opposite. Observe that in \cite[Theorem 4.21]{Qiang16}, when 
	$HD=0$, Qiang had established Theorem \ref{thm:conv} by similar 
	arguments to the ones we just gave.
\end{remark}
\subsection{The Chern character of a cone}%
	\label{subsec:chcone}
	We use the notation of Subsections \ref{subsec:moco} and \ref{subsec:scmoco}. In 
	particular, $\phi$ is a morphism $E\to\underline{E}$ of degree 
	$0$, and $\left(C,A^{C \prime \prime}_{\phi}\right)$ is the 
	corresponding cone. Here we reobtain a result of Qiang 
	\cite[Proposition 4.24]{Qiang16}. 
	\begin{theorem}\label{thm:idc}
	The following identity holds:
	\begin{equation}\label{eq:cor2}
\cBC\left(A^{C \prime 
\prime}_{\phi}\right)=\ch_{\mathrm{BC}}\left(A^{\underline{E} \prime \prime 
}\right)-\ch_{\mathrm{BC}}\left(A^{E \prime \prime }\right)\, 
\mathrm{in}\, H^{(=)}_{\mathrm{BC}}\left(X,\R\right).
\end{equation}
In particular, if $\phi$ is a 
quasi-isomorphism, then
\begin{equation}\label{eq:cor2a1}
\cBC\left(A^{E \prime \prime }\right)=\cBC\left(A^{\underline{E} 
\prime \prime }\right) \, 
\mathrm{in}\, H^{(=)}_{\mathrm{BC}}\left(X,\R\right).
\end{equation}
\end{theorem}
\begin{proof}
	 For $t\in \C$, let $\left(C_{t}, A^{C \prime \prime 
	 }_{t\phi}\right)$ be 
	the complex of the cone associated with $t\phi$ as in Subsection 
	\ref{subsec:scmoco}. Recall that 
	$M_{t}$ is given by (\ref{eq:for1}).  If  $t\in \C^{*}$, $M_{t}$ 
	is invertible, and equation (\ref{eq:co1}) holds.	 By Theorem \ref{thm:carch}, if $t\in\C^{*}$,
	\begin{equation}\label{eq:cor2a2}
\cBC\left(A^{C \prime \prime }_{t\phi}\right)=\cBC\left(A^{C \prime 
\prime }_{\phi}\right).
\end{equation}
Moreover, as $t\to 0$, we get
\begin{equation}\label{eq:cor2a3}
\cBC\left(A^{C \prime \prime }_{t\phi}\right)\to\cBC\left(
A^{C \prime \prime }_{0}\right).
\end{equation}
We have the trivial identity
\begin{equation}\label{eq:cor2a4}
\cBC\left(A^{C \prime \prime }_{0}\right)=\cBC\left(A^{\underline{E} 
\prime \prime }\right)-\cBC\left(A^{E \prime \prime }\right).
\end{equation}
By (\ref{eq:cor2a2})--(\ref{eq:cor2a4}), we get (\ref{eq:cor2}).

If $\phi$ is a quasi-isomorphism, using equation (\ref{eq:iv36a4b}) in 
Theorem \ref{thm:conv}, we get
\begin{equation}\label{eq:cor5a1}
\cBC\left(A^{C \prime \prime }_{\phi}\right)=0.
\end{equation}
By combining (\ref{eq:cor2}) and (\ref{eq:cor5a1}), we get 
(\ref{eq:cor2a1}).
The 
proof of our theorem is completed. 
\end{proof}
\subsection{The Chern character  on $\Db\left(X\right)$}%
\label{subsec:und}
By Theorem \ref{thm:exib},  if $\mathscr F$ is an object in $\Db\left(X\right)$, there is 
an antiholomorphic superconnection  $\mathscr E=\left(E,A^{E \prime \prime }\right)$ such 
that $\mathscr E$ is  isomorphic to $\mathscr F$ in 
$\Db\left(X\right)$. 

Here we reobtain a result of Qiang \cite[Theorem 4.25]{Qiang16}.
\begin{theorem}\label{thm:cla}
	The class $\ch_{\mathrm{BC}}\left(A^{E \prime \prime }\right)\in 
	H^{(=)}_{\mathrm{BC}}\left(X,\R\right)$ 
	 depends only on the isomorphism class of $\mathscr F$ in 
	$\Db\left(X\right)$.
\end{theorem}
\begin{proof}
	Let $\underline{\mathscr F}\in \Db\left(X\right)$ be isomorphic 
	to $\mathscr F$, and let $\underline{\mathscr E}$ be a 
	corresponding antiholomorphic superconnection.  By Theorem 
	\ref{thm:eqcat}, $\underline{F}_{X}$ is an equivalence of 
	categories. Therefore,   there exists a 
	quasi-isomorphism $\phi: \mathscr E\to \underline{\mathscr E}$.
	Our theorem follows from Theorem \ref{thm:idc}.
	\end{proof}
	\begin{definition}\label{def:cDb}
		Let 
		\index{cBCF@$\cBC\left(\mathscr F\right)$}%
		$\cBC\left(\mathscr F\right)\in 
		H^{(=)}_{\mathrm{BC}}\left(X,\R\right)$ 
		be the common Bott-Chern class of the above $\cBC\left(A^{E \prime 
		\prime }\right)$.
	\end{definition}
	
	In the sequel, if $\mathscr E=\left(E,A^{E \prime \prime 
	}\right)$, we will often identify $\mathscr E$ and 
	$\underline{F}_{X}\left(\mathscr E\right)$, and we will also use 
	the notation 
	\index{cBCE@$\cBC\left(\mathscr E\right)$}%
	$\cBC\left(\mathscr E\right)$ instead of 
	$\cBC\left(A^{E \prime \prime }\right)$.

	Let $X,Y$ be compact complex manifolds, and let $f:X\to Y$ be a 
	holomorphic map. If $\mathscr F$ is an object in 
	$\Db\left(Y\right)$, then $Lf^{*} \mathscr F$ is an object in 
	$\Db\left(X\right)$.
	\begin{theorem}\label{thm:pu}
		The following identity holds:
		\begin{equation}\label{eq:pu1a}
\cBC\left(Lf^{*} \mathscr F\right)=f^{*}\cBC\left(\mathscr F\right).
\end{equation}
	\end{theorem}
	\begin{proof}
		As before, we may and we will assume that instead $\mathscr 
		E$ is an object in $\mathrm{B}\left(Y\right)$.  Our 
		theorem is now a consequence of Propositions \ref{prop:puba} 
		and    
		\ref{prop:idpu}.
	\end{proof}
	
 If $\mathscr F, \underline{\mathscr F}$ are objects in 
	$\Db\left(X\right)$, their derived tensor product $\mathscr 
	F\ho_{\mathcal{O}_{X}}^{\mathrm{L}}\underline{\mathscr F}$ was 
	defined in Subsection \ref{subsec:ten}.
	\begin{theorem}\label{thm:tepr}
		The following identity holds:
		\begin{equation}\label{eq:pu1aZ1}
\cBC\left(\mathscr 
	F\ho_{\mathcal{O}_{X}}^{\mathrm{L}}\underline{\mathscr 
	F}\right)=\cBC\left(\mathscr 
	F\right)\cBC\left(\underline{\mathscr 
	F}\right)\,\mathrm{in}\,H^{(=)}_{\mathrm{BC}}\left(X,\R\right).
\end{equation}
\end{theorem}
\begin{proof}
	We may as replace $\mathscr F, \underline{\mathscr F}$ by 
	antiholomorphic superconnections $\mathscr E, \underline{\mathscr 
	E}$. Our Theorem 
	follows from Propositions \ref{prop:pte} and 
	\ref{prop:ptens}. 
\end{proof}
\subsection{The Chern character on $K\left(X\right)$}%
	\label{subsec:chk}
	Let 
	\index{KDb@$K\left(\Db\left(X\right)\right)$}%
	$K\left(\Db\left(X\right)\right)$ denote the $K$-theory of 
	$\Db\left(X\right)$. By definition, this is the Grothendieck group
	generated  by the objects in $\Db\left(X\right)$, with  relations coming from the 
	construction of cones. Namely if $\mathscr F,\underline{\mathscr 
	F} \in \Db\left(X\right)$, if $\phi:\mathscr 
	F\to\underline{\mathscr  F}$ is a morphism of 
	$\mathcal{O}_{X}$-complexes, let $\mathscr C\in 
	\Db\left(X\right)$ be the cone of $\mathscr F,\underline{\mathscr 
	F}$. Then\footnote{This condition is equivalent to the standard 
	condition using instead  distinguished triangles in $\Db\left(X\right)$
	\cite[\href{https://stacks.math.columbia.edu/tag/0FCN}{Tag 0FCN}]{stacks-project}.}
	\begin{equation}\label{eq:form1}
\mathscr C+ \mathscr F=\underline{\mathscr F}\ \mathrm{in}\ 
K\left(\Db\left(X\right)\right).
\end{equation}

Let 
\index{KX@$K\left(X\right)$}%
$K\left(X\right)$ be the Grothendieck group of coherent sheaves 
on $X$. By (\ref{eq:hipi1}), (\ref{eq:hipi2}), we have the identity
\begin{equation}\label{eq:form1a1}
K\left(\Db\left(X\right)\right)=K\left(X\right).
\end{equation}
The derived tensor product on $\Db\left(X\right)$ induces a 
corresponding tensor product on $K\left(X\right)$, so that 
$K\left(X\right)$ is a commutative ring.

By Theorems \ref{thm:idc} and \ref{thm:cla}, and by (\ref{eq:form1a1}), the Chern character 
$\cBC$ can be viewed as a morphism of rings from 
$K\left(X\right)$ into $H^{(=)}_{\mathrm{BC}}\left(X,\R\right)$. 

If $\mathscr F$ is an object in $ \Db\left(X\right)$, we will often identify $\mathscr F$ 
to its image in $K\left(X\right)$. 

Since the $\mathscr H^{i} \mathscr F$ are coherent sheaves, they can 
be viewed as elements of $\Db\left(X\right)$ equipped with the $0$ 
differential.
Put
\index{RHF@$\mathscr R \mathscr H \mathscr F$}%
\begin{equation}\label{eq:dirim}
\mathscr R \mathscr H \mathscr F=\sum_{}^{}\left(-1\right)^{i}\mathscr H^{i} 
\mathscr F,
\end{equation}
so that $\mathscr R \mathscr H \mathscr F\in 
K\left(X\right)$.
\begin{theorem}\label{thm:kth}
		If $\mathscr F\in \Db\left(X\right)$, then
		\begin{equation}\label{eq:iv43}
\ch_{\mathrm{BC}}\left(\mathscr 
F\right)=\ch_{\mathrm{BC}}\left(\mathscr R\mathscr H\mathscr 
F\right).
\end{equation}
	\end{theorem}
	\begin{proof}
		 By (\ref{eq:hipi1}), (\ref{eq:hipi2}), we have the identity,
		 \begin{equation}\label{eq:coc1}
\mathscr F=\mathscr R \mathscr H \mathscr F\ \mathrm{in}\
K\left(X\right).
\end{equation}
Taking the Chern character of (\ref{eq:coc1}), we get 
(\ref{eq:iv43}). The proof of our theorem is completed. 
\end{proof}
\begin{remark}\label{rem:cons}
	By Theorem \ref{thm:kth}, we deduce that if $\mathscr 
	E=\left(E,A^{E \prime \prime }\right)$ is an antiholomorphic 
	superconnection, then
	\begin{equation}\label{eq:coc1a1}
\cBC\left(\mathscr E\right)=\cBC\left(\mathscr R \mathscr H \mathscr 
E\right).
\end{equation}
As we saw in Remark \ref{rem:loccst}, $\mathscr H \mathscr E$ depends 
only on $D,v_{0},\n^{D}$. By (\ref{eq:coc1a1}), the same is true for 
$\cBC\left(\mathscr E\right)$. When $\mathscr H \mathscr E= \mathscr 
H HD$ is locally free, this is also a consequence of Theorem 
\ref{thm:conv}. 
\end{remark}

If $\n^{D \prime \prime }$ defines a holomorphic structure on $D$, i.e., if $\n^{D 
\prime \prime ,2}=0$, we can define $\cBC\left(D\right)$. 
\begin{proposition}\label{prop:phol}
If 	$\n^{D \prime \prime ,2}=0$, then
\begin{equation}\label{eq:coc1a2}
\cBC\left(\mathscr E\right)=\cBC\left(D\right).
\end{equation}
\end{proposition}
\begin{proof}
	For $t\in \C$, put
	\begin{equation}\label{eq:coc1a3}
A^{E_{0} \prime \prime }_{t}=\n^{D \prime \prime }+tv_{0}.
\end{equation}
Set
\begin{equation}\label{eq:coc1a4}
\mathscr E_{t}=\left(E_{0},A^{E_{0} \prime \prime }_{t}\right).
\end{equation}
By Remarks \ref{rem:loccst} and \ref{rem:cons}, for $t\neq 0$,  $\mathscr H \mathscr E= 
\mathscr H \mathscr E_{t}$. By (\ref{eq:coc1a1}), for $t\neq 0$,  we get
\begin{equation}\label{eq:coc1a4r1}
\cBC\left(\mathscr E\right)=\cBC\left(\mathscr E_{t}\right). 
\end{equation}
Also as $t\in \C^{*}\to 0$, 
\begin{equation}\label{eq:coc1a5}
\cBC\left(\mathscr E_{t}\right)\to\cBC\left(D\right).
\end{equation}
By (\ref{eq:coc1a4}), (\ref{eq:coc1a5}), we get (\ref{eq:coc1a2}). The proof of our proposition is completed. 
\end{proof}

If $\mathscr H \mathscr E$ is locally free, our proposition is 
also a consequence of Theorem \ref{thm:conv}.
\subsection{Spectral truncations}%
\label{subsec:trunc}
We fix a splitting as in (\ref{eq:iv4a-1}), (\ref{eq:iv4}),  and also a Hermitian 
metric $g^{D}$ on $D$.  We write $A^{E_{0}\prime \prime }$ as in 
(\ref{eq:iv36a6}).  In particular, 
\begin{equation}\label{eq:tro0}
A^{E_{0} \prime \prime }=v_{0}+A^{E_{0 }\prime \prime (\ge 1)}.
\end{equation}
Put
\begin{equation}\label{eq:tro1}
	\mathbb A^{E_{0}}=A^{E_{0} \prime \prime }+v_{0}^{*}.
\end{equation}
As in (\ref{eq:iv18}), (\ref{eq:iv19}), we get
\begin{align}\label{eq:tro2}
&\mathbb A^{E_{0},2}=\left[A^{E_{0}, \prime \prime },v_{0}^{*}\right],
&\left[A^{E_{0} \prime \prime },\mathbb A^{E_{0},2}\right]=0,
\qquad \left[v_{0}^{*},\mathbb A^{E_{0},2}\right]=0.
\end{align}
Note that $\mathbb A^{E_{0},2}$ is the piece of $A^{E_{0},2}$ in 
which the terms that contain factors of positive degree in 
$\Lambda\left(T^{*}X\right)$   have been killed.

Put
\begin{equation}\label{eq:tro3}
V_{0}=v_{0}+v_{0}^{*}.
\end{equation}
Then 
\begin{equation}\label{eq:tro4}
V_{0}^{2}=\left[v_{0},v_{0}^{*}\right].
\end{equation}
Also $\mathrm{Sp} V_{0}^{2} \subset \R_{+}$. 

Put
\begin{equation}\label{eq:tro4a1}
J=\left[A^{E_{0} \prime \prime (\ge 
1)},v_{0}^{*}\right].
\end{equation}
Then $J$ is a smooth section of degree $0$ in 
$\Lambda\left(\overline{\TsX}\right)\ho\End\left(D\right)$, which 
only contains terms of positive degrees in 
$\Lambda\left(\overline{\TsX}\right)$.

Then
\begin{equation}\label{eq:tro5}
\mathbb A^{E_{0},2}=V_{0}^{2}+J.
\end{equation}
Since $\Lambda^{(\ge 1)}\left(\overline{\TsX}\right)$ is 
nilpotent, from (\ref{eq:tro5}), we get
\begin{equation}\label{eq:tro6}
\mathrm{Sp}\mathbb A^{E_{0},2}=\mathrm{Sp}  
V_{0}^{2}.
\end{equation}
If $\lambda\in \C,\lambda\notin\mathrm{Sp}\left(V_{0}^{2}\right)$, 
then
\begin{equation}\label{eq:tro7}
\left(\lambda-\mathbb 
A^{E_{0},2}\right)^{-1}=\left(\lambda-V_{0}^{2}\right)^{-1}+
\left(\lambda-V_{0}^{2}\right)^{-1}J\left(\lambda-V_{0}^{2}\right)^{-1}+\ldots
\end{equation}
and the expansion only contains a finite number of terms.

For $a>0$, set
\begin{equation}\label{eq:tro7a}
U_{a}=\left\{x\in X,a\notin \mathrm{Sp}\,V_{0}^{2}\right\}.
\end{equation}
Then $U_{a}$ is an open set in $X$. 
\begin{definition}\label{def:tro8}
	Over $U_{a}$, put
\index{Pa@$P_{a,-}$}%
	\begin{equation}\label{eq:tro9}
P_{a,-}=\frac{1}{2i\pi}\int_{\substack{\lambda\in\C\\\left\vert  \lambda\right\vert=a}}^{}\frac{d\lambda}{\lambda-\mathbb A^{E_{0},2}}.
\end{equation}
\end{definition}
Then $P_{a,-}$ is a projector that acts on $E_{0}$, and 
commutes with the action of $\Lambda\left(\overline{\TsX}\right)$. By 
(\ref{eq:tro2}), we get
\begin{align}\label{eq:tro9a1}
&\left[A^{E_{0} \prime \prime},P_{a,-}\right]=0, 
&\left[v_{0}^{*},P_{a,-}\right]=0.
\end{align}

Put
\begin{equation}\label{eq:tro9a2}
P_{a,+}=1-P_{a,-}.
\end{equation}
Then $P_{a,+}$ is also a projector.

 By (\ref{eq:tro7}), 
(\ref{eq:tro9}), we can write $P_{a,\pm}$ in the form
\begin{equation}\label{eq:tro10}
P_{a,\pm}=P_{a,\pm}^{0}+P_{a,\pm}^{(\ge 1)}.
\end{equation}
Let $D_{a,\pm}$ be the direct sums of the eigenspaces of $V_{0}^{2}$ for 
eigenvalues  $\lambda>a$ or $\lambda<a$. Then $P^{0}_{a,\pm}$ is the 
orthogonal projectors $D\to D_{a,\pm}$. Since $P_{a,\pm}=P_{a,\pm}^{2}$, we get
\begin{equation}\label{eq:tro11}
P_{a,\pm}^{(\ge 1)}=\left[P^{0}_{a,\pm},P_{a,\pm}^{(\ge 1)}\right]_{+}+P_{a,\pm}^{(\ge 1),2}.
\end{equation}
In particular $P^{1}_{a,\pm}$ maps $D_{a,\pm}$ into $\overline{\TsX}\ho 
D_{a,\mp}$.
\begin{definition}\label{def:Ea}
	On $U_{a}$, put
	\index{Ea@$E_{0,a,\pm}$}%
	\begin{equation}\label{eq:tro12}
E_{0,a,\pm}=P_{a,\pm}E_{0}.
\end{equation}
\end{definition}
On $U_{a}$,  $E_{0,a,\pm}$ is a subbundle of $E_{0}$ which is also a 
$\Lambda\left(\overline{\TsX}\right)$-module. Let $i_{a,\pm}$ be 
the corresponding embedding in $E_{0}$. In the sequel, we use the 
notation in Subsection \ref{subsec:filal}, and in particular equation 
(\ref{eq:bc6r1}).
\begin{theorem}\label{thm:pea}
	On $U_{a}$, we have  a splitting of 
	$\Lambda\left(\overline{\TsX}\right)$-modules,
	\begin{equation}\label{eq:tro12a1}
E_{0}=E_{0,a,+} \oplus E_{0,a,-}.
\end{equation}
Moreover, 
\begin{equation}\label{eq:mod1}
F^{p}E_{0,a,\pm}=E_{0,a,\pm}\cap F^{p}E_{0}.
\end{equation}

Also $P_{a,\pm}$ induces a filtered isomorphism of 
$\Lambda\left(\overline{\TsX}\right)$-modules,
\begin{equation}\label{eq:tro12az1}
	 \Lambda\left(\overline{\TsX}\right)\ho D_{a,\pm} \simeq E_{0,a,\pm}.
\end{equation}
As a $\Lambda\left(\overline{\TsX}\right)$-module, $E_{0,a,\pm}$ 
verifies the conditions  in (\ref{eq:bc12}), and the associated diagonal bundle  is 
just $D_{a,\pm}$. Also it is equipped with a splitting as in 
(\ref{eq:iv4a-1}),(\ref{eq:iv4}).

Moreover, 
	$A^{E_{0}\prime \prime }$ preserves the smooth sections of 
	$E_{0,a,\pm}$, and it induces an antiholomorphic 
	superconnection  $A^{E_{0,a,\pm} \prime \prime }$ on 
	$E_{0,a,\pm}$, so that on $U_{a}$, we have the splitting
	\begin{equation}\label{eq:tro13}
\mathscr E_{0}=\mathscr E_{0,a,+} \oplus \mathscr E_{0,a,-}.
\end{equation}
Finally, on $U_{a}$,  $P_{a,-}:\mathscr E_{0}\to \mathscr E_{0,a,-}$ and 
$i_{a,-}:\mathscr E_{0,a,-}\to \mathscr E_{0}$ are quasi-isomorphisms of 
$\mathcal{O}_{U_{a}}$-complexes, and $\mathscr H \mathscr 
E_{0,a,+}=0$.
\end{theorem}
\begin{proof}
Equation (\ref{eq:tro12a1}) is obvious. Since  $P_{a,\pm}$ commutes 
with $\Lambda\left(\overline{\TsX}\right)$, it maps $F^{p}E_{0}$ into 
$F^{p}E_{0,a,\pm}$, from which we get (\ref{eq:mod1}).

 Let $Q_{a,\pm}$ 
be the morphism $ \Lambda^{p}\left(\overline{\TsX}\right)\ho 
D_{a,\pm}\to F^{p}E_{0,a,\pm}/F^{p+1}E_{0,a,\pm}$ 
induced by $P_{a,\pm}$. If $e\in 
\Lambda^{p}\left(\overline{\TsX}\right)\ho D_{a,\pm}$, and  if $P_{a,\pm}e\in 
F^{p+1}E_{0,a,\pm}$,  by (\ref{eq:tro10}), we find that $e=0$, so that $Q_{a,\pm}$ is injective. Let us 
prove $Q_{a,\pm}$ is surjective. If $e\in E_{0,a,\pm}$, by (\ref{eq:tro10}), we 
deduce that
\begin{equation}\label{eq:tro13a}
e=P^{0}_{a,\pm}e+P^{(\ge1)}_{a,\pm}e, 
\end{equation}
so that
\begin{equation}\label{eq:tro14}
e=P_{a,\pm}P^{0}_{a,\pm}e+P_{a,\pm}P^{(\ge 1)}_{a,\pm}e.
\end{equation}
If $e\in F^{p}E_{0,a,\pm}$, by (\ref{eq:tro14}), 
 $e-P_{a,\pm}P^{0}_{a,\pm}e\in F^{p+1}E_{0,a,\pm}$, which shows that 
$Q_{a,\pm}$ is surjective. By (\ref{eq:tro12az1}), we conclude that 
$E_{0,a,\pm}$ verifies the conditions in (\ref{eq:bc12}), and it is 
equipped with a splitting as in (\ref{eq:bc12a1}) which is induced 
by (\ref{eq:tro12az1}).

By the first identity in (\ref{eq:tro9a1}), $A^{E_{0} \prime \prime 
}$ acts on smooth sections of $E_{0,a,\pm}$. Therefore, it induces an 
antiholomorphic superconnection $A^{E_{0,a,\pm}\prime \prime }$ on $E_{0,a,\pm}$, and (\ref{eq:tro13}) 
holds.

Clearly, $\mathbb A^{E_{0,a,+},2}$ acts as an invertible operator on 
$E_{0,a,+}$, $v_{0}^{*}$ acts on $E_{0,a,+}$, and by (\ref{eq:tro2}), 
\begin{equation}\label{eq:acy1}
1\vert_{E_{0,a,+}}=\left[A^{E_{0,a,+}\prime \prime },v_{0}^{*}
\left[\mathbb A^{E_{0,a,+},2}\right]^{-1}\right].
\end{equation}
By (\ref{eq:acy1}), we find that $\mathscr H \mathscr E_{a,+}=0$.

By fiberwise Hodge theory, for any $x\in X$, 
$P_{a,-}^{0}:D\to D_{a,-}$ and $i_{a,-}^{0}: D_{a,-}\to D$ are 
quasi-isomorphisms. By Proposition \ref{prop:prophtpeq}, we conclude 
that $P_{a,-}$ and $i_{a,-}$ are quasi-isomorphisms. 

The proof of our theorem is completed. 
\end{proof}

We will view $A^{E_{0,a,\pm} \prime \prime }$ as a superconnection 
$ \mathsf A^{\prime \prime }_{a,\pm}$  on 
$\Lambda\left(\overline{\TsX}\right)\ho D_{a,\pm}$. By construction, we 
have the identity
\begin{equation}\label{eq:tro15}
P_{a,\pm}\mathsf A^{\prime \prime }_{a,\pm}=A^{E_{0} \prime \prime 
}P_{a,\pm}.
\end{equation}

Let us now express $\mathsf A^{\prime \prime }_{a,\pm}$ as in (\ref{eq:iv36a6}), 
i.e.,
\begin{equation}\label{eq:tro16}
\mathsf A^{\prime \prime }_{a,\pm}=v_{0,a,\pm}+\n^{D_{a,\pm} 
\prime \prime }+\sum_{i\ge 2}^{}v_{i,a,\pm}.
\end{equation}
By  (\ref{eq:tro10}),  (\ref{eq:tro15}), we can determine the various terms in 
(\ref{eq:tro16}) by recursion.
\begin{proposition}\label{prop:rec1}
	The following identities hold:
	\begin{align}\label{eq:tro17}
&v_{0,a,\pm}=v_{0\vert_{D_{a,\pm}}},
&\n^{D_{a,\pm}\prime \prime }=P^{0}_{a,\pm}\n^{D \prime \prime 
}P^{0}_{a,\pm}.
\end{align}
\end{proposition}
\begin{proof}
	In degree $0$, equation (\ref{eq:tro15}) gives the first equation 
	in (\ref{eq:tro17}). In degree $1$, we get
	\begin{equation}\label{eq:tro18}
\n^{D_{a,\pm} \prime \prime }+P^{1}_{a,\pm}v_{0,a,\pm}=\n^{D \prime 
\prime }P_{a,\pm}^{0}+v_{0}P_{a,\pm}^{1}.
\end{equation}
Since $P^{1}_{a,\pm}$ exchanges $D_{a,+}$ and $D_{a,-}$, from 
(\ref{eq:tro18}), we get the second equation in (\ref{eq:tro17}). The proof of our proposition is completed. 
\end{proof}
\begin{remark}\label{rem:holdet}
	By Theorem \ref{thm:pdet}, we know that $\n^{D \prime \prime 
	},\n^{D_{a,\pm}\prime \prime}$ induce holomorphic structures on the lines 
	$\det D, \det D_{a,\pm}$. Since $D=D_{a,+} \oplus D_{a,-}$, we 
	conclude that we have the smooth isomorphism
	\begin{equation}\label{eq:tro19}
\det D=\det D_{a,+}\ho\det D_{a,-}.
\end{equation}
By (\ref{eq:tro18}), it is elementary to verify that (\ref{eq:tro19}) 
is an isomorphism of holomorphic line bundles on $U_{a}$. As 
explained in \cite{BismutGilletSoule88a}, since 
$\left(D_{a,+},v_{0}\vert_{D_{a,+}}\right)$ is exact, $\det D_{a,+}$ 
has a canonical section $\tau_{a,+}$. Since $\n^{D_{a,+ 
}\prime \prime }v_{0,a,+}=0$, the section $\tau_{a,+}$ is holomorphic. By the above, 
the map $s\in \det D_{a,-}\to s\ho \tau_{a,+}\in \det D$ is a 
holomorphic identification of line bundles. 

The above completely elucidates the arguments given in an infinite 
dimensional context in \cite[Theorem 1.3]{BismutGilletSoule88c},  in 
which $\det D_{a,-}$ was shown directly to carry a holomorphic 
structure, The explanation is now obvious: $E_{a,-}$ carries an 
antiholomorphic superconnection.
\end{remark}

Observe that for $a<b$, on $U_{a}\cap U_{b}$, $E_{0,a,-}\subset 
E_{0,b,-}$, and that
\begin{equation}\label{eq:aga1}
P_{a,-}=P_{a,-}P_{b,-}.
\end{equation}
In particular, on $U_{a}\cap U_{b}$, $P_{a,-}\vert_{\mathscr 
E_{0,b,-}}:\mathscr E_{0,b,-}\to 
\mathscr E_{0,a,-}$, and $i_{a,b}=\mathscr E_{0,a,-}\to \mathscr E_{0,b,-}$ 
are quasi-isomorphisms.

Let $g^{D ,1}$ be another Hermitian metric on $D$. We 
denote with an extra superscript $1$ the objects we just constructed 
that are associated with $g^{D,1}$.
\begin{proposition}\label{prop:proj}
	For  $a>0,a^{1}>0$, on $U_{a}\cap U^{1}_{a_{1}}$, the map 
	$P_{a,-}\vert_{\mathscr E^{1}_{0,a_{1},-}}:\mathscr E^{1}_{0,a_{1},-}\to 
	\mathscr E_{0,a,-}$ is a quasi-isomorphism.
\end{proposition}
\begin{proof}
	Clearly,
\begin{equation}\label{eq:acy2}
P_{a,-}\vert_{\mathscr E^{1}_{0,a_{1},-}}=P_{a,-}i^{1}_{a_{1},-}.
\end{equation}
Our proposition follows from  Theorem \ref{thm:pea} and from 
(\ref{eq:acy2}).
\end{proof}
\section{The case of embeddings}%
\label{sec:imm}
The purpose of this Section is to establish our 
Riemann-Roch-Grothendieck theorem in the case of an embedding 
$i_{X,Y}:X\to Y$. Our proof uses all the properties we established 
in the previous Sections on $\cBC$, and also the deformation to the 
normal cone. As a consequence we show that our Chern character 
$\cBC:K\left(X\right)\to H_{\mathrm{BC}}^{(=)}\left(X,\R\right)$ verifies the 
uniqueness conditions stated by Grivaux \cite{Grivaux10}.

This Section is organized as follows. In Subsection 
\ref{subsec:emditr}, we recall elementary facts on embeddings, direct 
images, and transversality.

In Subsection \ref{subsec:defnc}, given an embedding $i_{X,Y}:X\to 
Y$, we describe  the deformation  to the normal cone.

In Subsection \ref{subsec:rrim}, we establish our main theorem for 
the embedding $i_{X,Y}$.

Finally, in Subsection \ref{subsec:uni}, from our 
Riemann-Roch-Grothendieck formula for embeddings, we show that our Chern 
character $\cBC$ verifies the uniqueness conditions of Grivaux 
\cite{Grivaux10}.
\subsection{Embeddings, direct images, and transversality}%
\label{subsec:emditr}
Let $Z$ be a compact complex manifold. Let $i_{X,Z}: X\to Z, 
i_{Y,Z}:Y\to Z$ be two holomorphic embeddings of compact complex 
manifolds.  We assume that $U=Y\cap Z$ is a compact submanifold of 
$Z$, and also that
\begin{equation}\label{eq:tra0}
TU=TX\vert_{U}\cap TY\vert_{U}.
\end{equation}
We denote by $i_{U,X},i_{U,Y}$ the embeddings of $U$ in $X,Y$.

The excess normal bundle $ N$ on $U$ is defined to be
\begin{equation}\label{eq:zeph}
N=TZ\vert_{U}/\left(TX\vert_{U}+TY\vert_{U}\right).
\end{equation}
We have the exact sequence of vector bundles 
\begin{equation}\label{eq:tra1}
0\to N_{U/Y}\to N_{X/Z}\vert_{U}\to 
N\to 0.
\end{equation}
The manifolds $X$ and $Y$ are said to be transverse if $N=0$, which 
is equivalent to
\begin{equation}\label{eq:tra2}
N_{U/Y}=N_{X/Z}\vert_{U}.
\end{equation}

We use the notation of Section \ref{sec:derblo}.
Let $\mathscr F\in \Db\left(X\right)$, and let $i_{X,Z,*}\mathscr 
F\in\Db\left(Z\right)$ be its direct image, which coincides with the 
right derived image $Ri_{X,Z*}\mathscr F$.   Let $Li_{U,X}^{*}\mathscr 
F\in \Db\left(U\right)$ be the left derived pull-back of $\mathscr 
F$. Other similar objects will be denoted in the same way.
\begin{proposition}\label{prop:fund}
	If $X$ and $Y$ are transverse,  there exists an isomorphism 
	in $\Db\left(Y\right)$, 
	\begin{equation}\label{eq:tra3}
Li_{Y,Z}^{*}i_{X,Z,*} \mathscr F \simeq i_{U,Y,*}Li_{U,X}^{*}\mathscr F.
\end{equation}
\end{proposition}
\begin{proof}
	Let $\mathscr E_{X}=\left( E_{
	X}, A^{E_{X}\prime \prime } \right) , 
	\mathscr E_{Z}=\left(E_{Z},A^{E_{Z} \prime \prime }\right)$ be  antiholomorphic 
	superconnections 
	on $X,Z$ that represent   $\mathscr F$ and $i_{X,Z,*} \mathscr 
	F$.  
	 By  Proposition \ref{prop:puba}, $Li_{U,X}^{*} 
\mathscr F$ is represented by $i^{*}_{U,X,b} \mathscr E_{X}$, 
and $Li_{Y,Z}^{*}i_{X,Z,*} \mathscr F $ is represented by $i_{Y,Z,b}^{*} 
\mathscr E_{Z}$.

	As in Subsection \ref{subsec:dirim}, $C^{\infty }\left(X,E_{X}\right)$ can  be viewed 
	as a 
	$\mathcal{O}_{Z}^{\infty 
	}\left(\Lambda\left(\overline{T^{*}Z}\right)\right)$-module, with 
	the convention that if $\alpha\in \Omega^{0,\Ou}\left(Z,\C\right), s\in C^{\infty }\left(X,E_{X}\right)$, then
	\begin{equation}\label{eq:aja1}
\alpha .s=\left( i_{X,Z}^{*}\alpha \right)s.
\end{equation}
We will denote the corresponding $\mathcal{O}_{Z}$-complex by 
$i_{X,Z,*} \mathscr E_{X}$. As we saw before,    
 $i_{X,Z,*} \mathscr E_{X}$  defines an object in
 $\Db\left(Z\right)$, that 
coincides with its right-derived direct image.

By Proposition \ref{prop:exibb}, there exists a morphism of  
	$\mathcal{O}_{Z}$-complexes $r_{Z,X}:\mathscr E_{Z} \to i_{X,Z,*} \mathscr E_{X} $  which induces also a morphism of 
	$\mathcal{O}_{Z}^{\infty 
	}\Lambda\left(\overline{T^{*}Z}\right)$-modules,
and  $r_{Z,X}$  is a quasi-isomorphism of 
$\mathcal{O}_{Z}$-complexes.

To make our notation simpler, we will use the notation $\mathscr 
E_{X}$ instead of $i_{X,Z,*} \mathscr E_{X}$. A similar notation will 
be used when considering the embedding $i_{U,Y}$.

The above maps can be combined in a commutative diagram\footnote{As 
explained in Subsection \ref{subsec:pullten}, the 
subscript $b$ is used to define the pull-back of antiholomorphic 
superconnections.}
	\begin{equation}\label{eq:commux1}
\xymatrix{&\mathscr 
E_{Z}\ar[d]^{i_{Y,Z,b}^{*}}\ar[r]^{r_{Z,X}}&\mathscr 
E_{X}\ar[d]^{i_{U,X,b}^{*}}\\
&i_{Y,Z,b}^{*}\mathscr E_{Z} \ar[r]^{r_{Y,U}}&
i_{U,X,b}^{*} \mathscr E_{X}.}
\end{equation}
in which  $r_{Y,U}$ is induced by $r_{Z,X}$. The first row consists 
of objects in $\Db\left(Z\right)$, the second row of objects in 
$\Db\left(Y\right)$. Also $r_{Y,U}$ is a morphism of 
$\mathcal{O}^{\infty }_{Y}\Lambda\left(\overline{T^{*}Y}\right)$-modules, 
and of $\mathcal{O}_{Y}$-complexes.

As we saw before,  $r_{Z,X}$ is a quasi-isomorphism of 
$\mathcal{O}_{Z}$-complexes. To establish our 
proposition, we  need to show that $r_{Y,U}$ is a 
quasi-isomorphism of objects in $\Db\left(Y\right)$. The proof will be obtained via local arguments.

 Since $r_{Z,X}$ is a quasi-isomorphism, on $Z\setminus X$, $\mathscr 
 H \mathscr E_{Z}=0$. By a local version of Theorem  \ref{thm:cohco}, on
 $ Z\setminus X$,  $HD_{Z}=0$, so that on $Y\setminus U$, 
 $Hi^{*}_{Y,Z}D_{Z}=0$.    Using again Theorem \ref{thm:cohco}, we 
 find that on $Y\setminus U$, $\mathscr H i_{Y,Z,b}^{*} \mathscr E_{Z}=0$.
 
 Take $x\in U$.  If $V \subset Z$ is a small neighborhood of  $x$, we  choose a holomorphic coordinate 
system  on $V $ such that $x$ is represented by $0\in 
\C^{n}$, and   $X,Y$ are 
two transverse vector subspaces $H_{X}, H_{Y}$ of $\C^{n}$, so that 
$U$ is represented by $H_{X}\cap H_{Y}$. If $K $ is a 
vector subspace of $H_{Y}$  such that $H_{X} \oplus K=\C^{n}$,  then $K$ 
represents both $N_{X/Z}$ and $N_{U/Y}$.  Let $V_{H_{X}}, V_{K}$ be 
open neighborhoods of $0$ in $H_{X},K$ so that $V_{H_{X}}\times V_{K} 
\subset V$.  By Theorem \ref{thm:Conj}, if $V_{H_{X}}$ is small enough, after a 
gauge transformation of total degree $0$, $A^{E_{X} 
\prime \prime }$ can be written in the form
\begin{equation}\label{eq:mult2}
A^{E_{X} \prime \prime }=\n^{D_{X} \prime \prime }+v_{X,0}.
\end{equation}

Let $\pi_{H_{X}},\pi_{K}$ be the projections of $\C^{n}$ on $H_{X}, K$. 
Let $y$ be the generic section of $K$, and let $\left(\Lambda  K^{*}, 
i_{y}\right)$ denote the Koszul complex of $K$. Set
\begin{equation}\label{eq:mult3-1}
\mathbf{E}_{Z}=\Lambda\left(\overline{\C}^{n}\right)\ho\pi_{H_{X}}^{*}D_{X}\ho \pi^{*}_{K}\Lambda K^{*}.
\end{equation}
Let $A^{\mathbf{E}_{Z} \prime \prime}$ be the antiholomorphic superconnection,
\begin{equation}\label{eq:mult3}
A^{\mathbf{E}_{Z} \prime \prime }=\pi_{H_{X}}^{*}\left(\n^{D _{X}\prime \prime 
}+v_{X,0}\right)+\pi_{K}^{*}\left(\overline{\pa}^{K}+i_{y}\right).
\end{equation}
Then $A^{\mathbf{E}_{Z} \prime \prime }$ is an antiholomorphic 
superconnection on $\C^{n}$ near $0$. Let $\mathcal{E}_{Z}$ denote 
the corresponding complex of $\mathcal{O}_{V}$-modules.  Let 
$r_{\C^{n},H_{X}}$ be the
projection $\pi^{*}_{H_{X}}D_{X}\ho\pi^{*}_{K}\Lambda K^{*}\to 
D_{X}$.  Then $r_{\C^{n},H_{X}}$ extends to a morphism 
$\mathcal{E}_{Z}\to \mathscr E_{X}$ which has the same properties as $r_{Z,X}$. 
Because of known properties of the Koszul complex, if the considered 
neighborhoods are small enough, $r_{\C^{n},H_{X}}$ is a 
quasi-isomorphism.

Near $z\in Z$, both $r_{Z,X}$ and $r_{\C^{n},H_{X}}$ provide 
quasi-isomorphisms of $\mathscr E_{Z}, \mathcal{E}_{Z}$ with $\mathscr 
E_{X}$,  so that for $V$ small 
enough, $\mathscr E_{Z}$ and $\mathbf{E}_{Z}$ are isomorphic as 
objects in $\Db\left(V_{H_{X}}\times V_{K}\right)$. By a local 
version of Theorem 
\ref{thm:eqcat}, there is a corresponding quasi-isomorphism 
$\phi:\mathscr E_{Z}\to \mathcal{E}_{Z}$. As we saw in Subsection 
\ref{subsec:pullten}, $\phi$ induces a morphism $i_{Y,Z,b}^{*}\phi: 
i_{Y,Z,b}^{*} \mathscr E_{Z}\to i^{*}_{Y,Z,b}\mathcal{E}_{Z}$.  By  a 
local version of Proposition  \ref{prop:prophtpeq}, 
for $z\in V_{H_{X}}\times V_{K}$, $\phi_{z}: D_{Z,z}\to  
\left(\pi_{H^{X}}^{*}D_{X}\ho \pi_{K}^{*}\Lambda K^{*}\right)_{z}$ is a 
quasi-isomorphism. In particular, this will be true for $z\in \left( 
V_{H_{X}\times }V_{K} \right) \cap Y$. 
This shows that near $z\in Y$, $\phi$ induces a quasi-isomorphism 
$i_{Y,Z,b}^{*} \mathscr E_{Z}\to i_{Y,Z,b}^{*}\mathcal{E}_{Z}$.

Since near $x\in Y$, the restriction of the above Koszul complex to 
$Y$ is still a Koszul complex,  $r_{K,U}: 
i_{Y,Z,b}^{*}\mathcal{E}_{Z}\to  i_{U,X,b}^{*} \mathscr E_{X}$ is a 
quasi-isomorphism. This shows that $r_{Y,U}:i^{*}_{Y,Z,b}\mathscr E_{Z}\to 
i_{U,X,b}^{*} \mathscr E_{X}$ is a quasi-isomorphism. The proof of our proposition is completed. 
\end{proof}
 \subsection{Deformation to the normal cone}%
\label{subsec:defnc}
Let $i_{X,Y }:X\to Y$ be a holomorphic embedding of compact complex 
manifolds. Let $N_{X/Y}$ be the normal bundle to $X$ in $Y$, i.e., 
\begin{equation}\label{eq:coh1}
N_{X/Y}=i^{*}_{X,Y}TY/TX.
\end{equation}

First,  we construct the deformation of $X$ to the normal cone to 
$Y$ as in  \cite[Section 
4]{BismutGilletSoule90a}. 
Let	$W=W_{X/Y}$ be the blow-up of $Y\times 
\mathbf{P}^{1}$ along $X\times \infty $.  Then 
$X\times \mathbf{P}^{1}$ embeds in $W$.  Let $P$ be the exceptional divisor of the blow-up, i.e., 
\begin{equation}\label{eq:bc2}
P=\mathbf{P}\left(N_{X\times \infty /Y\times\mathbf{P}^{1}}\right).
\end{equation}
Let $p_{X},p_{\infty }$ be the projections $X\times \infty \to X, 
X\times \infty \to \infty $. Put
\begin{equation}\label{eq:bc4a}
A=p_{X}^{*}N_{X/Y} \otimes p_{\infty }^{*}N^{-1}_{\infty /\mathbf{P}^{1}}.
\end{equation}
Then
\begin{equation}\label{eq:bc3}
P=\mathbf{P}\left(A
\oplus \C\right),
\end{equation}
so that $P$ is the projective completion of $A $, and its divisor at 
$\infty $ is given by $\mathbf{P}\left(A\right)=\mathbf{P}\left(N_{X/Y}\right)$.

Let $\widetilde Y$ be the blow-up of $Y$ along $X$. Then 
$\mathbf{P}\left(N_{X/Y}\right)$ is the exceptional divisor in 
$\widetilde Y$. Let $q_{W,Y}: W\to Y, q_{W,\mathbf{P}^{1}}:W\to 
\mathbf{P}^{1}$ be the obvious maps. 
For $z\in \mathbf{P}^{1}$, put
\begin{equation}\label{eq:bc4-a-1}
Y_{z}=q^{-1}_{W,\mathbf{P}^{1}}z.
\end{equation}
Then
\begin{align}\label{eq:bc4-a}
Y_{z}=
 &Y\ \mathrm{if}\,z\neq \infty,\\
&P\cup \widetilde Y\ \mathrm{if}\,  z= \infty. \nonumber 
\end{align}
For $z= \infty $, $P$ and $\widetilde Y$ meet transversally along 
$\mathbf{P}\left(N_{X/Y}\right)$. Also the projection 
$q_{W,\mathbf{P}^{1}}$ is a submersion except on  
$\mathbf{P}\left(N_{Y/X}\right)$ where it has ordinary singularities.
\begin{figure}
    \includegraphics[width=3in]{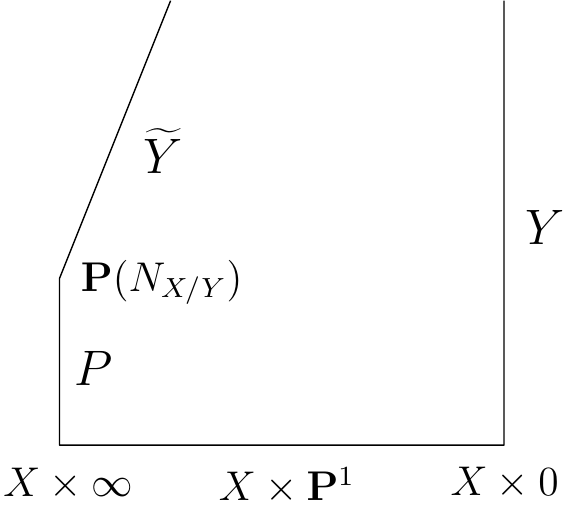}
	\caption
	\centering
    \end{figure}

Let $U=\mathcal{O}_{P}\left(-1\right)$ be the universal line bundle 
on $P$. We have the exact sequence of holomorphic vector bundles 
on $P$,
\begin{equation}\label{eq:bc5r1}
0\to U\to A  \oplus  \C\to \left(A \oplus \C\right)/U\to 0.
\end{equation}
The image $\sigma$ of $1\in \C$ in $\left(A\oplus \C\right)/U$ is a holomorphic section of 
$\left(A \oplus \C\right)/U$  that vanishes exactly on $X\times \infty$. On 
$\mathbf{P}\left(A\right)$, $U$ restricts to the corresponding universal line 
bundle,  the exact sequence (\ref{eq:bc5r1}) restricts to
\begin{equation}\label{eq:bc5a}
0\to U\to A \oplus\C \to A/U \oplus \C\to 0, 
\end{equation}
and $\sigma$ restricts to the section $1$ of $A/U \oplus \C$.

Consider the Koszul complex 
$\left(\Lambda 
\left(  \left(A+\C\right)/U\right)^{*}   ,i_{\sigma}\right)$ on $P$. This complex 
provides a resolution of $i_{X\times \infty ,P,*}\mathcal{O}_{X\times \infty }$. On 
$\mathbf{P}\left(N_{X/Y}\right)$, this complex is just  the split 
complex
$\Lambda \left(A/U\right)^{*}\ho\left(\Lambda\C,i_{1}\right)$.
\subsection{A Riemann-Roch-Grothendieck theorem for embeddings}%
\label{subsec:rrim}
Consider the Bott-Chern class
$\Td_{\mathrm{BC}}\left(N_{X/Y}\right)\in 
H^{(=)}_{\mathrm{BC}}\left(X,\R\right)$.

Let $\mathscr{F}\in \mathrm{D}^{\mathrm{b}}_{\mathrm{coh}}\left(X\right)$, and let 
$i_{X,Y,*}\mathscr F\in 
\mathrm{D}^{\mathrm{b}}_{\mathrm{coh}}\left(Y\right)$ be its direct image.
\begin{theorem}\label{thm:tim}
	The following identity holds:
	\begin{equation}\label{eq:bc1}
\ch_{\mathrm{BC}}\left(i_{X,Y,*}\mathscr 
F\right)=i_{X,Y,*}\frac{\ch_{\mathrm{BC}}\left(\mathscr 
F\right)}{\Td_{\mathrm{BC}}\left(N_{X/Y}\right)}\,\mathrm{in}\,H^{(=)}_{\mathrm{BC}}\left(Y,\R\right).
\end{equation}
\end{theorem}
\begin{proof}
Let $i_{X\times \mathbf{P}^{1},W}: X\times \mathbf{P}^{1}\to W$ be 
the obvious embedding, 
let $\pi_{X\times\mathbf{P}^{1},X}:X\times\mathbf{P}^{1}\to X$ be 
the  natural projection. 
Then $\pi_{X\times \mathbf{P}^{1},X}^{*}\mathscr F\in 
\mathrm{D}^{\mathrm{b}}_{\mathrm{coh}}\left(X\times 
\mathbf{P}^{1}\right)$\footnote{By 
\cite[\href{https://stacks.math.columbia.edu/tag/00R4}{Tag 00R4}]{stacks-project}, $\pi_{X\times \mathbf{P}^{1},X}$ is flat, so 
that $\pi^{*}_{X\times\mathbf{P^{1},X}}\mathscr 
F=L\pi^{*}_{X\times\mathbf{P}^{1},X} \mathscr F$.}, so that 
\begin{equation}\label{eq:bc2r1}
i_{X\times \mathbf{P}^{1},W,*}\pi_{X\times 
\mathbf{P}^{1},X}^{*}\mathscr {F}\in 
\mathrm{D}^{\mathrm{b}}_{\mathrm{coh}}\left(W\right).
\end{equation}

Let 
\begin{align}\label{eq:emb1}
&i_{Y,W}:Y= q^{-1}_{W,\mathbf{P}^{1}}\left(0\right)\to W,
& i_{ P,W}:
P\to W,\\
&i_{\widetilde Y,W}: \widetilde Y\to W, &i_{X\times 
\infty,P}:X\times \infty \to P.\nonumber 
\end{align}
 be the obvious 
embeddings. Other embeddings will be denoted in the same way. 

Note that $\pi_{X\times\mathbf{P}^{1},X}i_{X,X\times \mathbf{P}^{1}}$ 
is the identity of $X$, so that 
\begin{equation}\label{eq:tom1}
Li^{*}_{X,X\times\mathbf{P}^{1}}L\pi^{*}_{X\times\mathbf{P}^{1},X}
\mathscr F \simeq  \mathscr F\ \mathrm{in}\ \Db\left(X\right).
\end{equation}
Also $\pi_{X\times\mathbf{P}^{1},X}i_{X\times \infty ,X\times 
\mathbf{P}^{1}}$ is the identification of $X\times \infty $ with 
$X$. When identifying these two spaces, we get
\begin{equation}\label{eq:tom2}
Li^{*}_{X\times \infty ,X\times 
\mathbf{P}^{1}}L\pi^{*}_{X\times\mathbf{P}^{1},X} \mathscr F \simeq  \mathscr 
F\ \mathrm{in}\ \Db\left(X\right).
\end{equation}

Observe that $Y=q_{W,\mathbf{P}^{1}}^{-1}\left(0\right)$ and $P$ are 
 transverse to $X\times\mathbf{P}^{1}$. Using Proposition 
\ref{prop:fund}, in which $X,Y,Z, \mathscr F$ are replaced by 
$X\times\mathbf{P}^{1}$, $Y$ or $P$, $W$, 
$L\pi^{*}_{X\times\mathbf{P}^{1},X}\mathscr F$, and also 
(\ref{eq:tom1}), (\ref{eq:tom2}),  we get
 \begin{align}\label{eq:bc3a1}
&\left( Li_{Y,W}^{*}\right)i_{X\times\mathbf{P}^{1},W*}\left( L\pi_{X\times 
\mathbf{P}^{1},X}^{*} \right)
\mathscr F \simeq  i_{X,Y,*}\mathscr F\ \mathrm{in}\ \Db\left(Y\right),\\
&\left( Li_{P,W}^{*}\right) i_{X\times \mathbf{P}^{1},W,*}\left( L\pi_{X\times 
\mathbf{P}^{1},X}^{*} \right)
\mathscr F \simeq  i_{X\times \infty,P,*}\mathscr F\ 
\mathrm{in}\ \Db\left(P\right). \nonumber 
\end{align}
 Using Theorem \ref{thm:pu} and 
 (\ref{eq:bc3a1}),we get
\begin{align}\label{eq:bc4}
&i_{Y,W}^{*}\ch_{\mathrm{BC}}\left(i_{X\times 
\mathbf{P}^{1},W,*}\left( L\pi_{X\times \mathbf{P}^{1},X}^{*}\right) \mathscr 
F\right)=\ch_{\mathrm{BC}}\left(i_{X,Y,*} \mathscr F\right)\,\mathrm{in}\,
H^{(=)}_{\mathrm{BC}}\left
(Y,\R\right),  \\
&i_{P,W}^{*}\ch_{\mathrm{BC}}\left(i_{X\times 
\mathbf{P}^{1},W,*}\left( L\pi_{X\times \mathbf{P}^{1},X}^{*}\right) \mathscr 
F\right)=\ch_{\mathrm{BC}}\left(i_{X\times \infty ,P,*} \mathscr F\right)
\mathrm{in}\,
H^{(=)}_{\mathrm{BC}}\left
(P,\R\right). \nonumber 
\end{align}

Let $\alpha\in \Omega^{(=)}\left(W,\R\right)$ be a smooth closed form on $W$ 
representing 
\begin{equation}\label{eq:alph1}
\ch_{\mathrm{BC}}\left(i_{X\times 
\mathbf{P}^{1},W,*}\left(L\pi_{X\times \mathbf{P}^{1},X}^{*}\right)\mathscr 
F\right)\in H^{(=)}_{\mathrm{BC}}\left(W,\R\right).
\end{equation}
 By  
(\ref{eq:bc4}),  $i_{Y,W}^{*}\alpha\in \Omega^{(=)}\left(Y,\R\right)$ represents
 $\ch_{\mathrm{BC}}\left(i_{X,Y,*} \mathscr F\right)\in 
 H_{\mathrm{BC}}^{(=)}\left(Y,\R\right)$,  and $i_{P,W}^{*}\alpha\in \Omega^{(=)}\left(P,\R\right)$ represents
 $\ch_{\mathrm{BC}}\left(i_{X\times \infty ,P,*} 
 \mathscr F\right)\in 
 H_{\mathrm{BC}}^{(=)}\left(P,\R\right)$.

Let $z\in \C$ be the canonical meromorphic coordinate on 
$\mathbf{P}^{1}$ that vanishes at $0$. We have the Poincaré-Lelong 
equation (\ref{eq:bc5}). 

Let $\delta_{Y_{0}}, \delta_{Y_{\infty } }$ be the currents of 
integration on $Y_{0}, Y_{ \infty }$.  Since $q_{W,\mathbf{P}^{1}}$ 
has ordinary singularities near $\mathbf{P}\left(N_{X/Y}\right)$,  
 there is a well-defined  integrable current 
 $q^{*}_{W,\mathbf{P}^{1}}\log\left(\left\vert  z\right\vert^{2}\right)$ 
 on $W$, which is such that
 \begin{equation}\label{eq:bc5a1}
\frac{\overline{\pa}^{W}\pa^{W}}{2i\pi}q^{*}_{W,\mathbf{P}^{1}}\log\left(\left\vert  z\right\vert^{2}\right)=\delta_{Y_{0}}-\delta_{Y_{\infty }}.
\end{equation}
 For $z\in\mathbf{P}^{1}$, let $i_{z}$ be the embedding $Y_{z}\to W$. By (\ref{eq:bc5a1}), we deduce that
\begin{equation}\label{eq:bc6}
\frac{\overline{\pa}^{W}\pa^{W}}{2i\pi}\left[\alpha q^{*}_{W,\mathbf{P}^{1}}\log\left(\left\vert  
z\right\vert^{2}\right)\right]
=i_{0}^{*}\alpha\delta_{Y_{0}} -i_{\infty }^{*}\alpha\delta_{Y_{\infty } }.
\end{equation}
Let $q_{Y_{ \infty },Y}$ be the restriction of $q_{W,Y}$ to $Y_{ 
\infty }$. By (\ref{eq:bc6}), we deduce the identity of currents on $Y$, 
\begin{equation}\label{eq:bc7}
\frac{\overline{\pa}^{Y}\pa^{Y}}{2i\pi}q_{W,Y,*}\left[\alpha q^{*}_{W,\mathbf{P}^{1}}\log\left(\left\vert  
z\right\vert^{2}\right)\right]
=i_{0}^{*}\alpha-q_{Y_{\infty },Y,*}i^{*}_{\infty }\alpha.
\end{equation}

  By the considerations in the paragraph following  (\ref{eq:bc4}) and 
by (\ref{eq:bc7}), we get
\begin{equation}\label{eq:bc7a1}
\ch_{\mathrm{BC}}\left(i_{X,Y,*}\mathscr F\right)=
\left\{q_{Y_{\infty },Y,*}i_{\infty }^{*}\alpha\right\} \,\mathrm{in}\, H^{(=)}_{\mathrm{BC}}\left(Y,\R\right).
\end{equation}

Let $q_{\widetilde Y,Y}, q_{P,Y}$ be the restriction of $q_{Y_{\infty},Y}$ to 
 $\widetilde Y, P$. Note that
\begin{equation}\label{eq:bc8}
q_{Y_{\infty },Y,*}i_{\infty }^{*}\alpha
=q_{\widetilde Y,Y,*}i_{\widetilde Y,W }^{*}\alpha \\
+q_{P,Y,*}i_{P,W}^{*}\alpha.
\end{equation}

Since $\widetilde Y \cap \left( X\times \infty \right) 
=\emptyset$,  we find that
\begin{equation}\label{eq:bc9a-1}
Li_{\widetilde 
Y,W}^{*}i_{X\times\mathbf{P}^{1},W,*}\left( 
L\pi^{*}_{X\times\mathbf{P}^{1},X}\right) 
\mathscr F \simeq 0\ \mathrm{in}\ \Db\left(\widetilde Y\right).
\end{equation}
Using Theorem \ref{thm:pu}, the fact that $\left\{\alpha\right\}$ is 
given by (\ref{eq:alph1}), and (\ref{eq:bc9a-1}), 
we get
\begin{equation}\label{eq:bc9}
\left\{i^{*}_{\widetilde Y,W}\alpha\right\}=0\,\mathrm{in}\, H^{(=)}_{\mathrm{BC}}\left(\widetilde 
Y,\R\right),
\end{equation}
so that
\begin{equation}\label{eq:bc10}
\left\{q_{\widetilde Y,Y,*}i^{*}_{\widetilde Y,W}\alpha\right\}=0\,\mathrm{in}\, H^{(=)}_{\mathrm{BC}}\left(Y,\R\right).
\end{equation}

Using again the considerations in the paragraph following 
(\ref{eq:bc4}), (\ref{eq:bc8}), and (\ref{eq:bc10}), we obtain
\begin{equation}\label{eq:bc10z1}
\cBC\left(i_{X,Y,*} \mathscr F\right)=q_{P,Y*}\cBC\left(i_{X\times 
\infty ,P,*} \mathscr F\right).
\end{equation}
We will now evaluate the right-hand side of (\ref{eq:bc10z1}).

The Koszul complex $\left(\mathcal{O}_{P}\left(\Lambda \left(\left(A \oplus 
\C\right)/U\right)^{*}\right), i_{\sigma}\right)$ is a projective resolution of 
the sheaf $i_{X\times \infty ,P,*}\mathcal{O}_{X\times \infty }$. By 
the projection formula
\cite[\href{https://stacks.math.columbia.edu/tag/0B54}{Tag 
0B54}]{stacks-project}, we get 
\begin{equation}\label{eq:bc5b}
i_{X\times \infty,P,*}\mathscr F \simeq Lq_{P,X\times \infty}^{*} 
\mathscr F \ho_{\mathcal{O}_{P}}^{\mathrm{L}}\left( \mathcal
O_{P}\left(\Lambda \left( \left( A \oplus \C\right)/U \right)^{*}\right)  
,i_{\sigma}\right) \ \mathrm{in}\ \Db\left(P\right).
\end{equation}
By Theorems \ref{thm:pu},  \ref{thm:tepr} and by (\ref{eq:bc5b}), we deduce that
\begin{equation}\label{eq:bc5c}
\ch_{\mathrm{BC}}\left(i_{X\times \infty,P,*} \mathscr 
F\right)=q^{*}_{P,X\times \infty}\ch_{\mathrm{BC}}\left(\mathscr 
F\right)\ch_{\mathrm{BC}}\left(\Lambda \left(\left( A\oplus 
\C\right)/U \right) ^{*}\right).
\end{equation}
By (\ref{eq:bc5c}), we get
\begin{equation}\label{eq:bc5d}
q_{P,Y,*}\ch_{\mathrm{BC}}\left(i_{X\times \infty,P,*} \mathscr 
F\right)=i_{X,Y,*}\left[\ch_{\mathrm{BC}}\left(\mathscr 
F\right)q_{P,X\times \infty, *}\ch_{\mathrm{BC}}\left(\Lambda \left(\left( A 
\oplus \C/U \right) ^{*}\right)\right)\right].
\end{equation}
Observe that
\begin{equation}\label{eq:bc5da1}
N_{\left( X\times \infty \right) /P }=A.
\end{equation}
By \cite[eq. (2.8) in Theorem 2.5]{BismutGilletSoule90b} and by (\ref{eq:bc5da1}), we get
\begin{equation}\label{eq:bc5e}
q_{P,X\times \infty,*}\ch_{\mathrm{BC}}\left(\Lambda \left( \left( A \oplus 
\C\right)/U \right) ^{*}\right)=\Td_{\mathrm{BC}}^{-1}\left(A\right).
\end{equation}
Recall that the choice of the coordinate $z$ gives a canonical 
trivialization of $N_{\infty /\mathbf{P}^{1}}$, so that by 
(\ref{eq:bc4a}), we have the canonical identification
\begin{equation}\label{eq:bc5eb1}
A \simeq N_{X/Y}.
\end{equation}

By (\ref{eq:bc10z1}), (\ref{eq:bc5d}), and (\ref{eq:bc5e}), we 
get (\ref{eq:bc1}).
The 
proof of our theorem is completed.  
\end{proof}
\begin{remark}\label{rem:compa}
	We claim that Proposition \ref{prop:fund} and Theorem 
	\ref{thm:tim} are compatible. Indeed, under the assumptions of 
	Proposition \ref{prop:fund}, by Theorem \ref{thm:tim}, we get
	\begin{equation}\label{eq:nuf1}
\cBC\left(i_{X,Z,*} \mathscr F\right)=i_{X,Z,*}\frac{\cBC\left(\mathscr 
F\right)}{\Td_{\mathrm{BC}}\left(N_{X/Z}\right)}\ \mathrm{in}\ 
H^{(=)}_{\mathrm{BC}}\left(Z,\R\right).
\end{equation}
If $\beta\in \Omega^{(=)}\left(X,\R\right)$ is a smooth form 
representing $\frac{\cBC\left(\mathscr 
F\right)}{\Td_{\mathrm{BC}}\left(N_{X/Z}\right)}$, the right-hand 
side of (\ref{eq:nuf1}) is represented by the current 
$\beta\delta_{X}$ on $Z$. 

If $X$ and $Y$ are transverse, as explained in Subsection 
\ref{subsec:unpro},
$i_{Y,Z}^{*}\beta \delta_{X}$ is a 
well-defined current on $Y$, and it represents $i_{Y,Z}^{*}i_{X,Z,*}\frac{\cBC\left(\mathscr 
F\right)}{\Td_{\mathrm{BC}}\left(N_{X/Z}\right)}$ in 
$H_{\mathrm{BC}}^{(=)}\left(Y,\R\right)$. Moreover, one has the 
equality of currents on $Y$,
\begin{equation}\label{eq:nuf2}
 i_{Y,Z}^{*}\left( \beta  \delta_{X}\right) =\left( i^{*}_{U,X}\beta \right) \delta_{U}.
\end{equation}
By Theorem \ref{thm:pu} and by (\ref{eq:tra2}),  the 
Bott-Chern class of the current in the right-hand side of 
(\ref{eq:nuf2}) is just $i_{U,Y,*}\frac{\cBC\left(Li^{*}_{U,X} \mathscr 
F\right)}{\Td_{\mathrm{BC}}\left(N_{U/Y}\right)}$.
Using again Theorem \ref{thm:tim} for the embedding $i_{U,Y}$,  we 
find that when taking the Chern character of both sides of 
(\ref{eq:tra3}), we get a known equality.
\end{remark}
\subsection{The uniqueness of the Chern character}%
\label{subsec:uni}
First, we recall a version of a result of Grivaux \cite[Theorem 2]{Grivaux10}.
\begin{theorem}\label{thmuni}
	There is at most a unique morphism of groups $\cBC: K(X)\to H^{(=)}_{\rm 
	BC}(X,\mathbf R)$ such that 
		\begin{enumerate}
			\item If $E$ is a holomorphic vector bundle,  
		$\cBC\left(E\right)$ 
		is defined as in Subsection \ref{subsec:bovb}.
		\item  $\cBC$ is functorial under pull-backs.
		\item ${\rm ch_{BC}}$ verifies   Riemann-Roch-Grothendieck 
		with respect to   embeddings.
	\end{enumerate} 
	
	The Chern character $\cBC$ constructed in Definition 
	\ref{def:cDb} verifies the above uniqueness conditions.
\end{theorem}
\begin{proof}
	According to a modified version of Grivaux \cite{Grivaux10}, for 
	uniqueness, we only need to check that 
	if $X$ is a compact complex manifold, if $Y$ is a compact complex 
	submanifold of $X$, and if $\sigma:\widetilde X\to X$ is the 
	blow-up of $X$ along $Y$, then 
	$\sigma^{*}:H^{(=)}_{\mathrm{BC}}\left(X,\R\right)\to 
	H^{(=)}_{\mathrm{BC}}\left(\widetilde X,\R\right)$ is injective. Since the exceptional 
	divisor in $\widetilde X$ is negligible, if $\alpha\in 
	\Omega\left(X,\C\right)$,
	\begin{equation}\label{eq:negli1}
\int_{\widetilde X}^{}\sigma^{*}\alpha=\int_{X}^{}\alpha.
\end{equation}
By (\ref{eq:negli1}), we deduce that if $\alpha\in 
\Omega\left(X,\C\right)$, we have the identity of currents on $X$, 
\begin{equation}\label{eq:negli2}
\sigma_{*}\sigma^{*}\alpha=\alpha.
\end{equation}
From (\ref{eq:negli2}), we conclude that on 
$H^{(=)}_{\mathrm{BC}}\left(X,\R\right)$, $\sigma_{*}\sigma^{*}$ is the identity, 
and so $\sigma^{*}$ is injective. This proves the uniqueness of 
$\cBC$. 

By construction, if $E$ is a holomorphic vector bundle on $X$, our 
$\cBC\left(E\right)$ is exactly the classical Chern character 
obtained in Subsection \ref{subsec:bovb}. By Theorems \ref{thm:pu} 
and \ref{thm:tim}, our $\cBC$ verifies the above two remaining 
conditions.

The proof of our theorem is completed.
\end{proof} 
\begin{remark}\label{rem:gri}
	The original unicity theorem of Grivaux would require that we 
	prove if $E$ is a holomorphic vector bundle on $X$, and if 
	$\pi:P\to X$ is the total space of the 
	projectivization $\mathbf{P}\left(E\right)$, then 
	$\pi^{*}:H^{(=)}_{\mathrm{BC}}\left(X,\R\right)\to 
	H^{(=)}_{\mathrm{BC}}\left(P,\R\right)$ is injective. We do not 
	need to check this   (which would be trivial here), because 
	$\cBC$ is known on all holomorphic vector bundles.
	\end{remark}

If $H^{(=)}_{\rm BC}(X,\mathbf Q)$ is the rational Bott-Chern 
cohomology \cite[Section 4.e]{Schweitzer07}, Wu \cite{Wu20} has 
constructed ${\rm ch_{BC,\mathbf Q}}:K(X)\to H^{(=)}_{\rm 
BC}(X,\mathbf Q)$ a rational Chern character. Let $\iota 
:H^{(=)}_{\rm BC}(X,\mathbf Q)\to H^{(=)}_{\rm BC}(X,\mathbf R)$ be 
the canonical map \cite[Section 7.a]{Schweitzer07}.

\begin{cor}
	The following identity holds:
\begin{align}\label{eq:rat}
\iota	{\rm ch_{BC,\mathbf Q}}={\rm ch_{BC}}. 
\end{align} 	
\end{cor}
\begin{proof}
By \cite[Theorem 2]{Wu20}, $\iota {\rm ch_{BC,\mathbf Q}}$ satisfies 
the uniqueness conditions of  Theorem \ref{thmuni}, and so 
(\ref{eq:rat}) holds.
\end{proof} 
\section{Submersions and elliptic superconnections}%
\label{sec:suel}
The purpose of this Section is to state our main result in the case 
of a submersion $p:M=X\times S\to S$, and to define infinite 
dimensional
superconnection forms on $S$, which will be used to establish this main 
result.

This Section is organized as follows. In Subsection \ref{subsec:geo}, 
we state our main result, that involves $\mathscr F\in 
\Db\left(M\right)$.

In Subsection \ref{subsec:rebl}, using the results of Section 
\ref{sec:blo}, we show that in the proof,  $\mathscr F$ can be 
replaced by $\mathscr E\in \mathrm{B}\left(M\right)$. We will then 
view $p_{*} \mathscr E$ as equipped with an antiholomorphic 
superconnection $A^{p_{*}\mathscr E \prime \prime }$.

In Subsection \ref{subsec:adj}, given a splitting $\mathscr E 
\simeq \mathscr E_{0}$, and Hermitian metrics $g^{TX},g^{D}$ on 
$TX,D$, we obtain the adjoint $A^{p_{*} \mathscr E_{0} \prime }$ of 
$A^{p_{*}\mathscr E_{0} \prime \prime}$.

In Subsection \ref{subsec:cox}, we construct natural connections on 
$T_{\R}X$.

In Subsection \ref{subsec:lichbi},  a Lichnerowicz formula is 
established for 
the curvature of $A^{p_{*} \mathscr E_{0}}$, when $S$ is 
reduced to a point.

In Subsection \ref{subsec:curv}, we give a Lichnerowicz formula for 
$A^{p_{*} \mathscr E_{0}, 2}$ by taking a proper adiabatic limit of 
the corresponding formula in the case of a single fiber. This 
curvature is a second order elliptic differential operator along the 
fibers $X$. 

Finally, in Subsection \ref{subsec:elscn}, by imitating the constructions of 
Section \ref{sec:chfoge} in an infinite-dimensional context, we 
obtain infinite-dimensional Chern character forms 
$\ch\left(A^{p_{*} \mathscr E_{0}\prime \prime},\omega^{X},g^{D}\right)$, and we 
show that they have the same formal properties as the forms 
$\ch\left(A^{E_{0}\prime \prime },h\right)$ that were considered in 
Section \ref{sec:chfoge}. In particular, their class in Bott-Chern 
cohomology does not depend on the metric data.
\subsection{A theorem of Riemann-Roch-Grothendieck  for submersions}%
\label{subsec:geo}
Let $X,S$ be compact complex manifolds of dimension $n,n'$. Put
\begin{equation}\label{eq:new1}
M=X\times S. 
\end{equation}
Then $M$ has dimension  $m=n+n'$.
Let $p:M\to S,q:M\to X$ be the   projections.
 Then
\begin{equation}\label{eq:new1a1}
TM=p^{*}TS \oplus q^{*}TX.
\end{equation}

Let $\mathscr F$ be an object in $\Db\left(M\right)$. As we saw in 
Subsection \ref{subsec:dira}, by a theorem by Grauert 
\cite[Theorem 10.4.6]{Grauert60a},  $Rp_{*}\mathscr F\in 
\Db\left(S\right)$. 

The purpose of the next sections is to prove a special case of our general 
result.
\begin{theorem}\label{thm:sub}
	The following identity holds:
	\begin{equation}\label{eq:dir1}
\cBC\left(Rp_{*} \mathscr 
F\right)=p_{*}\left[q^{*}\Td_{\mathrm{BC}}\left(TX\right)\cBC\left(\mathscr F\right)\right]\,\mathrm{in}\, H^{(=)}_{\mathrm{BC}}\left(S,\R\right).
\end{equation}
\end{theorem}
\begin{remark}\label{rem:extsub}
We will exploit as much 
as possible the fact that $M$ is product. However, the methods of 
\cite{Bismut10b} can also be used to give a direct proof of Theorem 
\ref{thm:sub} when $p$ is an arbitrary holomorphic submersion.
\end{remark}
\subsection{Replacing $ \mathscr F$ by $\mathscr E$}%
\label{subsec:rebl}
Let $\mathscr F\in \Db\left(M\right)$. By Theorem \ref{thm:exib},  
there exists $\mathscr 
E=\left(E,A^{E \prime \prime }\right)\in \mathrm{B}\left(M\right)$  and a 
quasi-isomorphism of $\mathcal{O}_{X}$-complexes $\phi:\mathscr E\to 
\overline{\mathscr F}^{\infty}$,  so that 
$\mathscr H \mathscr E= \mathscr H \overline{\mathscr F}^{\infty }=\mathscr H \mathscr F$.  To prove our 
Theorem, we may as well assume that $\mathscr F=\mathscr E$.

We will use the notation of Sections \ref{sec:ahscn} and 
\ref{sec:ascngm}. In particular $D$ denotes the diagonal vector bundle 
on $M$ which is associated with $E$. 

Note that $p_{*} \mathscr E$ is a 
$\Omega^{0,\Ou}\left(S,\C\right)$-module, and  $A^{E 
\prime \prime }$ can be viewed as an antiholomorphic superconnection 
\index{Ape@$A^{p_{*} \mathscr E \prime \prime}$}%
$A^{p_{*} \mathscr E \prime \prime}$ on $p_{*}\mathscr E$.
 The difference with respect to Subsections 
\ref{subsec:def}--\ref{subsec:applco} is that  $p_{*} 
\mathscr E$  is infinite-dimensional.

As we saw in Subsection \ref{subsec:catdb}, a 
theorem of Grauert \cite[Theorem 10.4.6]{GrauertRemmert84} asserts 
that $Rp_{*}\mathscr E$ is an object in 
$\Db\left(S\right)$. Also by (\ref{eq:fon1}), 
\begin{equation}\label{eq:turp1}
Rp_{*}\mathscr E= p_{*}\mathscr E.
\end{equation}
Therefore the $\mathcal{O}_{S}$-complex $\left(p_{*}\mathscr E,A^{p_{*} \mathscr E\prime \prime }
\right)$ defines an object in $\Db\left(S\right)$. This result will 
be given a direct proof in Theorem \ref{thm:exi}.

Let $i$ be the embedding of the fibers $X$ into $M$. By
Propositions \ref{prop:rest} and \ref{prop:puba}, the pull-back 
$i^{*}_{b}\mathscr E=\left(i^{*}_{b}E,A^{i^{*}_{b}E \prime \prime 
}\right)$ is a family of 
antiholomorphic superconnections   along the fibers $X$, whose 
corresponding associated element in $\Db\left(X\right)$ is just 
$Li^{*}\mathscr E$.

Let 
\index{D@$\mathcal{D}$}%
$\mathcal{D}$ be the diagonal bundle  associated with $p_{*} \mathscr 
E$. Then
\begin{equation}\label{eq:solv1}
\mathcal{D}=p_{*} i_{b}^{*}\mathscr 
E,
\end{equation}
and the corresponding 
\index{AD@$A^{\mathcal{D}\prime \prime }$}%
$A^{\mathcal{D}\prime \prime }$ (which is an 
analogue of $v_{0}$ for $p_{*} \mathscr E$) is given by 
\begin{equation}\label{eq:solv2}
A^{\mathcal{D} \prime \prime }=A^{i^{*}_{b}E \prime \prime }.
\end{equation}
The projection $P$ defined in 
Subsection \ref{subsec:def} is given here by $i^{*}_{b}: p_{*}\mathscr 
E\to \mathcal{D}=
p_{*}i^{*}_{b}\mathscr E$.

By (\ref{eq:new1a1}), we get
\begin{equation}\label{eq:didim2}
\Lambda\left(\overline{T^{*}M}\right)=p^{*}\Lambda\left(\overline{T^{*}S}\right)\ho q^{*}\Lambda\left(\overline{\TsX}\right).
\end{equation}
By (\ref{eq:bot1}), (\ref{eq:didim2}), we obtain
\begin{equation}\label{eq:didim2a1}
E_{0} = p^{*}\Lambda\left(\overline{T^{*}S}\right)\ho 
q^{*}\Lambda\left(\overline{\TsX}\right)\ho D.
\end{equation}
Note that
\begin{equation}\label{eq:didim2a2}
i^{*}_{b}E_{0}=q^{*}\Lambda\left(\overline{\TsX}\right)\ho D.
\end{equation}

Over $X\times S$, we fix splittings as in (\ref{eq:iv4a-1}), 
(\ref{eq:iv4}), so that
\begin{equation}\label{eq:dididm2a2}
E \simeq E_{0}.
\end{equation}
These splittings induce corresponding splittings of $i_{b}^{*}E$, and we 
have corresponding isomorphism,
\begin{equation}\label{eq:didim2a3}
i^{*}_{b}E \simeq i^{*}_{b}E_{0}.
\end{equation}
By (\ref{eq:didim2a3}), we get the non-canonical isomorphism
\begin{equation}\label{eq:didim2a4}
\mathcal{D}=\Omega^{0,\Ou}\left(X,D\vert_{X}\right).
\end{equation}

Put
\begin{equation}\label{eq:didim2a4x1}
p_{*} \mathscr E_{0}=\Lambda \left(\overline{T^{*}S}\right)\ho 
p_{*}i_{b}^{*} \mathscr E.
\end{equation}
By (\ref{eq:solv1}), (\ref{eq:didim2a1}), and (\ref{eq:didim2a2}), we 
have the  identification\footnote{It would be more appropriate here 
to introduce the sub-filtration of $E$ associated with 
$p^{*}\Lambda\left(\overline{T^{*}S}\right)$.},
\begin{equation}\label{eq:didim3}
p_{*} \mathscr E \simeq p_{*} \mathscr E_{0}.
\end{equation}
To keep in line with the previous notation, we will denote 
\index{ApE@$A^{p_{*} \mathscr E_{0}}$}%
$A^{p_{*} \mathscr E_{0} \prime \prime}$ the operator corresponding to $A^{p_{*} 
\mathscr E \prime \prime }$ 
via the isomorphism (\ref{eq:didim3}).

Using (\ref{eq:baba2}), we get
\begin{equation}\label{eq:new2a1}
A^{p_{*} \mathscr E_{0} \prime \prime }=\n^{D \prime \prime 
}+B.
\end{equation}
\subsection{The adjoint of $A^{p_{*} \mathscr E_{0} \prime \prime}$}%
\label{subsec:adj}
Since 
$M=X\times S$, then $\n^{D \prime \prime }$ splits as
\begin{equation}\label{eq:coc7a1}
\n^{D \prime \prime }=\n^{D \prime \prime,X}+\n^{D \prime \prime, S},
\end{equation}
where the terms in the right-hand side  of (\ref{eq:coc7a1}) differentiate along 
$X,S$.

Let 
\index{gTX@$g^{TX}$}%
$g^{TX}$ be a Hermitian metric on $TX$, and let 
\index{oX@$\omega^{X}$}%
$\omega^{X}$ 
denote the corresponding Kähler $(1,1)$-form along the 
fibers $X$.\footnote{To simplify the presentation, we 
exploit as much as possible the product structure of $M=X\times S$, 
to avoid the difficulties with the more general situation considered 
in  \cite{Bismut10b}.} Let $g^{T_{\R}X},g^{T^{*}_{\R}X}$ be the 
induced metrics on 
$T_{\R}X,T^{*}_{\R}X$. We denote by $\left\langle  \,\right\rangle$ the 
corresponding scalar product on $T_{\R}X$. If $J^{T_{\R}X}$ denotes the complex structure of 
$T_{\R}X$, if $U,V\in T_{\R}X$, then\footnote{Here, we choose the  
sign
conventions of Kobayashi-Nomizu \cite[Chapter 9.5]{KobayashiNomizu69} 
and of \cite{BismutGilletSoule88b,BismutGilletSoule88c}. The opposite 
sign  is also commonly used.}
\begin{equation}\label{eq:cru1a1}
\omega^{X}\left(U,V\right)=\left\langle  U,J^{T_{\R}X}V\right\rangle.
\end{equation}
Let $g^{\Lambda\left(\overline{\TsX}\right)}$ denote the metric 
induced by $g^{TX}$ on $\Lambda\left(\overline{\TsX}\right)$. Let 
$dx$ be the volume form on $X$ which is induced by $g^{TX}$.

 Let $g^{D}$ be a 
Hermitian metric on the $\Z$-graded vector bundle $D$. Then $g^{D}$ 
defines a pure metric $h$ in $\mathscr M^{D}$.

We equip  $\Lambda \left(\overline{\TsX}\right)\ho D$ with the 
metric $g^{\Lambda \left(\overline{\TsX}\right) } \ho g^{D}$. 
\begin{definition}\label{def:her}
	If $s,s'\in  \mathcal{D}=\Omega^{0,\Ou}\left(X,D\vert_{X}\right)$, put
	\begin{equation}\label{eq:coc7}
\alpha\left(s,s'\right)=\left(2\pi\right)^{-n}\int_{X}^{}\left\langle
s,s'\right\rangle_{g^{\Lambda \left(\overline{\TsX}\right)} \otimes 
g^{D}}dx.
\end{equation}
\end{definition}

Then $\alpha$  is a Hermitian product  on 
$\mathcal{D}$, that defines a pure metric in $\mathscr M^{\mathcal{D}}$.

Here we use the notation of Subsection \ref{subsec:adjscn}, where the adjoint 
of an antiholomorphic superconnection with respect to a Hermitian 
metric was 
defined.
\begin{definition}\label{def:adj}
	Let 
	\index{ApE@$A^{p_{*}\mathscr E_{0}\prime}$}%
	$A^{p_{*}\mathscr E_{0}\prime}$ denote the 
	 adjoint of $A^{p_{*}\mathscr E_{0} \prime \prime }$ with respect to 
	$\alpha$.
\end{definition}

Although $A^{E_{0} \prime \prime 
}$ and $A^{p_{*} \mathscr E_{0}\prime \prime}  $ are essentially the same 
object, their adjoints $A^{E_{0} \prime}$ and $A^{p_{*} \mathscr E_{0} \prime 
}$ are distinct. In particular, the metric $g^{TX}$ plays no role 
in the definition of $A^{E_{0} \prime }$.  Still, we will show how to 
derive $A^{p_{*} \mathscr E_{0}\prime } $ from $A^{E_{0} \prime}$.

If $H$ is a smooth section of 
$\Lambda\left(\overline{T^{*}M}\right)\ho\End\left(D\right)$, 
$H^{*}$ still denotes  the adjoint of $H$
associated with the metric $h=g^{D}$, as  defined in Subsection 
\ref{subsec:dua}.

Then $H$ can be viewed as  a smooth section of 
$$\Lambda\left(\overline{T^{*}S}\right)\ho\End\left(\Lambda\left(\overline{\TsX}\right)\right)\ho\End\left(D\right).$$
\begin{definition}\label{def:adjB}
	Let $H^{*}_{\alpha}$ denote the adjoint of $H$ with respect to the 
	metric $\alpha$ on $\mathcal{D}$. 
\end{definition}

We use the conventions of Subsection \ref{subsec:cli}, with $V=TX, g^{V^{*}}=g^{\TsX}$. 
Then
\index{c@${}^{c}$}%
${}^{c}$ is a map  $\Lambda\left(T^{*}_{\C}M\right)\to
\Lambda \left(T^{*}_{\C}S\right)\ho_{\R} 
c\left(T^{*}_{\R}X\right)$. 
We extend ${}^{\mathrm{c}}$ to a map  
$$\Lambda\left(T^{*}_{\C}M\right)\ho\End\left(D\right)\to \Lambda \left(T^{*}_{\C}S\right)\ho _{\R}
c\left(T^{*}_{\R}X\right)\ho_{\R}\End\left(D\right).$$
Also we use the 
identification in (\ref{eq:bing1}), 
\begin{equation}\label{eq:bing1bi}
c\left(T^{*}_{\R}X\right) 
\otimes _{\R}\C=\End\left(\Lambda \left( \overline{T^{*}X}\right) \right).
\end{equation}
\begin{proposition}\label{prop:psimple}
	If $H$ is a smooth section of 
	$\Lambda\left(\overline{T^{*}M}\right)\ho\End\left(D\right)$, then 
	\begin{equation}\label{eq:gor1}
H^{*}_{\alpha}={}^{c}H^{*}.
\end{equation}
\end{proposition}
\begin{proof}
	We may as well assume that $D$ is trivial. We use the notation in 
	(\ref{eq:xiao2a1}). If $f\in \TsX$, $\overline{f}$ and 
	$i_{f_{*}}$ are adjoint to each other with respect to the metric 
	of $\Lambda\left(\overline{T^{*}X}\right)$. From 
	(\ref{eq:xiao2a1}), we get (\ref{eq:gor1}).
\end{proof}

Recall that 
\index{nD@$\n^{D \prime}$}%
$\n^{D \prime}$ was defined after (\ref{eq:iv16}). As in (\ref{eq:coc7a1}),  we have the splitting
\begin{equation}\label{eq:spli1}
\n^{D \prime }=\n^{D \prime ,X}+\n^{D \prime ,S}.
\end{equation}
We can extend 
$\n^{D \prime, S}$  to smooth sections of  
$\Lambda\left(T^{*}_{\C}S\right)\ho 
\Lambda\left(\overline{\TsX}\right)\ho D$. 

Let 
\index{nD@$\n^{D \prime \prime, X*}_{\alpha}$}%
$\n^{D \prime \prime, X*}_{\alpha}$ denote the formal adjoint of $\n^{D 
\prime \prime, X}$ with respect to $\alpha$. Then $\n^{D \prime \prime,X *}_{\alpha}$ is a classical 
Hermitian adjoint. \begin{proposition}\label{prop:padj}
	The following identity holds:
	\begin{equation}\label{eq:coc7a3}
A^{p_{*} \mathscr E_{0} \prime}=\n^{D \prime,S }+\n^{D  
\prime \prime ,X*}_{\alpha}+{}^{c}B^{*}.
\end{equation}
\end{proposition}
\begin{proof}
	This is a consequence of (\ref{eq:new2a1}), (\ref{eq:coc7a1}), 
	and of Proposition \ref{prop:psimple}.
\end{proof}
\begin{definition}\label{def:comsc}
	Put
	\index{ApE@$A^{p_{*} \mathscr E_{0}}$}%
	\begin{equation}\label{eq:coc4}
A^{p_{*} \mathscr E_{0}}=A^{p_{*} \mathscr E_{0} \prime \prime 
}+A^{p_{*} \mathscr E_{0}\prime}.
\end{equation}
Then $A^{p_{*} \mathscr E_{0}}$ is a superconnection on $p_{*} 
\mathscr E_{0}$. 
\end{definition}

Put
\index{nDS@$\n^{D,S}$}%
\begin{equation}\label{eq:coc4az1}
\n^{D,S}=\n^{D \prime \prime,S}+\n^{D \prime,S}.
\end{equation}
By (\ref{eq:dep1}), (\ref{eq:new2a1}), and (\ref{eq:coc7a3}), we get
\begin{equation}\label{eq:dep2}
A^{p_{*} \mathscr E_{0}}=\n^{D,S}+\n^{D\prime\prime ,X}+\n^{D\prime 
\prime, X*}+{}^{c}C.
\end{equation}

Let 
\index{ADs@$A^{\mathcal{D}\prime \prime *}$}%
$A^{\mathcal{D}\prime \prime *}$  denote the standard fiberwise 
$L_{2}$ adjoint of $A^{\mathcal{D} \prime \prime }$. We use the notation
\index{AD@$A^{\mathcal{D}}$}%
\begin{equation}\label{eq:dep2a1}
A^{\mathcal{D}}=A^{\mathcal{D} \prime \prime }+A^{\mathcal{D} \prime 
\prime *}.
\end{equation}

\subsection{Some connections on $TX$}%
\label{subsec:cox}
Put
\index{b@$\beta$}%
\begin{equation}\label{eq:dep10a1}
\beta=\frac{1}{2}\left(\overline{\pa}^{X}-\pa^{X}\right)i\omega^{X}.
\end{equation}
Then $\beta$ is a smooth section of 
$\Lambda^{3}\left(T^{*}_{\R}X\right)$.

Let 
\index{nTRX@$\n^{T_{\R}X,\mathrm{LC}}$}%
$\n^{T_{\R}X,\mathrm{LC}}$ be the Levi-Civita connection on $T_{\R}X$, and 
let 
\index{nTX@$\n^{TX}$}%
$\n^{TX}$ be the Chern connection on $TX$. Let $\n^{T_{\R}X}$ be 
the connection on $T_{\R}X$ which is induced by $\n^{TX}$. The connection 
$\n^{T_{\R}X,\mathrm{LC}}$ preserves $TX$ if and only if $g^{TX}$ is Kähler, 
and in this case, it coincides with $\n^{T_{\R}X}$.

Let 
\index{nLTX@$\n^{\Lambda\left(\overline{\TsX}\right)}$}%
$\n^{\Lambda\left(\overline{\TsX}\right)}$ be the connection on 
$\Lambda\left(\overline{\TsX}\right)$ which is induced by $\n^{TX}$. 
This connection preserves the $\Z$-grading. 

Let 
\index{th@$\vartheta$}%
$\vartheta$ be the canonical $1$-form on $X$ with values in 
$T_{\R}X$ that corresponds to the identity. Let 
\index{t@$\tau$}%
$\tau$ denote the 
torsion of $\n^{T_{\R}X}$. Then $\tau$ is the sum of a 
$\left(2,0\right)$-form $\tau^{1,0}$ with values in $TX$, and of a 
$\left(0,2\right)$-form $\tau^{0,1}$ with values in $\overline{TX}$. 

Let $\left\langle  \tau\we,\vartheta\right\rangle$ be the 
antisymmetrization of the tensor $\left\langle  
\tau,\vartheta\right\rangle$. By 
\cite[Proposition 2.1]{Bismut89a}, or by (\ref{eq:cla1}), we get
\begin{equation}\label{eq:dep9}
\left\langle  \tau\we_{,}\vartheta\right\rangle=-2\beta.
\end{equation}
Also equation (\ref{eq:dep9}) determines $\tau$. More precisely, 
\begin{align}\label{eq:dep9a1}
&\overline{\pa}^{X}i\omega^{X}=-\left\langle  
\tau^{0,1}\we,\vartheta^{1,0}\right\rangle,
&\pa^{X}i\omega^{X}=\left\langle  
\tau^{1,0}\we,\vartheta^{0,1}\right\rangle.
\end{align}

By (\ref{eq:cla2}),  the connection 	
 $\n^{T_{\R}X,-1/2}$ on $T_{\R}X$ associated with $\n^{TX}$ is such that	if $U,V\in T_{\R}X$, then
	\begin{equation}\label{eq:dee3}
\n^{T_{\R}X,-1/2}_{U}V=\n^{T_{\R}X}_{U}V-\frac{1}{2}\tau\left(U,V\right).
\end{equation}
As we saw at the end of Subsection  \ref{subsec:remdi}, $\n^{T_{\R}X,-1/2}$ induces a connection $\n^{TX,-1/2}$ on $TX$, 
and the analogue of (\ref{eq:clap5}) holds.

Here, we follow \cite[Section 2 b)]{Bismut89a}.
\begin{definition}\label{def:nb}
	Let 
	\index{ntRX@$\overline{\n}^{T_{\R}X}$}%
	\index{nTRX@$\underline{\n}^{T_{\R}X}$}%
	$\overline{\n}^{T_{\R}X},\underline{\n}^{T_{\R}X}$ be the 
	Euclidean connections on 
	$T_{\R}X$ whose torsions $\overline{\tau},\underline{\tau}$ are 
	such that the tensors $\left\langle  
	\overline{\tau},\vartheta\right\rangle,\left\langle  
	\underline{\tau},\vartheta\right\rangle$
	are totally 
	antisymmetric, and moreover, 
	\begin{align}\label{eq:dep10}
&\left\langle  
\overline{\tau}\we_{,}\vartheta\right\rangle=2\beta,
&\left\langle  
\underline{\tau}\we_{,}\vartheta\right\rangle=-2\beta,
\end{align} 
and let 
\index{RTR@$\overline{R}^{T_{\R}X}$}%
\index{RTR@$\underline{R}^{T_{\R}X}$}%
$\overline{R}^{T_{\R}X},\underline{R}^{T_{\R}X}$ 
be their curvatures.
\end{definition}

By \cite[Proposition 2.5]{Bismut89a}, the connection 
$\overline{\n}^{T_{\R}X}$ preserves the complex structure, i.e., it 
induces a corresponding Hermitian connection 
\index{nTX@$\overline{\n}^{TX}$}%
$\overline{\n}^{TX}$ on 
$TX$. 

 Let $A$ 
be the $1$-form on $X$ with values in antisymmetric sections of 
$\End\left(T_{\R}X\right)$  such that if $U,V,W\in T_{\R}X$, then
\begin{equation}\label{eq:dep10a}
\left\langle  A\left(U\right)V,W\right\rangle=-\left\langle  
\tau\left(U,V\right),W\right\rangle+\left\langle  
\tau\left(U,W\right),V\right\rangle.
\end{equation}
From the properties of $\tau$ given before, we find that 
for $U\in T_{\R}X$, $A\left(U\right)$ is a complex endomorphism of 
$T_{\R}X$. By  \cite[Proposition 2.5 and eq. 
(2.37)]{Bismut90d}, we have the identity 
\begin{equation}\label{eq:dep10b}
\overline{\n}^{TX}=\n^{TX}+A.
\end{equation}

For $t>0$, if $\omega^{X}$ is replaced by $\omega^{X}/t$,  $\beta$ 
is replaced by $\beta/t$. However, the above connections are 
unchanged. Also if $\omega^{X}$ is closed, all the 
considered connections on $T_{\R}X$ coincide with $\n^{T_{\R}X}$.

By following \cite[Section 2 b)]{Bismut89a} and \cite[Section 
2.5]{Bismut10b}, let us describe the connection $\underline{\n}^{T_{\R}X}$ in 
more detail.
Put
\begin{equation}\label{eq:sup2}
F=T^{*}X \oplus TX
\end{equation}
Then $F$ is a holomorphic Hermitian vector bundle  on 
$X$. Let $\n^{F}=\n^{\TsX} \oplus \n^{TX}$ denote the corresponding 
Chern connection.    

Let $b$ be the smooth section of $\overline{T^{*}X}\otimes 
\Hom\left(TX,\TsX\right)$ such that if $u,v\in TX$,then
\begin{equation}\label{eq:sup3}
b\left(\overline{u}\right)v=i\pa^{X}\omega^{X}\left(\overline{u},v,\cdot\right).
\end{equation}

Put
\begin{equation}\label{eq:sup4}
\mathfrak 
b=\left(g^{TX}\right)^{-1}\overline{b \left( g^{TX} \right)^{-1}} .
\end{equation}
Then  $\mathfrak b$ is a section  of $\TsX \otimes 
\Hom\left(\TsX,TX\right)$.

Let 
\index{nF@$\underline{\n}^{F}$}%
$\underline{\n}^{F}$ be the connection on $F$,
\begin{equation}\label{eq:sup5}
\underline{\n}^{F}=
\begin{bmatrix}
	\n^{\TsX} & b \\
	\mathfrak b & \n^{TX}
\end{bmatrix}.
\end{equation}

Then $g^{\overline{TX}}$ 
identifies $\overline{TX}$ with $T^{*}X$, so that
\begin{equation}\label{eq:sup2a1}
F \simeq T_{\C}X = \overline{TX} \oplus TX.
\end{equation}

In  \cite[Theorem 2.9]{Bismut89a}, \cite[Theorem 2.5.3]{Bismut10b}, it was shown that via the identification $F \simeq 
T_{\C}X$, the connection $\underline{\n}^{F}$ is real and it induces the Euclidean 
connection $\underline{\n}^{T_{\R}X}$ on $T_{\R}X$. If  $\underline{\n}^{T_{\C}X}$ denote the corresponding 
connection on $T_{\C}X$, by (\ref{eq:sup5}), we get
\begin{equation}\label{eq:sup6}
\underline{\n}^{T_{\C}X}=
\begin{bmatrix}
	\n^{\overline{TX}} & \left(g^{\overline{TX}}\right)^{-1}b \\
\left(g^{TX}\right)^{-1}\overline{b} & \n^{TX}
\end{bmatrix}.
\end{equation}
By (\ref{eq:sup6}), we have the confirmation that 
$\underline{\n}^{T_{\R}X}$ is unchanged when  replacing $\omega^{X}$ 
by $\omega^{X}/t$.

By (\ref{eq:sup5}), we get
\begin{equation}\label{eq:sup6a1}
\underline{\n}^{F \prime \prime}=
\begin{bmatrix}
	\n^{\TsX \prime \prime } &b \\
0 & \n^{TX \prime \prime }
\end{bmatrix}.
\end{equation}

If $\overline{\pa}^{X}\pa^{X}i\omega^{X}=0$,  
by \cite[Theorem 2.7]{Bismut89a}, \cite[Theorem 2.5.3]{Bismut10b},   
$\underline{\n}^{F\prime \prime}$ induces a new holomorphic 
structure on $F$. We have the exact sequence of holomorphic vector 
bundles,
\begin{equation}\label{eq:exi1}
0\to \TsX\to F\to TX\to 0.
\end{equation}
Equivalently,  $\underline{\n}^{T_{\C}X \prime 
\prime }$ defines a holomorphic structure on $T_{\C}X$, and 
the exact sequence (\ref{eq:exi1})
 can  be rewritten in the form,
\begin{equation}\label{eq:exi2}
0\to \left( \overline{TX},\n^{\overline{TX} \prime \prime }\right)  \to 
\left( T_{\C}X,\underline{\n}^{T_{\C}X \prime \prime } \right) \to \left(TX,\n 
^{TX \prime \prime }\right)\to 0.
\end{equation}
In (\ref{eq:exi2}),  $\n^{\overline{TX} \prime \prime 
}=\overline{\n^{TX \prime }}$ depends on the metric $g^{TX}$.

If $\overline{\pa}^{X}\pa^{X}i\omega^{X}=0$, then $\underline{\n}^{T_{\C}X \prime \prime }$ is a holomorphic 
structure,  so that 
$\underline{R}^{T_{\R}X}$ is of type $(1,1)$. 
By \cite[Theorem 1.6]{Bismut89a}, \cite[Theorem 
2.5.6]{Bismut10b}, under the same assumption,  if $A,B,C,D\in T_{\R}X$, then
\begin{equation}\label{eq:sup7}
\left\langle  \overline{R}^{T_{\R}X}\left(A,B\right)C,D\right\rangle=
\left\langle  \underline{R}^{T_{\R}X}\left(C,D\right)A,B\right\rangle.
\end{equation}
Equation (\ref{eq:sup7}) is compatible with the fact that 
$\overline{R}^{T_{\R}X}$ takes its values in complex endomorphisms 
of $T_{\R}X$.

Let 
\index{nLTX@$\overline{\n}^{\Lambda\left(\overline{\TsX}\right)}$}%
$\overline{\n}^{\Lambda\left(\overline{\TsX}\right)}$ be the  connection on $\Lambda\left(\overline{\TsX}\right)$  induced by the unitary connection $ \overline{\n}^{TX}$ on $TX$.

 Assume for the moment that $X$ is spin, or equivalently that the line 
bundle $\det TX$ has a square root $\lambda$. Then $\lambda$ is a 
holomorphic Hermitian line bundle. Let $\n^{\lambda}$ be the 
corresponding Chern connection. Put
\begin{equation}\label{eq:spi1}
S^{T_{\R}X}=\Lambda\left(\overline{\TsX}\right) \otimes \lambda^{-1}.
\end{equation}
Then $S^{T_{\R}X}$ is the Hermitian vector bundle of 
$\left(T_{\R}X,\,g^{T_{\R}X} \right) $-spinors associated with the 
underlying spin structure,  the positive spinors 
 corresponding to even forms, and the negative spinors to odd 
forms. Also $S^{T_{\R}X}$ inherits a  Hermitian 
connection $\n^{S^{T_{\R}X},\mathrm{LC}}$ corresponding to $\n^{T_{\R}X,\mathrm{LC}}$,  that 
preserves its $\Z_{2}$-grading. Let 
$\n^{\Lambda\left(\overline{\TsX} \right),\mathrm{LC}}$ denote the connection on 
$\Lambda\left(\overline{\TsX}\right)$ that is induced by 
$\n^{S^{T_{\R}X}},\n^{\lambda}$. The connection 
$\n^{\Lambda\left(\overline{\TsX}\right),\mathrm{LC}}$ preserves the 
$\Z_{2}$-grading of $\Lambda\left(\overline{\TsX}\right)$.

The above considerations still make sense even if $X$ is not spin, since locally, a square 
root $\lambda$ always exists.

By proceeding the way we did for $\n^{T_{\R}X,\mathrm{LC}}$, we define the connection  
$\overline{\n}_{s}^{\Lambda\left(\overline{\TsX}\right)}$ to  be the 
connection induced by $\overline{\n}^{T_{\R}X},\n^{\lambda}$ on 
$\Lambda\left(\overline{\TsX}\right)$.  

Let $e_{1},\ldots,e_{2n}$ be an orthonormal basis of $T_{\R}X$, and 
let $e^{1},\ldots,e^{2n}$ be the corresponding dual basis of 
$T^{*}_{\R}X$. Let ${}^{c}A$ be the section of 
$\Lambda\left(T^{*}_{\R}X\right)\ho 
c\left(T_{\R}X\right)$ given by
\begin{equation}\label{eq:clic1}
{}^{c}A=\frac{1}{2}\left\langle  Ae_{i},e_{j}\right\rangle 
{}^{c}e_{i}{}^{c}e_{j}.
\end{equation}
By equation 
(\ref{eq:dep10b}), we get
\begin{equation}\label{eq:dep10bx1}
\overline{\n}_{s}^{\Lambda\left(\overline{\TsX}\right)}=\n^{\Lambda\left(\overline{\TsX}\right)}+{}^{c}A.
\end{equation}

Observe that $i_{\cdot}\beta$ is a $1$-form  with values 
in $\Lambda^{2}\left(T^{*}_{\R}X\right)$, so that ${}^{c}i_{\cdot}\beta$ 
is a $1$-form  with values in 
$\End\left(\Lambda\left(\overline{\TsX}\right)\right)$.
By \cite[eq. (2.34)] {Bismut89a} and by (\ref{eq:dep9}),  we 
have the identity\footnote{When $\omega^{X}$ is replaced by 
$\omega^{X}/t$, $\beta$ is replaced by $\beta/t$. However, ${}^{c}$ now depends on $t$. When suitably 
interpreted, (\ref{eq:dep10a1r1}) does not depend on the scaling.} 
\begin{equation}\label{eq:dep10a1r1}
\overline{\n}_{s}^{\Lambda\left(\overline{\TsX}\right)}=\n^{\Lambda\left(\overline{\TsX}\right),\mathrm{LC}}+{}^{c}
\left(i_{\cdot}\beta\right).
\end{equation}

In general,  the connections 
$\overline{\n}_{s}^{\Lambda\left(\overline{\TsX}\right)}$ and 
$\overline{\n}^{\Lambda\left(\overline{\TsX}\right)}$ 
do  not 
coincide. More precisely,  we have the identity\footnote{This point was clear in \cite{Bismut89a}, and was only 
stated implicitly in \cite[eq. (3.72)]{Bismut11a}. Still observe that 
these two connections induce the same connection on 
$\End\left(\Lambda\left(\overline{\TsX}\right)\right)$.}
\begin{equation}\label{eq:dep10bx2}
\overline{\n}^{\Lambda\left(\overline{\TsX}\right)}=\overline{\n}_{s}^{\Lambda\left(\overline{\TsX}\right)}+\frac{1}{2}\Tr^{TX}\left[A\right].
\end{equation}

Let 
\index{nLTX@$\n^{\Lambda\left(\overline{\TsX}\right)\ho D,\mathrm{LC}}$}%
$\n^{\Lambda\left(\overline{\TsX}\right)\ho D,\mathrm{LC}}$ be the 
connection on $\Lambda\left(\overline{\TsX}\right)\ho D$ that is 
induced by $\n^{\Lambda\left(\overline{\TsX}\right),\mathrm{LC}},\n^{D}$.
Let 
\index{nLTX@$\overline{\n}^{\Lambda\left(\overline{\TsX}\right)\ho D}$ }%
$\overline{\n}^{\Lambda\left(\overline{\TsX}\right)\ho D}$ 
denote the connection on $\Lambda\left(\overline{\TsX}\right)\ho 
D$ which is associated with 
$\overline{\n}^{\Lambda\left(\overline{\TsX}\right)},\n^{D}$.
\subsection{A Lichnerowicz formula for $A^{\mathcal{D},2}$}%
\label{subsec:lichbi}
In this Subsection, we assume  that $S$ is reduced to a point. 
Let 
\index{DX@$D^{X,\mathrm{LC}}$}%
$D^{X,\mathrm{LC}}$ denote the classical    Dirac operator acting on 
$C^{\infty }\left(X,\Lambda\left(\overline{\TsX}\right)\ho D\right)$. Namely, if 
$e_{1},\ldots,e_{2n}$ is a basis  $T_{\R}X$, and if 
$e^{1},\ldots,e^{2n}$ is the corresponding dual basis of 
$T^{*}_{\R}X$, then
\begin{equation}\label{eq:zep2}
D^{X,\mathrm{LC}}={}^\mathrm{c}e^{i}\n^{\Lambda 
\left( \overline{\TsX} \right) \ho D,\mathrm{LC}}_{e_{i}}.
\end{equation}

Now we assume that $e_{1},\ldots,e_{2n}$ is an orthonormal basis of 
$T_{\R}X$. Let 
\index{KX@$K^{X}$}%
$K^{X}$ be the scalar curvature of $X$, and let $R^{D}$ be the 
curvature of $\n^{D}$.  By Lichnerowicz formula, we get
\begin{equation}\label{eq:zep3}
D^{X,\mathrm{LC},2}=-\frac{1}{2}\n_{e_{i}}^{\Lambda\left(\overline{\TsX}\right)\ho D,\mathrm{LC},2}+
\frac{1}{8}K^{X}+{}^{\mathrm{c}}\left( 
R^{D}+\frac{1}{2}\Tr\left[R^{TX}\right] \right) .
\end{equation}
In the right-hand side of (\ref{eq:zep3}), 
$\n_{e_{i}}^{\Lambda\left(\overline{\TsX}\right)\ho D,\mathrm{LC},2}$ 
is our notation for the Bochner Laplacian. More precisely if 
$e_{1},\ldots,e_{2n}$ is a locally defined orthonormal basis of 
$T_{\R}X$, then
\begin{equation}\label{eq:zep3r1}
\n_{e_{i}}^{\Lambda\left(\overline{\TsX}\right)\ho 
D,\mathrm{LC},2}=\sum_{}^{}\n_{e_{i}}^{\Lambda\left(\overline{\TsX}\right)\ho D,\mathrm{LC},2}-
\n_{\sum_{}^{}\n^{T_{\R}X,\mathrm{LC}}_{e_{i}}e_{i}}^{\Lambda\left(\overline{\TsX}\right)\ho D,\mathrm{LC}}
\end{equation}
Other Bochner Laplacians will be denoted in the same way, with the 
same correction in the right-hand side.
\begin{theorem}\label{thm:Pide}
	The following identities hold:
	\begin{align}\label{eq:cod2}
&A^{\mathcal{D}}=D^{X,\mathrm{LC}}+{}^{\mathrm{c}} \left( \beta+C \right), \nonumber \\
&A^{\mathcal{D},2}
=-\frac{1}{2}\left(\n^{\Lambda\left(\overline{\TsX}\right)\ho 
D,\mathrm{LC}}_{e_{i}}+{}^{\mathrm{c}}\left(i_{e_{i}}\left( \beta+C \right) \right)\right)^{2}+\frac{K^{X}}{8} \nonumber \\
&+^{\mathrm{c}}\left( 
A^{E,2}+\frac{1}{2}\Tr\left[R^{TX}\right]-\overline{\pa}^{X}\pa^{X}i\omega^{X}\right)+\left(^{\mathrm{c}}\left(\beta+C\right) \right) ^{2}\nonumber \\
&+\frac{1}{2}\sum_{}^{}\left(^{\mathrm{c}}\left(i_{e_{i}}\left(\beta+C\right)\right)\right)^{2}- {}^{\mathrm{c}}\left( \left(\beta+C\right)^{2}\right) ,\\
&\left(^{\mathrm{c}}\left(\beta+C\right)\right)^{2}+\frac{1}{2}\sum_{}^{}\left(^{\mathrm{c}}\left(i_{e_{i}}\left(\beta+C\right)\right)\right)^{2}-^{\mathrm{c}}\left(\left(\beta+C\right)^{2}\right) \nonumber \\
&=\sum_{\substack{i_{1<\ldots<i_{k}}\\k\ge 
2}}^{}\left(-1\right)^{k\left(k+1\right)/2}\frac{\left(1-k\right)}{2^{k/2}} {}^{\mathrm{c}}\left( \left( i_{e_{i_{1}}}\ldots i_{e_{i_{k}}}\left(\beta+C\right) \right) ^{2} \right) . \nonumber 
\end{align}
\end{theorem}
\begin{proof}
	We  use  \cite[Theorems 1.1, 1.3,  and 2.2]{Bismut89a} and 
	(\ref{eq:dep2}). 
As explained in Subsection \ref{subsec:cli},  our normalization on the Clifford algebra 
 is different from the one \cite{Bismut89a}. Also  $\beta+C$ is an odd 
	section of $\Lambda\left(T^{*}_{\C}X\right)\ho 
	\End\left(D\right)$, while $\End \left( D \right) $ was absent in \cite[Theorem 2.3]{Bismut89a}.  
	Still, for formal reasons, it is easy to check that the  result of 
	\cite{Bismut89a} remains correct. The proof of our theorem is completed.  
\end{proof}
\begin{remark}\label{rem:Hodge}
	Observe that $A^{\mathcal{D},2}=\left[A^{\mathcal{D} \prime 
	\prime },A^{\mathcal{D} \prime }\right]$. The standard 
	consequences of  Hodge theory hold true for 
	$A^{\mathcal{D},2}$.
	
	By (\ref{eq:dep10a1r1}), (\ref{eq:dep10bx2}),  we get
	\begin{multline}\label{eq:sup1x-1}
-\frac{1}{2}\left(\n^{\Lambda\left(\overline{\TsX}\right)\ho 
D,\mathrm{LC}}_{e_{i}}+{}^{\mathrm{c}}\left(i_{e_{i}}\left( \beta+C \right) 
\right)\right)^{2}\\
=-\frac{1}{2}\left(\overline{\n}^{\Lambda\left(\overline{\TsX}\right)\ho 
D}_{e_{i}}-\frac{1}{2}\Tr^{TX}\left[A\left(e_{i}\right)\right]+{}^{\mathrm{c}}\left(i_{e_{i}}C
\right)\right)^{2}.
\end{multline}
Equation (\ref{eq:sup1x-1}) makes clear that as it should be, the 
operator (\ref{eq:sup1x-1}) is of total degree $0$. Equations 
(\ref{eq:cod2}), (\ref{eq:sup1x-1}) give a false sense of symmetry 
between the roles of $\beta$ and $C$. In particular both ${}^{c}$ and 
$\beta$ do depend on $g^{TX}$. 
\end{remark}
\subsection{A Lichnerowicz formula for 
$A^{p_{*}\mathscr E_{0},2}$}%
\label{subsec:curv}
Since $M=X\times S\to S$ is product, the connections on $TX$ defined in 
Subsection \ref{subsec:lichbi} lift to $q^{*}TX$. Also the form 
$\beta$ in (\ref{eq:dep10a1}) lifts to the form $q^{*}\beta$ on $M$. 

Let $A^{\mathcal{D}, \mathrm{LC}}$ denote the Levi-Civita 
superconnection in the sense of \cite{Bismut86d}, \cite[Section 2 
a)]{BismutGilletSoule88b},  
\begin{equation}\label{eq:dep13}
A^{\mathcal{D},\mathrm{LC}}=\n^{D,S}+D^{X,\mathrm{LC}}.
\end{equation}
\begin{theorem}\label{thm:lichg}
	The following identities hold:
	\begin{align}\label{eq:dep15}
&A^{p_{*}\mathscr 
E_{0}}=A^{\mathcal{D},\mathrm{\mathrm{LC}}}+{}^{\mathrm{c}}\left( 
q^{*}\beta+C \right) , \nonumber\\
&A^{p_{*}\mathscr E_{0},2}
=-\frac{1}{2}\left(\n^{\Lambda\left(\overline{\TsX}\right)\ho 
D,\mathrm{LC}}_{e_{i}}+{}^{\mathrm{c}}\left(i_{e_{i}}\left(q^{*}\beta+C\right) \right) \right)^{2}+\frac{K^{X}}{8} \nonumber \\
&+^{\mathrm{c}}\left( 
A^{E_{0},2}+\frac{1}{2}\Tr\left[R^{TX}\right]-q^{*}\overline{\pa}^{X}\pa^{X}i\omega^{X}\right)+\left(^{\mathrm{c}}\left(q^{*}\beta+C\right) \right) ^{2}\nonumber \\
&+\frac{1}{2}\sum_{}^{}\left(^{\mathrm{c}}\left(i_{e_{i}}\left(q^{*}\beta+C\right)\right)\right)^{2}- {}^{\mathrm{c}}\left( \left(q^{*}\beta+C\right)^{2}\right) ,\\
&\left(^{\mathrm{c}}\left(q^{*}\beta+C\right)\right)^{2}+\frac{1}{2}\sum_{}^{}\left(^{\mathrm{c}}\left(i_{e_{i}}\left(q^{*}\beta+C\right)\right)\right)^{2}-^{\mathrm{c}}\left(\left(q^{*}\beta+C\right)^{2}\right) \nonumber \\
&=\sum_{\substack{i_{1<\ldots<i_{k}}\\k\ge 
2}}^{}\left(-1\right)^{k\left(k+1\right)/2}\frac{\left(1-k\right)}{2^{k/2}} {}^{\mathrm{c}}\left( \left( i_{e_{i_{1}}}\ldots i_{e_{i_{k}}}\left(q^{*}\beta+C\right) \right) ^{2} \right) . \nonumber 
\end{align}
\end{theorem}
\begin{proof}
	By  
	(\ref{eq:dep2}), by the first identity in (\ref{eq:cod2}), and by
	(\ref{eq:dep13}), we get the first identity in (\ref{eq:dep15}).
	
Let $U$ be a small open set in $S$. We will view $U$ as an open ball 
in $\C^{n'}$, and we replace $S\times X$ by $U\times X$.
	 Given $\epsilon>0$, we  equip $T\C^{n'}$ with the constant metric  
	$g^{T\C^{n'}}$ which is 
	Kähler,  and $TM$ with the metric $p^{*}\frac{g^{T\C^{n'}}}{\epsilon}\oplus 
	q^{*}g^{TX}$. Let $\omega^{\mathbf C^{n'}}$ be the Kähler form of 
	$\C^{n'}$. If $f_{1},\ldots,f_{n'}$ is an orthonormal basis of 
	$\C^{n'}$, and if $f^{1},\ldots,f^{n'}$ is the corresponding dual 
	basis, then
	\begin{equation}\label{eq:co1x1}
\omega^{\mathbf C^{n'}}=-i\fa\overline{f}^{\alpha}.
\end{equation}
Also $\sqrt{\epsilon}f_{1},\ldots,\sqrt{\epsilon}f_{n'}$ is an 
orthonormal basis of $T\C^{n'}$ for the metric 
$g^{T\C^{n'}}/\epsilon$.
 
Let $A^{\mathcal{D}}_{\epsilon}$ be the operator on $U\times X$, 
which is  analogue of the operator 
 $A^{\mathcal{D}}$
considered in Theorem \ref{thm:Pide}. 	 Since $\omega^{\mathbf C^{n'}}$ is closed, the analogue 
	of $\beta$ on $U\times X$ is just $q^{*}\beta$. Using Theorem 
	\ref{thm:Pide}, we get a Lichnerowicz formula for 
	$A^{\mathcal{D},2}_{\epsilon}$. In this formula, ${}^{c}$ now 
	depends on $\epsilon$, and will instead be denoted 
	${}^{c_{\epsilon}}$. 	Also the adjoint of $\fa$ is now 
	$\epsilon i_{f_{\alpha}}$. We will take a proper limit of this formula as 
	$\epsilon\to 0$. 
	
Set
\begin{equation}\label{eq:co2}
\mathsf A^{\mathcal{D}}_{\epsilon}=e^{-i\omega^{\mathbf C^{n'}}/\epsilon}A^{\mathcal{D}}_{\epsilon}e^{i\omega^{\mathbf C^{n'}}/\epsilon}.
\end{equation}
The effect of the conjugation is to change $\epsilon 
i_{\overline{f}_{\alpha}}$ into $\epsilon 
i_{\overline{f}_{\alpha}}-\fa$. In particular, as $\epsilon\to 0$, 
after conjugation, we have the uniform convergence together with all the derivatives,
\begin{equation}\label{eq:co2a1}
{}^{c_{\epsilon}}\left(q^{*}\beta+C\right)\to {}^{c}\left(q^{*}\beta+C\right).
\end{equation}
By  (\ref{eq:zep2}), by the first equation in (\ref{eq:cod2}), by 
(\ref{eq:dep13}), and by (\ref{eq:co2a1}), we find that  as 
$\epsilon\to 0$, 
\begin{equation}\label{eq:co3}
\mathsf A^{\mathcal{D}}_{\epsilon}\to A^{p_{*} \mathscr  E_{0}}, 
\end{equation}
in the sense that the coefficients of the considered differential 
operators and their derivatives of any order converge uniformly over 
compact subsets. 

Now we take the limit as $\epsilon\to 0$ of the analogue of the 
equation for $\mathsf 
A^{\mathcal{D},2}_{\epsilon}$  one derives from 
 (\ref{eq:cod2}).  It is now easy to take the limit as $\epsilon\to 
 0$ of our equation for $\mathsf 
A^{\mathcal{D},2}_{\epsilon}$ and to obtain (\ref{eq:dep15}). The proof of our theorem is completed. 
\end{proof}
\begin{remark}\label{rem:prod}
	Since the fibration $\pi:M\to S$ is product, none of the 
	subtleties contained in 
	\cite{Bismut86d,BismutGilletSoule88b,Bismut11a} appears in our 
	formula.
\end{remark}
\subsection{The elliptic superconnection forms}%
\label{subsec:elscn}
By equation (\ref{eq:dep15}), the operator $A^{p_{*}\mathscr E_{0},2}$ is elliptic along the 
fibers $X$, so that $\exp\left(-A^{p_{*}\mathscr E_{0},2}\right)$ is a 
fiberwise trace class operator.  

We will now imitate Definition \ref{def:chf} in an infinite 
dimensional context.
\begin{definition}\label{def:Dscf}
	Put
	\index{chA@$\ch\left(A^{p_{*}\mathscr E_{0}\prime \prime 
	},\omega^{X},g^{D}\right)$}%
	\begin{equation}\label{eq:zu1}
\ch\left(A^{p_{*}\mathscr E_{0}\prime \prime },\omega^{X},g^{D}\right)=\varphi\Trs\left[\exp\left(-A^{p_{*}\mathscr E_{0},2}\right)\right].
\end{equation}
\end{definition}
Then $\ch\left(A^{p_{*}\mathscr E_{0}\prime \prime 
},\omega^{X},g^{D}\right)$ is a smooth even form on $S$.

Let 
\index{P@$\mathscr P$}%
$\mathscr P$ be the collection of parameters $\omega^{X},g^{D}$. Let $d^{\mathscr P}$ be the de Rham operator on $\mathscr P$.

The metric $\alpha$ in (\ref{eq:coc7}) depends on $\omega^{X},g^{D}$. 
Then $\alpha^{-1} d^{\mathscr P}\alpha $ is a $1$-form with values in 
$\alpha$-self-adjoint endomorphisms.

Let $w_{1},\ldots,w_{n}$ be an orthonormal basis of $TX$ with respect 
to $g^{TX}$. Using (\ref{eq:coc7}), and proceeding as in 
	\cite[eq. (1.109) and Proposition 1.19]{BismutGilletSoule88c},  we get
\begin{equation}\label{eq:zu3}
\alpha^{-1} d^{\mathscr P}\alpha  =-{}^{\mathrm{c}}
d^{\mathscr P}i
\omega^{X}+\frac{1}{2}d^{\mathscr P}i\omega^{X}\left(w_{i},\overline{w}_{i}\right)+\left(g^{D}\right)^{-1}d^{\mathscr P}g^{D}.
\end{equation}

Now we have the obvious analogue of Theorem \ref{thm:carch}.
\begin{theorem}\label{thm:chca}
	The form $\ch\left(A^{p_{*}\mathscr E_{0}\prime \prime },\omega^{X},g^{D}\right)$ lies in $\Omega^{(=)}\left(S,\R\right)$, it is 
	closed, and its Bott-Chern cohomology class does not depend on 
	$\omega^{X},g^{D}$, or on the splitting of $E$.  
Also
 $\varphi\Trs \left[  \alpha^{-1}d^{\mathscr P}\alpha\exp\left(A^{p_{*}\mathscr E_{0},2}\right)\right]$ is a 
	$1$-form on $\mathscr P$ with values in
	$\Omega^{(=)}\left(S,\R\right)$, and moreover, 
	\begin{equation}\label{eq:zu2}
d^{\mathscr P}\ch\left(A^{p_{*}\mathscr E_{0}\prime \prime 
},\omega^{X},g^{D}\right)\\
=-\frac{\overline{\pa}^{S}\pa^{S}}{2i\pi}\varphi\Trs \left[  
\alpha^{-1} d^{\mathscr 
P}\alpha \exp\left(-A^{p_{*}\mathscr E_{0},2}\right)\right].
\end{equation}
\end{theorem}
\begin{proof}
	The proof of the fact that our forms are closed and lie in 
	$\Omega^{(=)}\left(S,\R\right)$ is formally the same 
	as the proof of Theorem \ref{thm:carch}. By proceeding as in the 
	proof of Theorem \ref{thm:carch}, we get (\ref{eq:zu2}). 
	
	Let us now prove that the Bott-Chern class of 
	$\ch\left(A^{p_{*}\mathscr E_{0}\prime \prime 
	},\omega^{X},g^{D}\right)$ does not depend on the splitting of 
	$E$. Let $\mathscr E'_{0}=\left( E'_{0},A^{E'_{0} \prime \prime 
	}\right) $ be another splitting of $E$. If $M'=M\times 
	\mathbf{P}^{1}$, and let $p_{1},p_{2},p_{3}$ be the projections $M'\to S, M'\to 
	M,M'\to \mathbf{P}^{1}$. Consider  
	$\mathscr E_{M'}=p^{*}_{2,b} \mathscr E$, so that $\mathscr 
	E_{M'}=\left(E_{M'},A^{E_{M'} \prime \prime }\right)$. The diagonal bundle 
	associated with $\mathscr E_{M'}$ is $p_{2}^{*}D$. Given $z\in 
	\mathbf{P}^{1}$, let $j_{z}$ denote the embedding 
	$M\times\left\{z\right\}\to M'$. For any $z\in \mathbf{P}^{1}$,  
	then
	$j_{z,b}^{*}\mathscr 
	E_{M'}=\mathscr E$. Let $\mathscr E_{M',0}=\left(E_{M',0}, A^{E_{M',0} \prime \prime }\right)$ be a splitting of 
	$\mathscr E_{M'}$ such that $j^{*}_{0,b}\mathscr E'_{0}=\mathscr 
	E_{0}$, and $j^{*}_{\infty ,b}\mathscr E_{M',0}=\mathscr E'_{0}$.    Let $q$ be the projection 
	$M'\to S'= S\times \mathbf{P}^{1}$. We equip $p_{2}^{*}D$ with the 
	metric $p_{2}^{*}g^{D}$. We will use  our previous results on the form
 $\ch\left(A^{\mathscr E_{M',0}\prime \prime 
 },\omega^{X},p_{2}^{*}g^{D}\right)$. The pull-backs of this form 
 to $S\times 0$, $S\times \infty $ are just 
 $\ch\left(A^{p_{*}\mathscr E_{0}\prime \prime 
 },\omega^{X},g^{D}\right), \ch\left(A^{p_{*}\mathscr E'_{0}\prime 
 \prime },\omega^{X},g^{D}\right)$.  Using the Poincaré-Lelong equation 
 (\ref{eq:bc5}), we get
 \begin{multline}\label{eq:pole}
\frac{\overline{\pa}^{S}\pa^{S}}{2i\pi}p_{1,*}\left[\ch\left(A^{q_{*}\mathscr E_{M',0}\prime \prime 
 },\omega^{X},p_{2}^{*}g^{D}\right)p^{*}_{3}\log\left(\left\vert  
 z\right\vert^{2}\right)\right]\\
 =\ch\left(A^{p_{*}\mathscr E_{0}\prime \prime 
 },\omega^{X},g^{D}\right)-\ch\left(A^{p_{*}\mathscr E'_{0}\prime \prime 
 },\omega^{X},g^{D}\right),
\end{multline}
from which we get the last statement in our theorem.
\end{proof}
\begin{definition}\label{def:ellscn}
Let
\index{cBC@$\cBC\left(A^{p_{*} \mathscr E \prime \prime }\right)$}%
$\cBC\left(A^{p_{*} \mathscr E \prime \prime }\right)\in 
H_{\mathrm{BC}}^{(=)}\left(S,\R\right)$ be the common Bott-Chern cohomology 
class of the forms $\ch\left(A^{p_{*} \mathscr E_{0}\prime \prime 
},\omega^{X},g^{D}\right)$.	
\end{definition}
\section{Elliptic superconnection forms and direct images}%
\label{sec:lift}
The purpose of this Section is to prove that, with the notation of 
Section \ref{sec:suel}, $\cBC\left(A^{p_{*}\mathscr E\prime \prime}\right)$ 
coincides with $\cBC\left(Rp_{*} \mathscr E\right)$.

This Section is organized as follows.  In Subsection 
\ref{subsec:sptrin}, we apply the technique of spectral truncation 
of Subsection \ref{subsec:trunc} to the infinite-dimensional $p_{*} 
\mathscr E_{0}$. This will permit us to extend to $p_{*} \mathscr 
E_{0}$ results we had established before only for ordinary 
antiholomorphic superconnections.

In Subsection \ref{subsec:existence}, 
we prove the critical result that there exists a classical 
antiholomorphic superconnection $\underline{\mathscr E}$ on $S$ such 
that $p_{*} \mathscr E$ and $\underline{\mathscr E}$ are 
quasi-isomorphic, the main point being that $\underline{\mathscr E}$ is 
finite-dimensional.

In Subsection \ref{subsec:diri}, we prove our main identity.

Finally, in Subsection \ref{subsec:lofra}, we briefly consider the 
case where $\mathscr HRp_{*} \mathscr E$ is locally free, in which 
case our main result has a much simpler proof.

We make the same assumptions, and we use the same notation as in 
Section \ref{sec:suel}.
\subsection{Spectral truncations in infinite dimensions}%
\label{subsec:sptrin}
We will now proceed in infinite dimensions as in Subsection 
\ref{subsec:trunc}. Recall that $A^{\mathcal{D} \prime \prime *}$ is  
the standard $L_{2}$-adjoint of $A^{\mathcal{D} \prime \prime }$. Put
\begin{equation}\label{eq:train1}
\mathbb A^{p_{*} \mathscr E_{0}}=A^{p_{*} \mathscr E_{0} \prime 
\prime }+A^{\mathcal{D} \prime\prime*}.
\end{equation}
Equation (\ref{eq:train1}) is the analogue of equation 
(\ref{eq:tro1}) for $\mathbb A^{E_{0}}$.

As in (\ref{eq:tro2}), we get
\begin{align}\label{eq:tro2b}
&\mathbb A^{p_{*} \mathscr E_{0},2}=\left[A^{p_{*} \mathscr E_{0} 
\prime \prime },A^{\mathcal{D}\prime \prime * }\right],
&\left[A^{p_{*} \mathscr E_{0} \prime \prime },\mathbb A^{p_{*} \mathscr E_{0},2}\right]=0,
\qquad \left[A^{\mathcal{D} \prime\prime*},\mathbb A^{p_{*} \mathscr E_{0},2}\right]=0.
\end{align}

The operator $A^{\mathcal{D},2}$ being a second order elliptic 
operator along the fibers $X$, it has discrete spectrum. As in 
(\ref{eq:tro6}), we get
\begin{equation}\label{eq:train2}
\mathrm{Sp}\,\mathbb 
A^{p_{*} \mathscr 
E_{0},2}=\mathrm{Sp}\,A^{\mathcal{D},2}.
\end{equation}

As in (\ref{eq:tro7a}), for $a>0$, set
\begin{equation}\label{eq:train3}
U_{a}=\left\{s\in S,a\notin \mathrm{Sp}\,A^{\mathcal{D},2}\right\}.
\end{equation}
Then $U_{a}$ is open in $S$. 
\begin{definition}\label{def:tro8b}
	Over $U_{a}$, put
	\index{Pa@$P_{a,-}$}%
	\begin{equation}\label{eq:tro9b}
P_{a,-}=\frac{1}{2i\pi}\int_{\substack{\lambda\in\C\\\left\vert  
\lambda\right\vert=a}}^{}\frac{d\lambda}{\lambda-\mathbb A^{p_{*} \mathscr E_{0},2}}.
\end{equation}
\end{definition}
Then $P_{a,-}$ is a fiberwise smoothing projector acting on $p_{*} 
\mathscr E_{0}$.

By 
(\ref{eq:tro2}), we get
\begin{align}\label{eq:tro9a1b}
&\left[A^{p_{*}\mathscr E_{0} \prime \prime},P_{a,-}\right]=0, 
&\left[A^{\mathcal{D} \prime\prime*},P_{a,-}\right]=0.
\end{align}

Put
\begin{equation}\label{eq:tro9a2b}
P_{a,+}=1-P_{a,-}.
\end{equation}
Then $P_{a,+}$ is also a projector. 

As in (\ref{eq:tro10}), we get
\begin{equation}\label{eq:tro10b}
P_{a,\pm}=P_{a,\pm}^{0}+P_{a,\pm}^{(\ge 1)}.
\end{equation}
Let 
\index{Da@$\mathcal{D}_{a,\pm}$}%
$\mathcal{D}_{a,\pm}$ be the direct sum of the eigenspaces of 
$A^{\mathcal{D},2}$ for eigenvalues $\lambda>a$ or $\lambda<a$.  Then  $P^{0}_{a,\pm}$ is the 
orthogonal projectors $\mathcal{D}\to \mathcal{D}_{a,\pm}$. Also the 
 analogue of (\ref{eq:tro11}) holds. Also $\mathcal{D}_{a,-}$ is a 
finite-dimensional vector bundle on $U_{a}$.

\begin{definition}\label{def:Eab}
	Over $U_{a}$, put
	\index{pEa@$ p_{*} \mathscr E_{0,a,\pm}$}%
	\begin{equation}\label{eq:tro12b}
 p_{*} \mathscr E_{0,a,\pm}=P_{a,\pm}p_{*}\mathscr E_{0}.
\end{equation}
\end{definition}
Then $ p_{*} \mathscr E_{0,a,\pm}$ is a subbundle of $p_{*} \mathscr E_{0}$ which is also a 
$\Lambda\left(\overline{T^{*}S}\right)$-module. Let $i_{a,\pm}$ be 
the corresponding embedding in $p_{*} \mathscr E_{0}$. 

We use the 
notation of Subsection \ref{subsec:filal}. We will establish an 
analogue of Theorem \ref{thm:pea}. The main difference is that here, 
$p_{*}\mathscr E_{0}$ is infinite-dimensional.
\begin{theorem}\label{thm:peab}
	On $U_{a}$, we have  a splitting of 
	$\Lambda\left(\overline{T^{*}S}\right)$-modules,
	\begin{equation}\label{eq:tro12a1b}
p_{*}\mathscr E_{0}=p_{*}\mathscr E_{0,a,+} \oplus p_{*}\mathscr E_{0,a,-}.
\end{equation}
Moreover, 
\begin{equation}\label{eq:mod1b}
F^{p}p_{*}\mathscr E_{0,a,\pm}=p_{*}\mathscr E_{0,a,\pm}\cap F^{p}p_{*}\mathscr E_{0}.
\end{equation}

Also $P_{a,\pm}$ induces a filtered isomorphism of 
$\Lambda\left(\overline{T^{*}S}\right)$-modules,
\begin{equation}\label{eq:tro12az1b}
	 \Lambda\left(\overline{T^{*}S}\right)\ho \mathcal{D}_{a,\pm} \simeq p_{*}\mathscr E_{0,a,\pm}.
\end{equation}
As a $\Lambda\left(\overline{T^{*}S}\right)$-module, $p_{*}\mathscr 
E_{0,a,-}$ 
verifies the conditions  in (\ref{eq:bc12}), and the associated diagonal bundle 
 is 
just $\mathcal{D}_{a,-}$. Also it is equipped with a splitting as in 
(\ref{eq:iv4a-1}), (\ref{eq:iv4}).

Moreover, 
	$A^{p_{*}\mathscr E_{0}\prime \prime }$ preserves the smooth sections of 
	$p_{*}\mathscr E_{0,a,\pm}$, and it induces an antiholomorphic 
	superconnection  $A^{p_{*}\mathscr E_{0,a,\pm} \prime \prime }$ on 
	$p_{*}\mathscr E_{0,a,\pm}$.
	
On $U_{a}$,  $P_{a,-}:\mathscr E_{0}\to \mathscr E_{0,a,-}$ and 
$i_{a,-}=\mathscr E_{0,a,-}\to \mathscr E_{0}$ are quasi-isomorphisms of 
$\mathcal{O}_{U_{a}}$-complexes, and $\mathscr H \mathscr 
E_{0,a,+}=0$. Finally, for $s\in U_{a}$, the complex 
$\mathcal{D}_{a,+,s}$ is exact.
\end{theorem}
\begin{proof}
	The proof is essentially the same as the proof of Theorem 
	\ref{thm:pea}. With respect to that Theorem, here, only 
	$D_{a,-}$ is finite-dimensional. Let us now prove that if
	$s\in U_{a}$, $\mathcal{D}_{a,+}$ is exact. By proceeding 
	as in the proof of Proposition \ref{prop:rec1}, i.e., using the 
	analogue of 
	(\ref{eq:tro15}),  the part of degree $0$ over 
	$\Lambda\left(\overline{T^{*}S}\right)$ of $A^{p_{*}\mathscr 
	E}_{0,a,+}$ is just the restriction of $A^{\mathcal{D} \prime 
	\prime }$ to $\mathcal{D}_{a,+}$. By Hodge theory, the complex 
	$\mathcal{D}_{a,+,s}$ is exact. The proof of our theorem is completed. 
\end{proof}

When considering another pair of metrics $g^{TX}, g^{D}$ on $TX,D$, 
the obvious analogue of Proposition \ref{prop:proj} still holds. 
Details are left to the reader.

As a consequence of Theorem \ref{thm:peab}, we give another proof of 
a theorem of Grauert \cite[Theorem 10.4.6]{GrauertRemmert84}.
\begin{theorem}\label{thm:dbc}
	The complex of $\mathcal{O}_{S}$-modules $p_{*} \mathscr E_{0}$ 
	defines an element in $\Db\left(S\right)$. Moreover, the Leray 
	spectral sequence associated with the filtration of $p_{*} 
	\mathscr E_{0}$ by $\Lambda\left(\overline{T^{*}S}\right)$ 
	degenerates at $E_{2}$.
\end{theorem}
\begin{proof}
	By Theorem \ref{thm:peab}, over $U_{a}$, $p_{*} \mathscr E_{0}$ and $p_{*} 
	\mathscr E_{0,a,-}$ are quasi-isomorphic. By Theorem 
	\ref{thm:cohco}, $\mathscr H p_{*} \mathscr E_{0,a,-}$ is a 
	coherent sheaf. This shows that $p_{*} \mathscr E_{0}$ defines an 
	object in $\Db\left(S\right)$. By Theorem \ref{thm:cohco}, the 
	spectral sequence associated with $p_{*}\mathscr E_{0,a,-}$ 
	degenerates at $E_{2}$. Also on $D_{a,+}$, we have the identity
	\begin{equation}\label{eq:trem1}
1=\left[A^{\mathcal{D} \prime \prime }, A^{\mathcal{D}\prime \prime *}\left[A^{\mathcal{D},2}\right]^{-1}\right].
\end{equation}
Equation (\ref{eq:trem1}) shows that on $p_{*} \mathscr E_{0,a,+}$, 
the spectral sequence stops at $E_{1}$. This concludes the proof of 
our theorem.
\end{proof}
\begin{remark}\label{rem:bgs}
	In the context of the theory of determinant bundles of direct 
	images, Theorems \ref{thm:peab} and \ref{thm:dbc} can be used as 
	a substitute of many arguments in 
	\cite[Section 3]{BismutGilletSoule88c}, where non-projective 
	Kähler manifolds are also considered. 
\end{remark}
\subsection{The existence theorem}%
\label{subsec:existence}
Recall that (\ref{eq:didim2a4x1}), (\ref{eq:didim3}) 
hold, and that $A^{p_{*} \mathscr E_{0} \prime \prime}$ is given by 
(\ref{eq:new2a1}).

Recall that  if $\mathscr F\in \Db\left(M\right)$, 
\index{RHF@$\mathscr R \mathscr H \mathscr F$}%
$\mathscr R \mathscr H 
\mathscr F\in K\left(M\right)$ was defined in (\ref{eq:dirim}).

As we saw in Theorem \ref{thm:dbc},  $p_{*}\mathscr E_{0}$ is an object in  $\Db\left(S\right)$.
\begin{theorem}\label{thm:exi}
	There exists a finite-dimensional antiholomorphic superconnection 
	$\underline{\mathscr E}=\left(\underline{E},A^{\underline{E} 
	\prime \prime }\right)$ on $S$, and a morphism of 
	$\mathcal{O}^{\infty 
	}_{S}\left(\Lambda\left(\overline{T^{*}S}\right)\right)$-modules 	
	$\phi:\underline{\mathscr E}\to p_{*}\mathscr E_{0}$  which is  a 
	quasi-isomorphism of $\mathcal{O}_{S}$-complexes,   such that
	for any $s\in S$, 
	$\phi_{s}:\left(\underline{D},\underline{v}_{0}\right)_{s}\to\left(\mathcal{D},A^{\mathcal{D}\prime \prime }\right)_{s} $ is a quasi-isomorphism.
	
We have the identity
	\begin{equation}\label{eq:carb1}
\underline{\mathscr E}\simeq Rp_{*}\mathscr E\ 
\mathrm{in}\,\Db\left(S\right).
\end{equation}
In particular, we have the identity,
\begin{equation}\label{eq:dirim2}
\mathscr R \mathscr H \underline{\mathscr E}=p_{!} \left[\mathscr R 
\mathscr H \mathscr E\right]\ \mathrm{in}\ K\left(S\right).
\end{equation}
\end{theorem}
\begin{proof}
	By Proposition \ref{prop:exibb}, there exists an object
	$\underline{\mathscr E}$ in $\mathrm{B}\left(S\right)$ and a morphism of 
	$\mathcal{O}^{\infty 
	}_{S}\left(\Lambda\left(\overline{T^{*}S}\right)\right)$-modules 
	$\phi:\underline{\mathscr E}\to p_{*} \mathscr E_{0}$, which is also 
	a quasi-isomorphism of $\mathcal{O}_{S}$-complexes. Using Theorem 
	\ref{thm:peab}, over $U_{a}$, if $\phi_{a}=P_{a,-}\phi$, 
	$\phi_{a}$ induces a quasi-isomorphism $\underline{\mathscr E}\to 
	p_{*}\mathscr E_{0,a,-}$.  
	
	Over $U_{a}$, we now deal with finite 
	dimensional antiholomorphic superconnections. By Proposition 
	\ref{prop:prophtpeq},  for  $s\in U_{a}$, 
	$\phi_{a}$ induces a quasi-isomorphism $\underline{D}_{s}\to 
	\mathcal{D}_{a,-,s}$. By Theorem \ref{thm:peab}, for $s\in U_{a}$,
	the complex $\mathcal{D}_{a,+,s}$ is exact. 	Therefore, for 
	any $s\in U_{a}$, $\phi_{s}:\underline{D}_{s}\to\mathcal{D}_{s}$ 
	is a quasi-isomorphism. 
	
	Since $\phi$ is a quasi-isomorphism, we get (\ref{eq:carb1}),  from which (\ref{eq:dirim2}) follows. The proof of our theorem is completed. 
	\end{proof}

\begin{remark}\label{rem:uniq}
By equation (\ref{eq:carb1}) in Theorem \ref{thm:exi},  if $\underline{E}'$ has the same properties as 
	$\underline{E}$, then 
	\begin{equation}\label{eq:confi1}
\underline{\mathscr E}\simeq \underline{\mathscr 
E}'\ \mathrm{in}\ \Db\left(S\right).
\end{equation}
Using Theorem \ref{thm:eqcat} and (\ref{eq:confi1}), we deduce that 
there exists a quasi-isomorphism $\psi:\underline{\mathscr E}\to 
\underline{\mathscr E}'$. By Proposition \ref{prop:prophtpeq}, this 
just says that if $\underline{D},\underline{D'}$ are the 
corresponding diagonal bundles on $S$, for any $s\in S$, 
$\psi_{s}:\underline{D}_{s}\to \underline{D}'_{s}$ is a 
quasi-isomorphism.
\end{remark}
\subsection{The Chern character of the direct 
image}%
\label{subsec:diri}
We use the notation of Theorem \ref{thm:exi}. In particular 
$\underline{\mathscr E}$ is taken as in this Theorem. 
\begin{theorem}\label{thm:simb}
		The following identity holds:
		\begin{equation}\label{eq:coc15}
\cBC\left(A^{p_{*} \mathscr E \prime \prime 
}\right)=\cBC\left(\underline{\mathscr E}\right)\ \mathrm{in}\ 
H^{(=)}_{\mathrm{BC}}\left(S,\R\right).
\end{equation}
	\end{theorem}
	\begin{proof}
		The proof is  an adaptation of the proof of Theorems 
		\ref{thm:conv} and \ref{thm:idc}. We will use the notation in
		Theorem \ref{thm:exi}. 		For $t\in \C$, we  form the cone 
	$\mathscr C_{t}=\mathrm{cone}\left(\underline{\mathscr E}, 
		p_{*} \mathscr E\right)$ associated with the 
		morphism $t\phi$, with underlying infinite-dimensional vector bundle $C_{t}$ on 
		$S$. Note that $\phi$ induces a morphism $\underline{D}\to 
		\mathcal{D}$. The  diagonal bundle $\mathbf{D}$ associated 
		with $\mathscr C_{t}$ is given 
		$\mathbf{D}_{t}=\mathrm{cone}\left(\underline{D},\mathcal{D}\right)$, where the morphism is induced by $t\phi_{0}$. Also
		\begin{equation}\label{eq:sum1}
\mathbf{D}^{\Ou}_{t}=\underline{D}^{\Ou+1} \oplus \mathcal{D}^{\Ou}.
\end{equation}

By  Theorem \ref{thm:exi}, for any  $s\in S$, when $t\neq 0$, $\mathbf{D}_{t}$ is 
 exact.

We fix splittings of $E,\underline{E}$ as in (\ref{eq:iv4a-1}), 
(\ref{eq:iv4}). In the sequel, we may as well replace 
$\underline{\mathscr E},\mathscr E$ by $\underline{\mathscr E}_{0}, 
\mathscr E_{0}$. In particular $\mathscr C_{t}$ also splits and can 
be replaced by $\mathscr C_{t,0}$.

		 Let 
$g^{\underline{D}}$ be 
		a Hermitian metric on $\underline{D}$.  Let 
$g^{\mathbf{D}}$ denote the direct sum 
		 Hermitian metric on $\mathbf{D}$ associated with 
		 $g^{\underline{D}}, g^{\mathcal{D}}$. We equip $C_{t,0}$ 
		 with the pure metric $g^{\mathbf{D}}$.  Let 
		 $\square^{\mathbf{D}}_{t}$ denote 
		the corresponding  fiberwise Hodge Laplacian acting on 
		$\mathbf{D}$.  Since for any $s\in S$, $\mathbf{D}_{t}$ is 
		exact, 
		$\square^{\mathbf{D}}_{t,s}$ is invertible, and it has a uniform positive 
		lower bound over $S$. Its heat kernel is again fiberwise 
		trace class. 
		
		Let $A^{C_{t,0} \prime \prime }$ denote the antiholomorphic 
		superconnection on $C_{t}$ that is associated with $A^{\underline{E}_{0} \prime 
		\prime }, A^{p_{*} \mathscr E_{0} \prime \prime }, t\phi$. 
		
		Given $t\in \C$, we can still define the Chern form
		$\ch\left(A^{C_{t,0} \prime \prime },\omega^{X}, 
		g^{D},g^{\underline{D}}\right)$, for which the analogue of 
		Theorem \ref{thm:chca} holds, with a similar proof. Let 
		$\cBC\left(A^{C_{t,0} \prime \prime }\right)\in 
		H_{\mathrm{BC}}^{(=)}\left(S,\R\right)$ denote the 
		corresponding common Bott-Chern cohomology class.

For $T>0$, let $g_{T}^{D},g_{T}^{\underline{D}}$ be the 
metrics on $D,\underline{D}$ constructed as in (\ref{eq:iv36a2}).  Let 
$g^{\mathbf{D}}_{T}$ be the metric on $\mathbf{D}$ that is associated with the 
metrics $\omega^{X}/T, g^{D}_{T}, g^{\underline{D}}_{T}T^{-n-1}$. If 
$N^{\mathbf{D}}$ is the number operator of $\mathbf{D}$, then
\begin{equation}\label{eq:extra2}
g^{\mathbf{D}}_{T}=T^{-n}g^{\mathbf{D}}T^{N^{\mathbf{D}}}.
\end{equation}
The proof of (\ref{eq:extra2}) is elementary. The extra factor 
$T^{-n}$ comes from the contribution of the volume form on $X$ to 
the  Hermitian product on $\mathcal{D}$.
This normalizing factor does not contribute to the computation of the 
adjoint $A^{C _{t,0}\prime }_{T}$ of $A^{C_{t,0}\prime \prime }$.
By (\ref{eq:extra2}), if $\square^{\mathbf{D}}_{t,T}$ denotes the 
associated fiberwise Laplacian,  then
\begin{equation}\label{eq:seni-4a1}
\square^{\mathbf{D}}_{t,T}=T\square^{\mathbf{D}}_{t}.
\end{equation}

 We claim 
that if $t\neq 0$, as $T\to + \infty $, 
\begin{equation}\label{eq:seni-4}
\ch\left(A^{C_{t,0} \prime \prime 
},\omega^{X}/T,g_{T}^{D}, g_{T}^{\underline{D}}T^{-n-1}\right)\to 0.
\end{equation}
The proof is an infinite-dimensional analogue of equation 
(\ref{eq:iv36a2b}) in  Theorem 
\ref{thm:conv}, with $HD=0$ which can be established by the methods 
of  \cite[Theorems 9.19 and 9.23]{BerlineGetzlerVergne},  \cite[Theorems 9.5 and 9.6]{Bismut97a}. 

By (\ref{eq:seni-4}), we deduce that for $t\neq 0$, 
\begin{equation}\label{eq:coc22}
\cBC\left(A^{C_{t,0} \prime \prime }\right)=0.
\end{equation}
By taking $t\to 0$ in (\ref{eq:coc22}), we get 
\begin{equation}\label{eq:coc23}
\cBC\left(A^{C_{0,0} \prime \prime }\right)=0.
\end{equation}

As in (\ref{eq:cor2a4}), we have the identity 
\begin{equation}\label{eq:coc24}
\cBC\left(A^{C_{0,0} \prime \prime }\right)=\cBC\left(A^{p_{*} 
\mathscr E_{0} \prime \prime 
}\right)-\cBC\left(\underline{\mathscr E}\right).
\end{equation}
By (\ref{eq:coc23}), (\ref{eq:coc24}), we get (\ref{eq:coc15}). The proof of our theorem is completed. 
\end{proof}
\begin{remark}\label{rem:clev}
	By   Theorem \ref{thm:cla} and Remark \ref{rem:uniq}, we already 
	know that $\cBC \left(\underline{\mathscr E}\right)$ does not 
	depend on the choice of  $\underline{E}$. Theorem 
	\ref{thm:simb} gives a direct proof of this fact.
\end{remark}
\begin{theorem}\label{thm:Tes}
	The following identity holds:
	\begin{equation}\label{eq:cox1}
\cBC\left(A^{p_{*} \mathscr E \prime \prime 
}\right)=\cBC\left(Rp_{*}\ \mathscr E\right)\ \mathrm{in}\ 
H^{(=)}_{\mathrm{BC}}\left(S,\R\right).
\end{equation}
\end{theorem}
\begin{proof}
	Using equation (\ref{eq:carb1}) in Theorem \ref{thm:exi} and 
	Theorem \ref{thm:simb}, we get (\ref{eq:cox1}). 
	\end{proof}
\subsection{The case where $\mathscr HRp_{*}\mathscr E$ is locally free}%
\label{subsec:lofra}
By Theorem \ref{thm:kth}, we know that
\begin{equation}\label{eq:lofr1}
\cBC\left(\underline{\mathscr E}\right)=\cBC\left(\mathscr R\mathscr 
H\underline{\mathscr E}\right) \ \mathrm{in}\ 
H^{(=)}_{\mathrm{BC}}\left(S,\R\right).
\end{equation}
Also $\mathscr H Rp_{*}\mathscr E$ is a $\Z$-graded coherent sheaf on 
$S$.  Using  equation (\ref{eq:coc15}) in  Theorem \ref{thm:Tes} and 
(\ref{eq:lofr1}),  we get\begin{equation}\label{eq:lofr2}
\cBC\left(A^{p_{*} \mathscr E \prime \prime 
}\right)=\cBC\left(   \mathscr H Rp_{*}\mathscr E\right)\ \mathrm{in}\ 
H^{(=)}_{\mathrm{BC}}\left(S,\R\right).
\end{equation}

When $\mathscr H Rp_{*}\mathscr 
E$ is locally free, there is a   simpler proof of Theorem 
\ref{thm:Tes}   not relying on Theorem 
\ref{thm:exi}. In this  case, the same arguments as in Subsection 
\ref{subsec:lofr} show that the cohomology of the complex 
$\left(\mathcal{D}, A^{\mathcal{D} \prime \prime }\right)$ is a 
finite-dimensional holomorphic vector bundle $H$ on $S$. Using 
fiberwise Hodge theory, we can identify $H$ with the corresponding 
fiberwise harmonic sections $\mathcal{H}$.  Let $g^{H}$ denote the 
corresponding   Hermitian metric on $H$. 

Under the above assumption, a strict analogue of Theorem \ref{thm:conv} can be proved. 
More precisely, by proceeding as in  Berline-Getzler-Vergne \cite[Theorems 9.19 and 
9.23]{BerlineGetzlerVergne}, \cite[Theorems 9.5 and 
9.6]{Bismut97a}, and using the obvious analogue of (\ref{eq:extra2}), 
one can show that as $T\to + \infty $, we have the uniform 
convergence of forms on $S$, 
\begin{equation}\label{eq:lofr3}
\ch\left(A^{p_{*} \mathscr E_{0} \prime \prime }, \omega^{X}/T, g_{T}^{D}\right)\to
\ch\left(H,g^{H}\right).
\end{equation}
By (\ref{eq:lofr3}), we deduce that
\begin{equation}\label{eq:lofr4}
\cBC\left(A^{p_{*} \mathscr E \prime \prime }\right)=\cBC\left(H\right) \ \mathrm{in}\ 
H^{(=)}_{\mathrm{BC}}\left(S,\R\right).
\end{equation}
Moreover, 
\begin{equation}\label{eq:lofr5}
\mathscr H Rp_{*} \mathscr E=H\, \mathrm{in}\, K\left(S\right).
\end{equation}
Using (\ref{eq:lofr4}), (\ref{eq:lofr5}), we reobtain Theorem 
\ref{thm:Tes}.

The above strategy  was used  in Bismut \cite[Theorems 
4.10.4 and 4.11.2]{Bismut10b} in 
the case where $D$ is a holomorphic vector bundle on $M$, and 
$\left(E,A^{E \prime \prime }\right)=\left(D,\n^{D \prime \prime 
}\right)$, in order to 
establish Theorem \ref{thm:Tes}.
\section{A proof of Theorem \ref{thm:sub} when 
$\overline{\pa}^{X}\pa^{X}\omega^{X}=0$.}%
\label{sec:prka}
We make the same assumptions as in Section \ref{sec:suel}.  We will
prove   Theorem \ref{thm:sub} under the assumption that 
$\overline{\pa}^{X}\pa^{X}\omega^{X}=0$,  by extending  \cite[Theorem 
5.1.2]{Bismut10b}. Under this assumption, by the results of 
Subsections \ref{subsec:bovb} and  \ref{subsec:cox}, 
$\underline{R}^{T_{\R}X}$ is a $(1,1)$ form, so that
$\widehat{A}\left(T_{\R}X,\underline{\n}^{T_{\R}X}\right)$ is a 
closed form in $\Omega^{(=)}\left(X,\R\right)$.
\begin{theorem}\label{thm:clo}
If $\overline{\pa}^{X}\pa^{X}\omega^{X}=0$, as $t\to 
0$, we have the uniform convergence of forms on $S$, 
\begin{multline}\label{eq:coc25a1}
\ch\left(A^{p_{*} \mathscr E_{0} \prime 
\prime},\omega^{X}/t,g^{D}\right)\\
\to p_{*}
\left[q^{*}\left[\widehat{A}\left(T_{\R}X,\underline{\n}^{T_{\R}X}\right)e^{c_{1}\left(TX,g^{TX}\right)/2}\right]\ch\left(A^{E_{0} \prime \prime 
},g^{D}\right)\right].
\end{multline}
In particular, if $\omega^{X}$ is closed, as $t\to 0$, 
	\begin{equation}\label{eq:coc25}
\ch\left(A^{p_{*} \mathscr E_{0} \prime \prime},\omega^{X}/t,g^{D}\right)\to p_{*}
\left[q^{*}\Td\left(TX,g^{TX}\right)\ch\left(A^{E_{0} \prime \prime 
},g^{D}\right)\right].
\end{equation}

If $\overline{\pa}^{X}\pa^{X}\omega^{X}=0$,  then
\begin{equation}\label{eq:dir1bi}
\cBC\left(Rp_{*} \mathscr 
E\right)=p_{*}\left[q^{*}\Td_{\mathrm{BC}}\left(TX\right)\cBC\left(\mathscr E\right)\right]\,\mathrm{in}\, H^{(=)}_{\mathrm{BC}}\left(S,\R\right).
\end{equation}
\end{theorem}
\begin{proof}
	In the proof, we will assume that 
	$\overline{\pa}^{X}\pa^{X}\omega^{X}=0$. Also, we fix $s\in S$. 
	The analysis will be done in the fiber $X\times\left\{s\right\}$.
	
	Given $t>0, x,x'\in X$, let $P_{t}\left(x,x'\right)$ denote the 
	smooth kernel for $\exp\left(-A^{p_{*} \mathscr 
	E_{0},2}_{t}\right)$ along the fiber $X\times\{s\}$ with respect 
	to the volume  $dx'/\left(2\pi\right)^{n}$ \footnote{In the 
	proof, the parameter $s$ will be usually omitted.}. Then 
	\begin{equation}\label{eq:dirbic}
\ch\left(A^{p_{*} \mathscr E_{0} \prime 
\prime},\omega^{X}/t,g^{D}\right)=
\varphi\int_{X}^{}\Trs\left[P_{t}\left(x,x\right)\right]\frac{dx}{\left(2\pi\right)^{n}}.
\end{equation}
To calculate the asymptotics of (\ref{eq:dirbic}) as $t\to 0$, we will use  the local index 
	theoretic methods of \cite{Bismut86d,BismutGilletSoule88c,Bismut11a}.

We denote by 
\index{ct@${}^{\mathrm{c}_{t}}$}%
	${}^{\mathrm{c}_{t}}$ the map ${}^{c}$ associated with 
the metric $g^{TX}/t$ on $TX$, that induces the metric $tg^{\TsX}$ on 
$\TsX$. By (\ref{eq:xiao2a1}), if $f\in \TsX$, if $f_{*}\in 
\overline{TX}$ corresponds to $f$ by the metric $g^{TX}$, then
\begin{align}\label{eq:roa1}
&c_{t}\left(\overline{f}\right)=\overline{f}\we,
&c_{t}\left(f\right)=-i_{tf_{*}}.
\end{align}

Let $A_{t}^{p_{*} \mathscr E_{0}}$ denote the superconnection 
$A^{p_{*} \mathscr E_{0}}$ associated with $\omega^{X}/t,g^{D}$. In  the right-hand 
	side of the  equation  for $A^{p_{*} \mathscr E_{0},2}_{t}$ in  
	(\ref{eq:dep15}),  by    (\ref{eq:sup1x-1}), in the first row 
	of the right-hand side, the first term is now 
	\begin{equation}\label{eq:coc26}
-\frac{1}{2}t \left( 
\overline{\n}_{e_{i}}^{\Lambda\left(\overline{T^{*}X}\right)\otimes 
D}-\frac{1}{2}\Tr^{TX}\left[A\left(e_{i}\right)\right]+{}^{\mathrm{c}_{t}}\left(i_{e_{i}}C\right) \right) ^{2}.
\end{equation}
The second term is also scaled by the 
	factor $t$.
	The term containing
	$\overline{\pa}^{X}\pa^{X}i\omega^{X}$ 
	is absent in the right-hand side.  Also the $e_{i}$ are replaced 
	by $\sqrt{t}e_{i}$.
	
	We will compute the limit as $t\to 0$ of 
	$\Trs\left[P_{t}\left(x,x\right)\right]$. 	We fix $x\in X$. We take  geodesic 
	coordinates on $X$ near $x$ with respect to the metric 
	$g^{TX}$. This way, for $\epsilon>0$ small enough,  for any $x\in 
	X$, we identify the open ball 
	$B^{T_{\R,x}X}\left(0,\epsilon\right)$ in $T_{\R,x}X$ to a corresponding open
	geodesic ball $B^{X}\left(x,\epsilon\right)$ in $X$.	Given $Y\in 
	T_{\R,x}X$, along the geodesic $u\in \R\to 
	\exp_{x}\left(uY\right)\in X$,  we 
	trivialize $TX,D$ by parallel transport with respect to the 
	connections  $\overline{\n}^{TX},\n^{D}$.  Also we trivialize $ 
    \Lambda\left(\overline{T^{*}X}\right)$ via the 
	connection 
	$\overline{\n}^{\Lambda\left(\overline{\TsX}\right)}$.  If 
	$w_{1},\ldots,w_{n}$ is an orthonormal basis of $T_{x}X$, 
	it can also be viewed as an orthonormal 
	trivialization of $TX$ near $x$. In this trivialization, we have 
	the identity
	\begin{equation}\label{eq:sup9}
\omega^{X}=-iw^{i}\overline{w}^{i}.
\end{equation}
If $e_{1},\ldots,e_{2n}$ is an orthonormal basis of $T_{\R,x}X$, via 
the parallel transport   with respect to $\overline{\n}^{T_{\R}X}$ along the above geodesics, we obtain an 
orthonormal basis of $T_{\R}X$ near $x$. 

Using the above trivialization, near
	$x$, the operator $A^{p_{*} \mathscr E_{0},2}_{t}$ 
	 acts on 
	 $$H_{x,\epsilon}= C^{\infty 
	}\left(B^{T_{\R,x}\left(0,\epsilon\right)},
	\Lambda\left(T^{*}_{\C,s}S\right)\ho\Lambda\left(\overline{T^{*}_{x}X}\right) \ho D_{s,x}\right).$$
	
	A first standard step is to construct an elliptic operator acting 
	on
	$$H_{x}=C^{\infty 
	}\left(T_{\R,x}X,\Lambda\left(T^{*}_{\C,s}S\right)\ho\Lambda\left(\overline{T^{*}_{x}X}\right) \ho D_{s,x} \right),$$
	that coincides with $A^{p_{*} \mathscr E_{0},2}_{t}$ on 
	$B^{T_{\R,x}X}\left(0,\epsilon/2\right)$, the point being that 
	the asymptotics of $\Trs\left[P_{t}\left(x,x\right)\right]$ as 
	$t\to 0$ can be computed using this new operator. This step being 
	standard, we will skip it, and pretend that $A^{p_{*} \mathscr 
	E_{0},2}_{t}$ acts on $H_{x}$, and that $P_{t}\left(Y,Y'\right)$ 
	is now a kernel on $T_{\R,x}X$.

	For $a>0,u\in H_{x}$,  put
	\begin{equation}\label{eq:coc26a1}
F_{a}u\left(Y\right)=u\left(aY\right).
\end{equation}
Set
\begin{equation}\label{eq:coc26a2}
L_{t,x}=F_{\sqrt{t}}A^{p_{*} \mathscr E_{0},2}_{t}F_{1/\sqrt{t}}.
\end{equation}
Let $Q_{t,x}\left(Y,Y'\right),Y,Y'\in T_{\R,x}X$ be the smooth kernel for 
$\exp\left(-L_{t,x}\right)$ with respect to  the volume form 
$\frac{dY'}{\left(2\pi\right)^{n}}$. Clearly
\begin{equation}\label{eq:coc26a1z1}
P_{t}\left(x,x\right)=t^{-n}Q_{t,x}\left(0,0\right).
\end{equation}

Among the monomials in the $\overline{w}^{i}, i_{\overline{w}_{j}}$, 
observe that up to obvious permutations, the only monomial whose 
supertrace on $\Lambda\left(\overline{T^{*}_{x}X}\right)$ is nonzero 
is $\prod_{i=1}^{n}\overline{w}^{i}i_{\overline{w}_{i}}$, and 
moreover
\begin{equation}\label{eq:glig1a1}
\Trs^{\Lambda\left(\overline{T^{*}_{x}X}\right)}\left[\prod_{i=1}^{n}\overline{w}^{i}i_{\overline{w}_{i}}\right]=\left(-1\right)^{n}.
\end{equation}
If $I$ is a strictly ordered multi-index with values in $1,\ldots,n$, 
put 
$\overline{w}^{I}=\prod_{i\in 
I}^{}\overline{w}^{i},i_{\overline{w}_{I}}=\prod_{i\in I}^{}i_{\overline{w}_{i}}$. Any $T\in\End\left(\Lambda\left(\overline{T^{*}_{x}X}\right)\right)$ can be written uniquely  as a linear combination 
of operators of the form $\overline{w}^{I}i_{\overline{w}_{J}}$. Put
\begin{equation}\label{eq:glig1}
T_{t}=e^{-i\omega^{X}_{x}/t}Te^{i\omega^{X}_{x}/t}.
\end{equation}
Conjugation by the form $e^{-i\omega_{x}^{X}/t}$  leaves unchanged the operators
$\overline{w}^{i}\we$ and replaces the operators 
$i_{t\overline{w}_{i}}$ by $i_{t\overline{w}_{i}}-w^{i}\we$. 
Let $T_{t}^{\max}\in \C$ be the coefficient of 
$\left(-i\right)^{n}\prod_{1}^{n}\overline{w}^{i}w^{i}$ in the obvious expansion of 
$T_{t}$.  By (\ref{eq:glig1a1}), (\ref{eq:glig1}), we get
\begin{equation}\label{eq:glig2}
\Trs^{\Lambda\left(\overline{T^{*}_{x}X}\right)}\left[T\right]=\frac{t^{n}}{i^{n}}T_{t}^{\max}.
\end{equation}
If $T$ is instead an element of 
$\End\left(\Lambda\left(\overline{T_{x}^{*}X}\right)\right)\ho 
\End\left(D_{s,x}\right)$, we still define $T_{t}$ as in (\ref{eq:glig1}). 
Equation (\ref{eq:glig2}) is replaced by 
\begin{equation}\label{eq:glig2bi}
\Trs^{\Lambda\left(\overline{T^{*}_{x}X}\right)\ho 
D_{s,x}}\left[T\right]=\frac{t^{n}}{i^{n}}\Trs^{D_{s,x}}\left[T_{t}\right]^{\max}.
\end{equation}

When viewing $w^{i}, \overline{w}^{i}$  as differential forms on $X$ at 
$x$, then $\left(-i\right)^{n}\prod_{1}^{n}\overline{w}^{i}w^{i}$ is 
just the canonical volume $2n$-form.

In $L_{t,x}$, no matrix operator of the form $w^{i}\we 
$ appears. 
Put
\begin{equation}\label{eq:coc28}
M_{t,x}=e^{-i\omega_{x}^{X}/t}L_{t,x}e^{i\omega_{x}^{X}/t}.
\end{equation}
The above conjugation will play the role of 
Getzler rescaling \cite{Getzler86} in local index theory.

Let $R_{t,x}\left(Y,Y'\right),Y,Y'\in T_{\R,x}X$ be the smooth kernel 
for $\exp\left(-M_{t,x}\right)$ with respect to the 
volume $\frac{dY'}{\left(2\pi\right)^{n}}$.
By (\ref{eq:coc26a1z1}), (\ref{eq:glig2bi}), we get
\begin{equation}\label{eq:glig2bi1}
\Trs\left[P_{t}\left(x,x\right)\right]=\frac{1}{i^{n}}\Trs^{D_{s,x}}\left[R_{t,x}\left(0,0\right)\right]^{\max}.
\end{equation}

To compute the asymptotics of (\ref{eq:glig2bi1}) as $t\to 0$, we   use the second 
equation in (\ref{eq:dep15}) for $A^{p_{*} \mathscr E_{0},2}_{t}$ 
while  implementing  the above trivializations and conjugations.  
We will not write the conjugations explicitly.

We view $\overline{R}^{TX}_{x}$ as a section of 
$\Lambda^{2}\left(T^{*}_{\R,x}X\right) \otimes 
\End\left(T_{x}X\right)$. Let $\overline{\Gamma}^{TX}$ be the 
connection form for $\overline{\n}^{TX}$ in the above parallel 
transport trivialization. It is  well-known  that for $Y\in 
T_{\R,x}X$,
\begin{equation}\label{eq:coq1r1}
\overline{\Gamma}^{TX}_{Y}=\frac{1}{2}\overline{R}_{x}^{TX}\left(Y,\cdot\right)+\mathcal{O}\left(\left\vert  Y\right\vert^{2}\right).
\end{equation}
By (\ref{eq:coq1r1}), we deduce that if 
$\overline{\Gamma}^{\Lambda\left(\overline{\TsX}\right)}$ is the 
corresponding connection form for 
$\overline{\n}^{\Lambda\left(\overline{\TsX}\right)}$, then
\begin{equation}\label{eq:coq2r1}
	\overline{\Gamma}^{\Lambda\left(\overline{\TsX}\right)}_{Y}=-\frac{1}{2}\left\langle  \overline{R}_{x}^{TX}\left(Y,\cdot\right)\overline{w}_{i},w_{j}\right\rangle\overline{w}^{i}i_{\overline{w}^{j}}+\mathcal{O}\left(\left\vert  Y\right\vert^{2}\right),
\end{equation}
with $\mathcal{O}\left(\left\vert  Y\right\vert^{2}\right)$ still 
being linear combinations of the 
$\overline{w}^{i}i_{\overline{w}_{j}}$.

Let 
$\Gamma^{D}$ be the connection form in the considered trivialization 
of $D$. Then
\begin{multline}\label{eq:coc29a1}
-F_{\sqrt{t}}\frac{1}{2}t \left( 
\overline{\n}_{e_{i}}^{\Lambda\left(\overline{\TsX}\right) \ho 
D}-\frac{1}{2}\Tr^{TX}\left[A\left(e_{i}\right)\right]+{}^{\mathrm{c}_{t}}i_{e_{i}}C \right) ^{2}F_{1/\sqrt{t}}\\
=
-\frac{1}{2}\Biggl(\n_{e_{i}\left(\sqrt{t}Y\right)}+\sqrt{t}\left(\overline{\Gamma}^{\Lambda\left(\overline{\TsX}\right)}+\Gamma^{D} \right) _{\sqrt{t}Y}\left(e_{i}(\sqrt{t}Y)\right) \\
-\frac{\sqrt{t}}{2}\Tr^{TX}\left[A\left(e_{i}\right)\right]_{_{\sqrt{t}Y}}+\sqrt{t}{}^{c_{t}}i_{e_{i}}C_{\sqrt{t}Y}\Biggr)^{2}.
\end{multline}
We conjugate (\ref{eq:coc29a1}) by $e^{-i\omega_{x}^{X}/t}$. Using 
(\ref{eq:coq2r1}), we find that as $t\to 0$, 
\begin{multline}\label{eq:coc29a3}
\n_{e_{i}\left(\sqrt{t}Y\right)}+\sqrt{t}\left(\overline{\Gamma}^{\Lambda\left(\overline{\TsX}\right)}+\Gamma^{D} \right) _{\sqrt{t}Y}\left(e_{i}(\sqrt{t}Y)\right)\\
\to
\n_{e_{i}}+\frac{1}{2}\left\langle  
\overline{R}^{TX}_{x}\left(Y,e_{i}\right)\overline{w}_{i},w_{j}\right\rangle
\overline{w}^{i}w^{j}.
\end{multline}

Note that  ${}^{c_{t}}B=B$, and ${}^{c_{t}}B^{*}$ is obtained from 
$B^{*}$ by replacing the exterior variables $w^{i}$ by 
$-i_{t\overline{w}_{i}}$. By the considerations that follow 
(\ref{eq:glig1a1}), when conjugating $-i_{t\overline{w}_{i}}$ by 
$e^{-i\omega^{X}_{x}/t}$, we 
obtain $w^{i}-i_{t\overline{w}_{i}}$, so that as $t\to 0$,  after conjugation,
\begin{equation}\label{eq:coc29a4}
{}^{c_{t}}i_{e_{i}}C_{\sqrt{t}Y}\to i_{e_{i}} C_{x}.
\end{equation}
By (\ref{eq:coc29a4}), 
we deduce that
\begin{equation}\label{eq:coc29a4r1}
\sqrt{t}{}^{c_{t}}i_{e_{i}}C_{\sqrt{t}Y}\to 0.
\end{equation}

By (\ref{eq:coc29a3}), (\ref{eq:coc29a4r1}), we find that after 
conjugation, as $t\to 0$,
\begin{multline}\label{eq:coc29a5}
-\frac{1}{2}t \left( \overline{\n}_{e_{i}}^{\Lambda\left(\overline{\TsX}\right)\ho 
D}-\frac{1}{2}\Tr^{TX}\left[A\left(e_{i}\right)\right]+{}^{\mathrm{c}_{t}}i_{e_{i}}C \right) ^{2}\\
\to
-\frac{1}{2}\left(\n_{e_{i}}+\frac{1}{2}\left\langle  
\overline{R}_{x}^{TX}\left(Y,e_{i}\right)\overline{w}_{i},w_{j}\right\rangle \overline{w}^{i}w^{j}\right)^{2}.
\end{multline}
We will use the fact that $\overline{\pa}^{X}\pa^{X}\omega^{X}=0$, and also 
(\ref{eq:sup7}). Recall that  $w^{i}, \overline{w}^{j}$ are considered as 
standard differential forms on $X$ at $x$. Then (\ref{eq:coc29a5}) 
can be rewritten in the form
\begin{multline}\label{eq:coc29a6}
-\frac{1}{2}t \left( 
\overline{\n}_{e_{i}}^{\Lambda\left(\overline{\TsX} \right) \ho 
D}-\frac{1}{2}\Tr^{TX}\left[A\left(e_{i}\right)\right]+{}^{\mathrm{c}_{t}}i_{e_{i}}C \right) ^{2}\\
\to
-\frac{1}{2}\left(\n_{e_{i}}+\frac{1}{2}\left\langle  
\underline{R}_{x}^{TX}Y,e_{i}\right\rangle \right)^{2}.
\end{multline}
  
Consider now the remaining terms in the right-hand side of 
(\ref{eq:dep15}), while still using the fact that 
$\overline{\pa}^{X}\pa^{X}\omega^{X}=0$.  By proceeding 
as before, we find that after conjugation,
\begin{equation}\label{eq:roa2}
{}^{\mathrm{c}_{t}}\left(A^{E_{0},2}+\frac{1}{2}\Tr\left[R^{TX}\right]\right)_{\sqrt{t}Y}\to
\left( A^{E_{0},2}+\frac{1}{2}\Tr\left[R^{TX}\right] \right) _{s,x}.
\end{equation}

To obtain the behavior as $t\to 0$ of the remaining terms in the 
right-hand side of the second equation in (\ref{eq:dep15}), we will 
use the third identity with $\beta$  replaced by $\beta/t$, 
${}^{c}$ by ${}^{c_{t}}$, and  $e_{i}$ by $\sqrt{t}e_{i}$. After 
conjugation,  for $k\ge 3$, the limit of the corresponding terms 
vanishes. Since $\beta$ is a $3$-form, 
$i_{e_{i_{1}}}i_{e_{i_{2}}}\beta$ is a $1$-form, whose square 
vanishes. Ultimately, after conjugation, the limit of all the  terms 
vanishes.

Put
\begin{equation}\label{eq:roa3}
M_{0,s,x}=-\frac{1}{2}\left(\n_{e_{i}}+\frac{1}{2}\left\langle  
\underline{R}_{x}^{TX}Y,e_{i}\right\rangle\right)^{2}+A^{E_{0},2}_{s,x}+\frac{1}{2}\Tr\left[R^{TX}_{x}\right].
\end{equation}
By (\ref{eq:dep15}), (\ref{eq:coc29a6}), (\ref{eq:roa2}), and by the 
considerations that follow, we see that as $t\to 0$, 
\begin{equation}\label{eq:roa4}
M_{t,x}\to M_{0,s,x}, 
\end{equation}
in the sense that we have a  uniform convergence of the coefficients 
of the considered operators together with all their derivatives  over compact 
subsets of $T_{\R,x}X$.

Let $R_{0,s,x}\left(Y,Y'\right),Y,Y'\in T_{\R,x}X$ denote the smooth 
kernel associated with the operator $\exp\left(-M_{0,s,x}\right)$ with 
respect  to $\frac{dY'}{\left(2\pi\right)^{n}}$.  By proceeding as in 
\cite[Theorem 4.12]{Bismut86d}, \cite[Theorem 10.23]{BerlineGetzlerVergne} using (\ref{eq:roa4}), as $t\to 0$, we have the 
uniform convergence,
\begin{equation}\label{eq:roa4a1}
R_{t,x}\left(0,0\right)\to  R_{0,s,x}\left(0,0\right).
\end{equation}

Note that 
$\Trs^{D_{s,x}}\left[ R_{0,x}\left(0,0\right)\right]$ takes its values in 
$\Lambda\left(T^{*}_{\C}S\right)\ho\Lambda\left(T^{*}_{\C}X\right)$, 
to which the full operator $\varphi$ now applies.
Let $\left\{\varphi\Tr^{D_{s,x}}\left[
Q_{0,s,x}\left(0,0\right)\right]\right\}^{\max}\in 
\Lambda\left(T^{*}_{\C}S\right)$ be the coefficient of the volume 
form $dx$ viewed as a $2n$ form on $T_{\R,x}X$.
Using (\ref{eq:glig2bi1}), (\ref{eq:roa4a1}),   as $t\to 
0$, we have the uniform convergence over $X$, 
\begin{equation}\label{eq:roa4a1r1}
\varphi\left(2\pi\right)^{-n}\Trs\left[P_{t}\left(x,x\right)\right]\to 
\left\{\varphi\Tr^{D_{s,x}}\left[
R_{0,s,x}\left(0,0\right)\right]\right\}^{\max}.
\end{equation}

In the right-hand side of (\ref{eq:roa4}) one  recognizes the 
classical harmonic oscillator in index theory \cite{Getzler86}, 
\cite[Proposition 4.19]{BerlineGetzlerVergne}. Using Mehler's formula 
\cite[Theorem 1.5.10]{GlimmJaffe87}  and 
(\ref{eq:iv21}), we get
\begin{equation}\label{eq:roa5}
\varphi\Tr^{D_{s,x}}\left[R_{0,s,x}\left(0,0\right)\right]=
q^{*}\left[\widehat{A}\left(T_{\R}X,\underline{\n}^{T_{\R}X}\right)
e^{c_{1} \left(TX,g^{TX}\right)/2}\right]\ch\left(A^{E_{0} \prime \prime 
},g^{D}\right).
\end{equation}

Using (\ref{eq:dirbic}), (\ref{eq:roa4a1r1}), (\ref{eq:roa5}), and 
dominated convergence, we get (\ref{eq:coc25a1}). If $\omega^{X}$ is 
closed, $\overline{\n}^{T_{\R}X}=\n^{T_{\R}X}$, so that 
(\ref{eq:coc25a1}) gives (\ref{eq:coc25}). 

If $\omega^{X}$ is closed, (\ref{eq:dir1bi}) follows from 
(\ref{eq:coc25}). We will now prove (\ref{eq:dir1bi}) when we only 
assume that
$\overline{\pa}^{X}\pa^{X}\omega^{X}=0$. 
Recall that if $\alpha\in \Omega^{(=)}\left(X,\R\right)$ is closed, 
\index{a@$\left\{\alpha\right\}$}%
$\left\{\alpha\right\}$ denotes its Bott-Chern cohomology class in 
$H^{(=)}_{\mathrm{BC}}\left(X,\R\right)$. By Theorem \ref{thm:Tes} and by (\ref{eq:coc25a1}), we get
\begin{equation}\label{eq:roai1}
\cBC\left(Rp_{*} \mathscr E\right)=p_{*}\left[q^{*}\left\{\widehat{A}\left(T_{\R}X,\underline{\n}^{T_{\R}X}\right)
\right\}q^{*}e^{c_{1,\mathrm{BC}}\left(TX\right)}\cBC\left(\mathscr E\right)\right].
\end{equation}
Now we 
proceed as in  \cite[Theorem 5.2.1]{Bismut10b}. We have the 
equivalent exact sequence of holomorphic vector bundles on $X$ in 
(\ref{eq:exi1}), (\ref{eq:exi2}). Using the results of 
Bismut-Gillet-Soulé
\cite[Theorem 1.29]{BismutGilletSoule88a} on Bott-Chern classes, from 
the exact sequence in (\ref{eq:exi1}), using the notation in 
(\ref{eq:cw6}), we get
\begin{equation}\label{eq:coai2}
\left\{\widehat{A}\left(T_{\R}X,\underline{\n}^{T_{\R}X}\right)\right\}=\widehat{A}_{\mathrm{BC}}\left(TX\right).
\end{equation}
By (\ref{eq:crisp1}), we obtain
\begin{equation}\label{eq:coai3}
\Td_{\mathrm{BC}}\left(TX\right)=\widehat{A}_{\mathrm{BC}}\left(TX\right)
e^{c_{1,\mathrm{BC}}\left(TX\right)/2}.
\end{equation}

By (\ref{eq:roai1})--(\ref{eq:coai3}), we get (\ref{eq:dir1bi}).
The proof of our theorem is completed. 
\end{proof}
\begin{remark}\label{rem:locind}
	It is not possible to deform geometrically the exact 
	sequence (\ref{eq:exi2}) by scaling $\omega^{X}$. However, on the exact sequence (\ref{eq:exi1}), 
	the deformation arguments of \cite{BismutGilletSoule88a} can 
	still be used. More precisely, that if $d\in \C^{*}$,   replacing 
	$\omega^{X}$ by $d\omega^{X}$ in the construction of 
	$\underline{\n}^{F \prime \prime }$ in (\ref{eq:sup6a1})
	does not change the Bott-Chern class, which ultimately allows us 
	to make $d=0$, which is equivalent to what we did in the proof of 
	Theorem \ref{thm:clo}. Similar ideas will be used in a more sophisticate 
	context in the rest of the paper.
\end{remark}
\section{The hypoelliptic superconnections}%
\label{sec:hypo}
The purpose of this Section is to   construct an antiholomorphic superconnection over $S$ 
with fiberwise hypoelliptic curvature. This superconnection is a 
deformation of the elliptic superconnections of Section 
\ref{sec:suel}. Our constructions extend what was done in 
\cite[Chapter 6]{Bismut10b} for holomorphic vector bundles on $M$ to the case of antiholomorphic 
superconnections $A^{E \prime \prime}$. 

This Section is organized as follows. In Subsection \ref{subsec:tot}, 
we introduce the total space $\mathcal{X}$ of an extra copy 
$\widehat{TX}$ of $TX$, and also the antiholomorphic superconnection 
$\mathscr N_{X}=\pi_{*} \mathscr C^{\mathcal{X}}$ on $X$.

In Subsection \ref{subsec:hemer}, given a Hermitian metric 
$g^{\widehat{TX}}$ on $\widehat{TX}$, we obtain a corresponding 
splitting of $\mathscr N_{X}$.

In Subsection \ref{subsec:lift}, we introduce the total space 
$\underline{\pi}:\mathcal{M}\to M$ of 
$q^{*}\widehat{TX}$,  and we construct a natural 
infinite-dimensional antiholomorphic superconnection $\mathscr 
E_{M}=\mathscr E\ho_{b}\underline{\pi}_{*} \mathscr C^{M}$ 
over $M$.

In Subsection \ref{subsec:metde}, we introduce the superconnection 
$\mathcal{A}^{\prime \prime }$ that corresponds to $A^{E_{M} \prime 
\prime }$ via a natural 
splitting of $E_{M}$.

In Subsection \ref{subsec:enls}, given metrics 
$g^{\widehat{TX}},g^{D}$ on $\widehat{TX},D$, we construct a 
nonpositive Hermitian form on the diagonal bundle associated with 
$E_{M}$, and the corresponding adjoint superconnection  $\mathcal{A}^{\prime }$ of 
$\mathcal{A^{\prime \prime }}$.

In Subsection \ref{subsec:amz}, if $\underline{q}:\mathcal{M}\to X$ 
is the obvious projection, given a holomorphic section $z$ of 
$\underline{q}^{*}TX$ on $\mathcal{M}$, we construct an 
antiholomorphic superconnection $A^{E_{M} \prime \prime }_{z}$ on 
$\Lambda\left(T^{*}M\right)\ho \mathscr E_{M}$.

In Subsection \ref{subsec:adaz}, still using the above splittings, we 
obtain a superconnection $\mathcal{A}^{\prime \prime }_{Z}$ and its 
companion $\mathcal{A}'_{Z}$, which also depends on the choice of a 
Kähler form $\omega^{X}$ on $X$. 

In Subsection \ref{subsec:enlsb}, we construct a Hermitian form 
$\epsilon_{X}$ along the fibers $\mathcal{X}$.

In Subsection \ref{subsec:sadj}, we compute the proper adjoint of 
$\mathcal{A}_{Z}$.

In Subsection \ref{subsec:coc}, we give a Lichnerowicz formula for 
the curvature
$\mathcal{A}_{Z}^{2}$.

In Subsection \ref{subsec:hypcurv}, we specialize our constructions 
with $z$ is the canonical section $y$, in which case 
$\mathcal{A}_{Z}$ is self-adjoint, and its curvature is fiberwise 
hypoelliptic.

In Subsection \ref{subsec:sca}, we study the behavior of 
$\mathcal{A}_{Y}$ under the scaling of $g^{\widehat{TX}}$.

Finally, in Subsection \ref{subsec:hypel}, we give the arguments 
which show that when replacing 
$g^{\widehat{TX}}$ by $b^{4}g^{\widehat{TX}}$, if 
$g^{TX}=g^{\widehat{TX}}$, as $b\to 0$, the hypoelliptic superconnection 
$\mathcal{A}_{Y,b}$ is a deformation of $A^{p_{*} \mathscr E_{0}}$.

In this Section, we make the same assumptions as in Section 
\ref{sec:suel}, and we use the corresponding notation. 
\subsection{The total space of $\widehat{TX}$}%
\label{subsec:tot}
Let $X$ be a compact complex manifold. Let
\index{X@$\mathcal{X}$}%
 $\pi:\mathcal{X}\to X$ be the total space of $TX$, the fiber being  denoted 
 \index{TX@$\widehat{TX}$}%
 $\widehat{TX}$
 to distinguish it from the 
 usual tangent bundle $TX$, so that $TX$ and $\widehat{TX}$ are 
 canonically isomorphic.  The conjugate bundle to 
$\widehat{TX}$ will be denoted   
$\overline{\widehat{TX}}$  and its dual  $\overline{\widehat{\TsX}}$. The real bundle associated with $\widehat{TX}$ is 
denoted by
$\widehat{T_{\R}X}$.

We have the exact sequences of holomorphic vector bundles on $\mathcal{X}$,
\begin{align}\label{eq:exa1b}
&0\to \pi^{*}\widehat{TX}\to T\mathcal{X}\to \pi^{*}TX\to 0,\\
&0\to \pi^{*}\TsX\to T^{*}\mathcal{X}\to 
\pi^{*}\widehat{T^{*}X}\to 0. \nonumber 
\end{align}

 Recall that the trivial antiholomorphic superconnection 
 \index{CX@ $\mathscr C^{X}$}%
 $\mathscr C^{X}$ was defined 
 in (\ref{eq:triva1}). We define $\mathscr C^{\mathcal{X}}$ on 
 $\mathcal{X}$ in the same way. Then
\begin{equation}\label{eq:triva1x1}
\mathscr C^{\mathcal{X}}=\pi^{*}_{b} \mathscr C^{X}.
\end{equation}
Put
\index{NX@$\mathscr N_{X}$}%
\begin{equation}\label{eq:triva2}
\mathscr N_{X}=\pi_{*}\mathscr C^{\mathcal{X}}.
\end{equation}
Then $\mathscr N_{X}$ is a 
$\Lambda\left(\overline{\TsX}\right)$-module. Let 
\index{NX@$N_{X}$}%
$N_{X}$ 
be the  infinite-dimensional vector bundle associated with $\mathscr N_{X}$. The operator 
$\overline{\pa}^{\mathcal{X}}$ can be viewed as an antiholomorphic 
superconnection $A^{N_{X} \prime \prime  }$ on 
$N_{X}$, so that $\mathscr N_{X}=\left(N_{X},A^{N_{X} \prime 
\prime }\right)$.

Put 
\index{I@$\mathbf{I}$}%
\begin{equation}\label{eq:defa1a}
\mathbf{I}=C^{\infty }\left(\widehat{TX}, 
\pi^{*}\Lambda\left(\overline{\widehat{\TsX}}\right)\right).
\end{equation}
Then $\mathbf{I}$  is the diagonal vector bundle 
associated with $N_{X}$. Also 
$\mathbf{I}^{0}$ is the 
vector bundle on $X$ of smooth complex functions along the fibre 
$\widehat{TX}$, and 
\begin{equation}\label{eq:bim1}
\mathbf{I}=\Lambda\left(\overline{\widehat{T^{*}X}}\right) \otimes 
\mathbf{I}^{0}.
\end{equation}
\subsection{A Hermitian metric on $\widehat{TX}$}%
\label{subsec:hemer}
Here, we follow \cite[Section 3.12]{Bismut06d} and \cite[Section 
6.1]{Bismut10b}. The constructions given in these references will 
permit us to construct a splittings of $\mathscr 
N_{X}$.

Let 
$g^{\widehat{TX}}$ be a Hermitian metric on $\widehat{TX}$, let 
\index{nTX@$\n^{\widehat{TX}}$}%
$\n^{\widehat{TX}}$ be the holomorphic Hermitian connection on 
$\left(\widehat{TX},g^{\widehat{TX}}\right)$, and let 
$R^{\widehat{TX}}$
be its curvature. Let $\n^{\Lambda\left(\overline{\widehat{\TsX}}\right)}$ 
be the induced connection on 
$\Lambda\left(\overline{\widehat{\TsX}}\right)$.

Let 
\index{THX@$T^{H}\mathcal{X}$}%
$T^{H}\mathcal{X} \subset  T\mathcal{X}$ 
be the horizontal subbundle  associated with the 
connection $\n^{\widehat{TX}}$.  We have the identification of smooth 
vector bundles on $\mathcal{X}$,
\begin{equation}\label{eq:club-1}
T^{H}\mathcal{X}=\pi^{*}TX,
\end{equation}
and also
identification of smooth vector bundles on 
$\mathcal{X}$,
\begin{equation}\label{eq:club0}
T\mathcal{X}=T^{H}\mathcal{X} \oplus \pi^{*}\widehat{TX}.
\end{equation}
By (\ref{eq:club-1}),  (\ref{eq:club0}), we get  the  smooth
identification,
\begin{equation}\label{eq:club-2}
T\mathcal{X}=\pi^{*} \left(TX \oplus  \widehat{TX} \right).
\end{equation}
By   (\ref{eq:club-2}), we get the  
identification of smooth vector bundles,
\begin{equation}\label{eq:club1}
\Lambda\left(\overline{T^{*}\mathcal{X}}\right) 
    =\pi^{*} \left( \Lambda\left(\overline{T^{*}X}\right) \ho
    \Lambda\left(\overline{\widehat{\TsX}}\right)\right).
\end{equation}
	
	By (\ref{eq:club1}), we get the identification of smooth vector 
	bundles on $X$, 
	\begin{equation}\label{eq:clib1}
N_{X}=\Lambda\left(\overline{\TsX}\right)\ho\mathbf{I}.
\end{equation}

By (\ref{eq:club1}), we have the identity 
\begin{equation}\label{eq:club2}
\Omega^{0,\scriptsize\bullet}\left(\mathcal{X},\C\right)=\Omega^{0,\scriptsize\bullet}\left(X,\mathbf{I}^{\Ou}\right).
\end{equation}

Let 
\index{Ic@$\mathbf{I}^{\mathrm{c}}$}%
$\mathbf{I}^{\mathrm{c}}$ denote the vector bundle of elements of 
$\mathbf{I}$ which have compact support. Then 
$\mathbf{I}^{\mathrm{c}}$ inherits a $L_{2}$ Hermitian metric 
$g^{\mathbf{I}^{\mathrm{c}}}$ from the metric $g^{\widehat{TX}}$, 
such that if $r,t\in\mathbf{I}^{\mathrm{c}}$, then
\begin{equation}\label{eq:harm1}
\left\langle  
r,t\right\rangle_{g^{\mathbf{I}^{\mathrm{c}}}}=\int_{\widehat{T_{\R}X}}^{}
\left\langle  
r,t\right\rangle_{g^{\Lambda\left(\overline{\widehat{\TsX}}\right)}}\frac{dY}{\left(2\pi\right)^{n}}.
\end{equation}

If $U\in T_{\R}X$, let 
	\index{UH@$U^{H}$}%
	$U^{H}\in T_{\R}^{H}\mathcal{X}$ denote 
	the horizontal lift of $U$.

If $U\in T_{\R}X$,  if $s$ is a 
smooth section of $\mathbf{I}$ on $X$, set
\begin{equation}
    \n^{\mathbf{I}}_{U}s=\n^{\Lambda\left(\overline{\widehat{\TsX}}\right)}_{U^{H}}s.
    \label{eq:3.10az1}
\end{equation}
Then $\n^{\mathbf{I}}$ induces a unitary connection $\n^{\mathbf{I}^{\mathrm{c}}}$ on 
$\mathbf{I}^{\mathrm{c}}$ that preserves its $\Z$-grading. Let 
$\n^{\mathbf{I}^{0}}$ be the restriction of $\n^{\mathbf{I}}$ to 
$\mathbf{I}^{0}$. By 
(\ref{eq:bim1}), the connection $\n^{\mathbf{I}}$ is just the 
connection on $\mathbf{I}$ that is induced by 
$\n^{\Lambda\left(\overline{\widehat{\TsX}}\right)},\n^{\mathbf{I}^{0}}$.

Let 
\index{dV@$\overline{\pa}^{V}$}%
$\overline{\pa}^{V}$ denote the $\overline{\pa}$-operator acting 
along the fibers $\widehat{TX}$ of $\pi$. Using 
(\ref{eq:club2}), we 
see that $\n^{\mathbf{I} \prime \prime }$ and $\overline{\pa}^{V}$ 
act on $\Omega^{0,\scriptsize\bullet}\left(\mathcal{X},\C\right)$.

Let
\index{y@$\widehat{y}$}%
$\widehat{y}$ be the tautological section of 
$\pi^{*}\widehat{TX}$ on $\mathcal{X}$, let $\overline{\widehat{y}}$ 
be
the conjugate section of $\pi^{*}\overline{\widehat{TX}}$, and let 
$\widehat{Y}=\widehat{y}+\overline{\widehat{y}}$ be the tautological 
section of $\pi^{*}\widehat{T_{\R}X}$, so that 
$\left\vert  \widehat{Y}\right\vert^{2}_{g^{\widehat{T_{\R}X}}}=2\left\vert  
\widehat{y}\right\vert^{2}_{g^{\widehat{TX}}}$.

Let $w_{1},\ldots,w_{n}$ be a basis of $TX$, let $w^{1},\ldots,w^{n}$ 
be
the corresponding dual basis of $T^{*}X$. We denote with a hat the 
corresponding objects associated with $\widehat{TX}, 
\widehat{T^{*}X}$. 
 We have   results established in \cite[Theorem 
2.8]{BismutGilletSoule88b}, \cite[Propositions 6.1.2, 6.1.4, and 6.1.6]{Bismut10b}.
\begin{proposition}\label{prop:Pcurvni}
The following identities hold:
\begin{align}\label{eq:curvi}
&\n^{\mathbf{I}\prime \prime, 2}=0,\ 
    \n^{\mathbf{I} \prime,2}=0,\\
    &\n^{\mathbf{I},2}=-\n^{V}_{R^{\widehat{TX}}\widehat{Y}}-
\left\langle  
R^{\widehat{TX}}\overline{\widehat{w}}_{i},\overline{\widehat{w}}^{j}\right\rangle\overline{\widehat{w}}^{i}
i_{\overline{\widehat{w}}_{j}}. \nonumber 
\end{align}
Also
\begin{align}\label{eq:coop1}
&\left[\n^{\mathbf{I}},\overline{\pa}^{V}\right]=0,
&\left[\n^{\mathbf{I}},\overline{\pa}^{V*}\right]=0.
\end{align}
Finally, we have the identity of operators acting on 
$\Omega^{0,\Ou}\left(\mathcal{X},\C\right)$,
\begin{equation}\label{eq:coop1a1}
\overline{\partial}^{\mathcal{X}}=\n^{\mathbf{I}
\prime \prime }+\overline{\partial}^{V}.
\end{equation}
\end{proposition}
\begin{remark}\label{rem:inte}
	The decomposition (\ref{eq:coop1a1}) can be viewed as a special 
	case of equation (\ref{eq:iv36a6}) when $\mathscr E$ is replaced 
	by $\mathscr N_{X}$. In (\ref{eq:coop1a1}), $\overline{\pa}^{\mathcal{X}}$ does not 
	depend on $g^{\widehat{TX}}$, but the splitting in the right-hand 
	depends on $g^{\widehat{TX}}$.
\end{remark}

By the above, it follows that $\left(\mathbf{I}^{\mathrm{c}}, 
\n^{\mathbf{I}^{\mathrm{c}}
\prime \prime }\right)$ is a holomorphic Hermitian  vector bundle\footnote{Here, `holomorphic vector bundle' is taken in the naive 
sense.}. The holomorphic structure depends explicitly on the metric 
$g^{\widehat{TX}}$.
\subsection{The antiholomorphic superconnection $\mathscr 
E_{M}$}%
\label{subsec:lift}
We make the same assumptions as in Subsections \ref{subsec:geo},
\ref{subsec:rebl}, and \ref{subsec:tot}. Put
\begin{equation}\label{eq:lav1}
\mathcal{M}=S\times \mathcal{X}.
\end{equation}
Let $\underline{\pi}$ be the projection $\mathcal{M}\to M$, let 
$\underline{p},\underline{q}$ be the obvious maps $\mathcal{M}\to 
S,\mathcal{M}\to X$. 

As in (\ref{eq:triva1x1}), we have
\begin{equation}\label{eq:trova1}
\mathscr C^{\mathcal{M}}=\underline{\pi}^{*}_{b} \mathscr C^{M}.
\end{equation}
Put
\index{NM@$\mathscr N_{M}$}%
\begin{equation}\label{eq:stac2}
\mathscr N_{M}=\underline{\pi}_{*} \mathscr C^{\mathcal{M}}.
\end{equation}
Then
\begin{equation}\label{eq:stac1}
\mathscr N_{M}=\underline{q}^{*}_{b}\mathscr N_{X}.
\end{equation}
Let $N_{M}$ be the infinite-dimensional vector bundle on $M$ that is associated with $\mathscr 
N_{M}$.

The associated diagonal vector bundle is just $q^{*}\mathbf{I}$.

Let $\mathscr E=\left(E,A^{E \prime \prime }\right)$ be an 
antiholomorphic superconnection on $M$. Put
\index{EM@$\mathscr E^{\mathcal{M}}$}%
\begin{equation}\label{eq:pupa1}
\mathscr E^{\mathcal{M}}=\underline{\pi}^{*}_{b} \mathscr E.
\end{equation}
Then $\mathscr E^{\mathcal{M}}=\left(E^{\mathcal{M}}, 
A^{E^{\mathcal{M}} \prime \prime }\right)$ is an antiholomorphic 
superconnection on $\mathcal{M}$, and the associated diagonal bundle 
is just $\underline{\pi}^{*}D$. Also $\mathscr E^{\mathcal{M}}$ is 
a $\Lambda\left(\overline{T^{*}\mathcal{M}}\right)$-module.

Put
\index{EM@$ \mathscr E_{M}$}%
\begin{equation}\label{eq:pupa2}
\mathscr E_{M}=\underline{\pi}_{*} \mathscr E^{\mathcal{M}}.
\end{equation}
Recall that the tensor product 
\index{ob@$\ho_{b}$}%
$\ho_{b}$ was defined in 
(\ref{eq:redten1}). Then
\begin{equation}\label{eq:pupa2b1}
\mathscr E_{M}=\mathscr E\ho_{b} \mathscr N_{M}.
\end{equation}
Observe that 
\index{EM@$E_{M}$}%
$\mathscr E_{M}=\left(E_{M},A^{E_{M} \prime \prime }\right)$ is 
an 
antiholomorphic superconnection on $M$, with an infinite-dimensional 
$E_{M}$, and the associated diagonal bundle is $D\ho q^{*}\mathbf{I}$. 
\subsection{A metric description of $A^{E_{M} \prime \prime 
}$}%
\label{subsec:metde}
We fix a splitting of $E$ as in (\ref{eq:iv4a-1}), (\ref{eq:iv4}), so 
that $E \simeq E_{0}$. 

We proceed as in Subsections 
\ref{subsec:hemer} and \ref{subsec:lift}. We fix a metric $g^{\widehat{TX}}$ on the vector bundle  
$\widehat{TX}$ over $X$.   This induces a 
corresponding splitting of $\mathscr N_{X}$, and also of 
$\mathscr N_{M}$. 

The corresponding vector bundles $E_{0}, N_{M,0}$ are given by 
\begin{align}\label{eq:cloc1}
&E_{0}=\Lambda\left(\overline{T^{*}M}\right)\ho D,
&N_{M,0}=\Lambda\left(\overline{T^{*}M}\right)\ho q^{*}\mathbf{I}.
\end{align}

The above  two splittings induce a corresponding splitting of 
$\mathscr E_{M}$, so that
\begin{equation}\label{eq:spla1}
E_{M,0} =\Lambda\left(\overline{T^{*}M}\right)\ho D\ho q^{*}\mathbf{I}.
\end{equation}
Equation (\ref{eq:spla1}) can also be written in the form
\begin{equation}\label{eq:spla2}
E_{M,0}=E_{0}\ho q^{*}\mathbf{I}.
\end{equation}

\begin{definition}\label{def:app}
	Let 
	\index{A@$\mathcal{A}''$}%
	$\mathcal{A}''$ be the antiholomorphic 
	superconnection on 
	$E_{M,0}$ that corresponds to $A^{E_{M}\prime \prime 
	}$.
\end{definition}

We write $A^{E_{0} \prime \prime }$ as in (\ref{eq:baba2}). In 
particular, $D$ is equipped with an antiholomorphic connection $\n^{D 
\prime \prime }$. 

As we saw in Subsection \ref{subsec:hemer}, $\mathbf{I}$ is equipped 
with a holomorphic structure $\n^{\mathbf{I} \prime \prime}$.
\begin{definition}\label{def:nDI}
	 Let
\index{nDI@$\n^{D\ho q^{*}\mathbf{I} \prime \prime }$}%
$\n^{D\ho q^{*}\mathbf{I} \prime \prime }$ be the antiholomorphic 
connection on $D\ho q^{*}\mathbf{I}$ associated with $\n^{D \prime 
\prime },\n^{\mathbf{I} \prime \prime}$.
\end{definition}
\begin{proposition}\label{prop:pideA}
	The following identity holds:
	\begin{equation}\label{eq:ama1}
\mathcal{A}''=\n^{D\ho q^{*}\mathbf{I} \prime \prime }+B+\overline{\pa}^{V}.
\end{equation}
\end{proposition}
\begin{proof}
	Our proposition follows from equations (\ref{eq:baba2}), 
	(\ref{eq:coop1a1}).
\end{proof}
\subsection{The adjoint of $\mathcal{A}''$}%
\label{subsec:enls}
Let $g^{D}$ be a Hermitian metric on $D$. Recall that 
$g^{\widehat{TX}}$ induces a Hermitian metric on 
$\mathbf{I}^{\mathrm{c}}$. Therefore $D\ho 
q^{*}\mathbf{I}^{\mathrm{c}}$ is equipped with a Hermitian metric $g^{D\ho 
q^{*}\mathbf{I}^{\mathrm{c}}}$.

We are now in a situation formally similar to what was done in a 
finite-dimensional context in Section \ref{sec:ascngm}, and  in an 
infinite-dimensional context in  \cite{BismutGilletSoule88b}, in 
\cite[Chapter 4]{Bismut10b}, and in Section \ref{sec:suel}, with 
$M,S$ replaced by $\mathcal{M},M$, and  $p$ replaced by $\underline{\pi}$. We 
briefly explain the construction of the adjoint of $\mathcal{A}''$.

Let $s,s'\in 
\Omega\left(M,D\ho q^{*}\mathbf{I}^{\mathrm{c}}\right)$. We can write $s,s'$ 
in the form
\begin{align}\label{eq:iv14a1bi}
&s=\sum_{}^{}\alpha_{i}r_{i},
&s'=\sum_{}^{}\beta_{j}t_{j},
\end{align}
with $\alpha_{i},\beta_{j}\in \Omega\left(M,\C\right)$ and 
$r_{i},t_{j}\in C^{\infty}\left(M,D\ho 
q^{*}\mathbf{I}^{\mathrm{c}}\right)$. As in (\ref{eq:iv14a2}),
(\ref{eq:iv14a3}), put
\begin{equation}\label{eq:bim2}
\theta_{g^{D\ho 
q^{*}\mathbf{I}^{\mathrm{c}}}}\left(s,s'\right)=\frac{i^{m}}{\left(2\pi\right)^{m}}\sum_{}^{}\int_{M}^{} 
\widetilde \alpha_{i}\we\overline{\beta}_{j}\left\langle  
r_{i},t_{j}\right\rangle_{g^{D\ho q^{*}\mathbf{I}^{\mathrm{c}}}}.
\end{equation}
\begin{definition}\label{def:adscn}
Let 
\index{Ap@$\mathcal{A}'$}%
$\mathcal{A}'$ be the formal adjoint of $\mathcal{A}''$ with 
respect to $\theta_{g^{D\ho q^{*}\mathbf{I}^{\mathrm{c}}}}$.	
\end{definition}

Then $\mathcal{A}'$ is exactly the adjoint of the antiholomorphic superconnection 
$\mathcal{A}^{\prime \prime }$ in the sense of Definition 
\ref{def:adjo} with respect to the pure metric 
$g^{D\ho q^{*}\mathbf{I}^{\mathrm{c}}}$.

Recall that the holomorphic connection $\n^{D \prime }$ on $D$ was 
defined in (\ref{eq:coa1}). Since $g^{D}$ is a pure metric in 
$\mathscr M^{D}$, $\n^{D \prime}$ is just a standard holomorphic 
connection. 
\begin{definition}\label{def:nhol}
	Let 
	\index{nDi@ $\n^{D\ho q^{*}\mathbf{I} \prime}$}%
	$\n^{D\ho q^{*}\mathbf{I} \prime}$ be the holomorphic 
	connection on $D\ho q^{*}\mathbf{I}$ induced by $\n^{D \prime 
	},\n^{\mathbf{I} \prime }$.
\end{definition}

We use the notation in (\ref{eq:baba4}). Let $\overline{\pa}^{V*}$ 
denote the formal adjoint of $\overline{\pa}^{V}$ with respect to the 
metric $g^{\mathbf{I}^{\mathrm{c}}}$.
\begin{proposition}\label{prop:pap}
	The following identity holds:
	\begin{equation}\label{eq:pap1}
\mathcal{A}'=\n^{D\ho q^{*}\mathbf{I} \prime }+B^{*}+\overline{\pa}^{V*}.
\end{equation}
\end{proposition}
 \begin{proof}
 	This is a consequence of Proposition \ref{prop:pideA}.
 \end{proof}
 
 Then $\mathcal{A}^{\prime \prime },\mathcal{A}'$ act as differential operators on 
$\Omega\left(M,D\ho q^{*}\mathbf{I} \right)$, and 
\begin{align}\label{eq:nota3}
&\mathcal{A}^{\prime \prime 2}=0,&\mathcal{A}^{\prime 2}=0.
\end{align}
\subsection{The antiholomorphic superconnection $A^{E_{M} \prime \prime }_{z}$}%
 \label{subsec:amz}
Recall that $A^{E^{\mathcal{M}} \prime \prime }, A^{E_{M} \prime 
\prime }$ are the same operator, that  acts  on $C^{\infty 
}\left(\mathcal{M},E^{\mathcal{M}}\right)=C^{\infty }\left(M, 
E_{M}\right)$. We  extend these operators to operators acting on 
$C^{\infty 
}\left(\mathcal{M},\underline{\pi}^{*}\Lambda\left(T^{*}M\right)\ho 
E^{\mathcal{M}}\right)=C^{\infty 
}\left(M,\Lambda\left(T^{*}M\right)\ho E_{M}\right)$, that still 
verify Leibniz rule with respect to multiplication by smooth forms 
in $\underline{\pi}^{*}\Lambda\left(T^{*}M\right), 
\Lambda\left(T^{*}M\right)$.

Let $z$ be a holomorphic section of $\underline{q}^{*}TX$ on 
$\mathcal{M}$. 
Since $q^{*}TX \subset TM$, the contraction operator $i_{z}$ acts on the above 
vector spaces.
\begin{definition}\label{def:Az}
	Put
	\index{AEMz@$A^{E^{\mathcal{M}}\prime \prime}_{z}$}%
	\index{AEMz@$A^{E_{M} \prime \prime}_{z}$}%
	\begin{align}
&A^{E^{\mathcal{M}}\prime \prime }_{z}=A^{E^{\mathcal{M}}\prime \prime }+i_{z},
&A^{E_{M} \prime \prime }_{z}=A^{E_{M} \prime \prime }+i_{z}.
\end{align}
\end{definition}

Both operators have total degree $+1$.   Since $z$ is holomorphic, we get
\begin{align}\label{eq:pap2}
&A^{E^{\mathcal{M}} \prime \prime,2}_{z}=0,
&A^{E_{M} \prime \prime,2}_{z}=0.
\end{align}
Also $A^{E^{\mathcal{M}} \prime \prime}_{z}$ is an antiholomorphic 
superconnection on $\underline{\pi}^{*}\Lambda\left(T^{*}M\right)\ho 
\mathscr E^{\mathcal{M}}$, and $A^{E_{M} \prime \prime}_{z}$ is an 
antiholomorphic superconnection on $\Lambda\left(T^{*}M\right)\ho 
\mathscr E_{M}$. These two operators are in fact the same operator. 
In the sequel, we will deal mostly with the second one.

\subsection{The superconnection $\mathcal{A}_{Z}$}%
\label{subsec:adaz}
We now make the same assumptions as in Subsections \ref{subsec:metde} 
and \ref{subsec:amz}.

As in Subsection \ref{subsec:adj}, let 
\index{gTX@$g^{TX}$}%
$g^{TX}$ be a Hermitian metric 
on $TX$ with Kähler form 
\index{oX@$\omega^{X}$}%
$\omega^{X}$.  The exponential  
$e^{-i\omega^{X}}$ is taken in the exterior algebra 
$\Lambda\left(T^{*}_{\C}X\right)$. 

Let $z$ denote a holomorphic section of $\underline{q}^{*}TX$ on 
$\mathcal{M}$.  Put
\index{Z@$Z$}%
\begin{equation}\label{eq:sec1}
Z=z+\overline{z}.
\end{equation}
Then $Z$ is a smooth section of $\underline{q}^{*}T_{\R}X$.
Let $\overline{z}_{*}\in \underline{q}^{*}\TsX$ correspond to 
$\overline{z}\in \underline{q}^{*}\overline{TX}$ by the metric 
$g^{TX}$. By (\ref{eq:didim2}), $\overline{z}_{*}$ can also be viewed 
as a section of $T^{*}M$.

Note that $TX \subset TM$, and so  the contraction operators  
$i_{z},i_{\overline{z}}$ act as operators of degree 
$1,-1$ on 
$\Lambda\left(T^{*}_{\C}M\right)$, and so  they  act on 
$\Omega\left(M, D\ho q^{*}\mathbf{I}\right)$. 

Now we follow \cite[Section 6.3]{Bismut10b}.
\begin{definition}\label{def:bln}
	Put
\index{AZ@$\mathcal{A}^{\prime \prime }_{Z}$}%
\index{AZ@$\mathcal{A}'_{Z}$}%
\index{AZ@$\mathcal{A}_{Z}$}%
	\begin{align}\label{eq:coro4a2}
&\mathcal{A}^{\prime \prime }_{Z}=\mathcal{A}^{\prime \prime }+i_{z},
&\mathcal{A}'_{Z}=e^{q^{*}i\omega^{X}}\left( \mathcal{A}'+i_{\overline{z}} \right)
e^{-q^{*}i\omega^{X}},\quad \mathcal{A}_{Z}=\mathcal{A}^{\prime \prime 
}_{Z}+\mathcal{A}'_{Z}.
\end{align}
\end{definition}
Since $z$ is holomorphic, 
\begin{align}\label{eq:coro8}
&\mathcal{A}_{Z}^{\prime \prime 2}=0,&\mathcal{A}^{\prime 2}_{Z}=0.
\end{align}
By (\ref{eq:coro4a2}), we get
\begin{equation}\label{eq:coro9}
\mathcal{A}'_{Z}=\mathcal{A}' 
-q^{*}\pa^{X}i\omega^{X}+i_{\overline{z}}+\overline{z}_{*}\we.
\end{equation}
Also
\begin{equation}\label{eq:corobeb1}
\mathcal{A}^{2}_{Z}=\left[\mathcal{A}''_{Z},\mathcal{A}'_{Z}\right].
\end{equation}

Let $\n^{D\ho q^{*}\mathbf{I}}$ be the connection on $D\ho q^{*}\mathbf{I}$ which 
is induced by $\n^{D},\n^{\mathbf{I}}$. Recall that 
\index{C@$C$}%
$C$ was defined 
in (\ref{eq:dep1}).  By  (\ref{eq:ama1}), (\ref{eq:pap1}),   (\ref{eq:coro4a2}), and 
(\ref{eq:coro9}), we get
\begin{align}\label{eq:coro9a1}
&\mathcal{A}^{\prime \prime }_{Z}=\n^{D\ho q^{*}\mathbf{I} \prime \prime 
}+B+\overline{\pa}^{V}+i_{z}, \nonumber \\
&\mathcal{A}'_{Z}=\n^{D\ho q^{*}\mathbf{I} \prime }+B^{*}
 -q^{*}\pa^{X}i\omega^{X}+\overline{\pa}^{V*}+i_{\overline{z}}+\overline{z}_{*}\we,\\
&\mathcal{A}_{Z}=\n^{D\ho q^{*}\mathbf{I}}+C-q^{*}\pa^{X}i\omega^{X}+\overline{\pa}^{V}+i_{z}+\overline{\pa}^{V*}+i_{\overline{z}}+\overline{z}_{*}\we. \nonumber 
\end{align}

We denote by $L$ the operator of multiplication by $\omega^{X}$ and 
by $\Lambda$ its adjoint with respect to $g^{TX}$. Let 
$w_{1}\ldots,w_{n}$ be an orthonormal basis of $\left(TX,g^{TX}\right)$. Then
\begin{align}\label{eq:ret2}
&L=-iw^{i}\we \overline{w}^{i}, &\Lambda=ii_{\overline{w}_{i}}i_{w_{i}}.
\end{align}
Conjugation by $\exp\left(i\Lambda\right)$ 
changes $w^{i},\overline{w}^{i}$ into $w^{i}-i_{\overline{w}_{i}}, 
\overline{w}^{i}+i_{w_{i}}$.

Put
\index{BZ@$\mathcal{B}^{ \prime \prime }_{Z}$}%
\index{BZ@$\mathcal{B}'_{Z}$}%
\index{BZ@$\mathcal{B}_{Z}$}%
\begin{align} \label{eq:conjugrep}
&\mathcal{B}^{ \prime \prime }_{Z}=\exp\left(i\Lambda\right)\mathcal{A}''_{Z}\exp\left(-i\Lambda\right), \nonumber \\
&\mathcal{B}'_{Z}=\exp\left(i\Lambda\right)\mathcal{A}'_{Z}\exp\left(-i\Lambda\right),\\
&\mathcal{B}_{Z}=\exp\left(i\Lambda\right)\mathcal{A}_{Z}\exp\left(-i\Lambda\right). \nonumber 
\end{align}
Then
\begin{equation}\label{eq:conjugrepx1}
\mathcal{B}_{Z}=\mathcal{B}''_{Z}+\mathcal{B}'_{Z}.
\end{equation}

Set
\index{E@$\mathcal{E}$}%
\begin{equation}
    \mathcal{E}=\exp\left(i\Lambda\right)\left(\n^{D\ho q^{*}\mathbf{I}}+C-q^{*}\pa^{X}i\omega^{X}\right)\exp\left(-i
    \Lambda\right).
    \label{eq:club45}
\end{equation}
By (\ref{eq:coro9a1}), and by the considerations that follow 
(\ref{eq:ret2}),  we get
\begin{equation}
    \mathcal{B}_{Z}=\mathcal{E}+\overline{\pa}^{V}+i_{z}+\overline{\pa}^{V*}+\overline{z}_{*}.
    \label{eq:club46}
\end{equation}

We still count the degree in $\Lambda\left(T^{*}_{\C}M\right)$ as the 
difference of the antiholomorphic and of the holomorphic degree. Then $A^{\prime \prime }_{Z},B''_{Z}$ are 
operators of total degree $1$, and $A'_{Z},B'_{Z}$ are operators of 
total degree $-1$.

Observe that  the above operators can be 
considered as standard differential operators on $\mathcal{M}$. 
 \subsection{Two Hermitian forms}%
\label{subsec:enlsb}
Here we follow \cite[Sections 6.5 and 6.6]{Bismut10b}. We will 
slightly modify the construction of the Hermitian form 
$\theta_{g^{D\ho q^{*}\mathbf{I}}}$.

Set
\index{F@$\mathbb F$}%
\begin{equation}
    \mathbb F=q^{*}\left( \Lambda\left(T^{*}_{\C}X
\right)\ho\Lambda\left(\overline{\widehat{\TsX}}\right) \right)  \ho D.
    \label{eq:club19}
\end{equation}
Then $\mathbb F$ is a vector bundle on $M$. Recall that $\Lambda\left(T^{*}_{\C}X\right)$ is $\Z$-graded by the 
difference of the antiholomorphic  and the holomorphic degrees. Also 
$\Lambda\left(\overline{\widehat{\TsX}}\right)$ and $D$ are 
$\Z$-graded.  Then $\mathbb F$ inherits a corresponding $\Z$-grading. 
Similarly $\Lambda\left(T^{*}_{\C}S\right)$ inherits a similar 
grading.

As in \cite[Section 6.5]{Bismut10b}, we will use  the 
identification
\begin{equation}\label{eq:cret1}
\Lambda\left(\widehat{T^{*}_{\C}X}\right)=\Lambda\left(\TsX\right)
\ho\Lambda\left(\overline{\widehat{\TsX}}\right).
\end{equation}
By (\ref{eq:club19}), (\ref{eq:cret1}), we get
\begin{equation}\label{eq:cret1a1}
\mathbb F=q^{*}\left(\Lambda\left(\overline{T^{*}X}\right)\ho \Lambda\left(\widehat{T^{*}_{\C}X}\right) 
\right)\ho D.
\end{equation}

Let
\index{s@$\sigma$}%
$\sigma$ be the involution of $\mathcal{X}$ given by 
$\left( x,\widehat{y}\right) \to \left(  x, -\widehat{y}\right) $. Then $\sigma$ induces an involution of 
$\mathcal{M}$. 

Let
\index{s@$\sigma^{*}$}%
$\sigma^{*}$ denote the obvious action of $\sigma$ on 
$\Lambda\left(\widehat{T^{*}_{\C}X}\right)$. Namely, if $N^{\Lambda\left(\widehat{T^{*}_{\C}X}\right)}$
is the number operator of 
$\Lambda\left(\widehat{T^{*}_{\C}X}\right) $, then $\sigma^{*}$ acts like 
$\left(-1\right)^{N^{\Lambda\left(\widehat{T^{*}_{\C}X}\right)}}$.
Also, we make $\sigma^{*}$  act trivially  on   $D$.  Then $\sigma^{*}$ 
acts on $\mathbb F$ in (\ref{eq:cret1a1}).

We fix one fiber $X$ of the projection  $p$, and the 
corresponding fiber $\mathcal{X}$ of 
$\underline{p}$. Note that
\begin{equation}
    \Omega\left(X, D\vert_{X}\ho \mathbf{I}\right)=C^{\infty }\left(\mathcal{X},\underline{\pi}^{*}\mathbb 
    F\vert_{\mathcal{X}}\right).
    \label{eq:club21}
\end{equation}

Let 
\index{dvX@$dv_{\mathcal{X}}$}%
$dv_{\mathcal{X}}$ denote the volume form on $\mathcal{X}$ which 
is associated with the metrics $g^{TX}, g^{\widehat{TX}}$. If 
$s,s'\in  C^{\infty,\mathrm{c} }\left(\mathcal{X},\underline{\pi}^{*}\mathbb 
    F\vert_{\mathcal{X}}\right)$, set
	\begin{equation}\label{eq:crip1}
\left\langle  
s,s'\right\rangle_{L_{2}}=\left(2\pi\right)^{-2n}\int_{\mathcal{X}}^{}\left\langle  s,s'\right\rangle_{g^{\mathbb F}}dv_{\mathcal{X}}.
\end{equation}
\begin{definition}\label{def:other}
	If $s,s'\in C^{\infty,\mathrm{c} }\left(\mathcal{X},\underline{\pi}^{*}\mathbb 
	F\vert_{\mathcal{X}}\right)$, put 
\index{etaX@$\eta_{X}$}%
\begin{equation}\label{eq:cret2}
\eta_{X}\left(s,s'\right)=\left\langle  
\sigma^{*}s,s'\right\rangle_{L_{2}}.
\end{equation}
\end{definition}
Then $\eta_{X}$  is a Hermitian form on $C^{\infty,\mathrm{c} }\left(\mathcal{X},\underline{\pi}^{*}\mathbb 
	F\vert_{\mathcal{X}}\right)$. Note that $\eta_{X}$ is non-degenerate, 
	it is positive on 
	$\sigma$-invariant forms, and negative on $\sigma$-anti-invariant 
	forms. Also the $C^{\infty,\mathrm{c} }\left(\mathcal{X},\underline{\pi}^{*}\mathbb 
	F^{i}\vert_{\mathcal{X}}\right)$ are mutually orthogonal with respect to 
	$\left\langle  \,\right\rangle_{L_{2}}$ and $\eta_{X}$.
	
Let 
\index{s@$\underline{\sigma}^{*}$}%
$\underline{\sigma}^{*}$ be the  restriction of 
$\sigma^{*}$ to $\Lambda\left(\overline{\widehat{T^{*}X}}\right)$. 
We make $\underline{\sigma}^{*}$ act trivially on 
$\Lambda\left(T^{*}_{\C}X\right)$ and on $D$. By (\ref{eq:club19}),  $\underline{\sigma}^{*}$ also 
acts on $\mathbb F$. The  difference with $\sigma^{*}$ is that 
$\underline{\sigma}^{*}$ acts trivially on $\Lambda\left(\TsX\right)$.
Then $\underline{\sigma}^{*}$ acts on $D\vert_{X}\ho \mathbf{I}$. Let $s,s'\in 
\Omega\left(X,D\vert_{X}\ho\mathbf{I}^{\mathrm{c}}\right)$. Then 
$s,s'$ can be written in the form
\begin{align}\label{eq:bim3}
&s=\sum_{}^{}\alpha_{i}r_{i},&s'=\sum_{}^{}\beta_{j}t_{j},
\end{align}
with $\alpha_{i},\beta_{j}\in\Omega\left(X,\C\right)$, and 
$r_{i},t_{j}\in C^{\infty 
}\left(X,D\vert_{X}\ho\mathbf{I}^{\mathrm{c}}\right)$.

If $\alpha\in \Lambda\left(T^{*}_{\C}X\right)$, we define 
$\widetilde \alpha$ as in (\ref{eq:bc13ax1}).  Also we use the same 
notation as in Subsection \ref{subsec:triex}. As we observed in that 
Subsection, multiplication by $i\omega^{X}$ is a self-adjoint 
operator with respect to the form $\theta$ in (\ref{eq:iv13}). Also the 
exponential $e^{-i\omega^{X}}$ is taken in 
$\Lambda\left(T^{*}_{\C}X\right)$.
\begin{definition}\label{def:goug}
Put
\index{eX@$\epsilon_{X}$}%
\begin{equation}\label{eq:retpek0}
\epsilon_{X}\left(s,s'\right)
=\frac{i^{n}}{\left(2\pi\right)^{n}}\int_{X}^{} 
\widetilde\alpha_{i}\we \overline{ e^{-i\omega^{X}}\beta}_{j}
\left\langle  \underline{\sigma}^{*}r,t\right\rangle_{g^{D\vert_{X}\ho 
\mathbf{I}^{\mathrm{c}}}}.
\end{equation}
\end{definition}

Equation (\ref{eq:retpek0}) can be rewritten in the form,
\begin{equation}\label{eq:retpek0r1}
\epsilon_{X}\left(s,s'\right)
=\frac{i^{n}}{\left(2\pi\right)^{n}}\int_{\mathcal{X}}^{} 
\widetilde\alpha_{i}\we \overline{ e^{-i\omega^{X}}\beta}_{j}
\left\langle  \underline{\sigma}^{*}r,t\right\rangle_{g^{D\vert_{X}\ho 
\Lambda\left(\overline{\widehat{\TsX}}\right)}}\frac{dY}{\left(2\pi\right)^{n}}.
\end{equation}
By \cite[Section 6.5 and Theorem 6.6.1]{Bismut10b}, 
$\epsilon_{X}$ 
is  a Hermitian form on $C^{\infty,\mathrm{c} 
}\left(\mathcal{X},\underline{\pi}^{*}\mathbb F\vert_{\mathcal{X}}\right)=\Omega\left(X, 
\mathbf{I}^{\mathrm{c}}\ho 
D\vert_{X}\right)$. As explained in \cite[Section 6.6]{Bismut10b},
if instead, we replace integration over $X$ by integration over $M$, 
while replacing $\frac{i^{n}}{\left(2\pi\right)^{n}}$ by 
$\frac{i^{m}}{\left(2\pi\right)^{m}}$, we obtain a  
Hermitian form
\index{eM@$\epsilon_{M}$}%
$\epsilon_{M}$ on $\Omega\left(M,D\ho 
q^{*}\mathbf{I}^{\mathrm{c}}\right)$.
\begin{remark}\label{rem:genme}
	As explained in Subsection \ref{subsec:triex},  the term 
$e^{-i\omega^{X}} $ in (\ref{eq:retpek0}) can be interpreted as a 
generalized metric on the trivial line $\C$ over $M$. Equivalently, 
$e^{-i\omega^{X}}g^{D}\vert_{X}$ can be viewed as a generalized 
metric on $D\vert_{X}$.
\end{remark}

By \cite[eq. (6.5.10) and Theorem 6.6.1]{Bismut10b}, if 
$$s,s'\in 
\Omega\left(X,D\vert_{X}\ho q^{*}\mathbf{I}^{\mathrm{c}}\right)=C^{\infty 
}\left(\mathcal{X},\underline{\pi}^{*}\mathbb 
F\vert_{\mathcal{X}}\right),$$
we get
\begin{equation}\label{eq:cret3}
\epsilon_{X}\left(s,s'\right)=\eta_{X}\left(e^{i\Lambda}s, 
e^{i\Lambda}s'\right).
\end{equation}
Then  $\epsilon_{X}$ is also a non-degenerate Hermitian form, 
which can also be proved directly.
By (\ref{eq:cret3}), or by direct inspection, the 
$C^{\infty,\mathrm{c} }\left(\mathcal{X},\underline{\pi}^{*}\mathbb 
F^{i}\vert _{\mathcal{X}}\right)$ are  mutually orthogonal with respect to 
$\epsilon_{X}$. 

We will denote with an upper $\dag$ the adjoint of an operator with 
respect to $\epsilon_{X}$ or $\eta_{X}$, the corresponding Hermitian form 
being explicitly mentioned.

We have the following result established in  \cite[Proposition 6.5.1 
and Theorem 6.6.1]{Bismut10b}.
\begin{proposition}\label{prop:Padje}
 If adjoints are taken with respect to 
$\epsilon_{X}$, if $e\in TX, f\in T^{*}X$, then
\begin{align}\label{eq:club32x1}
    &i_{e}^{\dag}=-\overline{e}_{*}\we -i_{\overline{e}},
    &i_{\overline{e}}^{\dag}=e_{*}\we -i_{e},\\
&\left(f\we\right)^{\dag}=-\overline{f}\we 
,&\left(\overline{f}\we\right) ^{\dag}=-f \we. \nonumber 
\end{align}
If $e\in TX,E=e+\overline{e}\in T_{\R}X$, then 
$e_{*}\we-i_{E},\overline{e}_{*}\we+i_{E}$ are skew-adjoint with respect to $\epsilon_{X}$.
\end{proposition}
\subsection{Self-adjointness  of $\mathcal{A}_{Z}, 
\mathcal{B}_{Z}$}%
\label{subsec:sadj}
Note that
$p_{*}\left(\Lambda\left(T^{*}M\right)\ho E_{M}\right)$
is a $\Lambda\left(T_{\C}^{*}S\right)$-module.

We will  use the conventions of 
Subsection \ref{subsec:adj} when taking adjoints of antiholomorphic 
superconnections over the base $S$.  These adjoints will be taken with respect to the Hermitian forms 
$\epsilon_{X}$ or $\eta_{X}$, and will be denoted 
with a ${}^{\dag}$.  The adjoints of $\mathcal{A}_{Z}$ and of its 
components $\mathcal{A}^{ \prime \prime }_{Z},\mathcal{A}'_{Z}$  are 
taken with respect to $\epsilon_{X}$, and the adjoints of 
$\mathcal{B}_{Z}$ and of its components $\mathcal{B}^{\prime \prime 
}_{Z}, \mathcal{B}'_{Z}$ are taken with respect to 
$\eta_{X}$. In the case of $\mathcal{A}_{Z}$ or of its components, the 
adjoint could be taken as well with respect to the standard Hermitian 
form $\epsilon_{M}$.

Let 
\index{z@$z_{-}$}%
$z_{-}$ be the section of $\underline{q}^{*}TX$ given by
\begin{equation}\label{eq:club31-1}
z_{-}\left(x,Y\right)=-z\left(x,-Y\right).
\end{equation}
Then $z_{-}$ is still holomorphic.
Let 
\index{Z@$Z_{-}$}%
$Z_{-}$ be the corresponding section of 
$\underline{q}^{*}T_{\R}X$.
\begin{theorem}\label{thm:Tcadjoingt}
    The operators $\mathcal{A}^{ \prime \prime }_{Z},\mathcal{B}''_{Z}$
    are of total degree $1$, and the operators
    $\mathcal{A}^{\prime }_{Z},\mathcal{B}'_{Z}$ are of 
    total degree $-1$. Moreover, we have the identities
\begin{align}\label{eq:dadjoint}
    &\mathcal{A}^{ \prime \prime \dag}_{Z}=\mathcal{A}^{ \prime 
    }_{Z_{-}},
&\mathcal{B}^{ \prime \prime \dag}_{Z}=\mathcal{B}^{ \prime }_{Z_{-}},\\
&\mathcal{A}^{\dag}_{Z}=\mathcal{A}_{Z_{-}},& \mathcal{B}^{\dag}_{Z}=\mathcal{B}_{Z_{-}}. \nonumber 
\end{align}
\end{theorem}
\begin{proof}
	The proof is  the same as in \cite[Theorem 6.5.2]{Bismut10b}. For 
	$\mathcal{A}_{Z}$ and its components, it is an easy consequence of 
	(\ref{eq:retpek0}) and of Proposition \ref{prop:Padje}.
\end{proof}
\begin{remark}\label{rem:adj}
	When considering $\mathcal{A}_{Z}$ and its components, we may 
	take adjoints with respect to $\epsilon_{M}$, and the above 
	results would still hold. Adjoints  with respect to 
	$\epsilon_{M}$ are simply classical adjoints. Using integration 
	along the fiber $p_{*}$ to evaluate $\epsilon_{M}$ is enough to 
	explain the coincidence of the adjoints.
\end{remark}
\subsection{A formula for the  curvature of $\mathcal{A}_{Z}$}%
\label{subsec:coc}
When identifying $TX$ and $\widehat{TX}$, we denote by  
\index{gTX@$\widehat{g}^{TX}$ }%
$\widehat{g}^{TX}$  
the metric on $TX$ corresponding to $g^{\widehat{TX}}$, and by
\index{nTX@$\widehat{\n}^{TX}$}%
$\widehat{\n}^{TX}$  the connection on $TX$ that corresponds to 
$\n^{\widehat{TX}}$.  Equivalently, $\widehat{\n}^{TX}$ is the Chern 
connection on $\left(TX,\widehat{g}^{TX}\right)$. Let 
\index{t@$\widehat{\tau}$}%
$\widehat{\tau}$ be the 
torsion of $\widehat{\n}^{TX}$. Then $\widehat{\tau}$ has the same 
properties as $\tau$. \footnote{The torsion $\widehat{\tau}$ takes 
its values in $TX$ and not in $\widehat{TX}$.}

 Recall that $Z=z+\overline{z}$ is a smooth section on $\mathcal{M}$ 
of $\underline{q}^{*}T_{\R}X$. Let $\rho:\mathcal{M}\to\mathcal{X}$ 
be the obvious projection. We  identify $Z$ with its horizontal lift 
$Z^{H}$ in $\rho^{*}T^{H}_{\R}\mathcal{X} \subset T_{\R}\mathcal{M}$. 

 We denote by $\widehat{\n}^{T_{\R}X,H}Z$ the  section on $\mathcal{M}$ of
 $\underline{\pi}^{*} \left( T^{*}_{\R}M \otimes q^{*}T_{\R}X \right)  $ that is obtained by taking the covariant 
derivative of $Z$ with respect to the pull-back of  $\widehat{\n}^{T_{\R}X}$ along horizontal 
directions in $\mathcal{M}$.  Other horizontal covariant derivatives 
will be denoted in a similar way.

Let 
$\Delta^{V}_{g^{\widehat{TX}}}$ be the Laplacian along the 
fibers $\widehat{TX}$ with respect to the metric  $g^{\widehat{TX}}$.

In this Subsection, $\dag$ refers to the adjoint with respect to the form 
$\epsilon_{X}$ in Definition \ref{def:goug}. By Proposition 
\ref{prop:Padje} and by (\ref{eq:club31-1}), we get
\begin{equation}\label{eq:retour6}
 i_{z_{-}}^{\dag}=\overline{z}_{*}\we + i_{\overline{z}}.
\end{equation}

Now we  follow Subsection \ref{subsec:remdi}. For $a\in \R$, we 
introduce the connection 
\index{nTX@$\widehat{\n}^{TX,a}$}%
$\widehat{\n}^{TX,a}$ on $TX$ associated 
with 
 $\widehat{\n}^{TX}$\footnote{In \cite{Bismut10b}, the connection 
 $\n^{TX}$ on $TX$ was used instead. This accounts for minor 
 differences with respect to \cite{Bismut10b}.}. If 
 \index{nLTX@$\widehat{\n}^{\Lambda\left(T^{*}_{\R}X\right),a}$}%
$\widehat{\n}^{\Lambda\left(T^{*}_{\R}X\right),a}$ is the induced connection on 
$\Lambda\left(T^{*}_{\R}X\right)$, by (\ref{eq:cla2a1}), if 
$U\in T_{\R}X$, then
\begin{equation}\label{eq:cup7}
\widehat{\n}^{\Lambda\left(T^{*}_{\R}X\right),a}_{U}=\widehat{\n}_{U}^{\Lambda\left(T^{*}_{\R}X\right)}-ai_{\widehat{\tau}\left(U,\cdot\right)}.
\end{equation}
The connection $\widehat{\n}^{\Lambda\left(T^{*}_{\R}X\right),a}$ preserves 
the complex type of the forms. In particular, by (\ref{eq:cla4b1}), 
we get
\begin{equation}\label{eq:cupi1}
d^{X}=\widehat{\n}_{\mathrm{a}}^{\Lambda\left(T^{*}_{\R}X\right),-1/2},
\end{equation}
and (\ref{eq:cupi1}) splits into formulas for 
$\overline{\pa}^{X},\pa^{X}$. Equation (\ref{eq:cupi1}) can be 
extended to an equation for $d^{M}$,
\begin{equation}\label{eq:cupi2}
d^{M}=\widehat{\n}_{\mathrm{a}}^{\Lambda\left(T^{*}_{\R}X\right),-1/2}.
\end{equation}
In the right-hand side of (\ref{eq:cupi2}), exterior differentiation  
on $S$
is done implicitly, and does not require the choice of a connection on 
$TS$.

Recall that 
\index{F@$\mathbb F$}%
$\mathbb F$ was defined in (\ref{eq:club19}). Let 
$\widehat{\n}^{\mathbb F}$ be 
the connection on $\mathbb F$ induced by
$\widehat{\n}^{\Lambda\left(T^{*}_{\C}X\right)}, 
\n^{D},\n^{\Lambda\left(\overline{\widehat{T^{*}X}}\right)}$.

For $a\in \R$,  let 
\index{nF@$\widehat{\n}^{\mathbb F, a}$}%
$\widehat{\n}^{\mathbb F, a}$ be the connection on $\mathbb F$ 
induced by 
$\widehat{\n}^{\Lambda\left(T^{*}_{\C}X\right),a},\n^{D},\n^{\Lambda\left(\overline{\widehat{T^{*}X}}\right)}$. 
By (\ref{eq:cup7}), we get
\begin{equation}\label{eq:cup7a1}
\widehat{\n}^{\mathbb F,a}_{U}=\widehat{\n}^{\mathbb F}_{U}-ai_{\widehat{\tau}\left(U,\cdot\right)}.
\end{equation}
By  (\ref{eq:cla4}), we obtain
\begin{equation}\label{eq:cup7a1r}
\widehat{\n}^{\mathbb F,-1/2}_{\mathrm{a}}=\widehat{\n}^{\mathbb F}_{\mathrm{a}}+i_{\widehat{\tau}}.
\end{equation}
As explained after (\ref{eq:cupi2}), equation 
(\ref{eq:cup7a1r}) can be viewed as an identity of operators acting 
on $\Omega\left(M, q^{*}\Lambda\left(\overline{\widehat{\TsX}}\right)\ho D\right)$. 
Equation (\ref{eq:cup7a1r}) can be rewritten in the form
\begin{equation}\label{eq:cupi3}
\widehat{\n}^{\mathbb 
F,-1/2}_{\mathrm{a}}=\n^{q^{*}\Lambda\left(\overline{\widehat{\TsX}}\right)\ho D},
\end{equation}
where the right-hand side is simply the natural extension of a 
connection on the vector bundle $q^{*}\Lambda\left(\overline{\widehat{\TsX}}\right)\ho D$ to an operator acting on 
$\Omega\left(M,q^{*}\Lambda\left(\overline{\widehat{\TsX}}\right)\ho D\right)$.

\begin{definition}\label{def:Dnfbr}
    Let
	\index{nF@${}^{0}\widehat{\n}^{\mathbb F}$}%
${}^{0}\widehat{\n}^{\mathbb F}$ be the fibrewise connection  on $\mathbb F$ such that if 
$U\in T_{\R}X$, $U=u+\overline{u}, u\in TX$, 
\begin{equation}\label{eq:deffbr}
    {}^{0}\widehat{\n}^{\mathbb F}_{U}=
\widehat{\n}^{\mathbb F,-1}_{U}
-i_{u}q^{*}\pa^{X}i\omega^{X}.
\end{equation}
\end{definition}

Put
\index{F@$\mathbf{F}$}%
\begin{equation}\label{eq:cups1}
\mathbf{F}=q^{*}\Lambda\left(T^{*}_{\C}X\right)\ho D\ho q^{*}\mathbf{I}.
\end{equation}
By (\ref{eq:bim1}), (\ref{eq:club19}), we get
\begin{equation}\label{eq:cups2}
\mathbf{F}= \mathbb F\ho q^{*}\mathbf{I}^{0}. 
\end{equation}
Equation (\ref{eq:cups2}) can be rewritten in the form,
\begin{equation}\label{eq:cups3}
\mathbf{F}=C^{\infty }\left(\widehat{T_{\R}X}, 
\underline{\pi}^{*}\mathbb F\vert_{\widehat{T_{\R}X}}\right).
\end{equation}

 Let
\index{nF@$\widehat{\n}^{\mathbf{F}}$}%
$\widehat{\n}^{\mathbf{F}}$ be the connection on $\mathbf{F}$ 
associated with 
$\widehat{\n}^{\Lambda\left(T^{*}_{\C}X\right)},\n^{D},\n^{\mathbf{I}}$. Equivalently, $\widehat{\n}^{\mathbf{F}}$ is the connection on $\mathbf{F}$ induced by $\widehat{\n}^{\mathbb F},\n^{\mathbf{I}^{0}}$. More generally, when replacing $\mathbb F$ by $\mathbf{F}$, we obtain corresponding objects associated with $\mathbf{F}$. 

As in Subsection \ref{subsec:remdi}, 
\index{nF@$\widehat{\n}^{\mathbf{F},a}_{\mathrm{a}}$}%
$\widehat{\n}^{\mathbf{F},a}_{\mathrm{a}}$ denotes the antisymmetrization of 
$\widehat{\n}^{\mathbf{F},a}$. 
By (\ref{eq:coro9a1}), (\ref{eq:cupi3}) can be rewritten in the form
\begin{align}\label{eq:coro9a1bi}
&\mathcal{A}^{\prime \prime }_{Z}=\widehat{\n}^{\mathbf{F}, -1/2 \prime \prime 
}_{\mathrm{a}}+B+\overline{\pa}^{V}+i_{z}, \nonumber \\
&\mathcal{A}'_{Z}=\widehat{\n}^{\mathbf{F},-1/2\prime}_{\mathrm{a}}+B^{*}-q^{*}\pa^{X}i\omega^{X}+\overline{\pa}^{V*}+i_{\overline{z}}+\overline{z}_{*}\we,\\
&\mathcal{A}_{Z}=\widehat{\n}^{\mathbf{F},-1/2}_{\mathrm{a}}+C
-q^{*}\pa^{X}i\omega^{X} +\overline{\pa}^{V}+i_{z}+\overline{\pa}^{V*}+i_{\overline{z}}+\overline{z}_{*}\we. \nonumber 
\end{align}

Recall that the curvature $A^{E_{0},2}$ was defined in Subsection \ref{subsec:curva}. 
It is a section of degree $0$ of
$\Lambda\left(T^{*}_{\C}M\right)\ho\End\left(D\right)$.

Let $w_{1}, \ldots ,w_{n}$ be a 
basis of $TX$. The corresponding basis of $\widehat{TX}$ is
denoted  $\widehat{w}_{1}, \ldots ,\widehat{w}_{n}$. In the sequel, we 
assume that $\widehat{w}_{1},\ldots,\widehat{w}_{n}$  is an 
orthonormal basis of $\left(\widehat{TX},g^{\widehat{TX}}\right)$. As 
before, if $e\in T_{\R}X$, $e_{*}\in T^{*}_{\R}X$ corresponds to $e$ via 
the metric $g^{T_{\R}X}$.

We establish an extension of \cite[Theorem 6.8.1]{Bismut10b}.
\begin{theorem}\label{thm:Tcurvhyp}
The following identity holds:
\begin{multline}\label{eq:curbhypx0}
\mathcal{A}^{2}_{Z}=\left[\overline{\pa}^{V}+i_{z},\pa^{V\dag}+ 
i_{z_{-}}^{\dag}\right]
+\left[{}\widehat{\n}^{\mathbf{F},-1/2 \prime \prime}_{\mathrm{a}}+B, 
i_{z_{-}}^{\dag}\right]\\
+\left[\widehat{\n}^{\mathbf{F},-1/2 
\prime}_{\mathrm{a}}+B^{*}-q^{*}\pa^{X}i\omega^{X},i_{z}\right]
-q^{*}\overline{\pa}^{X}\pa^{X}i\omega^{X}\\
-\n^{V}_{R^{\widehat{TX}}\widehat{Y}} 
-\left\langle  
R^{\widehat{TX}}\overline{\widehat{w}}_{i},\overline{\widehat{w}}^{j}\right\rangle\overline{\widehat{w}}^{i}
i_{\overline{\widehat{w}}_{j}} 
+A^{E_{0},2}.
\end{multline}
Equation (\ref{eq:curbhypx0}) can be rewritten in the form, 
\begin{multline}\label{eq:curbhyp}
\mathcal{A}^{2}_{Z}=\left[\overline{\pa}^{V}+i_{z},\pa^{V\dag}+
i_{z_{-}} ^{\dag}\right]+{}^{0}\widehat{\n}^{\mathbf{F}}_{Z}+i_{Z}C
-q^{*}\overline{\pa}^{X}\pa^{X}i\omega^{X}\\
-\n^{V}_{R^{\widehat{TX}}\widehat{Y}}-
\left\langle  
R^{\widehat{TX}}\overline{\widehat{w}}_{i},\overline{\widehat{w}}^{j}\right\rangle\overline{\widehat{w}}^{i}
i_{\overline{\widehat{w}}_{j}}+A^{E_{0},2}
+i_{\widehat{\n}^{T_{\R}X,H}Z}+\n^{TX \prime  \prime ,H}\overline{z}_{*}.
\end{multline}

Moreover, 
\begin{multline}\label{eq:curbhypa1}
\left[\overline{\pa}^{V}+i_{z},\pa^{V\dag}+ 
i_{z_{-}} ^{\dag}\right]=\frac{1}{2}\left(-\Delta^{V}_{g^{\widehat{TX}}}
+\left\vert  
Z\right\vert^{2}_{g^{TX}} \right)\\
+\overline{\widehat{w}}^{i}\we\left(i_{\n_{\overline{\widehat{w}}_{i}}
\overline{z}}+\n_{\overline{\widehat{w}}_{i}}\overline{z}_{*} \right) 
-i_{\overline{\widehat{w}}_{i}}i_{\n_{\widehat{w}_{i}}z}.
\end{multline}
\end{theorem}
\begin{proof}
	Since $\widehat{\sigma}^{*}$ commutes with $\overline{\pa}^{V}$, 
	we get
	\begin{equation}\label{eq:ovb1}
\overline{\pa}^{V*}=\overline{\pa}^{V\dag}.
\end{equation}
	Equation (\ref{eq:curbhypx0}) follows from Proposition 
	\ref{prop:Pcurvni} and from equations (\ref{eq:corobeb1}),  
	(\ref{eq:retour6}),   
	(\ref{eq:coro9a1bi}), and 
	  (\ref{eq:ovb1}). Equation (\ref{eq:curbhyp}) is an easy 
	  consequence. The only mysterious point is the that in 
	  ${}^{0}\widehat{\n}^{\mathbf{F}}$, the connection $\widehat{\n}^{\mathbf{F},-1}$ 
	  appears, while there was $\widehat{\n}^{\mathbf{F},-1/2}$ in equation 
	  (\ref{eq:curbhypx0}). But this is entirely explained in 
	  (\ref{eq:clap3})--(\ref{eq:clap4ax1}).
	  
	 The proof of (\ref{eq:curbhypa1}) is left to the reader, which 
	 completes the proof of our theorem.
 \end{proof}
\subsection{The hypoelliptic curvature}%
\label{subsec:hypcurv}
Among the sections $z$, we will  distinguish the 
tautological section 
\index{y@$y$}%
$y$ of $\underline{q}^{*}TX$ on $\mathcal{M}$  
corresponding to the tautological section $\widehat{y}$ of 
$\widehat{TX}$. Note that $y,\widehat{y}$ already appeared  in 
Subsection \ref{subsec:hemer}. The section $y$
is such that
\begin{equation}\label{eq:ega}
y_{-}=y.
\end{equation}
We use the notation
\begin{equation}\label{eq:ega1}
Y=y+\overline{y}.
\end{equation}

By (\ref{eq:coro9a1bi}),  we get
\begin{align}\label{eq:tot1b}
	&\mathcal{A}^{\prime \prime }_{Y}=\widehat{\n}^{\mathbf{F}, -1/2 \prime \prime 
}_{\mathrm{a}}+B+\overline{\pa}^{V}+i_{y}, \nonumber \\
&\mathcal{A}'_{Y}=\widehat{\n}^{\mathbf 
F,-1/2\prime}_{\mathrm{a}}+B^{*}-q^{*}\pa^{X}i\omega^{X}+\overline{\pa}^{V*}+i_{\overline{y}}+\overline{y}_{*}\we,\\
&\mathcal{A}_{Y}=\widehat{\n}^{\mathbf{F},-1/2}_{\mathrm{a}}+C
-q^{*}\pa^{X}i\omega^{X} 
+\overline{\pa}^{V}+i_{y}+\overline{\pa}^{V*}+i_{\overline{y}}+\overline{y}_{*}\we. \nonumber 
\end{align}
By Theorem \ref{thm:Tcadjoingt}, we get
\begin{align}\label{eq:tot1z1}
&\mathcal{A}^{\prime \prime \dag}_{Y}=\mathcal{A}^{\prime }_{Y},
&\mathcal{A}_{Y}^{\dag}=\mathcal{A}_{Y}.
\end{align}

Put
\index{ga@$\gamma$}%
\begin{equation}\label{eq:veau0}
\gamma=\n^{TX}-\widehat{\n}^{TX}.
\end{equation}
Then $\gamma$ is a section of $\TsX \otimes \End\left(TX\right)$. If 
$u\in TX$, then $\overline{\gamma u_{*}}$ is a $(1,1)$-form on $X$. 
\begin{definition}\label{def:fbc}
	Let 
\index{nF@${}^{1}\n^{\mathbb F}$}%
${}^{1}\widehat{\n}^{\mathbb F}$ be the fiberwise connection on $\mathbb F$, 
which is such that if $u\in TX, U=u+\overline{u}\in T_{\R}X$, then
\begin{equation}\label{eq:sup1}
{}^{1}\widehat{\n}^{\mathbb F}_{U}={}^{0}\widehat{\n}^{\mathbb F}_{U}+\overline{\gamma u_{*}}.
\end{equation}
\end{definition}
As before, we may as well replace $\mathbb F$ by $\mathbf{F}$ and 
obtain a fiberwise connection ${}^{1}\widehat{\n}^{\mathbf{F}}$ on $\mathbf{F}$.

\begin{theorem}\label{thm:Tcurvhypbi}
The following identity holds:
\begin{multline}\label{eq:curbhypx0bi}
\mathcal{A}^{2}_{Y}=\left[\overline{\pa}^{V}+i_{y},\pa^{V\dag}+ 
i_{y} ^{\dag}\right]
+\left[{}\widehat{\n}^{\mathbf{F},-1/2 \prime \prime}_{\mathrm{a}}+B,
i_{y} ^{\dag}\right]\\
+\left[\widehat{\n}^{\mathbf{F},-1/2 
\prime}_{\mathrm{a}}+B^{*}-q^{*}\pa^{X}i\omega^{X},i_{y}\right]
-q^{*}\overline{\pa}^{X}\pa^{X}i\omega^{X}\\
-\n^{V}_{R^{\widehat{TX}}\widehat{Y}} 
-\left\langle  
R^{\widehat{TX}}\overline{\widehat{w}}_{i},\overline{\widehat{w}}^{j}\right\rangle\overline{\widehat{w}}^{i}
i_{\overline{\widehat{w}}_{j}} 
+A^{E_{0},2}.
\end{multline}
Equation (\ref{eq:curbhypx0bi}) can be rewritten in the form 
\begin{multline}\label{eq:curbhypbibi}
\mathcal{A}^{2}_{Y}=\left[\overline{\pa}^{V}+i_{y},\pa^{V\dag}+
i_{y} ^{\dag}\right]+{}^{1}\widehat{\n}^{\mathbf{F}}_{Y}+i_{Y}C
-q^{*}\overline{\pa}^{X}\pa^{X}i\omega^{X}\\
-\n^{V}_{R^{\widehat{TX}}\widehat{Y}}-
\left\langle  
R^{\widehat{TX}}\overline{\widehat{w}}_{i},\overline{\widehat{w}}^{j}\right\rangle\overline{\widehat{w}}^{i}
i_{\overline{\widehat{w}}_{j}}+A^{E_{0},2}.
\end{multline}

Moreover, 
\begin{multline}\label{eq:curbhypbe}
\left[\overline{\pa}^{V}+i_{y},\pa^{V\dag}+
i_{y} ^{\dag}\right]=\frac{1}{2}\left(-\Delta^{V}_{g^{\widehat{TX}}}
+\left\vert  
Y\right\vert^{2}_{g^{TX}} \right)\\
+\overline{\widehat{w}}^{i}\we\left(i_{\overline{w}_{i}}+\overline{w}_{i*}\right) 
-i_{\overline{\widehat{w}}_{i}}i_{w_{i}}.
\end{multline}
\end{theorem}
\begin{proof}
	This is a consequence of Theorem \ref{thm:Tcurvhyp}. In the proof 
	of (\ref{eq:curbhypbibi}), we have exploited the fact that 
	$\widehat{\n}^{T_{\R}X,H}Y=0$, and the related equation
	\begin{equation}\label{eq:sur1}
\n^{TX \prime \prime ,H}\overline{y}_{*}=\overline{\gamma y}_{*}.
\end{equation}
\end{proof}

By  Theorem \ref{thm:Tcurvhyp} and  by H\"ormander \cite{Hormander67},  
 $\mathcal{A}^{2}_{Y}$ is a hypoelliptic operator 
along the fibers $\mathcal{X}$. More precisely,  $\mathcal{A}^{2}_{Y}$ 
is a
hypoelliptic Laplacian in the sense of
\cite{Bismut05a,Bismut06d, Bismut11a}. 
\begin{remark}\label{rem:nec}
	In  (\ref{eq:curbhypbibi}), we may 
	as well replace ${}^{1}\widehat{\n}^{\mathbf{F}}_{Y} $ by 
	${}^{1}\widehat{\n}^{\mathbb F}_{Y^{H}}$.
\end{remark}
\subsection{Scaling the metric $g^{\widehat{TX}}$}%
\label{subsec:sca}
Here we follow \cite[Section 7.1]{Bismut10b}.
For $b>0$, when  
$g^{\widehat{TX}}$   is replaced by 
$b^{4}g^{\widehat{TX}}$, let
\index{eXb@$\epsilon_{X,b}$}%
\index{eMb@$\epsilon_{M,b}$}%
\index{AYb@$\mathcal{A}'_{Y,b}$}%
\index{AYb@$\mathcal{A}_{Y,b}$}%
$\epsilon_{X,b},\epsilon_{M,b},\mathcal{A}'_{Y,b}, \mathcal{A}_{Y,b}$ be the  analogues 
of 
$\epsilon_{X},\epsilon_{M},\mathcal{A}'_{Y}, \mathcal{A}_{Y}$.
\begin{definition}\label{Dda}
For $a>0$, let 
\index{da@$\delta_{a}$}%
$\delta_{a}$ be the dilation of $\mathcal{M}$: $\left(x,Y\right)\to \left(x,aY\right)$.
For $a>0$, let
\index{da@$\underline{\delta}^{*}_{a}$}%
$\underline{\delta}^{*}_{a}$ be the  action of 
$\delta_{a}$ on smooth sections of 
$\Lambda\left(\overline{\widehat{\TsX}}\right)$. The  action of 
$\underline{\delta}^{*}_{a}$ extends to $\mathbb F$.
\end{definition}

Also $y/b^{2}$ is a holomorphic section on $\mathcal{M}$ of 
$\underline{q}^{*}TX$.
One verifies easily that
\begin{equation}
   A_{Y/b^{2}}= \underline{\delta}^{*}_{1/b^{2}} \mathcal  
    {A}_{Y,b}\underline{\delta}^{*-1}_{1/b^{2}}.
    \label{eq:triv-6}
\end{equation}
 This identity is  obvious because $\delta^{*}_{1/b^{2}}$ maps the 
metric $b^{4}g^{\widehat{TX}}$ to $g^{\widehat{TX}}$, and $Y$ 
to $Y/b^{2}$. 

Now we follow \cite[Section 6.9]{Bismut10b}. For $a>0$, put
\index{Ka@$K_{a}$}%
\begin{equation}\label{eq:club41}
K_{a}s\left(x,Y\right)=s\left(x,aY\right).
\end{equation}
\begin{definition}\label{DAZb}
Put
\index{CY@$\mathcal{C}_{Y,b}$}%
\index{DY@$\mathcal{D}_{Y,b}$}%
\begin{align}\label{eq:club43}
&\mathcal{C}_{Y,b}=K_{b}\mathcal{A}_{Y/b^{2}}K_{b}^{-1},
&\mathcal{D}_{Y,b}=K_{b}\mathcal{B}_{Y/b^{2}}K_{b}^{-1}.
\end{align}
By (\ref{eq:conjugrep}), (\ref{eq:club43}), we get
\begin{equation}
   \mathcal{D}_{Y,b}=\exp\left(i\Lambda\right)\mathcal{C}_{Y,b}\exp\left(-i\Lambda\right).
    \label{eq:club47}
\end{equation}
\end{definition}

By (\ref{eq:coro9a1}), (\ref{eq:club46}), we get 
\begin{align}\label{eq:retpek6x1}
&\mathcal{C}_{Y,b}=\n^{D\ho q^{*}\mathbf{I}}+C-q^{*}\pa^{X}i\omega^{X}+
\frac{1}{b}\left(\overline{\pa}^{V}+i_{y}+\overline{\pa}^{V*}+i_{\overline{y}}+\overline{y}_{*}\right),\\
&\mathcal{D}_{Y,b}=\mathcal{E}+\frac{1}{b}\left(\overline{\pa}^{V}+i_{y}+\overline{\pa}^{V*}
    +\overline{y}_{*}\right). \nonumber 
\end{align}
\subsection{Hypoelliptic and elliptic superconnections}%
\label{subsec:hypel}
In this Subsection,  we assume that $\widehat{g}^{TX}=g^{TX}$. Let 
$w_{1},\ldots,w_{n}$ be an orthonormal basis of $TX$ with respect to 
$g^{TX}$, let $\widehat{w}_{1},\ldots,\widehat{w}_{n}$ be the 
corresponding basis of $\widehat{TX}$. We denote the dual basis in 
the usual way.

We use the identification 
in (\ref{eq:cret1}).  Set 
\index{L@$\widehat{L}$}%
\index{L@$\widehat{\Lambda}$}%
\begin{align}\label{eq:cohypoz0}
&\widehat{L}=-iw^{j}\we\overline{\widehat{w}}^{j},
&\widehat{\Lambda}=ii_{\overline{\widehat{w}}_{j}}i_{w_{j}}.
\end{align}
In (\ref{eq:cohypoz0}), $\widehat{L}$ is just the multiplication 
operator by the fibrewise K\"ahler form 
\index{oV@$\widehat{\omega}^{\mathcal{X},V}$}%
$\widehat{\omega}^{\mathcal{X},V}$.
As in \cite[eq. (6.9.8)]{Bismut10b}, we get 
\begin{equation}
    \left[\overline{\pa}^{V}+i_{y},\overline{\pa}^{V*}+\overline{y}_{*}\we\right]=
    \frac{1}{2}\left(-\Delta^{V}_{g^{TX}}+\left\vert  
    Y\right\vert_{g^{TX}}^{2}\right)-i\left(\widehat{L}-\widehat{\Lambda}\right).
    \label{eq:cohhypoz1}
\end{equation}
The right-hand side of (\ref{eq:cohhypoz1}) is a harmonic oscillator.

Let $\mathcal{S}^{(0,\Ou)}\left(\widehat{TX},\pi^{*}\Lambda\left(\TsX\right)
\right)$ be the Schwartz space of rapidly decreasing sections of 
$\pi^{*} \left( \Lambda\left(\overline{\widehat{\TsX}}\right)\ho 
\Lambda\left(T^{*}X\right) \right) $ along the  fibre $\widehat{TX}$.
 The operator in 
(\ref{eq:cohhypoz1}) is essentially self-adjoint on 
$\mathcal{S}^{(0,\Ou)}\left(\widehat{TX},\pi^{*}\Lambda\left(\TsX\right)
\right)$. By \cite[Proposition 1.5 and Theorem 
1.6]{Bismut90d}\footnote{In \cite{Bismut90d}, the operator 
$\overline{\pa}^{V}+i_{\sqrt{-1}y}$ and its adjoint were considered 
instead, which explains 
the minor discrepancies with respect to what is done here.},
\cite[Section 6.9]{Bismut10b},  its spectrum is 
$\N$, and its kernel is $1$-dimensional and spanned by 
\index{d@$\delta$}%
\begin{equation}
    \delta=\exp\left(i\widehat{\omega}^{\mathcal{X},V}-\left\vert  
    Y\right\vert^{2}_{g^{TX}}/2\right).
    \label{eq:forbeta}
\end{equation}
Recall that $dY$ is the volume form along the fibers 
$\widehat{TX}$ with respect to $g^{\widehat{TX}}$. Note  that
    \begin{align}
       &\sigma^{*}\delta=\delta, &\left(2\pi\right)^{-n} \int_{\widehat{TX}}^{}\left\vert  
       \delta\right\vert^{2}dY=1.
	\label{eq:bebe14}
    \end{align}
    Let  
	\index{P@$P$}%
	$P$ be the $L_{2}$
    orthogonal projection operator on vector space spanned by 
	$\delta$.
    
    We embed 
	\index{D@$\mathcal{D}$}%
	$\mathcal{D}=\Omega^{0,\Ou}\left(X,D\vert_{X}\right)$ into 
    $\Omega\left(X,D\vert_{X}\ho \mathbf{I}\right)$ by the  
    embedding 
    $\alpha\to \underline{\pi}^{*}\alpha\we \delta$, so that
    $\mathcal{D}$ 
	is  just  the $L_{2}$ kernel of  $
    \overline{\pa}^{V}+i_{y}+\overline{\pa}^{V*}+\overline{y}_{*}\we$.
	More generally, we embed $p_{*} \mathscr E_{0}$ into $p_{*}\left[\mathscr 
	E_{0}\ho q^{*}\left(\Lambda\left(\TsX\right)\ho 
	\mathbf{I}\right)\right]$ by the same formula.

	Recall that the elliptic superconnection  
	\index{ApE@$A^{p_{*} \mathscr E_{0}}$}%
	$A^{p_{*}\mathscr E_{0}}$  
 was defined in Definition \ref{def:comsc}, and that the differential 
 operator 
 \index{E@$\mathcal{E}$}%
 $\mathcal{E}$ was 
 defined in (\ref{eq:club45}). Now we give an extension of  \cite[Theorem 
6.9.2]{Bismut10b}.
\begin{theorem}\label{Tcompidebis}
    The following identity holds:
    \begin{equation}\label{eq:3.15}
	P\mathcal{E}P=A^{p_{*} \mathscr E_{0}}.
    \end{equation}
    \end{theorem}
\begin{proof}
Observe that (\ref{eq:3.15}) is an identity of differential operators 
over $M$, which we will prove by local methods on $M$. We will use the considerations that follow (\ref{eq:ret2})   to evaluate $\mathcal{E}$. 

Let $\n^{\Lambda\left(\TsX\right)}$ be the connection induced by 
$\n^{TX}$ on $\Lambda\left(\TsX\right)$. Since $g^{TX}=\widehat{g}^{TX}$, we get
\begin{equation}\label{eq:rom1}
\n^{\Lambda\left(\TsX\right)\ho \mathbf{I}}\delta=0.
\end{equation}

Recall that $C=B+B^{*}$. 
As explained 
after equation (\ref{eq:ret2}), conjugation by 
$\exp\left(i\Lambda\right)$ changes $\overline{w}^{i}$ into 
$\overline{w}^{i}+i_{w_{i}}$.   Recall that  $\delta$ in 
(\ref{eq:forbeta}) is of total degree $0$. The effect of the 
composition of the $i_{w_{j}}$ is to map  $\delta$ into forms of total
positive degree\footnote{Recall that the total degree is the 
difference of the antiholomorphic and of the holomorphic degree.}, 
which are orthogonal to $\delta$. Using (\ref{eq:rom1}) and proceeding as in the proof of \cite[Theorem 6.9.2]{Bismut10b}, we 
find  that
\begin{equation}\label{eq:rims2-}
P\exp\left(i\Lambda\right)\left( \n^{D\ho q^{*}\mathbf{I} \prime \prime 
}+B\right)\exp\left(-i\Lambda\right)  P=\n^{D \prime \prime }+B.
\end{equation}
In \cite{Bismut10b}, the proof of (\ref{eq:rims2-}) uses equations 
(\ref{eq:cup7a1r}), (\ref{eq:tot1b}).

Using (\ref{eq:new2a1}), we can rewrite (\ref{eq:rims2-}) in the form
\begin{equation}\label{eq:rims4z1}
P\exp\left(i\Lambda\right)\left( \n^{D\ho q^{*}\mathbf{I} \prime \prime 
}+B\right)\exp\left(-i\Lambda\right)  P=A^{p_{*}\mathscr E_{0} \prime 
\prime}.
\end{equation}

Conjugation by $\exp\left(i\Lambda\right)$ changes 
$w^{i}$ into $w^{i}-i_{\overline{w}_{i}}$.
By proceeding again as in \cite[Theorem 6.9.2]{Bismut10b}, we 
 get
\begin{multline}\label{eq:rims2}
P\exp\left(i\Lambda\right)\left( \n^{ D \ho q^{*}\mathbf{I}\prime  
}+B^{*}-q^{*}\pa^{X}i\omega^{X} \right) 
\exp\left(-i\Lambda\right)P\\
=\n^{D \prime,S}+\n^{D \prime \prime, X 
*}_{\alpha}+{}^{c}B^{*}.
\end{multline}
By  (\ref{eq:coc7a3}), we can rewrite (\ref{eq:rims2}) in the form,
\begin{equation}\label{eq:rims7}
P\exp\left(i\Lambda\right)\left( \n^{D\ho q^{*}\mathbf{I} \prime  
}+B^{*}-q^{*}\pa^{X}i\omega^{X} \right) 
\exp\left(-i\Lambda\right)P=A^{p_{*} \mathscr E_{0}\prime}.
\end{equation}

By (\ref{eq:rims4z1}), (\ref{eq:rims7}), we get (\ref{eq:3.15}). The proof of our theorem is completed. 
\end{proof}
\begin{remark}\label{Rdirect}
	Recall that 
\index{DY@$\mathcal{D}_{Y,b}$}%
$\mathcal{D}_{Y,b}$ was defined in (\ref{eq:club47}), and 
that it splits as 
\begin{equation}\label{eq:spli1bi}
\mathcal{D}_{Y,b}=\mathcal{D}^{\prime \prime 
}_{Y,b}+\mathcal{D}'_{Y,b}.
\end{equation}
By (\ref{eq:dadjoint}), when taking adjoints with respect to 
$\eta_{X}$, then
\begin{equation}\label{eq:retpek13}
\mathcal{D}^{ \prime  \prime \dag}_{Y,b}=\mathcal{D}'_{Y,b}.
\end{equation}
By (\ref{eq:cret2}),  (\ref{eq:bebe14}), on 
$\Omega^{0,\Ou}\left(X,D\vert_{X}\right)$,  
$\eta_{X}$ restrict to the Hermitian product  $\left\langle  \,\right\rangle
_{L_{2}}$.
Since  the contribution of $\n^{D\ho q^{*}\mathbf{I}\prime \prime 
}+B$ to $P\mathcal{E}P$ is $A^{p_{*} \mathscr E_{0}\prime \prime}$, the 
contribution of its adjoint 
$\n^{D\ho q^{*}\mathbf{I}\prime
}+B^{*}-q^{*}\pa^{X}i\omega^{X}$ has to be  
$A^{p_{*} \mathscr E_{0} \prime }$. This explains some aspects of 
the proof of Theorem \ref{Tcompidebis}.
\end{remark}
\section{The hypoelliptic superconnection forms}%
\label{sec:hypofo}
The purpose of this Section is to construct the superconnection forms 
associated with the hypoelliptic superconnections  of Section 
\ref{sec:hypo}, and to prove that their class in Bott-Chern 
cohomology coincides with the class of the elliptic superconnection 
forms.

This Section is organized as follows. In Subsection 
\ref{subsec:scnform}, we construct the hypoelliptic superconnection 
forms, that depend on metrics $g^{TX},g^{\widehat{TX}},g^{D}$.

In Subsection \ref{subsec:depb}, when $g^{TX}=\widehat{g}^{TX}$, we 
show that when replacing $g^{\widehat{TX}}$ by 
$b^{4}g^{\widehat{TX}}$, as $b\to 0$, the hypoelliptic superconnection 
 forms converge to the elliptic superconnection forms, which gives the main result of 
 this Section.
 
 In this Section, we make the same assumptions as in Sections 
 \ref{sec:suel} and  \ref{sec:hypo},  and we use the 
 corresponding notation.
 \subsection{The  hypoelliptic  forms}%
\label{subsec:scnform}
We will say that two metrics $g^{\widehat{TX}},g^{\widehat{TX}\prime 
}$ lie in the same projective class if there exists a constant $c>0$ 
such that $g^{\widehat{TX} \prime }=c g^{\widehat{TX}}$. Let 
$\mathbf{p}$ be such a projective class. Once we fix 
$g^{\widehat{TX}}\in \mathbf{p}$, then $\mathbf{p} \simeq \R_{+}^{*}$.

We fix $\mathbf{p}$.
Let 
 \index{Qp@ $\mathscr Q_{\mathbf{p}}$}%
$\mathscr Q_{\mathbf{p}}$ be the collection of parameters 
$\omega^{X},g^{D},g^{\widehat{TX}}$, with $g^{\widehat{TX}}\in 
\mathbf{p}$. We denote by $d^{\mathscr Q_{\mathbf{p}}}$ 
the de Rham operator on $\mathscr Q_{\mathbf{p}}$. A key fact is that 
when $g^{\widehat{TX}}\in \mathbf{p}$, then $\n^{\mathbf{I}}$ does 
not depend on $g^{\widehat{TX}}$.

By
\cite[Section 3.3]{BismutLebeau06a},  the operator 
$\exp\left(-\mathcal{A}^{2}_{Y}\right)$ is trace class, and is 
given by a smooth kernel along the fibers $\mathcal{X}$ that is 
rapidly decreasing along the fibre $\widehat{TX}$ together with its 
derivatives, and this uniformly over  $S$. 

\begin{definition}\label{def:DChernob}
Put
\begin{equation}\label{eq:7.8a}
\ch\left(\mathcal{A}^{\prime \prime}_{Y}, \omega^{X},g^{D},g^{\widehat{TX}}\right)=\varphi\Trs\left[\exp\left(- 
\mathcal{A}^{2}_{Y}\right)\right].
\end{equation}
\end{definition}

The Hermitian form $\epsilon_{X}$ was defined in Definition 
\ref{def:goug}, and depends on  the splitting of $E$, and 
also on $\omega^{X},g^{D},g^{\widehat{TX}}$. 

 Given $\mathbf{p}$, 
$\epsilon_{X}^{-1}d^{\mathscr Q_{\mathbf{p}}}\epsilon_{X}$ is a $1$-form on $\mathscr 
Q_{\mathbf{p}}$ with values in 
$\epsilon_{X}$-self-adjoint endomorphisms. It can be easily computed from 
equation (\ref{eq:retpek0}). 

We establish an extension of \cite[Theorem 7.3.2]{Bismut10b}.
\begin{theorem}\label{thm:Tdouble}
The form 
$ \ch\left(\mathcal{A}^{ \prime \prime}_{Y}, \omega^{X},g^{D},g^{\widehat{TX}}\right)$ 
lies in 
$\Omega^{(=)}\left(S,\R\right)$, it is closed,  and its  Bott-Chern cohomology 
class does not depend on $\omega^{X},g^{D},g^{\widehat{TX}}$, or on 
the splitting of $E$. 

Given $\mathbf{p}$, then
\begin{equation}\label{eq:spr1}
d^{\mathscr Q_{\mathbf{p}}}\ch\left(\mathcal{A}^{ \prime \prime}_{Y}, 
\omega^{X},g^{D},g^{\widehat{TX}}\right)=-\frac{\overline{\pa}^{S}\pa^{S}}{2i\pi}\Trs\left[\epsilon_{X}^{-1}d^{\mathscr Q_{\mathbf{p}}}\epsilon_{X}\exp\left(- 
\mathcal{A}^{2}_{Y}\right)\right].
\end{equation}
\end{theorem}
\begin{proof}
The proof  that the above forms are closed   is formally the same 
as the proof of Theorem \ref{thm:chca}.   Since $\mathcal{A}_{Y}$ is self-adjoint with 
respect to $\epsilon_{X}$,  the same argument as in 
the proof of Theorem \ref{thm:carch} shows that
 the above forms are real.  The fact that $\epsilon_{X}$ is non-positive is irrelevant.  By proceeding as in the proof of Theorem 
 \ref{thm:chca}, i.e., by using deformation over $\mathbf{P}^{1}$, one 
 can prove that their corresponding class in Bott-Chern cohomology 
 does not depend on the parameters.  Finally, if one restricts 
 $g^{\widehat{TX}}$ to vary in a class $\mathbf{p}$, the proof 
 of (\ref{eq:spr1}) is the same as the proof of  Theorems 
 \ref{thm:carch} and \ref{thm:chca}.  
\end{proof}

\begin{definition}\label{def:cobc}
	Let 
	\index{cAY@$\cBC\left(\mathcal{A}''_{Y}\right)$}
	$\cBC\left(\mathcal{A}''_{Y}\right)\in 
	H_{\mathrm{BC}}^{(=)}\left(S,\R\right)$ denote the common 
	classes in Bott-Chern cohomology of the forms 
	$\ch\left(\mathcal{A}^{\prime \prime}_{Y}, \omega^{X},g^{D},g^{\widehat{TX}}\right)$.
\end{definition}
\subsection{The limit as $b\to 0$ of the hypoelliptic superconnection 
forms}%
\label{subsec:depb}
In this Subsection, we fix $\omega^{X},g^{D}$, and we take 
$g^{\widehat{TX}}=b^{4}g^{TX}$.Then $\mathcal{A}''_{Y}$  does not depend on $b>0$.

Recall that the elliptic superconnection forms 
\index{chA@$\ch\left(A^{p_{*}\mathscr E_{0}\prime \prime 
	},\omega^{X},g^{D}\right)$}%
$\ch\left(A^{p_{*}\mathscr E_{0} \prime 
\prime},\omega^{X},g^{D}\right)$ were defined in Definition 
\ref{def:Dscf}, and their Bott-Chern class 
\index{cBC@$\cBC\left(A^{p_{*} \mathscr E \prime \prime }\right)$}%
$\cBC\left(A^{p_{*} 
\mathscr E \prime \prime}\right)$ in Definition \ref{def:ellscn}.
\begin{theorem}\label{thm:lim}
As $b\to 0$, 
\begin{equation}\label{eq:lim1}
\ch\left(\mathcal{A}^{\prime \prime 
}_{Y},\omega^{X},g^{D},b^{4}g^{TX}\right)\to 
\ch\left(A^{p_{*}\mathscr E_{0} \prime 
\prime },\omega^{X},g^{D}\right).
\end{equation}
\end{theorem}
 \begin{proof}
 We use the notation of Subsection \ref{subsec:sca}. By definition, 
 \begin{equation}\label{eq:lim2}
\ch\left(\mathcal{A}^{\prime \prime 
}_{Y},\omega^{X},g^{D},b^{4}g^{TX}\right)=\varphi\Trs\left[\exp\left(-\mathcal{A}^{2}_{Y,b}\right)\right].
\end{equation}
By (\ref{eq:triv-6})--(\ref{eq:club47}), we can rewrite 
(\ref{eq:lim2}) in the form
\begin{equation}\label{eq:lim3}
\ch\left(\mathcal{A}^{\prime \prime 
}_{Y},\omega^{X},g^{D},b^{4}g^{TX}\right)=\varphi\Trs\left[\exp\left(-\mathcal{D}^{2}_{Y,b}\right)\right].
\end{equation}
Using (\ref{eq:retpek6x1}), Theorem \ref{Tcompidebis}, and by 
proceeding exactly as in the proofs of  \cite[Theorem 
5.2.1]{BismutLebeau06a}, \cite[Theorem 7.13]{Bismut06d},  and 
\cite[Theorem 7.6.2]{Bismut10b}, we get (\ref{eq:lim1}). The proof of our theorem is completed. 
\end{proof}
\begin{theorem}\label{thm:fun1}
	The following identities hold:
	\begin{equation}\label{eq:lim4}
\ch_{\mathrm{BC}}\left(\mathcal{A}^{\prime \prime 
}_{Y}\right)=\ch_{\mathrm{BC}}\left(A^{p_{*} \mathscr E_{0}\prime \prime 
}\right)=
\ch_{\mathrm{BC}}\left(Rp_{*}\mathscr 
E\right)\,\mathrm{in}\,H^{(=)}_{\mathrm{BC}}\left(S,\R\right).
\end{equation}
\end{theorem}
\begin{proof}
	This is an obvious consequence of Theorems \ref{thm:Tes}, 
	\ref{thm:Tdouble}, and \ref{thm:lim}.
\end{proof}
\section{The hypoelliptic superconnection forms when 
$\overline{\pa}^{X}\pa^{X}\omega^{X}=0$}%
\label{sec:hypodb}
The purpose of this Section is to prove Theorem \ref{thm:sub} when 
$\overline{\pa}^{X}\pa^{X}\omega^{X}=0$ using the hypoelliptic 
superconnection forms of Section \ref{sec:hypofo}. This result had 
already been proved in Section \ref{sec:prka} using elliptic 
superconnection forms. While the hypoelliptic superconnection forms 
do not give a better result, the techniques used in the present Section 
prepare for Section \ref{sec:RBC}, where Theorem \ref{thm:sub} will 
be established in full generality using deformations of our 
hypoelliptic superconnections.

This Section is organized as follows. In Subsection 
\ref{subsec:lialge}, we compute finite and infinite-dimensional 
supertraces on certain model operators. These computations will later 
be used in a local index theoretic context.

In Subsection \ref{subsec:sepr}, given $t>0$, when replacing 
$\left(\omega^{X},g^{\widehat{TX}},g^{D}\right)$ by 
$\left(\omega^{X}/t,g^{\widehat{TX}}/t^{3},g^{D}\right)$, we compute 
the corresponding hypoelliptic curvature, and we establish various 
scaling identities.

Finally, in Subsection \ref{subsec:cava}, when 
$\overline{\pa}^{X}\pa^{X}\omega^{X}=0$, we obtain the limit as 
$t\to 0$ of our hypoelliptic superconnections forms, and we prove 
Theorem \ref{thm:sub}. The local index theoretic techniques of this 
Subsection will be used again in Subsection \ref{subsec:exo}.
\subsection{Finite and infinite-dimensional traces}%
\label{subsec:lialge}
Let $W$ be a complex vector space of dimension $n$, and let $W_{\R}$ be the corresponding real vector space. If
 $W_{\C}=W_{\R} \otimes _{\R}\C$ is its complexification, 
 then $W_{\C}=W \oplus \overline{W}$. 

Let $\mathfrak W$ be 
another copy of $W$. Observe that
\begin{equation}\label{eq:lwa}
\Lambda\left(\overline{\mathfrak W} \oplus \overline{\mathfrak 
W}^{*}\right)=\Lambda\left(\overline{\mathfrak 
W}\right)\ho\Lambda\left(\overline{\mathfrak W}^{*}\right).
\end{equation}

 Let 
$w_{1},\ldots, w_{n}$ be a basis of $W$, let $w^{1},\ldots,w^{n}$ be 
the associated dual basis of $W^{*}$. Let  $\mathfrak w_{1},\ldots, 
\mathfrak w_{n}$ be the  corresponding basis of $\mathfrak W$, and let
$\mathfrak w^{1},\ldots, \mathfrak w^{n}$ the associated dual basis of 
$\mathfrak W^{*}$. The supertrace maps $\Lambda\left(\overline{\mathfrak 
W} \oplus \overline{\mathfrak W}^{*}\right)\ho \End \left( \Lambda\left(\overline{W}^{*}\right) 
\right) $ into $\Lambda\left(\overline{\mathfrak W}\oplus 
\overline{\mathfrak W}^{*}\right)$ and vanishes 
on supercommutators.

Note that
$\Lambda^{n}\left(\overline{\mathfrak 
W}^{*}\right)\ho\Lambda^{n}\left(\overline{\mathfrak W}\right)$ is canonically trivial, and  $\overline{\alpha}=\prod_{1}^{n}\overline{\mathfrak w}^{i}\overline{\mathfrak 
w}_{i}$ is the canonical section. If $\theta\in 
\Lambda\left(\overline{\mathfrak 
W}^{*}\right)\ho\Lambda\left(\overline{\mathfrak W}\right)$, let $\theta^{\max}\in\C$ 
be  the coefficient of $\overline{\alpha}$ in the  expansion of $\theta$.

If $A\in \End\left(W\right)$, then we still denote by $A$ the 
conjugate action of $\overline{A}$ on $\overline{W}$.

We establish a form of a result of Mathai-Quillen \cite[eq. 
(2.13)]{MathaiQuillen86}, \cite[Proposition 6.10]{Bismut06d}.
\begin{proposition}\label{prop:gra}
	If $A\in \End\left(W\right)$, then
	\begin{equation}\label{eq:dee-1}
\Trs\left[\exp\left(-\overline{w}^{i}\overline{\mathfrak w}_{i}+i_{\overline{w}_{j}}
\overline{\mathfrak  
w}^{j}+\left\langle  
A\overline{w}_{i},\overline{w}^{j}\right\rangle\overline{w}^{i}i_{\overline{w}_{j}}\right)\right]^{\max}
=\Td^{-1}\left(-\overline{A}\right).
\end{equation}
\end{proposition}
\begin{proof}
	If $h\in \overline{W}$, let $\left[h\right]$ 
	denote the element that corresponds in $\overline{\mathfrak  
	W}$. A similar  
	 notation is used for elements in $\overline{W}^{*}$.
	 
	First, we assume  that $A$ is invertible. Let $\widetilde A\in 
	\End\left(W^{*}\right)$ be 
	the transpose of $A$. Put
\begin{equation}\label{eq:dee20}
M=-\overline{\mathfrak 
w}^{i}i_{A^{-1}\overline{w}_{i}}-\overline{\mathfrak 
w}_{i}\left[\widetilde A^{-1}\overline{w}^{i}\right].
\end{equation}
If $f\in W,u\in W^{*}$, then
\begin{align}\label{eq:dee21}
&\exp\left(M\right)\overline{u}\exp\left(-M\right)=\overline{u}-\left[\widetilde A^{-1}\overline{u}\right],\\
&\exp\left(M\right)i_{\overline{f}}\exp\left(-M\right)=i_{\overline{f}}-\left[A^{-1}\overline{f}\right]. \nonumber 
\end{align}
By (\ref{eq:dee21}), we deduce that
\begin{multline}\label{eq:dee22}
\exp\left(M\right)\left\langle  
A\overline{w}_{i},\overline{w}^{j}\right\rangle\overline{w}^{i}i_{\overline{w}_{j}}\exp\left(-M\right)=\left\langle  
A\overline{w}_{i},\overline{w}^{j}\right\rangle\overline{w}^{i}i_{\overline{w}_{j}}
\\-\overline{w}^{i}\overline{\mathfrak w}_{i}+i_{\overline{w}_{j}}\overline{\mathfrak  
w}^{j}+\left\langle  
A^{-1}\overline{w}_{i},\overline{w}^{j}\right\rangle 
\overline{\mathfrak w}^{i}\overline{\mathfrak w}_{j}.
\end{multline}
Since supertraces vanish on supercommutators, we deduce from 
(\ref{eq:dee22}) that
\begin{multline}\label{eq:dee23}
\Trs\left[\exp\left(-\overline{w}^{i}\overline{\mathfrak w}_{i}+i_{\overline{w}_{j}}\overline{\mathfrak  
w}^{j}+\left\langle  
A\overline{w}_{i},\overline{w}^{j}\right\rangle\overline{w}^{i}i_{\overline{w}_{j}}\right)\right]\\
=\exp\left(-\left\langle  
A^{-1}\overline{w}_{i},\overline{w}^{j}\right\rangle 
\overline{\mathfrak w}^{i}\overline{\mathfrak w}_{j}\right)
\Trs\left[\exp\left(\left\langle  
A\overline{w}_{i},\overline{w}^{j}\right\rangle 
\overline{w}^{i}i_{\overline{w}^{j}}\right)\right].
\end{multline}
Also, 
\begin{equation}\label{eq:dee24}
\Trs\left[\exp\left(\left\langle  
A\overline{w}_{i},\overline{w}^{j}\right\rangle 
\overline{w}^{i}i_{\overline{w}_{j}}\right)\right]=\det\left(1-\exp\left(\overline{A}\right)\right).
\end{equation}
If $A$ is diagonalizable, we get 
\begin{equation}\label{eq:dee25}
\left[\exp\left(-\left\langle  
A^{-1}\overline{w}_{i},\overline{w}^{j}\right\rangle 
\overline{\mathfrak w}^{i}\overline{\mathfrak 
w}_{j}\right)\right]^{\max}=\left[\det\left(-\overline{A}\right)\right]^{-1}.
\end{equation}
Equation (\ref{eq:dee25}) extends by continuity when $A$ is only 
supposed to be invertible.
By (\ref{eq:dee23})--(\ref{eq:dee25}), we get  (\ref{eq:dee-1}) when 
$A$ is invertible. This equation extends by continuity to arbitrary 
$A$. The proof of our proposition is completed. 
\end{proof}

Let $\widehat{W}$ be another copy of $W$. Let $g^{\widehat{W}}$ be a Hermitian metric on 
$\widehat{W}$, and let 
$\widehat{g}^{W}$ be the corresponding metric on $W$. Let $\left(U,Y\right)$ be the generic element of $W_{\R} 
\oplus \widehat{W}_{\R}$. Let $dUdY$ be 
the associated volume form on
$W_{\R}\oplus \widehat{W}_{\R}$. Let $\Delta^{\widehat{W}_{\R}}$ be the Laplacian on 
$\widehat{W}_{\R}$.    Let $F\in \End\left(\widehat{W}\right)$ be 
skew-adjoint. Then $F$ acts as antisymmetric endomorphism of 
$\widehat{W}_{\R}$. Let 
$\n^{\widehat{W}_{\R}}_{FY}$ denote differentiation along the vector 
field $FY$.  Let $\n^{W_{\R}}_{Y}$ denote differentiation along 
$W_{\R}$ in the direction $Y$.   

Let 
\index{p@$\mathfrak p_{0}$}%
$\mathfrak p_{0}$ be the  scalar operator,
\begin{equation}\label{eq:dee13}
\mathfrak 
p_{0}=-\frac{1}{2}\Delta^{\widehat{W}_{\R}}+
\n^{\widehat{W}_{\R}}_{FY}+\n^{W_{\R}}_{Y}.
\end{equation}
Let $s_{0}\left(\left(U,Y\right),\left(U',Y'\right)\right)$ denote 
the smooth kernel associated with $\exp\left(-\mathfrak 
p_{0}\right)$ with respect to the volume 
$\frac{dU'dY'}{\left(2\pi\right)^{2n}}$.

Now we state a result taken from  \cite[Theorem 6.8]{Bismut06d}. 
\begin{theorem}\label{thm:tracl}
	If $\left\vert  F\right\vert<2\pi$,  the function 
	$s_{0}\left(\left(0,Y\right),\left(0,Y\right)\right)$ is 
	integrable, and we have the identity,
\begin{equation}\label{eq:dee14}
\int_{\widehat{W}_{\R}}^{}s_{0}\left(\left(0,Y\right),\left(0,Y\right)\right)
\frac{dY}{\left(2\pi\right)^{n}}=\widehat{A}^{2}\left(F\vert_{W_{\R}}\right)=\Td\left(F\vert _{W}\right)\Td\left(-F\vert_{W}\right).
\end{equation}
\end{theorem}
\begin{proof}
	We will give a self-contained proof of (\ref{eq:dee14}). Let 
	$ \mathfrak p_{0,\xi}$ be the Fourier transform of $\mathfrak p_{0}$ 
	in the variable $U$. If $\xi\in W^{*}_{\R} \simeq W_{\R}$, then
	\begin{equation}\label{eq:dee14a1}
 \mathfrak 
 p_{0,\xi}=-\frac{1}{2}\Delta^{\widehat{W}_{\R}}+\n^{\widehat{W}_{\R}}_{FY}+
2i\pi\left\langle \xi,Y \right\rangle.
\end{equation}
Assume that $F$ is invertible. Put
\begin{equation}\label{eq:dee14a2}
\mathfrak q_{0,\xi}=e^{2i\pi \left\langle \xi,F^{-1}Y 
\right\rangle}\mathfrak p_{0,\xi}e^{-2i\pi\left\langle  
\xi,F^{-1}Y\right\rangle}.
\end{equation}
Since $F$ is skew-adjoint, we get
\begin{equation}\label{eq:dee14a3}
\mathfrak 
q_{0,\xi}=-\frac{1}{2}\Delta^{\widehat{W}_{\R}}+\n^{\widehat{W}_{\R}}_{FY-2i\pi 
 F^{-1}\xi}+2\pi^{2}\left\vert  
F^{-1}\xi\right\vert^{2}
\end{equation}

Let $\mathfrak r_{0,\xi}$ be the operator deduced from $\mathfrak q_{0,\xi}$ 
by the translation $Y\to Y+2i\pi F^{-2}\xi$. 
Then
\begin{equation}\label{eq:dee14a4}
\mathfrak 
r_{0,\xi}=-\frac{1}{2}\Delta^{\widehat{W}_{\R}}+\n^{\widehat{W}_{\R}}_{FY}+2\pi^{2}\left\vert   F^{-1}\xi\right\vert^{2}.
\end{equation}
Since $F$ is skew-adjoint, the operators in the right-hand side of (\ref{eq:dee14a4})   are 
commuting, and the smooth kernel $\mathbf{r}_{0,\xi}\left(Y,Y'\right)$ for 
$\exp\left(-\mathfrak r_{0,\xi}\right)$ with respect to 
$\frac{dY'}{\left(2\pi\right)^{n}}$ is given by
\begin{equation}\label{eq:dee14a5}
\mathbf{r}_{0,\xi}\left(Y,Y'\right)=\exp\left(-2\pi^{2}\left\vert  
F^{-1}\xi\right\vert^{2}-\frac{1}{2}\left\vert  
e^{-F}Y-Y'\right\vert^{2}\right).
\end{equation}
By (\ref{eq:dee14a5}), if $\mathbf{q}_{0,\xi}\left(Y,Y'\right)$ is the smooth 
kernel for $\exp\left(-\mathfrak q_{0,\xi}\right)$ with respect to 
$\frac{dY'}{\left(2\pi\right)^{n}}$, then
\begin{multline}\label{eq:dee14a6}
\mathbf{q}_{0,\xi} \left(Y,Y'\right)=\exp\left(-2\pi^{2}\left\vert  
F^{-1}\xi\right\vert^{2} \right) \\
\exp \left(  -\frac{1}{2}\left\vert  e^{-F}\left( Y-2i\pi F^{-2}\xi 
\right) -\left( Y'-2i\pi F^{-2}\xi \right) \right\vert^{2} \right) .
\end{multline}

By (\ref{eq:dee14a6}), we deduce that if $1-e^{-F}$ is invertible, 
then
\begin{equation}\label{eq:dee14a7}
\int_{\widehat{W}_{\R}}^{}\mathbf{q}_{0,\xi}\left(Y,Y\right)\frac{dY}{\left(2\pi\right)^{n}}=\frac{\exp\left(-2\pi^{2}\left\vert  F^{-1}\xi\right\vert^{2}\right)
}{\det\left(1-e^{-F}\right)\vert _{W_{\R}}}.
\end{equation}
Using (\ref{eq:dee14a7}), we deduce that if $1-e^{-F}$ is invertible, 
then
\begin{equation}\label{eq:dee14a8}
\int_{\widehat{W}_{\R}}^{}s_{0}\left(\left(0,Y\right),\left(0,Y\right)\right)
\frac{dY}{\left(2\pi\right)^{n}}=\det\left(\frac{F}{1-e^{-F}}\right)\vert_{W_{\R}},
\end{equation}
which is equivalent to (\ref{eq:dee14}). The proof of our theorem is completed. 
\end{proof}
\begin{remark}\label{rem:calc}
	Using (\ref{eq:dee14a6}) and inverting the Fourier transform, one 
	can give an exact formula for the kernel 
	$s_{0}\left(\left(U,Y\right),\left(U',Y'\right)\right)$.
\end{remark}
\subsection{The time parameter}%
\label{subsec:sepr}
Let $\mathsf A_{Y}$ be the superconnection
\begin{equation}\label{eq:coir1}
\mathsf{A}_{Y}=\widehat{\n}^{\mathbf{F},-1/2}_{\mathrm{a}}
-q^{*}\pa^{X}i\omega^{X} 
+\overline{\pa}^{V}+i_{y}+\overline{\pa}^{V*}+i_{\overline{y}}+\overline{y}_{*}\we.
\end{equation}
By (\ref{eq:tot1b}), we get
\begin{equation}\label{eq:coeq3}
\mathcal{A}_{Y}=\mathsf A_{Y}+C.
\end{equation}

First, we extend the constructions of \cite[Section 7.1]{Bismut10b}.
\begin{definition}\label{def:dAt}
	Let 
\index{AYt@$\mathcal{A}_{Y,t}$}%
	$\mathcal{A}_{Y,t}$ be the superconnection $\mathcal{A}_{Y}$ 
 associated with the metrics 
$\left( \omega^{X}/t,g^{D},g^{\widehat{TX}}/t^{3}\right) $. 
\end{definition}

 In the sequel,  $N,N^{V},N^{H},N^{\widehat{V}}$ denote the total 
number operators\footnote{Here, the number operators are the 
classical sums of antiholomorphic and holomorphic degrees.} of $\Lambda\left(T^{*}_{\C}M\right), 
\Lambda\left(T^{*}_{\C}X\right), \Lambda\left(T^{*}_{\C}S\right), 
\Lambda\left(\overline{\widehat{T^{*}X}}\right)$.  Recall that for 
$a>0$,
\index{Ka@$K_{a}$}%
$K_{a}$ was defined in (\ref{eq:club41}).
We will establish an extension of \cite[Proposition 
7.1.5]{Bismut10b}.
\begin{proposition}\label{prop:coeq}
	For $t>0$, the following identity holds:
	\begin{multline}\label{eq:coeq3a}
t^{3N^{\widehat{V}}/2+N^{V}/2}K_{t}\mathcal{A}_{Y,t}
K_{t}^{-1}t^{-3N^{\widehat{V}}/2-N^{V}/2}\\
=t^{-N^{H}/2}\left(\sqrt{t}\mathsf 
A_{Y}+t^{N/2}Ct^{-N/2}\right)t^{N^{H}/2}.
\end{multline}
\end{proposition}
\begin{proof}
	We split $\mathcal{A}_{Y,t}$ as in (\ref{eq:coeq3}), i.e., 
	\begin{equation}\label{eq:coeq4}
\mathcal{A}_{Y,t}=\mathsf A_{Y,t}+C.
\end{equation}
The same arguments as in \cite[Proposition 7.1.5]{Bismut10b}, or an easy explicit 
computation using (\ref{eq:tot1b}) show that
\begin{equation}\label{eq:coeq5}
t^{3N^{\widehat{V}}/2+N^{V}/2}K_{t}\mathsf{A}_{Y,t}K_{t}^{-1}
t^{-3N^{\widehat{V}}/2-N^{V}/2}=t^{-N^{H}/2}\sqrt{t}\mathsf 
A_{Y}t^{N^{H}/2}.
\end{equation}
By (\ref{eq:coeq3}), (\ref{eq:coeq5}),  since $N^{H}+N^{V}=N$, we get 
 (\ref{eq:coeq3a}). The proof of our proposition is completed. 
\end{proof}

Let 
\index{AtY@$\mathcal{A}^{t}_{Y}$}%
$\mathcal{A}^{t}_{Y}$ be the superconnection 
$\mathcal{A}_{Y}$ for the metrics 
$\left(\omega^{TX}/t, g^{D},g^{\widehat{TX}}\right)$. The same 
argument as in the proof of (\ref{eq:triv-6}) shows that
\begin{equation}\label{eq:bor1}
\mathcal{A}^{t}_{t^{3/2}Y}=\underline{\delta}^{*}_{t^{3/2}}\mathcal{A}_{Y,t}
\underline{\delta}^{*-1}_{t^{3/2}}.
\end{equation}
By (\ref{eq:coeq3a}), (\ref{eq:bor1}), we obtain
\begin{equation}\label{eq:bor1a1}
K_{1/\sqrt{t}}t^{N^{V}/2}\mathcal{A}^{t}_{t^{3/2}Y}t^{-N^{V}/2}K_{\sqrt{t}}=
t^{-N^{H}/2}\left( \sqrt{t}\mathsf A_{Y}+t^{N/2}Ct^{-N/2} \right) 
t^{N^{H}/2}.
\end{equation}

Set
\index{Mt@$\mathfrak M_{t}$}%
\begin{equation}\label{eq:bor2}
\mathfrak M_{t}=\mathcal A^{t,2}_{t^{3/2}Y}.
\end{equation}

When $\omega^{X}$ is replaced by $\omega^{X}/t$, the connections 
${}^{0}\widehat{\n}^{\mathbb F},{}^{1}\widehat{\n}^{\mathbb F}$ 
defined in (\ref{eq:deffbr}), (\ref{eq:sup1})  depend on $t$, and will be denoted 
\index{nFt@${}^{0}\widehat{\n}^{\mathbb F}_{t}$}%
\index{nFt@${}^{1}\widehat{\n}^{\mathbb F}_{t}$}%
${}^{0}\widehat{\n}^{\mathbb F}_{t},{}^{1}\widehat{\n}^{\mathbb F}_{t}$. A similar 
notation will be used for the connections on $\mathbf{F}$.

By  Theorems \ref{thm:Tcurvhyp} and \ref{thm:Tcurvhypbi}, we get
\begin{multline}\label{eq:curbhypbis1}
    \mathfrak M_{t}=-\frac{1}{2}\Delta^{V}_{g^{\widehat{TX}}}
+\frac{t^{2}}{2}\left\vert  
Y\right\vert^{2}_{g^{TX}}+\overline{\widehat{w}}^{i}\we\left(t^{3/2}i_{\overline{w}_{i}}
+t^{1/2}\overline{w}_{i*}\right)
-t^{3/2}i_{\overline{\widehat{w}}_{i}}i_{w_{i}}\\
+t^{3/2}{}^{1}\widehat{\n}^{\mathbf{F}}_{t,Y}+t^{3/2}i_{Y}C
-q^{*}\overline{\pa}^{X}\pa^{X}i\omega^{X}/t\\
-\n^{V}_{R^{\widehat{TX}}\widehat{Y}}-
\left\langle  
R^{\widehat{TX}}\overline{\widehat{w}}_{i},\overline{\widehat{w}}^{j}\right\rangle\overline{\widehat{w}}^{i}
i_{\overline{\widehat{w}}_{j}}+A^{E_{0},2}.
\end{multline}

Observe that $\underline{\delta}^{*-1}_{t}\mathcal 
A^{t}_{t^{3/2}Y}\underline{\delta}^{*}_{t}$ is the superconnection 
$\mathcal{A}_{\sqrt{t}Y}$ associated with the metrics 
$\left(\omega^{X}/t,g^{D},g^{\widehat{TX}}/t^{2}\right)$.

Set
\index{Mt@$\mathfrak  M'_{t}$}%
\begin{equation}\label{eq:cup1}
\mathfrak  M'_{t}=\underline{\delta}^{*-1}_{t}\mathfrak M_{t} 
\underline{\delta}^{*}_{t}.
\end{equation}
By (\ref{eq:curbhypbis1}), we get
\begin{multline}\label{eq:cup2}
    \mathfrak M'_{t}=-\frac{t^{2}}{2}\Delta^{V}_{g^{\widehat{TX}}}
+\frac{1}{2}\left\vert  
Y\right\vert^{2}_{g^{TX}}+\overline{\widehat{w}}^{i}\we\left(t^{1/2}i_{\overline{w}_{i}}
+t^{-1/2}\overline{w}_{i*}\right)
-t^{5/2}i_{\overline{\widehat{w}}_{i}}i_{w_{i}}\\
+ t^{1/2}{}^{1}\widehat{\n}^{\mathbf{F}}_{t,Y}+
t^{1/2}i_{Y}C
-q^{*}\overline{\pa}^{X}\pa^{X}i\omega^{X}/t\\
-\n^{V}_{R^{\widehat{TX}}\widehat{Y}}-
\left\langle  
R^{\widehat{TX}}\overline{\widehat{w}}_{i},\overline{\widehat{w}}^{j}\right\rangle\overline{\widehat{w}}^{i}
i_{\overline{\widehat{w}}_{j}}+A^{E_{0},2}.
\end{multline}

Recall that  $N^{V}$ is the total number operator of 
$\Lambda\left(T^{*}_{\C}X\right)$. Put
\index{Mt@$\mathfrak M''_{t}$}%
\index{At@$A_{t}^{E_{0},2}$}%
\index{nFt@${}^{2}\n^{\mathbb F}_{t}$}%
\begin{align}\label{eq:cup3}
&\mathfrak M^{\prime \prime }_{t}= t^{N^{V}/2} \mathfrak M'_{t}
   t^{-N^{V}/2},
&A_{t}^{E_{0},2}=t^{N^{V}/2}A^{E_{0},2}t^{-N^{V}/2},\\
&{}^{2}\widehat{\n}^{\mathbb F}_{t}=t^{N^{V}/2}{}^{1}\widehat{\n}^{\mathbb F}_{t}
t^{-N^{V}/2},
&C_{t}=t^{N^{V}/2}Ct^{-N^{V}/2}. \nonumber 
\end{align}
Again, in (\ref{eq:cup3}), we may as well obtain a connection 
\index{nFt@${}^{2}\widehat{\n}^{\mathbf{F}}_{t}$}%
${}^{2}\widehat{\n}^{\mathbf{F}}_{t}$ on $\mathbf{F}$.

Using (\ref{eq:deffbr}), (\ref{eq:sup1}), and (\ref{eq:cup3}), we get
\begin{equation}\label{eq:cup3a1}
{}^{2}\widehat{\n}^{\mathbb F}_{t,U}=\widehat{\n}^{\mathbb 
F,-1}_{U}-i_{u}\pa^{X}i\omega^{X}+\overline{\gamma u_{*}}.\end{equation}
Since ${}^{2}\widehat{\n}^{\mathbb F}_{t}$ does not depend on $t$, we will 
write instead ${}^{2}\widehat{\n}^{\mathbb F}$, and also obtain a corresponding 
connection  
\index{nF@${}^{2}\widehat{\n}^{\mathbf{F}}$}%
${}^{2}\widehat{\n}^{\mathbf{F}}$ on $\mathbf{F}$.

By (\ref{eq:cup2}), we obtain
\begin{multline}\label{eq:cup3a2}
   \mathfrak M^{\prime \prime }_{t}=-\frac{t^{2}}{2}\Delta^{V}_{g^{\widehat{TX}}}
+\frac{1}{2}\left\vert  
Y\right\vert^{2}_{g^{TX}}+\overline{\widehat{w}}^{i}\we\left(i_{\overline{w}_{i}}
+\overline{w}_{i*}\right)
-t^{2}i_{\overline{\widehat{w}}_{i}}i_{w_{i}}\\
+t^{1/2}{}^{2}\widehat{\n}^{\mathbf{F}}_{Y}-i_{Y}C_{t}
-tq^{*}\overline{\pa}^{X}\pa^{X}i\omega^{X}\\
-t\n^{V}_{R^{\widehat{TX}}\widehat{Y}}-
t\left\langle  
R^{\widehat{TX}}\overline{\widehat{w}}_{i},\overline{\widehat{w}}^{j}\right\rangle\overline{\widehat{w}}^{i}
i_{\overline{\widehat{w}}_{j}}+A_{t}^{E_{0},2}.
\end{multline}
Observe that in (\ref{eq:cup3a2}), there are no longer diverging 
terms as $t\to 0$.

Let
\index{dx@$\widehat{dx}$}%
$\widehat{dx}$ be the volume form on $X$ with respect to the 
metric $g^{\widehat{TX}}$. Given $s\in 
	S,t>0$, let 
	\index{Pt@$P_{t}\left(z,z'\right)$}%
	\index{Pt@$P''_{t}\left(z,z'\right)$}%
	$P_{t}\left(z,z'\right), 
	P''_{t}\left(z,z'\right)$\footnote{Again, we do not write 
	explicitly the 
	dependence of these kernels on $s$.} be the smooth kernel 
	for $\exp\left(-\mathfrak  M_{t}\right),\exp\left(-\mathfrak  M''_{t}\right)$ with respect to the volume 
	$dv_{\mathcal{X}}=\widehat{dx}dY/\left(2\pi\right)^{2n}$ on $\mathcal{X}_{s}$. 
	
	Let 
	\index{dx@$d\left(x,x'\right)$}%
	$d\left(x,x'\right)$ denote the Riemannian distance on $X$ with 
	respect to $\widehat{g}^{TX}$.
	
	Here is a result inspired by \cite[Proposition 
	4.7.1]{BismutLebeau06a}.
	\begin{proposition}\label{prop:estf}
	There exist $m\in\N,c>0,C>0$ such that for $s\in S,t\in]0,1]$, if 
	$z=\left(x,Y\right),z'=\left(x',Y'\right)$, then
	\begin{equation}\label{eq:cup4}
\left\vert  
P''_{t}\left(z,z'\right)\right\vert\le\frac{C}{t^{m}}\exp\left(-c\left(\left\vert  Y\right\vert^{2}+\left\vert  Y'\right\vert^{2}+d^{2}\left(x,x'\right)/t\right) \right) .
\end{equation}
	\end{proposition}
	\begin{proof}
		Comparing equation (\ref{eq:cup3a2}) for $\mathfrak M''_{t}$ 
		with \cite[eq. (4.7.4)]{BismutLebeau06a}, and using the fact 
		the above matrix terms are uniformly bounded,  the proof is 
		  the 
		same as the proof of \cite[Proposition 
		4.7.1]{BismutLebeau06a}.
	\end{proof}
	
	Observe that
	\begin{equation}\label{eq:cup5}
\ch\left(\mathcal{A}''_{Y},\omega^{X}/t,g^{D},g^{\widehat{TX}}/t^{3}\right)=\varphi\int_{\mathcal{X}_{s}}^{}\Trs\left[P_{t}\left(z,z\right)\right]dv_{\mathcal{X}}\left(z\right).
\end{equation}
In (\ref{eq:cup5}), $P_{t}$ can be replaced by 
$P''_{t}$.

By proceeding as in \cite[Proposition 4.8.2]{BismutLebeau06a} and in 
\cite[Proposition 6.3]{Bismut06d}, one 
essential consequence of Proposition \ref{prop:estf} is that to 
evaluate the asymptotics of (\ref{eq:cup5}) as $t\to 0$, we may 
proceed \textit{locally} near any $x\in X$, and this uniformly on $M=X\times 
S$. This means that given $x\in X$, the asymptotics of 
$\int_{\widehat{T_{\R,x}X}}^{}\Trs\left[P_{t}\left(\left(x,Y\right),\left(x,Y\right)\right)\right]dY$ can be evaluated locally near $x$, and ultimately that $X$ can be suitably replaced by $T_{\R,x}X$. This shows that the asymptotics of (\ref{eq:cup5}) as $t\to 0$ can ultimately be obtained locally over $X$. 
\subsection{The case when $\overline{\pa}^{X}\pa^{X}\omega^{X}=0$}%
\label{subsec:cava}
\begin{theorem}\label{thm:sepr}
	If $\overline{\pa}^{X}\pa^{X}\omega^{X}=0$,  then
	\begin{equation}\label{eq:lim5-a}
\ch\left(\mathcal{A}^{\prime \prime }_{Y},\omega^{X}/t,g^{D}, 
g^{\widehat{TX}}/t^{3}\right)\to 
p_{*}\left[q^{*}\Td\left(TX,\widehat{g}^{TX}\right)\ch\left(A^{E_{0} 
\prime \prime },g^{D}\right)\right].
\end{equation}

	If $\overline{\pa}^{X}\pa^{X}\omega^{X}=0$, then
\begin{equation}\label{eq:lim5-b}
\cBC\left(Rp_{*}\mathscr 
E\right)=p_{*}\left[q^{*}\Td_{\mathrm{BC}}\left(TX\right)\cBC\left(\mathscr E\right)\right]\,\mathrm{in}\,H_{\mathrm{BC}}^{(=)}\left(S,\R\right).
\end{equation}
\end{theorem}
\begin{proof}
	We proceed as in \cite[Theorem 6.1]{Bismut06d} and  \cite[Theorem 9.1.1]{Bismut10b}. 	
	We start from equation (\ref{eq:cup5}) which we write in the 
	form,
	\begin{multline}\label{eq:form1x1}
\ch\left(\mathcal{A}''_{Y},\omega^{X}/t,g^{D},g^{\widehat{TX}}/t^{3}\right)\\
=
\varphi\int_{X}^{}\left[\int_{\widehat{T_{\R,x}X}}^{}\Trs\left[P_{t}\left(
\left(x,Y\right),\left(x,Y\right)\right)\right]\frac{dY}{\left(2\pi\right)^{n}}\right]
\frac{\widehat{dx}}{\left(2\pi\right)^{n}}.
\end{multline}

Put
\begin{equation}\label{eq:form2}
m_{t}\left(x\right)=\varphi \frac{1}{\left(2\pi\right)^{n}} \int_{\widehat{T_{\R,x}X}}^{}
\Trs\left[P_{t}\left( \left(x,Y\right),\left(x,Y\right) \right) 
\right]\frac{dY}{\left(2\pi\right)^{n}},
\end{equation}
so that
\begin{equation}\label{eq:form3x-1}
\ch\left(\mathcal{A}''_{Y},\omega^{X}/t,g^{D},g^{\widehat{TX}}/t^{3}\right)=
\int_{X}^{}m_{t}\left(x\right)\widehat{dx}.
\end{equation}

	For $\epsilon>0$ small enough, given  $x\in X$, 	 using geodesic coordinates with respect to the 
	metric $\widehat{g}^{TX}$ centered at $x$,  we can identify the open ball 
	$B^{T_{\R}X}\left(0,\epsilon\right)$ with the corresponding open 
	ball $B^{X}\left(x,\epsilon\right)$ in $X$. Also along the 
	geodesics based at $x$, we trivialize the vector bundle 
	$\widehat{TX}$ by parallel transport with respect to 
	$\n^{\widehat{TX}}$. We identify the total space of 
	$\mathcal{X}$ over $B^{X}\left(x,\epsilon\right)$ with 
	$B^{T_{\R,x}X}\left(0,\epsilon\right)\times \widehat{T_{\R,x}X}$. 
	Near $x$, we trivialize $\mathbb F$ by parallel transport along 
	geodesics with respect to the connection $\widehat{\n}^{\mathbb 
	F,-1}$. This connection is associated with the connections 
	$\widehat{\n}^{TX,-1},\n^{\widehat{TX}},\n^{D}$ on 
	$TX,\widehat{TX},D$. Our operator  
	$\mathfrak M_{t}$ acts now on sections of $\mathbb F_{x}$ over 
	$B^{T_{\R,x}X}\left(0,\epsilon\right)\times \widehat{T_{\R}X}$. 
	
	In principle, we should extend the restriction of the  operator 
	$\mathfrak  M_{t}$ to 
	$B^{T_{\R,x}X}\left(0,\epsilon/2\right)\times 
	\widehat{T_{\R,x}X}$ to the full $\mathbb H_{x}=T_{\R,x}X\times 
	\widehat{T_{\R,x}X}$. Since our coordinate system on $X$ near $x$ is not 
	holomorphic, we cannot simply view $\mathbb H_{x}$ as a complex manifold. Still because of the 
	spinor interpretation of $\Lambda\left(\overline{\TsX}\right)$, 
	we can construct this extension using the methods of 
	\cite[Section 6.6]{Bismut06d}. Ultimately, we get an operator 
	$\mathfrak N_{t,x}$
	 acting on $C^{\infty }\left(\mathbb 
	H_{x},\Lambda\left(T^{*}_{\C,s}S \right) \ho \mathbb F_{x}\right)$ which is still hypoelliptic, and which 
	coincides with $\mathfrak M_{t}$ over 
	$B^{T_{\R,x}X}\left(0,\epsilon/2\right)\times 
	\widehat{T_{\R,x}X}$. Let
	$Q_{t,x}\left(\left(U,Y\right),\left(U',Y'\right)\right)$ denote 
	its smooth kernel on $\mathbb H_{x}$
	 with respect to the volume 
	$\frac{dU'dY'}{\left(2\pi\right)^{2n}}$  By proceeding as in \cite[Proposition 
	4.8.2]{BismutLebeau06a} and in \cite[Theorem 6.1]{Bismut06d}, 
	from the structure of $\mathfrak M''_{t}$, when studying the 
	asymptotics as $t\to 0$ of $m_{t}\left(x\right)$, we may as well 
	replace $\mathfrak M_{t}$ by $\mathfrak N_{t,x}$. This means that 
	given $x\in X$, we may as well replace $m_{t}\left(x\right)$ by 
	\begin{equation}\label{eq:form4}
n_{t}\left(x\right)=\varphi\frac{1}{\left(2\pi\right)^{n}}\int_{\widehat{T_{\R,x}X}}^{}\Trs\left[Q_{t,x}\left(\left(0,Y\right),\left(0,Y\right)\right)\right]\frac{dY}{\left(2\pi\right)^{n}}.
\end{equation}

	For $a>0$, let 
$I_{a}$ be the map acting on smooth functions on 
$\mathbb H_{x}=T_{\R,x}X \times \widehat{T_{\R,x}X}$ with values in $ 
 p^{*}\Lambda\left(T^{*}_{\C,s}S\right) \ho \mathbb   F_{s,x}$ that is 
 given by
\begin{equation}
    I_{a}s\left(U,Y\right)=s\left(aU,Y\right).
    \label{eq:froid1}
\end{equation}
Set
\begin{equation}
    \mathfrak O_{t,x}=I_{t^{3/2}} \mathfrak N_{t,x}I_{t^{-3/2}}.
    \label{eq:froid2}
\end{equation}

Let $R_{t,x}\left(\left(U,Y\right),\left(U',Y'\right)\right)$ be the 
smooth kernel on $\mathbb H_{x}$ associated with the operator 
$\exp\left(-\mathfrak O_{t,x}\right)$ with respect to the volume 
$\frac{dU'dY'}{\left(2\pi\right)^{2n}}$. By (\ref{eq:froid2}), we 
deduce that
\begin{equation}\label{eq:veau1}
Q_{t,x}\left(\left(0,Y\right),\left(0,Y\right)\right)=t^{-3n}R_{t,x}\left(\left(0,Y\right),\left(0,Y\right)\right).
\end{equation}
By (\ref{eq:form4}), (\ref{eq:veau1}), we obtain
\begin{equation}\label{eq:veau2}
n_{t}\left(x\right)=\varphi\frac{1}{\left(2\pi\right)^{n}t^{3n}}\int_{\widehat{T_{\R,x}X}}^{}\Trs\left[R_{t,x}\left(\left(0,Y\right),\left(0,Y\right)\right)\right]\frac{dY}{\left(2\pi\right)^{n}}.
\end{equation}

We will use the same notation as in Subsection \ref{subsec:lialge}. 
Let $\mathfrak W$ be another copy of $TX$. Recall that  $w_{1},\ldots,w_{n}$ is an orthonormal basis  of 
$T_{x}X$ with respect to $g^{\widehat{TX}}$. Let $\mathfrak 
w_{1},\ldots,\mathfrak w_{n}$ 
be the corresponding  basis of $\mathfrak W_{x}$. Then $\mathfrak w_{i},\overline{ \mathfrak w}_{i},1\le 
i\le n$ 
generate the algebra $\Lambda\left(\mathfrak W_{\C,x}\right)$.

Note that
\begin{equation}\label{eq:veau3}
\Trs^{\Lambda\left(T^{*}_{\C,x}X\right)}\left[\prod_{1}^{n}w^{i}i_{w_{i}}\overline{w}^{i}i_{\overline{w}_{i}}\right]=1,
\end{equation}
and that the supertrace of the monomials in the above operators 
having length less than $2n$ vanishes. If $T\in\End \left( 
\Lambda\left(T^{*}_{\C,x}X\right) \right) $, $T$ can be 
written as a linear combination of normally ordered 
products\footnote{This means that the products of operators  
$i_{w_{i}},i_{\overline{w}_{j}}$ appear to the right of the products 
of operators
$w^{i},\overline{w}^{j}$.} of the 
$w^{i},i_{w_{j}},\overline{w}^{k},i_{\overline{w}_{\ell}}$. 
If $T$ is a normally ordered monomial, its action on $1\in 
\Lambda\left(T^{*}_{\C,x}\right)$ is trivial unless it does not 
contain operators $i_{w_{i}},i_{\overline{w}_{j}}$, in which case it is multiplication by 
a form in $\Lambda\left(T^{*}_{\C,x}X\right)$.

Set
\index{o@$\mathfrak o$}%
\begin{equation}
    \mathfrak o=\sum_{i=1}^{n}\left(w^{i}\mathfrak 
    w_{i}+\overline{w}^{i} \overline{\mathfrak w}_{i}\right).
    \label{eq:froid3}
\end{equation}
Then $\mathfrak o$ does not depend on the 
choice of the base $w_{1}\ldots,w_{n}$.

 Let 
$T_{t}\in 
\Lambda\left(\mathfrak 
W_{\C,x}\right)\ho\End\left(\Lambda\left(T^{*}_{\C,x}X\right)\right)$ be given by
\begin{equation}\label{eq:veau5}
T_{t}=\exp\left(-\mathfrak o/t^{3/2}\right)T\exp\left(\mathfrak 
o/t^{3/2}\right).
\end{equation}
Then $T_{t}$ is obtained from 
   $T$ by replacing
the operators $i_{w_{i}},i_{\overline{w}_{i}}$ by 
	$i_{w_{i}}+\mathfrak 
    w_{i}/t^{3/2},i_{\overline{w}_{i}}+\overline{ \mathfrak 
    w}_{i}/t^{3/2}$, while leaving unchanged the 
	$w^{i},\overline{w}^{i}$. Note that $T_{t}1\in 
	\Lambda\left(\mathfrak W_{\C,x}\right)\ho\Lambda\left(T^{*}_{\C,x}X\right)$.	
	
	Note that $\beta=\prod_{1}^{n}w^{i} \mathfrak w_{i}$ is a canonical 
	section  of the trivial line bundle 
	$\lambda=\Lambda^{n}\left(T^{*}X\right)\ho\Lambda^{n}\left(\mathfrak W\right)$, 
	and that $\overline{\beta}=\prod_{1}^{n}\overline{w}^{i} \overline{\mathfrak w}_{i}$ is the 
	corresponding conjugate section of $\overline{\lambda}$.
	
	Let $\left[T_{t}1\right]^{\max}\in\C$ be the coefficient of $\beta\overline{\beta}$ in the 
	expansion of $T_{t}1$. By (\ref{eq:veau3}), using the previous 
	considerations, we conclude that
	\begin{equation}\label{eq:veau6}
\Trs^{\Lambda\left(T^{*}_{\C,x}X\right)}\left[T\right]=t^{3n}\left[T_{t}1\right]^{\max}.
\end{equation}

If instead, $T$ is a section of 
$\End\left(\Lambda\left(T^{*}_{\C,x}X\right)\right)\ho\End\left(\Lambda\left(\overline{\widehat{T^{*}X_{x}}}\right)\ho D_{s,x}\right)$, we  define $T_{t}$ exactly as before. Equation (\ref{eq:veau6}) takes the form
\begin{equation}\label{eq:veau7}
\Trs^{\Lambda\left(T^{*}_{\C,x}X\right)\ho\Lambda\left(\overline{\widehat{T^{*}_{x}X}}\right)\ho D_{s,x}}\left[T\right]=
t^{3n}\Trs^{\Lambda\left(\overline{\widehat{T^{*}_{x}X}}\right)\ho 
D_{s,x}}\left[\left[T_{t}1\right]^{\max}\right].
\end{equation}

	Put
\begin{equation}\label{eq:froid4}
    \mathfrak P_{t,x}=\exp\left(-\mathfrak o/t^{3/2}\right) 
    \mathfrak 
    O_{t,x}\exp\left(\mathfrak 
    o/t^{3/2}\right).
    \end{equation}
	Let $S_{t,x}\left(\left(U,Y\right),\left(U',Y'\right)\right)$ be 
	the smooth kernel for the operator $\exp\left(-\mathfrak 
	P_{t,x}\right)$ with respect to the volume 
	$\frac{dU'dY'}{\left(2\pi\right)^{2n}}$.
	By (\ref{eq:veau2}), (\ref{eq:veau7}), and (\ref{eq:froid4}), we 
	get
	\begin{equation}\label{eq:veau8}
n_{t}\left(x\right)=\varphi\frac{1}{\left(2\pi\right)^{n}}\int_{\widehat{T_{\R,x}X}}^{}\Trs^{\Lambda\left(\overline{\widehat{T^{*}_{x}X}}\right)\ho 
D_{s,x}}\left[\left[S_{t,x}\left(\left(0,Y\right),\left(0,Y\right)\right)1\right]^{\max}\right]\frac{dY}{\left(2\pi\right)^{n}}.
\end{equation}

		Let 
      $\Delta^{V}_{g^{\widehat{TX}}}$ be the 
      Laplacian along the fibre $\widehat{T_{\R,x}X}$ with respect 
      to the metric $g^{\widehat{TX}}$.
       Let 
      $\n_{Y}^{T_{\R,x}X}$ 
     be the obvious  differentiation operator along the fibre
     $T_{\R,x}X$.

Put
\index{Po@$\mathfrak P_{0,s,x}$}%
    \begin{multline}\label{eq:froid4x1}
\mathfrak P_{0,s,x}= 
-\frac{1}{2}\Delta^{V}_{g^{\widehat{TX}}}+
\overline{\widehat{w}}^{i} \overline{\mathfrak 
w}_{i}-i_{\overline{\widehat{w}}_{i}} \mathfrak w_{i}-
\n^{V}_{R^{\widehat{TX}}Y}\\
-\left\langle  
R_{x}^{\widehat{TX}}\overline{\widehat{w}}_{i},\overline{\widehat{w}}^{j}\right\rangle\overline{\widehat{w}}^{i}
i_{\overline{\widehat{w}}_{j}}+\n^{T_{\R,x}X}_{Y}+A^{E_{0},2},
\end{multline}
the tensors are evaluated at $\left(s,x\right)$. 

We claim that as in \cite[Theorem 6.7]{Bismut06d} and in 
\cite[Theorem 9.1.1]{Bismut10b}, as $t\to 0$, 
\begin{equation}
    \mathfrak P_{t,x}\to \mathfrak P_{0,s,x},
    \label{eq:froid5}
\end{equation}
in the sense that the coefficients of the operator $\mathfrak P_{t,x}$ 
converge uniformly together with all
their derivatives uniformly over compact sets towards the 
corresponding coefficients of $\mathfrak P_{0,x}$. To establish 
(\ref{eq:froid5}), we will use equation (\ref{eq:curbhypbis1}) for 
$\mathfrak M_{t}$, and we will exploit the fact that 
$\overline{\pa}^{X}\pa^{X}\omega^{X}=0$.

The first two terms in the right-hand side of (\ref{eq:curbhypbis1}) 
do not cause any problem. Using the considerations that follow 
(\ref{eq:veau5}), the limit as $t\to 0$ of the contribution of the 
last two terms in the first line of
(\ref{eq:curbhypbis1}) is given by $\overline{\widehat{w}}^{i} \overline{\mathfrak 
w}_{i}-i_{\overline{\widehat{w}}_{i}} \mathfrak w_{i}$.	

By (\ref{eq:deffbr}), (\ref{eq:sup1}), we get
\begin{equation}\label{eq:dee2}
t^{3/2}{}^{1}\widehat{\n}^{ \mathbf{F}}_{t,Y}=t^{3/2}\widehat{\n}^{\mathbf{F},-1}_{Y}+t^{1/2}\left(- i_{y}\pa^{X}i\omega^{X}+\overline{\gamma y_{*}} 
\right) .
\end{equation}

In our given trivialization, at $\left(U,Y\right)\in \mathbb H_{x}$, the section of $T_{\R,x}X$ that 
corresponds to the section $Y$ of $T_{\R}X$ is just the tautological 
section $Y\in T_{\R,x}X$. Let $\Gamma^{\widehat{TX}}$ denote the 
connection form for $\widehat{\n}^{TX}$ in the given trivialization. 
In the above coordinate system the vector field $Y^{H}$, as a section 
of $\mathbb H_{x}$,  splits as 
\begin{equation}\label{eq:vf1}
Y^{H}\left(U,Y\right)=Y \oplus \left( 
-\Gamma_{U}^{\widehat{TX}}\left(Y\right)Y \right).
\end{equation}

Recall that $TX,\widehat{TX}$ have been trivialized by parallel 
transport with respect to $\widehat{\n}^{TX,-1},\n^{\widehat{TX}}$. Let 
$\widehat{\Gamma}^{TX,-1}$ denote the connection form 
for  $\widehat{\n}^{TX,-1}$ in  this trivialization of $TX$. The  
lift $\widehat{\Gamma}^{\Lambda\left(T^{*}_{\R}X\right),-1}$ of  $\widehat{\Gamma}^{TX,-1}$ to 
$\Lambda\left(T^{*}_{\R}X\right)$ is given by
\begin{equation}\label{eq:dee3a1}
\widehat{\Gamma}^{\Lambda\left(T^{*}_{\R}X\right),-1}=
-\left\langle  \widehat{\Gamma}^{TX,-1}w_{i},w^{j}\right\rangle w^{i} i_{w_{j}}
-\left\langle  
\overline{\widehat{\Gamma}}^{TX,-1}\overline{w}_{i},\overline{w}^{j}\right\rangle\overline{w}^{i}i_{\overline{w}_{j}}.
\end{equation}
Also, near $U=0$, 
\begin{equation}\label{eq:dee4}
\widehat{\Gamma}^{TX,-1}_{U}=\mathcal{O}\left(U\right).
\end{equation}

Let $d^{T_{\R,x}X}$ be the de Rham operator on $T_{\R,x}X$. Let 
$\Gamma^{\Lambda\left(\overline{\widehat{\TsX}}\right)},\Gamma^{D}$ 
be the connection forms associated with 
$\n^{\Lambda\left(\overline{\widehat{\TsX}}\right)},\n^{D}$ in the 
given trivializations. By (\ref{eq:dee3a1}), we get
\begin{multline}\label{eq:dee5}
\widehat{\n}^{\mathbb F,-1}=d^{T_{\R,x}X}-\left\langle \widehat{\Gamma}^{TX,-1}w_{i},w^{j}\right\rangle w^{i} i_{w_{j}}
-\left\langle  
\overline{\widehat{\Gamma}}^{TX,-1}\overline{w}_{i},\overline{w}^{j}\right\rangle\overline{w}^{i}i_{\overline{w}_{j}}\\
+\Gamma^{\Lambda\left(\overline{\widehat{\TsX}}\right)}+\Gamma^{D}.
\end{multline}

By (\ref{eq:vf1}), (\ref{eq:dee5}), we get
\begin{multline}\label{eq:dee5ab1}
\widehat{\n}^{\mathbf{F},-1}_{Y}=\widehat{\n}^{\mathbb 
F,-1}_{Y^{H}}=\n^{T_{\R,x}X}_{Y}-\n^{\widehat{T_{\R,x}X}}_{\Gamma_{U}^{\widehat{TX}}\left(Y\right)Y}\\
-\left\langle \widehat{\Gamma}^{TX,-1}_{U}\left(Y\right)w_{i},w^{j}\right\rangle w^{i} i_{w_{j}}
-\left\langle  
\overline{\widehat{\Gamma}}^{TX,-1}_{U}\left(Y\right)\overline{w}_{i},\overline{w}^{j}\right\rangle\overline{w}^{i}i_{\overline{w}_{j}}\\
+\Gamma^{\Lambda\left(\overline{\widehat{\TsX}}\right)}_{U}\left(Y\right)+\Gamma^{D}_{U}\left(Y\right).
\end{multline}

By (\ref{eq:dee5}), we get
\begin{multline}\label{eq:dee6}
I_{t^{3/2}}t^{3/2}\widehat{\n}^{\mathbf{F},-1}_{Y}I_{t^{-3/2}}=\n^{T_{\R,x}X}_{Y}-t^{3/2}
\n^{\widehat{T_{\R,x}X}}_{\Gamma_{t^{3/2}U}^{\widehat{TX}}\left(Y\right)Y}\\
-\left\langle  \widehat{\Gamma}_{t^{3/2}U}^{TX,-1}\left(Y\right)w_{i},w_{j}\right\rangle 
t^{3/2} w^{i} i_{w_{j}}
-\left\langle  
\overline{\widehat{\Gamma}}_{t^{3/2}U}^{TX,-1}\left(Y\right)\overline{w}_{i},\overline{w}_{j}\right\rangle t^{3/2}\overline{w}^{i}i_{\overline{w}_{j}}\\
+t^{3/2}\Gamma_{t^{3/2}U}^{\Lambda\left(\overline{\widehat{\TsX}}\right)}\left(Y\right)+t^{3/2}\Gamma^{D}_{t^{3/2}U}\left(Y\right).
\end{multline}

Let us now consider the effect on (\ref{eq:dee6}) of the conjugation 
by $\exp\left(-\mathfrak o/t^{3/2}\right)$, that only affects the 
third  and fourth terms in the right-hand side. As mentioned before, 
the terms 
$t^{3/2}w^{i}i_{w_{j}},t^{3/2}\overline{w}^{i}i_{\overline{w}_{j}}$ 
are replaced by $w^{i}\left(t^{3/2}i_{w_{j}}+ \mathfrak w_{j}\right),
\overline{w}^{i}\left(t^{3/2}i_{\overline{w}_{j}}+ 
\overline{\mathfrak w}_{j} \right) $. Because of 
(\ref{eq:dee4}), in  (\ref{eq:dee6}), these terms do not survive as 
$t\to 0$. 

Ultimately, we find that after all the conjugations and rescalings, as 
$t\to 0$, 
\begin{equation}\label{eq:dee7}
t^{3/2}\widehat{\n}^{\mathbf{F},-1}_{Y}\to 
\n^{T_{\R,x}X}_{Y}.
\end{equation}

Let us now consider the second term in the right-hand side of 
(\ref{eq:dee2}). This term is a section of 
$\Lambda\left(T^{*}_{\C}X\right)$. In the parallel transport 
trivialization with respect to $\widehat{\n}^{TX,-1}$, it remains so. Also it 
is unaffected by the conjugation by $\exp\left(-\mathfrak 
o/t^{3/2}\right)$. Combining (\ref{eq:dee2}) and (\ref{eq:dee7}), we 
find that as $t\to 0$, after conjugations and rescalings, 
\begin{equation}\label{eq:dee8}
t^{3/2}{}^{1}\widehat{\n}^{\mathbf{F}}_{t,Y}\to \n^{T_{\R,x}X}_{Y}.
\end{equation}

For the same reasons as before, as $t\to 0$, in 
(\ref{eq:curbhypbis1}),  the term $t^{3/2}i_{Y}C$ 
disappears. Also we made the assumption that 
$\overline{\pa}^{X}\pa^{X}\omega^{X}=0$. The last three terms in 
the right-hand side of (\ref{eq:curbhypbis1}) can be handled exactly 
as before. 

Putting together the previous considerations gives a proof of 
(\ref{eq:froid5}). 

Let 
\index{Sx@$S_{0,s,x}\left(\left(U,Y\right), \left(U',Y'\right)\right)$}%
$S_{0,s,x}\left(\left(U,Y\right), \left(U',Y'\right)\right)$ be the 
smooth kernel associated with $\exp\left(-\mathfrak P_{0,s,x}\right)$ 
with respect to the volume $\frac{dU'dY'}{\left(2\pi\right)^{2n}}$. 
By proceeding as in \cite[Theorem 4.10.1]{BismutLebeau06a}, 
there exist $C>0,c>0$ such that for $t\in]0,1],x\in X$,
\begin{equation}\label{eq:dee9}
\left\vert  
S_{t,x}\left(\left(0,Y\right),\left(0,Y\right)\right)\right\vert\le 
C\exp\left(-c\left\vert  Y\right\vert^{2/3}\right).
\end{equation}
By proceeding as in the same reference, and using (\ref{eq:froid5}), 
we find that as $t\to 0 $, 
\begin{equation}\label{eq:dee10}
S_{t,x}\left(\left(U,Y\right),\left(U',Y'\right)\right)\to 
S_{0,s,x}\left(\left(U,Y\right),\left(U',Y\right)\right).
\end{equation}

By (\ref{eq:veau8}), (\ref{eq:dee9}), and (\ref{eq:dee10}), we find 
that as $t\to 0$,  
\begin{multline}\label{eq:dee11}
n_{t}\left(x\right)\to 
n_{0,s}\left(x\right)\\
=\varphi\frac{1}{\left(2\pi\right)^{n}}
\int_{\widehat{T_{\R,x}X}}^{}\Trs^{\Lambda\left(\overline{\widehat{T^{*}_{x}X}}\right)\ho D_{s,x}}\left[\left[S_{0,s,x}\left(\left(0,Y\right),\left(0,Y\right)\right)1\right]^{\max}\right]\frac{dY}{\left(2\pi\right)^{n}}.
\end{multline}
By (\ref{eq:dee9}), there is $C>0$ such that for $x\in X,t\in 
]0,1]$, 
\begin{equation}\label{eq:dee12}
\left\vert  n_{t}\left(x\right)\right\vert\le C.
\end{equation}

By (\ref{eq:form3x-1}), (\ref{eq:dee11}), and (\ref{eq:dee12}), we 
conclude that as $t\to 0$,
\begin{equation}\label{eq:dee13a1}
\ch\left(\mathcal{A}^{\prime \prime }_{Y}, 
\omega^{X}/t,g^{\widehat{TX}}/t^{3}\right)\to 
\int_{X}^{}n_{0,s}\left(x\right)\widehat{dx}.
\end{equation}

We will evaluate the integrand in the right-hand side of 
(\ref{eq:dee11}). Observe that
\begin{equation}\label{eq:dee12a1}
\beta\overline{\beta}=\left(-1\right)^{n}\prod_{1}^{n}w^{i}\overline{w}^{i}
\prod_{1}^{n}\mathfrak w_{i}\overline{\mathfrak w}_{i}.
\end{equation}

From now on, we will view the
 $w^{i},\overline{w}^{i}$ as sections of  $\Lambda\left(T^{*}_{\C,x}X\right)$. Then
 \begin{equation}\label{eq:dee12a2}
\prod_{1}^{n}w^{i}\overline{w}^{i}=\left(-i\right)^{n}\widehat{dx}.
\end{equation}

If $\epsilon\in \Lambda\left(T^{*}_{\C}M\right)$,  let 
$\epsilon^{(2n)}\in \Lambda\left(T^{*}_{\C}M\right)$  be the component 
of $\epsilon$ which is of top vertical degree $2n$. Let 
$\epsilon^{\max}\in \Lambda\left(T^{*}_{\C}S\right)$ be  such that\begin{equation}\label{eq:dee12a3}
\epsilon^{(2n)}=\epsilon^{\max}\widehat{dx}.
\end{equation}
We can identity $\mathfrak w_{i}$ and 
$\overline{\mathfrak w}^{i}$ by the metric $\widehat{g}^{T_{x}X}$. 
Since $R^{\widehat{TX}}$ takes its values in skew-adjoint 
endomorphisms of $\widehat{TX}$, 
using Proposition \ref{prop:gra} and Theorem \ref{thm:tracl}, by 
(\ref{eq:dee11}), we get
\begin{equation}\label{eq:dee12a4}
n_{0,s}\left(x\right)=\varphi\frac{1}{\left(2i\pi\right)^{n}}\left[\widehat{A}^{2}\left(R^{\widehat{TX}}\right)\Td^{-1}\left(R^{\widehat{TX}}\right)
\Trs\left[\exp\left(-A^{E_{0},2}\right)\right]\right]_{s,x}^{\max}.
\end{equation}
From the last identity in (\ref{eq:dee14}), we get
\begin{equation}\label{eq:12a5}
\widehat{A}^{2}\left(R^{\widehat{TX}}\right)\Td^{-1}\left(R^{\widehat{TX}}\right)
=\Td\left(-R^{\widehat{TX}}\right).
\end{equation}

By (\ref{eq:iv21}), (\ref{eq:dee12a4}), and (\ref{eq:12a5}), we obtain
\begin{equation}\label{eq:12a6}
\int_{X}^{}n_{0}\left(x\right)\widehat{dx}=
p_{*}\left[q^{*}\Td \left( TX,\widehat{g}^{TX}\right)\ch\left(A^{E_{0} \prime \prime },g^{D}\right)\right].
\end{equation}
By combining (\ref{eq:dee13a1}) and (\ref{eq:12a6}), we get 
(\ref{eq:lim5-a}). Using Theorem \ref{thm:fun1} 
and (\ref{eq:lim5-a}), we get (\ref{eq:lim5-b}). The proof of our theorem is completed.  
\end{proof}
\section{Exotic superconnections and  Riemann-Roch-Grothendieck}%
\label{sec:RBC}
The purpose of this Section is to establish Theorem \ref{thm:sub}, 
which is the final step in the proof of our main result. More 
precisely, we define deformed hypoelliptic superconnections, which 
when introducing a scaling parameter $t>0$, have the proper asymptotics 
when $t\to 0$, and this without any assumption on the Kähler form. Most of 
the technical tools of local index theory which are needed were 
already developed in Section \ref{sec:hypodb}.

This Section is organized as follows. In Subsection 
\ref{subsec:exot}, following \cite[Chapter 11]{Bismut10b}, we 
introduce a deformation $\omega^{X}_{\theta}$ of the Kähler form 
$\omega^{X}$, that depends on $\theta=\left(c,d\right)$, with $c\in 
\left[0,1\right], d>0$, and coincides with $\omega^{X}$ for $c=0$.  This deformation $\omega^{X}_{\theta}$ also depends on 
the coordinate $Y\in \widehat{T_{\R}X}$. Given a holomorphic section 
$z$ of $\underline{q}^{*}TX$, we construct a superconnection 
$A_{\theta,Z}$.

In Subsection \ref{subsec:exocurv}, we give a Lichnerowicz formula 
for $A^{2}_{\theta,Z}$.

In Subsection \ref{subsec:scnAt}, we take $z=y$, in which case the 
curvature is still hypoelliptic. We define corresponding deformed 
hypoelliptic superconnection forms, and we show that their Bott-Chern 
class is the same as the class of our previous hypoelliptic 
superconnection forms.

In Subsection \ref{subsec:unib}, we introduce the parameter $t>0$, and 
we establish various scaling identities.

In Subsection \ref{subsec:unies}, we make $c=1$, and we state uniform 
estimates on the heat kernels associated with the curvature.

In Subsection \ref{subsec:exo}, we recall a result of 
\cite{Bismut10b} on the explicit construction   of hypoelliptic 
superconnection forms of vector bundles.

In Subsection \ref{subsec:limtbe}, when $c=1$ and with $d$ 
now depending explicitly on $t$, we compute the asymptotics of the 
deformed hypoelliptic superconnection forms.

Finally, in Subsection \ref{subsec:finsu}, we prove Theorem 
\ref{thm:sub}.

We make the same assumptions as in Sections \ref{sec:hypofo} and 
\ref{sec:hypodb}, and we 
use the corresponding notation.
\subsection{A  deformation of the K\"{a}hler form $\omega^{X}$}%
\label{subsec:exot}
Here, we follow \cite[Chapter 11]{Bismut10b}.
\begin{definition}\label{Dexokahl}
If  
\index{th@$\theta$}%
$\theta=\left(c,d\right),0\le c\le 1,d> 0$, put
\index{ot@$\omega^{X}_{\theta}$}%
\begin{equation}\label{eq:andr1}
\omega^{X}_{\theta}=\left( 1-c+\frac{cd}{2}\left\vert  
Y\right\vert_{g^{\widehat{TX}}}^{2}
\right) \omega^{X}.
\end{equation}
Equivalently,
\begin{equation}\label{eq:andr1x1}
\omega^{X}_{\theta}=\omega^{X}+c\left(\frac{d}{2}\left\vert  
Y\right\vert^{2}_{g^{\widehat{TX}}}-1\right)\omega^{X}.
\end{equation}
\end{definition}

We repeat the constructions of Section 
\ref{sec:hypo}, except that $\omega^{X}$ is now replaced by 
$\omega^{X}_{\theta}$. An extra subscript $\theta$ will be introduced to 
distinguish the objects  constructed here from the ones in 
 Section \ref{sec:hypo}.
Instead of (\ref{eq:coro4a2}), we have 
\begin{align}\label{eq:andr2}
&\mathcal{A}''_{Z}=\mathcal{A}''+i_{z}, 
&\mathcal{A}'_{Z,\theta}=e^{q^{*}i\omega^{X}_{\theta}}\left(\mathcal{A}'+i_{\overline{z}}\right)e^{-q^{*}i\omega^{X}_{\theta}},
\qquad \mathcal{A}_{Z,\theta}=\mathcal{A}''_{Z}+\mathcal{A}'_{Z,\theta}.
\end{align}

The following result was established in \cite[Theorem 
11.1.2]{Bismut10b}.
\begin{theorem}\label{Tformbis}
The following identities hold:
\begin{align}\label{eq:club28bis}
&\mathcal{A}'_{Z,\theta}=\mathcal{A}'_{Z}+c\left(\frac{d}{2}\left\vert  
Y\right\vert_{g^{\widehat{TX}}}^{2}-1\right)\left( 
\overline{z}_{*}-q^{*}\pa^{X}i\omega^{X} \right) 
+cdq^{*}i\omega^{X}i_{\overline{\widehat{y}}},\\
&\mathcal{A}_{Z,\theta}=\mathcal{A}_{Z}+c\left(\frac{d}{2}\left\vert  
Y\right\vert_{g^{\widehat{TX}}}^{2}-1\right)\left( 
\overline{z}_{*}-q^{*}\pa^{X}i\omega^{X} 
\right)+cdq^{*}i\omega^{X}i_{\overline{\widehat{y}}}.
 \nonumber 
\end{align}
\end{theorem}
\begin{proof}
This is a  consequence of (\ref{eq:pap1}), (\ref{eq:coro9}), and 
(\ref{eq:andr1x1}).
\end{proof}

As in (\ref{eq:retpek0r1}), by replacing $\omega^{X}$ by 
$\omega^{X}_{\theta}$, we can define  a new Hermitian form
\index{eXt@$\eta_{X,\theta}$}%
$\epsilon_{X,\theta}$
associated with $\omega^{X}_{\theta}$.  We obtain an analogue of 
Theorem \ref{thm:Tcadjoingt}. 
\subsection{A formula for $\mathcal{A}^{2}_{Z,\theta}$}
\label{subsec:exocurv}
We use the  notation of Section \ref{subsec:coc}. In 
particular $\widehat{w}_{1},\ldots,\widehat{w}_{n}$ denotes 
an orthonormal basis of  $\left(\widehat{TX},g^{\widehat{TX}}\right)$. Let
$\widehat{y}_{*}\in 
\overline{\widehat{\TsX}}$ be dual to $\widehat{y}$ with respect to the 
metric $g^{\widehat{TX}}$.

Now, we give a form of \cite[Theorem 11.2.1]{Bismut10b}. 
\begin{theorem}\label{Tcurvhypex}
The following identity holds:
\begin{multline}\label{eq:curbhypx0bis}
\mathcal{A}^{2}_{Z,\theta}=\mathcal{A}^{2}_{Z}+c\left(\frac{d}{2}\left\vert  Y\right\vert_{g^{\widehat{TX}}}^{2}
-1\right)\left( \frac{1}{2}
\left\vert  Z\right\vert^{2}_{g^{TX}}-i_{z}q^{*}\pa^{X}i\omega^{X} \right) 
+cdz_{*}i_{\overline{\widehat{y}}}\\
+c\left(\frac{d}{2}\left\vert  Y\right\vert^{2}_{g^{\widehat{TX}}}-1\right)
\left( \n^{TX \prime \prime ,H}\overline{z}_{*}-q^{*}\overline{\pa}^{X}\pa^{X}
i\omega^{X} \right) 
+cd \left( q^{*}\overline{\pa}^{X}i\omega^{X}\right) 
i_{\overline{\widehat{y}}}\\
+c\Biggl(d\widehat{y}_{*}\left(\overline{z}_{*}-q^{*}\pa^{X}i\omega^{X}\right)
+\left(\frac{d}{2}\left\vert  
Y\right\vert^{2}_{g^{\widehat{TX}}}-1\right)\overline{\widehat{w}}^{i}
\overline{\n_{\widehat{w}_{i}}z}_{*}\\
+dq^{*}i\omega^{X}\left(\n^{V}_{\overline{\widehat{y}}}+ 
\overline{\widehat{w}}^{i}i_{\overline{\widehat{w}}_{i}}\right) \Biggr) .
\end{multline}
\end{theorem}
\begin{proof}
	Since our fibration is product, and the metric $g^{TX}$ is 
	constant over $S$, the tensor $\sigma$ defined in 
	\cite[Definition 2.2.3]{Bismut10b} vanishes identically.  Using \cite[Theorem 
	11.2.1]{Bismut10b} and equation (\ref{eq:curbhypx0}), we get 
	(\ref{eq:curbhypx0bis}). The extra terms in the right-hand side 
	of (\ref{eq:curbhypx0bis}) do not depend on $\left(E_{0},A^{E_{0} 
	\prime \prime}\right)$, because $C=B+B^{*}$ contains only exterior 
	products. The proof of our theorem is completed. 
\end{proof}
\subsection{The superconnection $\mathcal{A}_{Y,\theta}$}%
\label{subsec:scnAt}
Here we take $z=y$. Using    (\ref{eq:sur1}), (\ref{eq:curbhypx0bis}),   we get 
\begin{multline}\label{eq:curbhypbis1x-1a}
    \mathcal{A}^{2}_{Y,\theta}= \mathcal{A}^{2}_{Y}
+c\left(\frac{d}{2}\left\vert  Y\right\vert_{g^{\widehat{TX}}}^{2}
-1\right)\left( \frac{1}{2}
\left\vert  
Y\right\vert^{2}_{g^{TX}}-i_{y}q^{*}\pa^{X}i\omega^{X} 
\right)
+cdy_{*}i_{\overline{\widehat{y}}}\\ 
+c\left(\frac{d}{2}\left\vert  Y\right\vert^{2}_{g^{\widehat{TX}}}-1\right)
\left(\overline{\gamma y}_{*} 
-q^{*}\overline{\pa}^{X}\pa^{X}i\omega^{X} \right) 
+cd \left(q^{*} \overline{\pa}^{X}i\omega^{X} \right) 
i_{\overline{\widehat{y}}}\\
+c\left(d\widehat{y}_{*}\left(\overline{y}_{*}-q^{*}\pa^{X}i
\omega^{X}\right)+\left(\frac{d}{2}\left\vert  Y\right\vert^{2}_{g^{\widehat{TX}}}
-1\right)\overline{\widehat{w}}^{i}\overline{w}_{i*}
\right)\\
+cdq^{*}i\omega^{X}\left( \n^{V}_{\overline{\widehat{y}}}+
\overline{\widehat{w}}^{i}i_{\overline{\widehat{w}}_{i}} \right).
\end{multline}
By (\ref{eq:curbhypbis1x-1a}), $\mathcal{A}^{2}_{Y,\theta}$ is  
fibrewise hypoelliptic. 

Set
    \begin{equation}
        V_{\theta}\left(Y\right)=
	\left(1-c\right)\frac{1}{2}\left\vert  
	Y\right\vert^{2}_{g^{TX}}+\frac{cd}{4}\left\vert  
	Y\right\vert^{2}_{g^{TX}}\left\vert  
	Y\right\vert^{2}_{g^{\widehat{TX}}}.
        \label{eq:wolf5}
    \end{equation}
	By (\ref{eq:curbhypbibi}), (\ref{eq:curbhypbe}), and (\ref{eq:curbhypbis1x-1a}),     the potential $V_{\theta}$ 
	appears in $\mathcal{A}^{2}_{Y,\theta}$. 
    Given $d>0$, there 
    exists $C_{d}>0$ such that for $c\in 
    \left[0,1\right]$,
    \begin{equation}
        V_{\theta}\left(Y\right)\ge 
	C_{d}\left(\left\vert  Y\right\vert^{2}_{g^{TX}}-1\right).
        \label{eq:wolf6x-1}
    \end{equation}
    The condition $d>0$ is 
    crucial for (\ref{eq:wolf6x-1}) to hold also at $c=1$.

 The  same arguments as in \cite[Section 11.3]{Bismut10b} and 
   (\ref{eq:curbhypbis1x-1a}), (\ref{eq:wolf6x-1}) show that   $\exp\left(-\mathcal{A}
	^{2}_{{Y,\theta}}\right)$ is fibrewise trace class. 

We use the same notation as in Subsection \ref{subsec:scnform}. Given 
$\mathbf{p}$, let  
\index{Rp@$\mathscr R_{\mathbf{p}}$}%
$\mathscr R_{\mathbf{p}}$ be the collection of parameters 
$\omega^{X},g^{D},g^{\widehat{TX}},\theta$ with $g^{\widehat{TX}}\in 
\mathbf{p}$. Let $d^{\mathscr R_{\mathbf{p}}}$ be 
the de Rham operator on $\mathscr R_{\mathbf{p}}$. Then 
$\epsilon_{X,\theta}^{-1}d^{\mathscr R_{\mathrm{p}}}\epsilon_{X,\theta}$ is a $1$-form 
on $\mathscr R_{\mathbf{p}}$ with values in $\epsilon_{X,\theta}$-self-adjoint 
endomorphisms.
\begin{definition}\label{def:scnnew}
	For $\theta=\left(c,d\right)$, $c\in \left[0,1\right],d>0$, put
	\begin{equation}\label{eq:hyne}
\ch\left(\mathcal{A}^{\prime \prime 
}_{Y},\omega^{X},g^{D},g^{\widehat{TX}},\theta\right)=\varphi\Trs\left[\exp\left(-\mathcal{A}^{2}_{Y,\theta}\right)\right].
\end{equation}
\end{definition}
\begin{theorem}\label{thm:Tdoublebibi}
The form 
$ \ch\left(\mathcal{A}^{ \prime \prime}_{Y}, 
\omega^{X},g^{D},g^{\widehat{TX}},\theta\right)$ 
lies in 
$\Omega^{(=)}\left(S,\R\right)$, it is closed,  and its  Bott-Chern cohomology 
class does not depend on $\omega^{X},g^{D},g^{\widehat{TX}},\theta$, 
or on the splitting of $E$. 

Given $\mathbf{p}$, $\varphi\Trs\left[\epsilon_{X,\theta}^{-1}d^{\mathscr R_{\mathrm{p}}}\epsilon_{X,\theta}\exp\left(- 
\mathcal{A}^{2}_{Y,\theta}\right)\right]$ is a $1$-form on $\mathscr 
R_{\mathbf{p}}$ with values in $\Omega^{(=)}\left(S,\R\right)$, and 
moreover,
\begin{equation}\label{eq:spr1a1}
d^{\mathscr R_{\mathrm{p}}}\ch\left(\mathcal{A}^{ \prime \prime}_{Y}, 
\omega^{X},g^{D},g^{\widehat{TX}},\theta\right)=-\frac{\overline{\pa}^{S}\pa^{S}}{2i\pi}\varphi\Trs\left[\epsilon_{X,\theta}^{-1}d^{\mathscr R_{\mathrm{p}}}\epsilon_{X,\theta}\exp\left(- 
\mathcal{A}^{2}_{Y,\theta}\right)\right].
\end{equation}
The  class of the forms $\ch\left(\mathcal{A}^{ \prime \prime}_{Y}, 
\omega^{X},g^{D},g^{\widehat{TX}},\theta\right)$ in Bott-Chern cohomology coincides with 
$\ch_{\mathrm{BC}}\left(\mathcal{A}^{\prime \prime }_{Y}\right)$.
\end{theorem}
\begin{proof}
The proof follows the same lines as the proof of \cite[Theorem 
11.3.1]{Bismut10b} and of Theorem 	\ref{thm:Tdouble}. The arguments 
of \cite{Bismut10b} show that they depend smoothly on all the 
parameters. Also for $c=0$, by (\ref{eq:club28bis}), we get
\begin{equation}\label{eq:spe1}
\ch\left(\mathcal{A}^{ \prime \prime}_{Y}, 
\omega^{X},g^{D},g^{\widehat{TX}},\theta\right)=\ch\left(\mathcal{A}^{ \prime \prime}_{Y}, 
\omega^{X},g^{D},g^{\widehat{TX}}\right),
\end{equation}
which gives the last statement in our theorem, and concludes its 
proof. 
\end{proof}
\subsection{The scaling identities}%
\label{subsec:unib}
For $c\in \left[0,1\right], d>0,t>0$, put
\index{t@$\theta_{t}$}%
\begin{equation}\label{eq:spli0}
\theta_{t}=\left(c,dt\right).
\end{equation}

As in (\ref{eq:coeq3}), by (\ref{eq:club28bis}), we have a splitting
\begin{equation}\label{eq:spli1r1}
\mathcal{A}_{Y,\theta}=\mathsf A_{Y,\theta}+C.
\end{equation}

Let 
\index{AYt@$\mathcal{A}_{Y,\theta,t}$}%
$\mathcal{A}_{Y,\theta,t}$ be the superconnection $\mathcal{A}_{Y,\theta}$ 
 associated with 
$$\omega^{X}/t,g^{D},g^{\widehat{TX}}/t^{3},\omega^{X}/t, \theta.$$

We have an analogue of \cite[eq. 
(11.3.12)]{Bismut10b} and of Proposition \ref{prop:coeq}.
\begin{proposition}\label{prop:idebis}
	The following identity holds:
	\begin{multline}\label{eq:joue0}
t^{3N^{\widehat{V}}/2+N^{V}/2}K_{t}\mathcal{A}_{Y,\theta_{t},t}K_{t}^{-1}
t^{-3N^{\widehat{V}}/2-N^{V}/2}\\
=t^{-N^{H}/2}\left(\sqrt{t}\mathsf 
A_{Y,\theta}+t^{N/2}C t^{-N/2}\right)t^{N^{H}/2}.
\end{multline}
\end{proposition}
\begin{proof}
By combining Proposition 	\ref{prop:coeq} and Theorem 
\ref{Tformbis}, we get (\ref{eq:joue0}). The proof of our proposition is completed.  
\end{proof}

Let 
\index{AtY@$\mathcal{A}^{t}_{Y,\theta}$}%
$\mathcal{A}^{t}_{Y,\theta}$  be the superconnection 
$\mathcal{A}_{Y,\theta}$ associated with 
$\omega^{X}/t,g^{D},g^{\widehat{TX}},\theta$. By combining  
(\ref{eq:bor1}) and (\ref{eq:club28bis}), we get
\begin{equation}\label{eq:clip1}
\mathcal{A}^{t}_{t^{3/2}Y,\theta}=\underline{\delta}^{*}_{t^{3/2}}\mathcal{A}_{Y,\theta,t}\underline{\delta}^{*-1}_{t^{3/2}}.
\end{equation}

By (\ref{eq:joue0}), (\ref{eq:clip1}),  we get the analogue of 
\cite[eq. (11.5.2)]{Bismut10b},
(\ref{eq:bor1a1}),
\begin{equation}\label{eq:clip2}
K_{1/\sqrt{t}}t^{N^{V}/2}\mathcal{A}^{t}_{t^{3/2}Y,\theta_{t}}t^{-N^{V}/2}K_{\sqrt{t}}=t^{-N^{H}/2}\left(\sqrt{t}\mathsf A_{Y,\theta}+t^{N/2}Ct^{-N/2}\right)t^{N^{H}/2}.
\end{equation}

We use the notation, 
\begin{equation}\label{eq:clip3}
\mathfrak  M_{\theta,t}=\mathcal{A}^{t,2}_{t^{3/2}Y,\theta}.
\end{equation}
Recall that $\mathfrak M_{t}$ was defined in (\ref{eq:bor2}) and 
is given by (\ref{eq:curbhypbis1}). Using 
(\ref{eq:curbhypx0bis}), we get
\begin{multline}\label{eq:clip4}
\mathfrak M_{\theta,t}=\mathfrak M_{t}
+c\left(\frac{d}{2}\left\vert  Y\right\vert_{g^{\widehat{TX}}}^{2}
-1\right)\left( \frac{t^{2}}{2}
\left\vert  
Y\right\vert^{2}_{g^{TX}}-t^{1/2}i_{y}q^{*}\pa^{X}i\omega^{X} \right) 
+cdt^{1/2}y_{*}i_{\overline{\widehat{y}}}\\
+c\left(\frac{d}{2}\left\vert  Y\right\vert^{2}_{g^{\widehat{TX}}}-1\right)
\left(t^{1/2}\overline{\gamma y}_{*} 
-q^{*}\overline{\pa}^{X}\pa^{X}i\omega^{X}/t \right) 
+cd \left(q^{*} \overline{\pa}^{X}i\omega^{X} \right) 
i_{\overline{\widehat{y}}}/t\\
+c\left(d\widehat{y}_{*}\left(t^{1/2}\overline{y}_{*}-q^{*}\pa^{X}i
\omega^{X}/t\right)+\left(\frac{d}{2}\left\vert  Y\right\vert^{2}_{g^{\widehat{TX}}}
-1\right)t^{1/2}\overline{\widehat{w}}^{i}\overline{w}_{i*}
\right)\\
+\frac{cd}{t}q^{*}i\omega^{X}\left( \n^{V}_{\overline{\widehat{y}}}+
\overline{\widehat{w}}^{i}i_{\overline{\widehat{w}}_{i}} \right).
\end{multline}
\subsection{The uniform estimates}%
\label{subsec:unies}
In the sequel, we will make $c=1$, so that for $d>0,t>0$, 
\begin{equation}\label{eq:sca1}
\theta_{t}=\left(1,dt\right).
\end{equation}

Put
\begin{equation}\label{eq:clip3r1}
\mathcal{M}_{Y,\theta,t}=K_{\left(dt\right)^{-1/4}}\left(\sqrt{t}\mathsf A_{Y,\theta}+t^{N/2}Ct^{-N/2}\right)^{2}K_{\left(dt\right)^{1/4}}.
\end{equation}

Recall that 
\index{dx@$d\left(x,x'\right)$}%
$d\left(x,x'\right)$ denotes the Riemannian distance on 
$X$ with respect to $\widehat{g}^{TX}$. 
Here is an important result,  which is an analogue of \cite[Theorem 
11.5.1]{Bismut10b}.
\begin{theorem}\label{thm:Tunif}
Given $k>0$, there exist $m\in \N,c>0,C>0$ such that for $ t\in 
]0,1],d\in \left[kt^{3},1\right], 
z=\left(x,Y\right),z'=\left(x',Y'\right)\in \mathcal{X}$, we have 
\begin{multline}\label{eq:wolf16}
\left\vert  \exp\left(- 
\mathcal{M}_{Y,\theta,t}\right)\left(z,z'\right)\right\vert\\
\le \frac{C}{t^{m}}
\exp\left(-c\left(\left\vert  Y\right\vert_{g^{\widehat{TX}}}^{2}
+\left\vert  Y'\right\vert_{g^{\widehat{TX}}}^{2}+\left(d/t^{3}\right)^{1/2}d^{2}\left(x,x'\right)\right)\right).
\end{multline}
\end{theorem}
\begin{proof}
	If we make $C=0$ \footnote{As the reader may have guessed, here 
	$C$ refers to the tensor defined in (\ref{eq:dep1}), and not  to the 
	constant appearing in our theorem.}, the proof of our theorem was already given in 
	\cite[Theorem 11.5.1]{Bismut10b}. Let us explain why the 
	introduction of the extra term containing $t^{N/2}Ct^{-N/2}$ in 
	(\ref{eq:clip3r1}) does not change 
	anything to the estimates.  Indeed this term remains uniformly 
	bounded. When taking the square in the right-hand side of 
	(\ref{eq:clip3r1}), this introduces, after scaling, a term which is 
	of the form  
	$\left(t/d\right)^{1/4}i_{Y}t^{N/2}Ct^{-N/2}$.  
In $C$, the component of total  degree $0$ in 
$\Lambda\left(T^{*}_{\C}X\right)$ is 
	necessarily disregarded, which introduces an extra factor which 
	is at least 
	$\sqrt{t}$. The factor that appears is then of the 
	order at most $\left(t^{3}/d\right)^{1/4}Y$. Under our 
	assumptions on $d$, $t^{3}/d$ remains uniformly bounded, so that 
	this term is of order $\left\vert  Y\right\vert$.  
	Under the indicated rescaling, a potential of the order of 
	$\left\vert  Y\right\vert^{4}$ appears in 
	$K_{\left(dt\right)^{-1/4}}t\mathsf 
	A^{2}_{Y,\theta}K_{\left(dt\right)^{1/4}}$. The term $\left\vert  
	Y\right\vert^{4}$ being much bigger than $\left\vert  
	Y\right\vert$, there is no difficulty in extending these 
	estimates to the present situation.
\end{proof}
\begin{remark}\label{rem:loca}
	Using (\ref{eq:clip2}) and (\ref{eq:clip3r1}), Theorem 
	\ref{thm:Tunif} will play  the role of Proposition 
	\ref{prop:estf}. 
\end{remark}
\subsection{The exotic  superconnection forms of 
vector bundles}%
\label{subsec:exo}
Let $Z$ be a complex manifold, and let $F$ be a holomorphic vector 
bundle on $Z$. 

Let $g^{F}$ be a Hermitian metric on $F$, let 
$\omega^{Z}$ be a  smooth real $(1,1)$-form on $Z$. 
In \cite[Definition 10.2.1]{Bismut10b},  hypoelliptic superconnection forms
$a\left(F,\omega^{Z},g^{F}\right)$\footnote{With the 
notation of \cite{Bismut10b}, these forms are instead 
denoted $a_{g}\left(b,d,g^{\widehat{F}}\right)$. Here, we make 
$g=1,b=1$. Also $d$ stands for a factor multiplying $\omega^{Z}$.} on 
$Z$ were defined,   that have the following properties:
\begin{itemize}
	\item  They depend smoothly on the parameters.

	\item  They lie in $\Omega^{(=)}\left(Z,\R\right)$, they are 
	closed, and their Bott-Chern cohomology class  does not depend on 
	the parameters.
	
	\item They are unchanged by constant rescaling of 
	$g^{F}$.
	\item The following identity holds:
	\begin{equation}\label{eq:exo1}
a\left(F,0,g^{F}\right)=\Td\left(F,g^{F}\right).
\end{equation}
\end{itemize}
If $a_{\mathrm{BC}}\left(F\right)$ is the common 
Bott-Chern class of the forms $a\left(F,\omega^{Z},g^{F}\right)$, then 
\begin{equation}\label{eq:exo2}
a_{\mathrm{BC}}\left(F\right)=\Td_{\mathrm{BC}}\left(F\right).
\end{equation}
\subsection{The limit as $t\to 0$ of the forms 
$\ch\left(\mathcal{A}^{\prime \prime 
}_{Y},\omega^{X}/t,g^{D},g^{\widehat{TX}}/t^{3},\theta_{t}\right)$}%
\label{subsec:limtbe}
Using  (\ref{eq:curbhypbis1}), 
(\ref{eq:clip4}),  we obtain a formula closely related to \cite[eq. 
(11.5.1)]{Bismut10b},  
\begin{multline}\label{eq:curbhypbis1x1}
     \mathfrak M_{\theta_{t},t}= 
   -\frac{1}{2}\Delta^{V}_{g^{\widehat{TX}}}+\frac{dt^{3}}{4}\left\vert  
    Y\right\vert^{2}_{g^{TX}}\left\vert  
    Y\right\vert^{2}_{g^{\widehat{TX}}}
    +t^{3/2}\left(\overline{\widehat{w}}^{i}\we 
    i_{\overline{w}_{i}}-i_{\overline{\widehat{w}}_{i}}i_{w_{i}}\right)\\
	+t^{3/2}{}^{1}\widehat{\n}^{\mathbf{F}}_{t,Y}+t^{3/2}i_{Y}C-\n^{V}_{R^{\widehat{TX}}\widehat{Y}}-
\left\langle  
R^{\widehat{TX}}\overline{\widehat{w}}_{i},\overline{\widehat{w}}^{j}\right\rangle\overline{\widehat{w}}^{i}
i_{\overline{\widehat{w}}_{j}}\\
+\frac{dt^{3/2}}{2}\left\vert  
    Y\right\vert^{2}_{g^{\widehat{TX}}}\overline{\widehat{w}}^{i}\overline{w}_{i*}-\frac{1}{2}\left\vert  
    Y\right\vert^{2}_{g^{\widehat{TX}}}q^{*}\overline{\pa}^{X}\pa^{X}id\omega^{X}\\
    +t^{1/2}\left(\frac{dt}{2}\left\vert  
    Y\right\vert^{2}_{g^{\widehat{TX}}}-1\right)\left(\overline{\gamma y}_{*}
    -i_{y}q^{*}\pa^{X}i\omega^{X}\right)+
    dt^{3/2}\left( 
    y_{*}i_{\overline{\widehat{y}}}+\widehat{y}_{*}\overline{y}_{*} 
    \right) \\
    + \left( q^{*} \overline{\pa}^{X}id\omega^{X} \right) 
    i_{\overline{\widehat{y}}}
    -\widehat{y}_{*}q^{*}\pa^{X}id\omega^{X} 
    + q^{*}id\omega^{X}\left( \n^{V}_{\overline{ \widehat{y}}}
    +\overline{\widehat{w}}^{i}i_{\overline{\widehat{w}}_{i}} 
	\right)+A^{E_{0},2} .
\end{multline}

Given $\rho\in]0,3[$, put
\index{qt@$\underline{\theta}_{t}$}%
\begin{equation}\label{eq:sca2}
\underline{\theta}_{t}=\left( 1,dt^{\rho+1}\right) .
\end{equation}

We establish an extension of \cite[Theorem 11.6.1]{Bismut10b}. We use 
the notation of Subsection \ref{subsec:exo}, with $Z=X$,  
$F=TX$, $\omega^{Z}=\omega^{X}, g^{F}=\widehat{g}^{TX}$.
\begin{theorem}\label{thm:fina1}
	As $t\to 0$, then
	\begin{align}\label{eq:co1bi}
&\ch\left(\mathcal{A}^{\prime \prime }_{Y}, 
\omega^{X}/t,g^{D},g^{\widehat{TX}}/t^{3},\theta_{t}\right)\to 
p_{*}\left[q^{*}a\left(TX,d\omega^{X},\widehat{g}^{TX}\right)\ch\left(A^{E_{0} \prime 
\prime},g^{D}\right)\right],\\
&\ch\left(\mathcal{A}^{\prime \prime }_{Y}, 
\omega^{X}/t,g^{D},g^{\widehat{TX}}/t^{3},\underline{\theta}_{t}\right)\to 
p_{*}\left[q^{*}\Td\left(TX,\widehat{g}^{TX}\right)\ch\left(A^{E_{0} \prime 
\prime},g^{D}\right)\right].\nonumber 
\end{align}
\end{theorem}
\begin{proof} 
	By (\ref{eq:hyne}), (\ref{eq:clip1}), and (\ref{eq:clip3}),   we get
\begin{equation}\label{eq:nank1}
\ch\left(\mathcal{A}^{\prime \prime }_{Y}, 
\omega^{X}/t,g^{D},g^{\widehat{TX}}/t^{3},\theta\right)=
\varphi\Trs\left[\exp\left(-\mathfrak M_{\theta,t}\right)\right].
\end{equation}
Let $P_{\theta,t}\left(\left(x,Y\right),\left(x',Y'\right)\right)$ be 
the smooth kernel for $\exp\left(-\mathfrak M_{\theta,t}\right)$ with 
respect to $\frac{dx'dY'}{\left(2\pi\right)^{2n}}$. By 
(\ref{eq:nank1}), as in (\ref{eq:form1x1}), we get
\begin{multline}\label{eq:nak2a1}
\ch\left(\mathcal{A}^{\prime \prime }_{Y}, 
\omega^{X}/t,g^{D},g^{\widehat{TX}}/t^{3},\theta\right)\\
=\varphi\int_{X}^{}\left[\int_{\widehat{T_{\R,x}X}}^{}\Trs\left[P_{\theta,t}\left(
\left(x,Y\right),\left(x,Y\right)\right)\right]\frac{dY}{\left(2\pi\right)^{n}}\right]
\frac{\widehat{dx}}{\left(2\pi\right)^{n}}.
\end{multline}
Put
\begin{equation}\label{eq:form2r1}
m_{\theta,t}\left(x\right)=\varphi \frac{1}{\left(2\pi\right)^{n}} \int_{\widehat{T_{\R,x}X}}^{}
\Trs\left[P_{\theta,t}\left( \left(x,Y\right),\left(x,Y\right) \right) 
\right]\frac{dY}{\left(2\pi\right)^{n}},
\end{equation}
so that
\begin{equation}\label{eq:form3}
\ch\left(\mathcal{A}''_{Y},\omega^{X}/t,g^{D},g^{\widehat{TX}}/t^{3},\theta\right)=
\int_{X}^{}m_{\theta,t}\left(x\right)\widehat{dx}.
\end{equation}

We fix $d>0$. In (\ref{eq:form3}), we replace $\theta$ by $\theta_{t}$. As explained before, 
given $x\in  X$, 
 Theorem \ref{thm:Tunif} takes care of localization as $t\to 0$, and 
 this  uniformly in $\left(s,x\right)\in M$. If instead $\theta_{t}$ 
 is replaced by $\underline{\theta}_{t}$, i.e., we substitute $d$ by 
 $dt^{\rho}$ in $\theta$, since $\rho\in ]0,3[$, Theorem 
 \ref{thm:Tunif} also takes care of the localization.
 
 To proceed, we make exactly the same choice of coordinates and the 
 same choice of trivialization as in the proof of Theorem 
 \ref{thm:sepr}, as well as the same rescalings and conjugations. We introduce 
 corresponding rescaled or conjugated operators with the extra subscript 
 $\theta_{t}$, or $\underline{\theta}_{t}$. We construct an operator 
 $\mathfrak N_{\theta,t,x}$ on $\mathbb H_{x}$ exactly as in the proof of Theorem 
 \ref{thm:sepr}, that coincides with $\mathfrak M_{\theta,t}$ on 
 $B^{T_{\R,x}X}\left(0,\epsilon/2\right)\times \widehat{T_{\R,x}X}$. 
 Let $Q_{\theta,t,x}\left(\left(U,Y\right),\left(U',Y'\right)\right)$ 
 be the smooth kernel for $\exp\left(-\mathfrak 
 N_{\theta,t,x}\right)$ with respect to the volume 
 $\frac{dU'dY'}{\left(2\pi\right)^{2n}}$. When studying the 
 asymptotics of (\ref{eq:form3}) as $t\to 0$, we may as well 
 replace $m_{\theta,t}\left(x\right)$ by
 \begin{equation}\label{eq:fatR1}
n_{\theta,t}\left(x\right)=\varphi\frac{1}{\left(2\pi\right)^{n}}
\int_{\widehat{T_{\R,x}X}}^{}\Trs\left[Q_{\theta,t,x}\left(\left(0,Y\right),\left(0,Y\right)\right)\right]\frac{dY}{\left(2\pi\right)^{n}}.
\end{equation}

 Recall that 
 \index{Po@$\mathfrak P_{0,s,x}$}%
 $\mathfrak P_{0,s,x}$ is given by (\ref{eq:froid4x1}). Given $d\ge 0$, put
 \index{Pod@$ \mathfrak P_{0,d,x}$}%
 \begin{multline}\label{eq:wolf3}
      \mathfrak P_{0,d,s,x}=\mathfrak P_{0,s,x}     	  -\frac{1}{2}\left\vert  
      Y\right\vert^{2}_{g^{\widehat{TX}}}q^{*}\overline{\pa}^{X}\pa^{X}id\omega^{X}
      -\widehat{y}_{*}q^{*}\pa^{X}id\omega^{X}\\
       +\left( q^{*}\overline{\pa}^{X}id\omega^{X} \right) 
      i_{\overline{\widehat{y}}}
	   +q^{*}id\omega^{X}\left( 
	   \n_{\overline{\widehat{y}}}+
      \overline{\widehat{w}}^{i}i_{\overline{\widehat{w}}_{i}} 
	  \right),
      \end{multline}
the tensors in (\ref{eq:wolf3}) being evaluated at $\left(s,x\right)$.

	 We claim that as $t\to 0$,  
	 instead of (\ref{eq:froid5}), we have
	 \begin{align}\label{eq:gnan1}
&\mathfrak P_{\theta_{t},t,x}\to \mathfrak P_{0,d,s,x},
&\mathfrak P_{\underline{\theta}_{t},t,x}\to \mathfrak P_{0,s,x}.
\end{align}
To establish (\ref{eq:gnan1}), we will refer to the proof of 
(\ref{eq:froid5}), and also to equation (\ref{eq:clip4}), with $c=1$, 
and $d$ replaced by $dt$.

In the proof of (\ref{eq:froid5}), we made essential use of the fact 
$\overline{\pa}^{X}\pa^{X}i\omega^{X}=0$. In the sequel, we will view 
$\mathfrak P_{0,s,x}$ as the contribution to the limit of  the terms 
in $\mathfrak M_{t}$ in (\ref{eq:curbhypbis1}), with the exception of the term 
$-q^{*}\overline{\pa}^{X}\pa^{X}\omega^{X}/t$. However, in equation 
(\ref{eq:clip4}) where we replace $\theta$ by $\theta_{t}$, so that 
$c=1$, this term 
goes away.  Since the other terms in (\ref{eq:clip4}) do not contain 
operators like $i_{w_{i}}, i_{\overline{w}_{i}}$, they are  
unaffected by the conjugation by $\exp\left(-\mathfrak 
o/t^{3/2}\right)$. Therefore, while replacing $d$ by $dt$ in the 
right-hand side of (\ref{eq:clip4}), we can take their naive limit as 
$t\to 0$, and we get (\ref{eq:gnan1}). We can instead use equation 
(\ref{eq:curbhypbis1x1}) for $\mathfrak M_{\theta_{t},t}$ and obtain 
(\ref{eq:gnan1}).

 Let 
$S_{\theta_{t},t,x}\left(\left(U,Y\right),\left(U',Y'\right)\right)$ be 
the smooth kernel associated with $\exp\left(-\mathfrak 
P_{\theta_{t},t}\right)$ with respect to the volume 
$\frac{dU'dY'}{\left(2\pi\right)^{2n}}$. We claim that we can use the same arguments 
as in  \cite[eq. 
(11.6.9)]{Bismut10b} to control $\left\vert 
S_{\theta_{t},t,x}\left(\left(0,Y\right),\left(0,Y\right)\right) 
\right\vert $ when $t\in ]0,1]$. Indeed when comparing equation 
(\ref{eq:curbhypbis1x1}) for $\mathfrak M_{\theta_{t},t}$ with the 
corresponding equation \cite[eq. (11.5.1)]{Bismut10b}, we see that 
the main difference lies in the presence of the term $t^{3/2}i_{Y}C$. However note 
that given $c>0$, there is $c'>0$ such that for $a\in\R$, 
\begin{equation}\label{eq:est1}
\left\vert  a\right\vert\le c'+c a^{4}.
\end{equation}
By (\ref{eq:est1}), we deduce that
\begin{equation}\label{eq:est2}
\left\vert  t^{3/2}Y\right\vert_{g^{\widehat{TX}}}\le c'+ct^{6}\left\vert  
Y\right\vert^{4}_{g^{\widehat{TX}}}.
\end{equation}
By (\ref{eq:est2}), we deduce that given $\rho\in \left[0,3\right], d_{0}>0$, if $d\ge 
d_{0}t^{\rho}$, then
\begin{equation}\label{eq:est3}
\left\vert  t^{3/2}Y\right\vert_{g^{\widehat{TX}}}\le c'+cdt^{3}\left\vert  
Y\right\vert^{4}_{g^{\widehat{TX}}}.
\end{equation}

Using (\ref{eq:curbhypbis1x1}), (\ref{eq:est3}), the same arguments as in \cite[eq. 
(11.6.9)]{Bismut10b} show that
for  $d_{0}>0,\rho\in [0,3[$, there exist 
$c>0,C>0$ such that
for 
$ t\in ]0,1], 1\ge d\ge d_{0}
t^{\rho},x\in X, Y\in \widehat{T_{\R,x}X}$, \footnote{In 
\cite{Bismut10b}, the condition $d\le 1$ was omitted. This is an 
irrelevant oversight.}
\begin{equation}\label{eq:parjer14}
\left\vert S_{\theta_{t},t,x}\left(\left(0,Y\right),\left(0,Y\right)\right) \right\vert 
\le C
\exp\left(-c d_{0}\left\vert  
Y\right\vert_{g^{\widehat{TX}}}^{2\left(1-\rho/3\right)}\right).
\end{equation}
The estimate (\ref{eq:parjer14}) also gives a corresponding estimate 
for 
$S_{\underline{\theta}_{t},t}\left(\left(0,Y\right),\left(0,Y\right)\right)$.

Recall that the kernel 
\index{Sx@$S_{0,s,x}\left(\left(U,Y\right), \left(U',Y'\right)\right)$}%
$S_{0,s,x}\left(\left(U,Y\right), \left(U',Y'\right)\right)$ for 
$\exp\left(-\mathfrak P_{0,s,x}\right)$
was defined in the proof of Theorem \ref{thm:sepr}. Let $S_{0,d,s,x}\left(\left(U,Y\right),\left(U',Y'\right)\right)$ be the 
smooth kernel associated with $\exp\left(-\mathfrak P_{0,d,s,x}\right)$ 
with respect to the volume $\frac{dU'dY'}{\left(2\pi\right)^{2n}}$. 
Using the uniform estimates (\ref{eq:parjer14}) and proceeding as in 
the proof of
\cite[Theorem 11.6.1]{Bismut10b}, we find that for $\rho\in 
]0,3[$,  as $t\to 0$, 
\begin{align}\label{eq:est4}
&S_{\theta_{t},t,x}\left(\left(U,Y\right),\left(U',Y'\right)\right)\to 
S_{0,d,s,x}\left(\left(U,Y\right),\left(U',Y'\right)\right),\\
&S_{\underline{\theta}_{t},t,x}\left(\left(U,Y\right),\left(U',Y'\right)\right)\to 
S_{0,s,x}\left(\left(U,Y\right),\left(U',Y'\right)\right). \nonumber 
\end{align}

By proceeding as in the proof of (\ref{eq:dee11}), using 
(\ref{eq:est4}), we find  that as 
$t\to 0$, 
\begin{align}\label{eq:est5}
&n_{\theta_{t},t}\left(x\right)\to n_{0,d,s}\left(x\right) \nonumber \\
&=
\varphi\frac{1}{\left(2\pi\right)^{n}}
\int_{\widehat{T_{\R,x}X}}^{}\Trs^{\Lambda\left(\overline{\widehat{T^{*}_{x}X}}\right)\ho D_{s,x}}\left[\left[S_{0,d,s,x}\left(\left(0,Y\right),\left(0,Y\right)\right)1\right]^{\max}\right]\frac{dY}{\left(2\pi\right)^{n}}, \\
&n_{\underline{\theta}_{t},t}\left(x\right)\to n_{0,s}\left(x\right) 
\nonumber \\
&=
\varphi\frac{1}{\left(2\pi\right)^{n}}
\int_{\widehat{T_{\R,x}X}}^{}\Trs^{\Lambda\left(\overline{\widehat{T^{*}_{x}X}}\right)\ho D_{s,x}}\left[\left[S_{0,s,x}\left(\left(0,Y\right),\left(0,Y\right)\right)1\right]^{\max}\right]\frac{dY}{\left(2\pi\right)^{n}}. \nonumber 
\end{align}
By (\ref{eq:parjer14}),  given $\rho\in ]0,3[$, there exists $C>0$ such that for any $x\in X$, 
\begin{align}\label{eq:est6}
&\left\vert  n_{\theta_{t},t}\left(x\right)\right\vert\le C,
&\left\vert  n_{\underline{\theta}_{t},t}\left(x\right)\right\vert\le 
C.
\end{align}

By (\ref{eq:form3}), (\ref{eq:est5}), and (\ref{eq:est6}),  as $t\to 0$, 
\begin{align}\label{eq:est6a1}
&\ch\left(A^{\prime \prime 
}_{Y},\omega^{X}/t,g^{D},g^{\widehat{TX}}/t^{3},\theta_{t}\right)\to 
\int_{X}^{}n_{0,d,s}\left(x\right)\widehat{dx},\\
&\ch\left(A^{\prime \prime 
}_{Y},\omega^{X}/t,g^{D},g^{\widehat{TX}}/t^{3},\theta_{t}\right)\to 
\int_{X}^{}n_{0,s}\left(x\right)\widehat{dx}, \nonumber 
\end{align}
By (\ref{eq:12a6}), (\ref{eq:est6a1}), we get the second 
equation in (\ref{eq:co1bi}).

To compute the right-hand side of the first equation in 
(\ref{eq:est6a1}), we proceed as in  the proof of \cite[Theorem 
11.6.1]{Bismut10b}. Consider equations (\ref{eq:froid4x1}), 
(\ref{eq:wolf3}) that give a formula for $\mathfrak P_{0,d,s,x}$. When 
$A^{E_{0}}$ is the trivial superconnection on $\C$, this is exactly 
the operator obtained in \cite[eq. (11.6.5)]{Bismut10b}. As explained 
in detail in \cite{Bismut10b}, in this case, if we  replace formally
the variables $\mathfrak w_{i}, \overline{\mathfrak w}_{i}$ by 
variables $i_{w_{i}}, i_{\overline{w}_{i}}$ that act on an extra copy 
of
$\Lambda\left(T^{*}_{\C,x}X\right)$, the operator $\mathfrak P_{0,d,s,x}$ 
coincides with the operator $\mathcal{L}_{d,Y}$\footnote{In 
\cite{Bismut10b}, because of the presence of an extra parameter 
$b>0$, this operator is denoted $\mathcal{L}_{d,Y_{b}}$.} considered in 
\cite[eq. (10.2.4)]{Bismut10b} that is associated with the vector 
bundle $TX$, the metric $\widehat{g}^{TX}$, the $(1,1)$-form 
$\omega^{X}$, and the section $Y$. When $A^{E_{0}}$ is non trivial, 
in $\mathfrak P_{0,d,s,x}$, we have the extra term $A^{E_{0},2}_{s,x}$.

Using the sign conventions 
explained in \cite[eq. (10.2.2)]{Bismut10b}, that match 
(\ref{eq:dee12a1}), and using the notation in Subsection 
\ref{subsec:exo},  for $d>0$, we get
\begin{equation}\label{eq:est8}
\int_{X}^{}n_{0,d}\left(x\right)\widehat{dx}=p_{*}\left[q^{*}a\left(TX,d\omega^{X},\widehat{g}^{TX}\right)\ch\left(A^{E_{0} \prime \prime },g^{D}\right)\right]
\end{equation}
By (\ref{eq:est6a1}), (\ref{eq:est8}), we obtain the first equation 
in (\ref{eq:co1bi}). The proof of our theorem is completed. 
\end{proof}
\begin{theorem}\label{thm:tafin}
	The following identity holds:
	\begin{equation}\label{eq:nank2}
		\ch_{\mathrm{BC}}\left(\mathcal{A}''_{Y}\right)=p_{*}\left[q^{*}\Td_{\mathrm{BC}}\left(TX\right)\ch_{\mathrm{BC}}\left(A^{E \prime \prime }\right)\right]\,\mathrm{in}\,H^{(=)}_{\mathrm{BC}}\left(S,\R\right).
\end{equation}
\end{theorem}
\begin{proof}
	Using Theorem \ref{thm:Tdoublebibi} and either (\ref{eq:exo2}) 
	combined with the first equation in (\ref{eq:co1bi}), or  the second 
	row in (\ref{eq:co1bi}),  we get (\ref{eq:nank2}). The proof of our theorem is completed. 
\end{proof}
\subsection{A proof of Theorem \ref{thm:sub}}%
\label{subsec:finsu}
By combining Theorems \ref{thm:fun1} and  \ref{thm:tafin}, we get 
Theorem \ref{thm:sub}, when $\mathscr F= \mathscr E$. As explained in 
Subsection \ref{subsec:rebl}, this is enough to establish 
Theorem \ref{thm:sub} in full generality. This also completes the 
proof of our main result stated in Theorem \ref{thm:rrg}.

\printindex

\bibliography{Bismut,Others}
\end{document}